\documentclass[mn-ohne,fleqn]{w-art}        
\usepackage{times}
\usepackage{w-thm}
 \numberwithin{equation}{section}
  \usepackage[ngerman,english]{babel} 
 \usepackage{amssymb}

 \makeatletter

\newenvironment{proofc}[2][\indent\emph{P\,r\,o\,o\,f}]{\par
  \pushQED{\qed}%
  \normalfont \topsep0\p@\@plus6\p@\relax
  \trivlist
  \item[\hskip\labelsep
        \itshape
    #1\ {\rm#2}\@addpunct{.}]\ignorespaces
}{%
  \popQED\endtrivlist\@endpefalse
}

 \def\BlackBoxes{\global\overfullrule5\p@}
 \BlackBoxes

\@abstractcopyrightfalse

  \def\@oddfoot{%
    \if@finallayout%
      \normalfont\fontsize{8}{9\p@}\normalfont%
      \bfseries\journal@url \hfil
      \normalfont\fontsize{6}{7\p@}\normalfont%
      \copyright\ \copyrightyear\ \copyrightholder
    \else
      \normalfont\fontsize{6}{7\p@}\normalfont%
      \hfil \phantom{Copyright line will be provided by the publisher} 
    \fi}%
  \def\@evenfoot{%
    \if@finallayout%
      \normalfont\fontsize{6}{7\p@}\normalfont%
      \copyright\ \copyrightyear\ \copyrightholder\hfil
      \normalfont\fontsize{8}{9\p@}\normalfont%
      \bfseries\journal@url
    \else
      \normalfont\fontsize{6}{7\p@}\normalfont%
      \hfil \phantom{Copyright line will be provided by the publisher} 
    \fi}%

 \makeatother
\def\polhk#1{\setbox0=\hbox{#1}{\ooalign{\hidewidth
  \lower1.5ex\hbox{`}\hidewidth\crcr\unhbox0}}}
 \makeatletter
 \def\BlackBoxes{\global\overfullrule5\p@}
 \makeatother
 \BlackBoxes
 \newcommand*{\beq}{$$\refstepcounter{equation}}
 \newcommand*{\eeq}{\eqno(\theequation)$$}
 \newcommand*{\bspezeq}{$$}
 \newcommand*{\espezeq}{$$}
 \newcommand*{\bspezeqn}{$$\refstepcounter{equation}}
 \newcommand*{\espezeqn}{\eqno(\theequation)$$}
 \newcommand*{\bal}{\begin{aligned}}
 \newcommand*{\eal}{\end{aligned}}
 \newcommand*{\qa}{,\qquad}
 \newcommand*{\qb}{,\quad}
 \newcommand*{\mf}[1]{\boldsymbol{#1}} 
 \newcommand*{\ci}{\mathaccent"7017 }                    
 \newcommand*{\cqqncel}[2]{\ooalign{$\hfil#1\mkern1mu/\hfil$\crcr$#1#2$}}
 \newcommand*{\notsim}{\mathrel{\mathpalette\cqqncel\sim}}
 \newcommand*{\hb}[1]{\hbox{$#1$}} 
 \newcommand*{\sdot}{\!\cdot\!}
 \newcommand*{\sco}{\kern2pt\colon\kern2pt}
 \newcommand*{\slt}{\hbox{$\ltimes$}}
 \newcommand*{\sn}{\kern1pt|\kern1pt}
 \newcommand*{\bsn}{\kern1pt\big|\kern1pt}
 \newcommand*{\ssm}{\!\setminus\!}
 \newcommand*{\esdot}{\hbox{$[{}\sdot{}]$}}
 
 \newcommand*{\vsdot}{\hbox{$|\sdot|$}}
 \newcommand*{\Vsdot}{\hbox{$\|\sdot\|$}}
 \newcommand*{\Vvsdot}{\hbox{$\vd\sdot\vd$}}
 \newcommand*{\hh}[1]{{\textbf{#1}}}
 \newcommand*{\npb}{\postdisplaypenalty=10000}
 \newcommand*{\po}{\pagebreak[2] }
 \newcommand*{\pe}{\hbox{$[{}\sdot{},{}\sdot{}]$}}
 \newcommand*{\pg}{\hbox{$\{\cdot,\cdot\}$}}
 \newcommand*{\pr}{\hbox{$(\cdot,\cdot)$}}
 \newcommand*{\prsn}{\hbox{$(\cdot\sn\cdot)$}}
 \newcommand*{\pw}{\hbox{$\dl{}\sdot{},{}\sdot{}\dr$}}
 \newcommand*{\mfpw}{\hbox{$\mf{\dl}\!\!\!\mf{\dl}{}\sdot{},{}\sdot{}\mf{\dr}\!\!\!\mf{\dr}$}}
 \newcommand*{\mfdl}{\hbox{$\mf{\dl}\!\!\!\mf{\dl}$}}
 \newcommand*{\mfdr}{\hbox{$\mf{\dr}\!\!\!\mf{\dr}$}}
 \newcommand*{\dlim}[2]{\substack{{#1}\\{#2}}}
 \newcommand*{\bid}[2]{\Bigl({{#1}\atop{#2}}\Bigr)}
 \newcommand*{\ol}{\overline}
 \newcommand*{\ora}{\overrightarrow}     
 \newcommand*{\wh}{\widehat}
 \newcommand*{\wt}{\widetilde}
 \newsavebox{\Ptop}
 \sbox{\Ptop}{\begin{picture}(2,2)(-1,-1)\put(0,0){\circle*{2}}\end{picture}}
 \newcommand*{\thdot}[3]{\vbox{\ialign{##\crcr    
                    {\kern#2ex\raise#3ex\hbox{$\scriptscriptstyle\bt$}}\crcr\noalign{\kern-1pt\nointerlineskip}
                    $\hfil\displaystyle{#1}\hfil$\crcr}}}
 \newcommand*{\thB}{\thdot{B}{.7}{.4}\vph{B}}
 \newcommand*{\thmfB}{\thdot{\mf{B}}{.8}{.4}\vph{B}}
 \newcommand*{\thgK}{\thdot{\gK}{.5}{.4}}
 \newcommand*{\thV}{\thdot{V}{.5}{.4}}
 \newcommand*{\ithV}{{\thdot{\scriptstyle V}{.5}{.4}}}
 \newcommand*{\thW}{\thdot{W}{.9}{.4}}
 \newcommand*{\ithW}{{\thdot{\scriptstyle W}{.65}{.4}}}
 
 \newcommand*{\thg}{\thdot{g}{.5}{.4}}
 \newcommand*{\ithg}{{\thdot{\scriptstyle g}{.5}{.4}}}
 \newcommand*{\thh}{\thdot{h}{.5}{.4}}
 \newcommand*{\thchi}{\thdot{\chi}{.5}{.4}}
 \newcommand*{\thia}{\rlap{\thdot{\ia}{.5}{.4}}\ph{\ia}}
 \newcommand*{\thka}{\rlap{\thdot{\ka}{.5}{.4}}{\ph{\ka}}}                           
 \newcommand*{\ithka}{{\thdot{\scriptstyle\ka}{.5}{.4}}}
 \newcommand*{\ithwtka}{{\thdot{\scriptstyle\wt{\ka}}{.5}{.4}}}
 \newcommand*{\thna}{\rlap{\thdot{\na}{.5}{.4}}{\ph{\na}}}
 \newcommand*{\thpi}{\thdot{\pi}{.5}{.4}}
 \newcommand*{\thvp}{\rlap{\thdot{\vp}{.5}{.4}}\ph{\vp}}
 \newcommand*{\thpsi}{\rlap{\thdot{\psi}{.5}{.4}}\ph{\psi}}
 \newcommand*{\thrho}{\thdot{\rho}{.5}{.4}}
 \newcommand*{\ph}{\phantom}
 \newcommand*{\vph}{\vphantom}
 \newcommand*{\spcheckf}{\slap{\mbox{$f$}}\mbox{\kern.4ex\raisebox{1.6ex}{$\checkp$}}}
 \newcommand*{\starB}{\slap{\slap{\mbox{$B$}}\mbox{\kern.6ex\raisebox{1.6ex}{$*$}}}{\mbox{$\phantom B$}}} 
 \newcommand*{\hr}{\hookrightarrow}
 \newcommand*{\lllllora}{\relbar\joinrel\relbar\joinrel\relbar\joinrel\relbar\joinrel\longrightarrow}
 \newcommand*{\Ra}{\mathrel{\smash=}\mathrel{\mkern-11mu}\Rightarrow}
 \newcommand*{\ra}{\rightarrow}
 \newcommand*{\dl}{\langle}
 \newcommand*{\dr}{\rangle}
 \newcommand*{\sdh}{\stackrel{d}{\hookrightarrow}}
 \newcommand{\vd}{\hbox{$\vert\kern-.2ex\vert\kern-.2ex\vert$}}
 \newcommand*{\slap}[1]{\hspace{0pt}\hbox to 0pt{#1\hss}}
 \newcommand*{\bdry}{\mathop{\rm bdry}\nolimits}
 \newcommand*{\card}{{\rm card}}
 \newcommand*{\dom}{\mathop{\rm dom}\nolimits}
 \newcommand*{\ev}{\mathop{\rm ev}\nolimits}
 \newcommand*{\id}{{\rm id}}
 \newcommand*{\im}{\mathop{\rm im}\nolimits}
 \newcommand*{\inv}{{\rm inv}}
 \newcommand*{\norm}{{\rm norm}}
 \newcommand*{\pro}{{\rm pr}}
 \newcommand*{\surj}{{\rm surj}}
 \newcommand*{\supp}{\mathop{\rm supp}\nolimits}
 \newcommand*{\tr}{{\gst\gsr}}
 \newcommand*{\unif}{{\rm unif}}
 \newcommand*{\Hom}{{\rm Hom}}
 \newcommand*{\Laut}{{\mathcal L}{\rm aut}}
 \newcommand*{\Lis}{{\mathcal L}{\rm is}}
 \newcommand*{\loc}{{\rm loc}}
 \renewcommand*{\Im}{\mathop{\rm Im}\nolimits}
 \renewcommand*{\Re}{\mathop{\rm Re}\nolimits}
 \newcommand*{\imi}{{\mit i\kern1pt}}
 \newcommand*{\is}{\subset}
 \newcommand*{\isis}{\Subset} 
 \newcommand*{\Chrz}{\Ga_{\kern-1pt i\mu}^\nu}
 \newcommand*{\bt}{\bullet}
 \newcommand*{\es}{\emptyset}
 \newcommand*{\iy}{\infty}
 \newcommand*{\mt}{\mapsto}
 \newcommand*{\nat}{\natural}
 \newcommand*{\naX}{\na_{\cona X}}
 \newcommand*{\pl}{\partial}
 \newcommand*{\pa}{\partial^\alpha}
 \newcommand*{\sh}{\sharp}
 \newcommand*{\tri}{\triangle}
 \newcommand*{\cof}{\kern-2pt}
 \newcommand*{\coLda}{\kern-1pt}
 \newcommand*{\cona}{\kern-1pt}
 \newcommand*{\coU}{\kern-1pt}
 \newcommand*{\coV}{\kern-1pt}
 \newcommand*{\coW}{\kern-1pt}
 \newcommand*{\coX}{\kern-1pt}
 \newcommand*{\cimfgF}{\rlap{${\ci{\mf{\gF}}}$}{\ph{\mf{\gF}}}} 
 \newcommand*{\cimfW}{\rlap{${\ci{\mf{W}}}$}{\ph{\mf{W}}}} 
 \newcommand*{\wtmfW}{\rlap{${\wt{\mf{W}}}$}{\ph{\mf{W}}}} 
 \newcommand*{\al}{\alpha}
 \newcommand*{\ba}{\beta}
 \newcommand*{\da}{\delta}
 \newcommand*{\Ga}{\Gamma}
 \newcommand*{\ga}{\gamma}
 \newcommand*{\mfga}{\rlap{$\mf{\ga}$}{\kern.3pt\mf{\ga}}} 
 \newcommand*{\ia}{\iota}
 \newcommand*{\ka}{\kappa}
 \newcommand*{\Lda}{\Lambda}
 \newcommand*{\lda}{\lambda}
 \newcommand*{\na}{\nabla}
 \newcommand*{\Om}{\Omega}
 \newcommand*{\om}{\omega}
 \newcommand*{\Sa}{\Sigma}
 \newcommand*{\sa}{\sigma}
 \newcommand*{\Ta}{\Theta}
 \newcommand*{\ta}{\theta}
 \newcommand*{\za}{\zeta}
 \newcommand*{\ve}{\varepsilon}
 \newcommand*{\vp}{\varphi}
 \newcommand*{\vt}{\vartheta}
 
 \newcommand*{\DiDj}{(D_i,D_j)}
 \newcommand*{\EeEz}{(E_1,E_2)}
 \newcommand*{\GaG}{(\Ga,G)}
 \newcommand*{\GaGi}{(\Ga,G_i)}
 \newcommand*{\GanGi}{(\Ga_0,G_i)}
 \newcommand*{\GaJ}{(\Ga,J)}
 \newcommand*{\GathV}{(\Ga,\thV)}
 \newcommand*{\GathW}{(\Ga,\thW)}
 \newcommand*{\Hmgm}{(\BH^m,g_m)}
 \newcommand*{\HeHz}{(H_1,H_2)}
 \newcommand*{\HzHe}{(H_2,H_1)}
 \newcommand*{\HzsHes}{(H_2',H_1')}
 \newcommand*{\HK}{(H,\BK)}
 \newcommand*{\IV}{(I,V)}
 \newcommand*{\IthV}{(I,\thV)}
 \newcommand*{\IcX}{(I,\cX)}
 \newcommand*{\JB}{(J,B)}
 
 \newcommand*{\JcD}{(J,\cD)}
 \newcommand*{\mJcD}{(-J,\cD)}
 \newcommand*{\JTcD}{(J_T,\cD)}
 \newcommand*{\JcDka}{(J,\cD_\ka)}
 \newcommand*{\ciJcicD}{(\ci J,\ci\cD)}
 \newcommand*{\ciJcicDka}{(\ci J,\ci\cD_\ka)}
 
 \newcommand*{\ciJciM}{(\ci J,\ci M)}
 \newcommand*{\JV}{(J,V)}
 \newcommand*{\ciJV}{(\ci J,V)}
 \newcommand*{\ciJVs}{(\ci J,V')}
 \newcommand*{\JthV}{(J,\thV)}
 \newcommand*{\JVs}{(J,V')}
 \newcommand*{\JVe}{(J,V_1)}
 \newcommand*{\JVee}{(J,V_{11})}
 \newcommand*{\JVes}{(J,V_1')}
 \newcommand*{\JVez}{(J,V_{12})}
 \newcommand*{\JVn}{(J,V_0)}
 \newcommand*{\JVnn}{(J,V_0^0)}
 \newcommand*{\JVz}{(J,V_2)}
 \newcommand*{\JVze}{(J,V_{21})}
 \newcommand*{\JVzz}{(J,V_{22})}
 \newcommand*{\JVzs}{(J,V_2')}
 \newcommand*{\JVzd}{(J,V_{23})}
 \newcommand*{\JVd}{(J,V_3)}
 \newcommand*{\JVis}{(J,V_{\coV i}')}
 \newcommand*{\JX}{(J,X)}
 \newcommand*{\JXs}{(J,X')}
 \newcommand*{\JcX}{(J,\cX)}
 \newcommand*{\JTcX}{(J_T,\cX)}
 
 \newcommand*{\ME}{(M,E)}
 \newcommand*{\Mg}{(M,g)}
 \newcommand*{\MBK}{(M,\BK)}
 \newcommand*{\MTM}{(M,TM)}
 \newcommand*{\MTsM}{(M,T^*M)}
 \newcommand*{\MV}{(M,V)}
 \newcommand*{\MVn}{(M,V_0)}
 \newcommand*{\MVe}{(M,V_1)}
 \newcommand*{\MVee}{(M,V_{11})}
 \newcommand*{\MVez}{(M,V_{12})}
 \newcommand*{\MVz}{(M,V_2)}
 \newcommand*{\MVzs}{(M,V_2')}
 \newcommand*{\MVze}{(M,V_{21})}
 \newcommand*{\MVd}{(M,V_3)}
 \newcommand*{\MVi}{(M,V_{\coV i})}
 \newcommand*{\MVj}{(M,V_{\coV j})}
 \newcommand*{\MVs}{(M,V^*)}
 \newcommand*{\ciMV}{(\ci M,V)}
 \newcommand*{\ciMVs}{(\ci M,V')}
 \newcommand*{\MW}{(M,W)}
 \newcommand*{\MWed}{(M,W_{13})}
 
 \newcommand*{\MWzs}{(M,W_2')}
 \newcommand*{\MWi}{(M,W_{\coW i})}
 \newcommand*{\MWij}{(M,W_{\coW ij})}
 
 \newcommand*{\MWa}{(M,W')}
 \newcommand*{\MWs}{(M,W^*)}
 \newcommand*{\BRcD}{(\BR,\cD)}
 \newcommand*{\BRpcD}{(\BR^+,\cD)}
 \newcommand*{\BRdcX}{(\BR^d,\cX)}
 
 \newcommand*{\Rmgm}{(\BR^m,g_m)}
 \newcommand*{\RmE}{(\BR^m,E)}
 
 \newcommand*{\Rp}{(\BR^+)}
 \newcommand*{\BRcX}{(\BR,\cX)}
 \newcommand*{\BRpcX}{(\BR^+,\cX)}
 \newcommand*{\mBRpcX}{(-\BR^+,\cX)}
 
 \newcommand*{\BRV}{(\BR,V)}
 \newcommand*{\BRpV}{(\BR^+,V)}
 \newcommand*{\mBRpV}{(-\BR^+,V)}
 \newcommand*{\BRpVi}{(\BR^+,V_{\coV i})}
 \newcommand*{\rgK}{(\rho,\gK)}
 \newcommand*{\wrgK}{(\wt{\rho},\wt{\gK})}
 \newcommand*{\SE}{(S,E)}
 \newcommand*{\SK}{(S,\BK)}
 \newcommand*{\SV}{(S,V)}
 \newcommand*{\UKmm}{(U,\BK^{m\times m})}
 \newcommand*{\UV}{(U,V)}
 \newcommand*{\uv}{(u,v)}

 \newcommand*{\UVi}{(U,V_{\coV i})}

 \newcommand*{\Vh}{(V,h)}
 \newcommand*{\Vihi}{(V_{\coV i}, h_i)}
 \newcommand*{\VV}{(V,V)}
 \newcommand*{\vv}{(v,v)}
 \newcommand*{\vsv}{(v^*,v)}
 \newcommand*{\vvs}{(v,v^*)}
 \newcommand*{\VeVz}{(V_1,V_2)}
 \newcommand*{\VzVe}{(V_2,V_1)}
 \newcommand*{\vw}{(v,w)}
 \newcommand*{\vsws}{(v^*,w^*)}
 \newcommand*{\thWGi}{(\thW,G_i)} 
 \newcommand*{\wv}{(w,v)}
 
 \newcommand*{\WiWj}{(W_{\coW i},W_{\coW j})}
 \newcommand*{\BXE}{(\BX,E)}
 
 \newcommand*{\BXEe}{(\BX,E_1)}
 \newcommand*{\BXeBEe}{(\BX_1,\BE_1)}
 \newcommand*{\BXEn}{(\BX,E_0)}
 \newcommand*{\BXnBEn}{(\BX_0,\BE_0)}
 \newcommand*{\BXEz}{(\BX,E_2)}
 \newcommand*{\BXiBEi}{(\BX_i,\BE_i)}
 \newcommand*{\BXiEka}{(\BX_i,E_\ka)}
 \newcommand*{\BXkE}{(\BX_\ka,E)}
 \newcommand*{\ciBXE}{(\ci\BX,E)}
 \newcommand*{\ciBXkE}{(\ci\BX_\ka,E)}
 \newcommand*{\BXBF}{(\BX,\BF)}
 
 \newcommand*{\BXcX}{(\BX,\cX)}
 \newcommand*{\XE}{(X,E)}
 
 \newcommand*{\XnXe}{(X_0,X_1)}
 \newcommand*{\Xv}{(X,v)}
 \newcommand*{\Xw}{(X,w)}
 \newcommand*{\XcX}{(X,\cX)}
 \newcommand*{\cXBK}{(\cX,\BK)}
 \newcommand*{\cXcX}{(\cX,\cX)}
 \newcommand*{\cXncXe}{(\cX_0,\cX_1)}
 \newcommand*{\cXcY}{(\cX,\cY)}
 \newcommand*{\YE}{(Y,E)}
 \newcommand*{\YkE}{(Y_\ka,E)}
 \newcommand*{\BYE}{(\BY,E)}
 \newcommand*{\BYEe}{(\BY,E_1)}
 \newcommand*{\BYEn}{(\BY,E_0)}
 \newcommand*{\BYEz}{(\BY,E_2)}
 \newcommand*{\BYEs}{(\BY,E')}
 \newcommand*{\BYkE}{(\BY_\ka,E)}
 \newcommand*{\BYkEe}{(\BY_\ka,E_1)}
 \newcommand*{\BYkEn}{(\BY_\ka,E_0)}
 \newcommand*{\BYkEz}{(\BY_\ka,E_2)}
 \newcommand*{\ciBYE}{(\ci\BY,E)}
 \newcommand*{\ciBYkE}{(\ci\BY_\ka,E)}
 \newcommand*{\YBK}{(Y,\BK)}
 
 \newcommand*{\cYcX}{(\cY,\cX)}
 \newcommand*{\BUC}{BU\kern-.3ex C}
 \newcommand*{\BC}{{\mathbb C}}
 \newcommand*{\BE}{{\mathbb E}}
 \newcommand*{\BF}{{\mathbb F}}
 \newcommand*{\BH}{{\mathbb H}}
 \newcommand*{\BJ}{{\mathbb J}}
 \newcommand*{\BK}{{\mathbb K}}
 \newcommand*{\BM}{{\mathbb M}}
 \newcommand*{\BN}{{\mathbb N}}
 \newcommand*{\BR}{{\mathbb R}}
 \newcommand*{\BX}{{\mathbb X}}
 \newcommand*{\BY}{{\mathbb Y}}
 \newcommand*{\BZ}{{\mathbb Z}}
 \newcommand*{\cA}{{\mathcal A}}
 \newcommand*{\cB}{{\mathcal B}}
 \newcommand*{\cC}{{\mathcal C}}
 \newcommand*{\cD}{{\mathcal D}}
 \newcommand*{\cF}{{\mathcal F}}
 \newcommand*{\cG}{{\mathcal G}}
 \newcommand*{\cL}{{\mathcal L}}
 \newcommand*{\cQ}{{\mathcal Q}}
 \newcommand*{\cT}{{\mathcal T}}
 \newcommand*{\cX}{{\mathcal X}}
 \newcommand*{\cY}{{\mathcal Y}}
 \newcommand*{\gB}{{\mathfrak B}}
 
 \newcommand*{\gF}{{\mathfrak F}}
 \newcommand*{\gG}{{\mathfrak G}}
 \newcommand*{\gK}{{\mathfrak K}}
 \newcommand*{\gsm}{{\mathfrak m}} 
 \newcommand*{\gN}{{\mathfrak N}}
 \newcommand*{\gsr}{{\mathfrak r}}
 \newcommand*{\gS}{{\mathfrak S}}
 \newcommand*{\gT}{{\mathfrak T}}
 \newcommand*{\gst}{{\mathfrak t}}
 
 \newcommand*{\sA}{{\mathsf A}}
 \newcommand*{\sfa}{{\mathsf a}}
 \newcommand*{\sB}{{\mathsf B}}
 \newcommand*{\sfb}{{\mathsf b}}
 \newcommand*{\sC}{{\mathsf C}}
 
 \newcommand*{\sfm}{{\mathsf m}}
 \newcommand*{\sft}{{\mathsf t}}
 \newcommand*{\sfw}{{\mathsf w}}
 \renewcommand*{\ldots}{\mathinner{\ldotp\ldotp\ldotp}}

 \makeatletter
 \newif\ifinany@
 \newcount\column@
 \def\column@plus{%
    \global\advance\column@\@ne
 }
 \newcount\maxfields@
 \def\add@amps#1{%
    \begingroup
        \count@#1
        \DN@{}%
        \loop
            \ifnum\count@>\column@
                \edef\next@{&\next@}%
                \advance\count@\m@ne
        \repeat
    \@xp\endgroup
    \next@
 }
 \def\Let@{\let\\\math@cr}
 \def\restore@math@cr{\def\math@cr@@@{\cr}}
 \restore@math@cr
 \def\default@tag{\let\tag\dft@tag}
 \default@tag

 \newbox\strutbox@
 \def\strut@{\copy\strutbox@}
 \addto@hook\every@math@size{%
  \global\setbox\strutbox@\hbox{\lower.5\normallineskiplimit
         \vbox{\kern-\normallineskiplimit\copy\strutbox}}}

 \renewcommand{\start@aligned}[2]{%
    \RIfM@\else
        \nonmatherr@{\begin{\@currenvir}}%
    \fi
    \null\,%
    \if #1t\vtop \else \if#1b \vbox \else \vcenter \fi \fi \bgroup
        \maxfields@#2\relax
        \ifnum\maxfields@>\m@ne
            \multiply\maxfields@\tw@
            \let\math@cr@@@\math@cr@@@alignedat
        \else
            \restore@math@cr
        \fi
        \Let@
        \default@tag
        \ifinany@\else\openup\jot\fi
        \column@\z@
        \ialign\bgroup
           &\column@plus
            \hfil
            \strut@
            $\m@th\displaystyle{##}$%
           &\column@plus
            $\m@th\displaystyle{{}##}$%
            \hfil
            \crcr
 }
 \renewenvironment{aligned}[1][c]{%
    \start@aligned{#1}\m@ne
 }{%
    \crcr\egroup\egroup
 }
 \makeatother
\begin{document}
\DOIsuffix{mana.DOIsuffix}
\Volume{248}
\Month{01}
\Year{2005}
\pagespan{1}{}
    \Receiveddate{15 November 2005}
    \Reviseddate{30 November 2005}
    \Accepteddate{2 December 2005}
    \Dateposted{3 December 2005}
    \vspace*{3\baselineskip}    



\title[Anisotropic Function Spaces on Singular Manifolds]{Anisotropic 
Function Spaces on Singular Manifolds}


\author[H. Amann]{H. Amann\footnote{e-mail: {\sf herbert.amann@math.uzh.ch}}}
\address{Math.\ Institut, Universit\"at Z\"urich,
Winterthurerstr.~190, CH--8057 Z\"urich, Switz\-er\-land}
\begin{abstract} 
A~rather complete investigation of anisotropic 
Bessel potential, Besov, and H\"older spaces on cylinders over (possibly) 
noncompact Riemannian manifolds with boundary is carried out. The geometry 
of the 
underlying manifold near its `ends' is determined by a singularity 
function which leads naturally to the study of weighted function spaces. 
Besides of the derivation of Sobolev-type embedding results, sharp trace 
theorems, point-wise multiplier properties, and interpolation 
characterizations particular emphasize is put on spaces distinguished by 
boundary conditions. This work is the fundament for the analysis of 
time-dependent partial differential equations on singular manifolds. 
\end{abstract} 
\maketitle                   





 \def\cprime{$'$}
 \newif \ifcomments
 \commentsfalse
 \ifcomments \advance\hoffset by-15truemm\fi
 \def\slComm#1#2{\ifcomments%
     {}\vadjust{\kern#1%
      \vtop to 0pt{\vss\hbox to\hsize{\hfill\strut\rlap{ \rm#2}}\null}
      \kern-#1}\fi%
 }
 \def\Comm#1{\slComm{0mm}{$\bullet$ #1}}
 \def\Label#1{\label{#1}\slComm{0mm}{\hbox{\rm #1}}}  
 \def\LabelT#1{\label{#1}\strut\slComm{1mm}{\hbox{\rm #1}}}
 \newcommand*{\Eqref}[1]{\hbox{\rm(\ref{#1})}}
\section{Introduction}\LabelT{sec-I}%
In~\cite{Ama12b} we have performed an in-depth study of Sobolev, Bessel 
potential, and Besov spaces of functions and tensor fields on Riemannian 
manifolds which may have a boundary and may be noncompact and 
noncomplete.  That as well as the present research is motivated 
by~---~and provides the basis for~---~the study of elliptic and 
parabolic boundary value problems on piece-wise smooth manifolds, on 
domains in~$\BR^m$ with a piece-wise smooth boundary in particular. 

\smallskip 
A singular manifold~$M$ is to a large extent 
determined by a `singularity function'
\hb{\rho\in C^\iy\bigl(M,(0,\iy)\bigr)}. The behavior of~$\rho$ at the
`singular ends' of~$M$, that is, near that parts of~$M$ at which $\rho$~gets
either arbitrarily small or arbitrarily large, reflects the singular 
structure of~$M$. 

\smallskip 
The basic building blocks for a useful theory of function 
spaces on singular manifolds are weighted Sobolev spaces based on the 
singularity function~$\rho$. More precisely, we denote by~$\BK$ 
either $\BR$ or~$\BC$. Then, given
\hb{k\in\BN},
\ \hb{\lda\in\BR}, and
\hb{p\in(1,\iy)}, the weighted Sobolev space
\hb{W_{\coW p}^{k,\lda}(M)=W_{\coW p}^{k,\lda}\MBK} is the completion
of~$\cD(M)$, the space of smooth functions with compact support in~$M$,
in~$L_{1,\loc}(M)$ with respect to the norm
\beq\Label{I.n}
u\mt 
\Bigl(\sum_{i=0}^k\big\|\rho^{\lda+i}\,|\na^iu|_g\big\|_p^p\Bigr)^{1/p}.
\eeq 
Here $\na$~denotes the Levi-Civita covariant derivative and
$|\na^iu|_g$~is the `length' of the covariant tensor field~$\na^iu$
naturally derived from the Riemannian metric $g$ of~$M$. Of course,
integration is carried out with respect to the volume measure of~$M$.
It turns out that $W_{\coW p}^{k,\lda}(M)$ is well-defined,
independently~---~in the sense of equivalent norms~---~of the
representation of the singularity structure of~$M$ by means of the
specific singularity function.

\smallskip
A~very special and simple example of a singular manifold is provided by a
bounded smooth domain whose boundary possesses a conical point. More
precisely, suppose $\Om$~is a bounded domain in~$\BR^m$ whose topological
boundary, $\bdry(\Om)$, contains the origin, and
\hb{\Ga:=\bdry(\Om)\ssm\{0\}} is a smooth
\hb{(m-1)}-dimensional submanifold of~$\BR^m$ lying locally on one side
of~$\Om$. Also suppose that
\hb{\Om\cup\Ga} is near~$0$ diffeomorphic to a cone
\hb{\{\,ry\ ;\ 0<r<1,\ y\in B\,\}}, where $B$~is a smooth compact
submanifold of the unit sphere in~$\BR^m$. Then, endowing
\hb{M:=\Om\cup\Ga} with the Euclidean metric, we get a singular
manifold with a single conical singularity, as considered in
\cite{NaP94a} and~\cite{KMR97a}, for example. In this case the weighted
norm~\Eqref{I.n} is equivalent to
$$
u\mt\Bigl(\sum_{|\al|\leq k}
\|r^{\lda+|\al|}\pa u\|_{L_p(\Om)}^p\Bigr)^{1/p},
$$
where $r(x)$~is the Euclidean distance from
\hb{x\in M} to the origin. Moreover,
$W_{\coW p}^{k,\lda}(M)$ coincides with the
space $V_{\coV p,\lda+k}^k(\Om)$ employed by
S.A. Nazarov and B.A. Plamenevsky \cite[p.~319]{NaP94a} and, in the case
\hb{p=2}, by 
V.A. Kozlov, V.G. Maz{\cprime}ya, and J.~Rossmann (see Section~6.2 
of~\cite{KMR97a}, for instance). 

\smallskip 
In~\cite{Ama12b} we have exhibited a number of examples of singular 
manifolds. For more general classes, comprising notably manifolds 
with corners and non-smooth cusps, we refer to H.~Amann~\cite{Ama12c}. 
It is worthwhile to point out that our concept of singular manifolds 
encompasses, as a very particular case, manifolds with 
bounded geometry (that is, Riemannian manifolds without boundary 
possessing a positive injectivity radius and having all covariant 
derivatives of the curvature tensor bounded). In this case we can set 
\hb{\rho=\mf{1}}, the function constantly equal to~$1$, so that 
$W_{\coW p}^{k,\lda}(M)$~is independent of~$\lda$ and equal to the 
standard Sobolev space~$W_{\coW p}^k(M)$. 

\smallskip 
The weighted Sobolev spaces~$W_{\coW p}^{k,\lda}(M)$ and their 
fractional order relatives, that is, Bessel potential and Besov spaces, 
come up naturally in, and are especially useful for, the study of elliptic 
boundary value problems for differential and pseudodifferential operators 
in non-smooth settings. This is known since the 
seminal work of V.A. Kondrat{\cprime}ev~\cite{Kon67a} 
and has since been exploited and amplified by numerous 
authors in various levels of generality, predominantly however in the 
Hilbertian case 
\hb{p=2} (see~\cite{Ama12b} for further bibliographical remarks). 

\smallskip 
For an efficient study of evolution equations on singular manifolds we 
have to have a good understanding of function spaces on space-time 
cylinders 
\hb{M\times J} with 
\hb{J\in\{\BR,\BR^+\}}, where 
\hb{\BR^+=[0,\iy)}. Then, in general, the functions (or distributions) 
under consideration posses different regularity properties with 
respect to the space and time variables. Thus we are led to study 
anisotropic Sobolev spaces and their fractional order relatives. 

\smallskip 
Anisotropic weighted Sobolev spaces depend on two additional parameters, 
namely 
\hb{r\in\BN^\times:=\BN\ssm\{0\}} and 
\hb{\mu\in\BR}. More precisely, we denote throughout by 
\hb{\pl=\pl_t} the vector-valued distributional 
`time' derivative. Then, given 
\hb{k\in\BN^\times}, 
\beq\Label{I.Nkr} 
\bal 
W_{\coW p}^{(kr,k),(\lda,\mu)}(M\times J) 
&\text{ is the linear subspace of $L_{1,\loc}(M\times J)$ 
 consisting of all $u$ satisfying}\\
&\rho^{\lda+i+j\mu}\,|\na^i\pl^ju|_g\in L_p(M\times J) 
\quad\text{for}\quad
i+jr\leq kr\text,\\ 
&\text{endowed with its natural norm.}  
\eal 
\npb 
\eeq  
It is a Banach space, a~Hilbert space if 
\hb{p=2}. 

\smallskip 
Spaces of this type, as well as fractional order versions thereof, provide 
the natural domain for an \hbox{$L_p$-theory} of linear differential 
operators of the form 
$$ 
\sum_{i+jr\leq kr}a_{ij}\cdot\na^i\pl^j, 
$$ 
where $a_{ij}$~is a time-dependent contravariant tensor field of order~$i$ 
and 
\hb{{}\cdot{}} indicates complete contraction. In this connection the values 
\hb{\mu=0}, 
\ \hb{\mu=1}, and 
\hb{\mu=r} are of particular importance. If 
\hb{\mu=1}, then space and time derivatives carry the same weight. If also 
\hb{r=1}, then we get isotropic weighted Sobolev spaces on 
\hb{M\times J}. 

\smallskip 
If 
\hb{\mu=0}, then the intersection space characterization 
$$ 
W_{\coW p}^{(kr,k),(\lda,0)}(M\times J) 
\doteq L_p\bigl(J,W_{\coW p}^{kr,\lda}(M)\bigr) 
\cap W_{\coW p}^k\bigl(J,L_p^\lda(M)\bigr) 
$$ 
is valid, where 
\hb{{}\doteq{}}~means: equal except for equivalent norms. Spaces of this 
type (with 
\hb{k=1}) have been used by 
S.~Coriasco, E.~Schrohe, and J.~Seiler \cite{CSS03a},~\cite{CSS07a} for 
studying parabolic equations on manifolds with conical points. In this case 
$\rho$~is (equivalent to) the distance from the singular points. 
Anisotropic spaces with 
\hb{\mu=0} are also important for certain classes of degenerate parabolic 
boundary value problems (see~\cite{Ama12c}). 

\smallskip 
The spaces $W_{\coW p}^{(kr,k),(\lda,r)}(M\times J)$ constitute, perhaps, 
the most natural extension of the `stationary' 
spaces $W_{\coW p}^{k,\lda}(M)$ to the space-time cylinder 
\hb{M\times J}. They have been employed by V.A. Kozlov 
\hbox{\cite{Koz88b}\nocite{Koz88a}\nocite{Koz89a}--\cite{Koz91a}} ---~in 
the Hilbertian setting 
\hb{p=2}~---~for the study of general parabolic boundary value problems 
on a cone~$M$. (Kozlov, as well as the authors mentioned below, write 
$W_{\lda+kr}^{(kr,k)}$ for~$W_2^{(kr,k),(\lda,r)}$.) The space 
$W_{\coW p}^{(2,1),(\lda,2)}(M\times J)$ occurs in the works on second 
order parabolic equations on smooth infinite wedges by 
V.A. Solonnikov~\cite{Sol01a} and 
A.I. Nazarov~\cite{Naz01a} (also see 
V.A. Solonnikov and E.V. Frolova \cite{SoF90a},~\cite{SoF91a}), as well as 
in the studies of W.M. 
Zaj{\polhk{a}}czkowski 
\hbox{\cite{Zaj07a}\nocite{Zaj11a}\nocite{Zaj11b}--\cite{Zaj11c}}, 
A.~Kubica and W.M. Zaj{\polhk{a}}czkowski \cite{KuZ07a},~\cite{KuZ07b}, and 
K.~Pileckas \hbox{\cite{Pil05a}\nocite{Pil07a}--\cite{Pil08a}} (see 
the references in these papers for earlier work) on Stokes and 
Navier-Stokes equations. In all these papers, except the ones of Pileckas, 
$\rho$~is the distance to the singularity set, where in 
Zaj{\polhk{a}}czkowski's publications $M$~is obtained from a smooth 
subdomain of~$\BR^m$ by eliminating a line segment. Pileckas considers 
subdomains of~$\BR^m$ with outlets to infinity and~$\rho$ having possibly  
polynomial or exponential growth. 

\smallskip 
In this work we carry out a detailed study of anisotropic Sobolev, 
Bessel potential, Besov, and H\"older spaces on singular manifolds and 
their interrelations. Besides of this introduction, the paper is structured 
by the following sections on whose principal content we comment below. 

\medskip 
\noindent 
\begin{tabular}{rlrl} 
\ref{sec-V}     &Vector Bundles 
    &\quad
     \ref{sec-P}     &Point-Wise Multipliers\\ 
\ref{sec-U}     &Uniform Regularity 
    &\ref{sec-C}     &Contractions\\ 
\ref{sec-S}     &Singular Manifolds 
    &\ref{sec-E}     &Embeddings\\ 
\ref{sec-L}     &Local Representations 
    &\ref{sec-M}     &Differential Operators\\ 
\ref{sec-J}     &Isotropic Bessel Potential and Besov Spaces 
    &\ref{sec-ER}    &Extensions and Restrictions\\ 
\ref{sec-R}     &The Isotropic Retraction Theorem 
    &\ref{sec-T}     &Trace Theorems\\ 
\ref{sec-A}     &Anisotropic Bessel Potential and Besov Spaces 
    &\ref{sec-VT}    &Spaces With Vanishing Traces\\ 
\ref{sec-RA}    &The Anisotropic Retraction Theorem 
    &\ref{sec-B}     &Boundary Operators\\ 
\ref{sec-N}     &Renorming of Besov Spaces 
    &\ref{sec-IP}    &Interpolation\\ 
\ref{sec-PH}    &H\"older Spaces in Euclidean Settings 
    &\ref{sec-BC}    &Bounded Cylinders\\  
\ref{sec-H}     &Weighted H\"older Spaces 
\end{tabular} 

\medskip   
\noindent 
We have already pointed out in~\cite{Ama12b} that it is not sufficient to 
study function spaces on singular manifolds since spaces of tensor 
fields occur naturally in applications. In order to pave the way for a 
study of \emph{systems} of differential and pseudodifferential operators it 
is even necessary to deal with tensor fields taking their values in general 
vector bundles. This framework is adopted here. 

\smallskip 
Sections \ref{sec-V} and~\ref{sec-U} are of preparatory character. In the 
former, besides of fixing notation and introducing conventions used 
throughout, we present the background material on vector bundles on which 
this paper is based. We emphasize, in particular, duality properties and 
local representations which are fundamental for our approach. 

\smallskip 
Since we are primarily interested in noncompact manifolds we have to impose 
suitable regularity conditions `at infinity'. This is done in 
Section~\ref{sec-U} where we introduce the class of `fully uniformly 
regular' vector bundles. They constitute the `image bundles' for the 
tensor fields on the singular manifolds which we consider here. 

\smallskip 
After these preparations, singular manifolds are introduced in 
Section~\ref{sec-S}. There we also install the geometrical frame which 
we use from thereon without further mention. 

\smallskip 
Although we study spaces of tensor fields taking their values in uniformly 
regular vector bundles, the vector bundles generated by these tensor fields 
are \emph{not} uniformly regular themselves, in general. In fact, their 
metric and their covariant derivative depend on the metric~$g$ of the 
underlying 
singular Riemannian manifold. Since the singularity behavior of~$g$ is 
controlled by the singularity function~$\rho$, due to our very definition 
of a singular manifold, we have to study carefully the dependence of all 
relevant parameters on~$\rho$ as well. This is done in Section~\ref{sec-L}. 
On its basis we can show in later sections that the various function spaces 
are independent of particular representations; they depend on the underlying 
geometric structure only. 

\smallskip 
Having settled these preparatory problems we can then turn to the main 
subject of this paper, the study of function spaces (more precisely, spaces 
of vector-bundle-valued tensor fields) on singular manifolds. We begin in 
Sections \ref{sec-J} and~\ref{sec-R} by recalling and amplifying some 
results from our previous paper~\cite{Ama12b} on isotropic spaces. On the 
one hand this allows us to introduce some basic concepts and on the other 
hand we can point out the changes which have to be made to cover the more 
general setting of of vector-bundle-valued tensor fields. 

\smallskip 
The actual study of anisotropic weighted function spaces begins in 
Section~\ref{sec-A}. First we introduce Sobolev spaces which can be easily 
described invariantly. They form the building blocks for the theory of 
anisotropic weighted Bessel potential and Besov spaces. The latter are 
invariantly defined by interpolation between Sobolev spaces and by duality. 

\smallskip 
This being done, it has to be shown that these spaces coincide in the most 
simple situation in which $M$~is either the Euclidean space~$\BR^m$ or a 
closed half-space~$\BH^m$ thereof with the `usual' anisotropic Bessel 
potential and Besov spaces, respectively. In the Euclidean model setting a 
thorough investigation has been carried out in H.~Amann~\cite{Ama09a} by 
means of Fourier analytic techniques. That work is the fundament upon which 
the present research is built. The basic result which settles this 
identification and is fundamental for the whole theory as well as for the 
study of evolution equations is Theorem~\ref{thm-RA.R}. In particular, it 
establishes isomorphisms between the function spaces on 
\hb{M\times J} and certain countable products of corresponding spaces on 
model manifolds. By these isomorphisms we can transfer the known properties 
of the `elementary' spaces on 
\hb{\BR^m\times J} and 
\hb{\BH^m\times J} to 
\hb{M\times J}. With this method we establish the most fundamental 
properties of anisotropic Bessel potential and Besov spaces which are 
already stated in Section~\ref{sec-A}. 

\smallskip 
In Section~\ref{sec-N} we take advantage of the fact that the anisotropic 
spaces we consider live in cylinders over~$M$ so that the `time variable' 
plays a distinguished role. This allows us to introduce some useful 
semi-explicit equivalent norms for Besov spaces. 

\smallskip 
It is well-known that spaces of H\"older continuous functions are 
intimately related to the theory of partial differential equations on 
Euclidean spaces. They occur naturally, even in the \hbox{$L_p$-theory}, 
as point-wise multiplier spaces, in particular as coefficient spaces for 
differential operators. Although it is fairly easy to study H\"older 
continuous functions on subsets of~$\BR^m$, it is surprisingly difficult to 
do this on manifolds. Our approach to this problem is similar to the way in 
which we defined Bessel potential and (\hbox{$L_p$-based}) Besov spaces on 
manifolds. Namely, first we introduce spaces of bounded and continuously 
differentiable functions. Then we define H\"older spaces, more generally 
Besov-H\"older spaces, by interpolation. This is not straightforward since 
we can only interpolate between spaces of bounded \hbox{$C^k$-functions} 
whose derivatives are uniformly continuous. Due to the fact that we are 
mainly interested in noncompact manifolds, the concept of uniform continuity 
is not a~priori clear and has to be clarified first. Then the next problem 
is to show that H\"older spaces introduced in this invariant way can be 
described locally by their standard anisotropic counterparts on 
\hb{\BR^m\times J} and 
\hb{\BH^m\times J}. Such representations in local coordinates are, 
of course, fundamental for the study of concrete equations, for example. 

\smallskip 
In order to achieve these goals we set up the preliminary 
Section~\ref{sec-PH} in which we establish the needed properties of 
(vector-valued) H\"older and Bessel-H\"older spaces in Euclidean settings.  
In Section~\ref{sec-H} we can then settle the problems alluded to above. 
It should be mentioned that in these two sections we consider 
time-independent isotropic as well as time-dependent anisotropic spaces, 
thus complementing the somewhat ad hoc results on H\"older spaces 
in~\cite{Ama12b}.

\smallskip 
Having introduced all these spaces and established their basic properties 
we proceed now to more refined features. In Section~\ref{sec-P} we show 
that, similarly as in the Euclidean setting, H\"older spaces are universal 
point-wise multiplier spaces for Bessel potential and Besov spaces 
modeled on~$L_p$. For this we establish the rather general (almost) 
optimal Theorem~\ref{thm-P.M}. 

\smallskip 
In practice point-wise multiplications occur, as a rule, through 
contractions of tensor fields. For this reason we carry out in 
Section~\ref{sec-C} a~detailed study of mapping properties of contractions 
of tensor fields, one factor belonging to a H\"older space and the other 
one to a Bessel potential or a Besov space, in particular. It should be 
noted that we impose minimal regularity assumptions for the multiplier 
space. The larger part of Section~\ref{sec-C} is, however, devoted to the 
problem of the existence of a continuous right inverse for a multiplier 
operator induced by a complete contraction. The main result of this section 
thus is Theorem~\ref{thm-C.R}. It is basic for the theory of boundary value 
problems. 

\smallskip 
Section~\ref{sec-E} contains general Sobolev-type embedding theorems for 
parameter-dependent weighted Bessel potential and Besov spaces. They are 
natural extensions of the corresponding classical results in the Euclidean 
setting. 

\smallskip 
Making use of our point-wise multiplier and Sobolev-type embedding theorems 
we study in Section~\ref{sec-M} mapping properties of differential 
operators in anisotropic spaces. In view of applications to quasilinear 
equations we strive for minimal regularity requirements for the coefficient 
tensors.  

\smallskip 
All results established up to this point hold both for 
\hb{J=\BR} and 
\hb{J=\BR^+}. In contrast, Section~\ref{sec-ER} is specifically concerned 
with anisotropic spaces on the half-line~$\BR^+$. It is shown that in many 
cases properties of function spaces on~$\BR^+$ can be derived from the 
corresponding results on the whole line~$\BR$. This can simplify the 
situation since 
\hb{M\times\BR} is a usual manifold (with boundary), whereas 
\hb{M\times\BR^+} has corners if 
\hb{\pl M\neq\es}. 

\smallskip 
In Section~\ref{sec-T} we consider the important case where $M$~has a 
nonempty boundary and establish the fundamental trace theorem for 
anisotropic weighted Bessel potential and Besov spaces, both on the 
`lateral boundary' 
\hb{\pl M\times J} and on the `initial boundary' 
\hb{M\times\{0\}} if 
\hb{J=\BR^+}. 

\smallskip 
In the next section we characterize spaces of functions having vanishing 
initial traces. Section~\ref{sec-B} is devoted to  extending the 
boundary values. Here we rely, besides the trace theorem, in particular on 
the `right inverse theorem' established in Section~\ref{sec-C}. The 
results of this section are of great importance in the theory of 
boundary value problems. 

\smallskip 
Section~\ref{sec-IP} describes the behavior of anisotropic weighted 
Bessel potential, Besov, and H\"older spaces under interpolation. In 
addition to this, we also derive interpolation theorems for `spaces with 
vanishing boundary conditions'. These results are needed for~a `weak 
\hbox{$L_p$-theory}' of parabolic evolution equations. 

\smallskip 
Our investigation of weighted anisotropic function spaces is greatly 
simplified by the fact that we consider full and half-cylinders over~$M$. 
In this case we can take advantage of the dilation invariance of~$J$. In 
practice, cylinders of finite height come up naturally and are of 
considerable importance. For this 
reason it is shown, in the last section, that all embedding, interpolation, 
trace theorems, etc.\ are equally valid if $J$~is replaced by $[0,T]$ for 
some 
\hb{T\in(0,\iy)}. 

\smallskip 
In order to cover the many possibilities due to the (unavoidably) large set 
of parameters our spaces depend upon, and to eliminate repetitive 
arguments, we use rather condensed notation in which we exhibit the locally 
relevant information only. This requires a great deal of concentration on 
the part of the reader. However, everything simplifies drastically in the 
important special case of Riemannian manifolds with bounded geometry. In 
that case there are no singularities and all spaces are 
parameter-independent. Readers interested in this situation only can simply 
ignore all mention of the parameters $\lda$, $\mu$, and~$\vec\om$ and set  
\hb{\rho=\mf{1}}. Needless to say that even in this `simple' situation the 
results of this paper are new. 
\section{Vector Bundles}\LabelT{sec-V}%
First we introduce some notation and conventions from functional analysis. 
Then we recall some relevant facts from the theory of vector bundles. 
It is the main purpose of this preparatory section to create a firm basis 
for the following. We emphasize in particular duality properties and local 
representations, for which we cannot refer to the literature. 
Background material on manifolds and vector bundles is found in  
J.~Dieudonn\'e~\cite{Die69b} or J.~Jost~\cite{Jos08a}, for example. 

\smallskip 
Given locally convex (Hausdorff topological vector) spaces $\cX$ and~$\cY$, 
we denote by~$\cL\cXcY$ the
space of continuous linear maps from~$\cX$ into~$\cY$, and
\hb{\cL(\cX):=\cL\cXcX}. By $\Lis\cXcY$ we mean the set of all 
isomorphisms in~$\cL\cXcY$, and 
\hb{\Laut(\cX):=\Lis\cXcX} is the automorphism group in~$\cL(X)$. If 
$\cX$ and~$\cY$ are Banach spaces, then
$\cL\cXcY$~is endowed with the uniform operator norm. In this situation 
$\Lis\cXcY$ is open in~$\cL\cXcY$. We write
\hb{\pw_\cX} for the duality pairing between 
\hb{\cX':=\cL\cXBK} and~$\cX$, that is,
$\dl x',x\dr_\cX$ is the value of
\hb{x'\in\cX'} at
\hb{x\in\cX}.

\smallskip 
Let 
\hb{H=\bigl(H,\prsn\bigr)} be a Hilbert space. Then the \emph{Riesz 
isomorphism} is the conjugate linear isometric isomorphism 
\hb{\vt=\vt_H\sco H\ra H'} defined by 
\beq\Label{V.R}
\dl\vt x,y\dr=(y\sn x) 
\qa x,y\in H, 
\eeq 
where 
\hb{\pw=\pw_H}. Then 
\beq\Label{V.ad} 
(x'\sn y')^*:=(\vt^{-1}y'\sn\vt^{-1}x') 
\qa x',y'\in H',
\eeq 
defines the \emph{adjoint} inner product on~$H'$, and 
\hb{H^*:=\bigl(H',\prsn^*\bigr)} is the adjoint Hilbert space. Denoting by 
\hb{\Vsdot} and~%
\hb{\Vsdot^*} the inner product norms associated with 
\hb{\prsn} and~%
\hb{\prsn^*}, respectively, we obtain from \Eqref{V.R} and \Eqref{V.ad} 
\beq\Label{V.CS1}
|\dl x',x\dr|\leq\|x'\|^*\,\|x\| 
\qa x'\in H' 
\qb x\in H. 
\eeq 
It follows from \hbox{\Eqref{V.R}--\Eqref{V.CS1}} 
and the fact that 
$\vt$~is an isometry that 
\hb{\|x'\|^*=\sup\bigl\{\,|\dl x',x\dr|,\ \|x\|\leq1\,\bigr\}} for 
\hb{x'\in H'}. Thus 
\hb{\Vsdot^*}~is the norm in 
\hb{H'=\cL\HK}, the dual norm. In other words, 
\hb{H'=H^*} as Banach spaces. For this and historical reasons we use the 
`star notation' for the dual space in the finite-dimensional setting and 
in connection with vector bundles, 
whereas the `prime notation' is more appropriate in functional analytical 
considerations. 

\smallskip 
If $H_1$ and~$H_2$ are Hilbert spaces and 
\hb{A\in\cL\HeHz}, then it has to be carefully distinguished between 
the dual 
\hb{A'\in\cL\HzsHes}, defined by 
\hb{\dl A'x_2',x_1\dr_{H_1}=\dl x_2',Ax_1\dr_{H_2}}, and the adjoint 
\hb{A^*\in\cL\HzHe}, given by 
\hb{(A^*x_2\sn x_1)_{H_1}=(x_2\sn Ax_1)_{H_2}} for 
\hb{x_i\in H_i} and 
\hb{x_2'\in H_2'}. 

\smallskip 
Suppose $H_1$ and~$H_2$ are finite-dimensional. Then $\cL\HeHz$ is a 
Hilbert space with the Hilbert-Schmidt inner product~%
\hb{\prsn_{HS}} defined by 
\hb{(A\sn B)_{HS}:=\tr(B^*A)} for 
\hb{A,B\in\cL\HeHz}, where $\tr$~denotes the trace. The corresponding 
norm~%
\hb{\vsdot_{HS}} is the Hilbert-Schmidt norm 
\hb{A\mt\sqrt{\tr(A^*A)}}. 

\smallskip 
Throughout this paper, we use the summation convention for indices 
labeling coordinates or bases. This means that such a repeated index, 
which appears once as a superscript and once as a subscript, implies 
summation over its whole range. 

\smallskip 
By~a \emph{manifold} we always mean a smooth, that is, $C^\iy$~manifold
with (possibly empty) boundary such that its underlying topological space
is separable and metrizable. Thus, in the context of manifolds,  we work in
the smooth category. A~manifold need not be connected, but all connected
components are of the same dimension.

\smallskip 
Let $M$ be an \hbox{$m$-dimensional} manifold and 
\hb{V=(V,\pi,M)} a \hbox{$\BK$~vector} bundle of rank~$n$ over~$M$. For 
a nonempty subset~$S$ of~$M$ we denote by $V_S$, or~$V_{|S}$, the 
restriction~$\pi^{-1}(S)$ of~$V$ to~$S$. If $S$~is a submanifold, or 
\hb{S=\pl M}, then $V_S$~is a vector bundle of rank~$n$ over~$S$. As usual, 
\hb{V_p:=V_{\{p\}}} is the fibre~$\pi^{-1}(p)$ of~$V$ over~$p$. 
Occasionally, we use the symbolic notation 
\hb{V=\bigcup_{p\in M}V_{\coV p}}. 

\smallskip 
By $\Ga\SV$ we mean the \hbox{$\BK^S$~module} of all sections of~$V$ 
over~$S$ (no smoothness). If $S$~is a submanifold, or 
\hb{S=\pl M}, then $C^k\SV$ is for 
\hb{k\in\BN\cup\{\iy\}} the Fr\'echet space of \hbox{$C^k$~sections} 
over~$S$. It is a 
\hb{C^k(S):=C^k\SK} module. In the case of a trivial bundle 
\hb{M\times E=(M\times E,\pro_1,M)} for some \hbox{$n$-dimensional} 
Banach space~$E$, a~section over~$S$ is a map from~$S$ into~$E$, that is, 
\hb{\Ga(S,M\times E)=E^S}. Accordingly 
\hb{C^k(S,M\times E)=C^k\SE} is the Fr\'echet space of all \hbox{$C^k$ maps} 
from~$S$ into~$E$. As usual, $\pro_i$~denotes the natural projection onto 
the \hbox{$i$-th} factor of a Cartesian product (of~sets). 

\smallskip  
Let 
\hb{\wt{V}=(\wt{V},\wt{\pi},\wt{M})} be a vector bundle over a 
manifold~$\wt{M}$. A~\hbox{$C^k$ map} 
\hb{(f_0,f)\sco\MV\ra(\wt{M},\wt{V})}, that is, 
\hb{f_0\in C^k(M,\wt{M})} and 
\hb{f\in C^k(V,\wt{V})}, is a \emph{\hbox{$C^k$~bundle} 
morphism} if the diagram 
 \bspezeq 
 \bal 
 \begin{picture}(82,63)(-41,-5)        
 \put(0,50){\makebox(0,0)[b]{\small{$f$}}}
 \put(0,5){\makebox(0,0)[b]{\small{$f_0$}}}
 \put(-30,45){\makebox(0,0)[c]{\small{$V$}}}
 \put(-30,0){\makebox(0,0)[c]{\small{$M$}}}
 \put(30,45){\makebox(0,0)[c]{\small{$\wt{V}$}}}
 \put(30,0){\makebox(0,0)[c]{\small{$\wt{M}$}}}
 \put(-35,22.5){\makebox(0,0)[r]{\small{$\pi$}}}
 \put(35,22.5){\makebox(0,0)[l]{\small{$\wt{\pi}$}}}
 \put(-20,45){\vector(1,0){40}}
 \put(-20,0){\vector(1,0){40}}
 \put(-30,35){\vector(0,-1){25}}
 \put(30,35){\vector(0,-1){25}}
 \end{picture}
 \eal 
 \espezeq 
is commuting, and 
\hb{f\sn V_p\in\cL(V_{\coV p},V_{f_0(p)})} for 
\hb{p\in M}. It is a \emph{conjugate linear} bundle morphism if 
\hb{f\sn V_{\coV p}} is a conjugate linear map. By defining 
compositions of bundle morphisms in the obvious way one gets, 
in particular, the category of smooth, that is~$C^\iy$, bundles 
in which we work.  Thus a \emph{bundle isomorphism} is an isomorphism 
in the category of smooth vector bundles. If 
\hb{M=\wt{M}}, then 
$f$~is called \emph{bundle morphism} if $(\id_M,f)$ is one. 

\smallskip 
A~\emph{bundle metric} on~$V$ is a smooth section~$h$ of the tensor product 
\hb{V^*\otimes V^*} such that $h(p)$~is an inner product on~$V_{\coV p}$ 
for  
\hb{p\in M}. Then the continuous map 
$$ 
\vsdot_h\sco V\ra C(M) 
\qb v\mt\sqrt{h\vv} 
\npb 
$$ 
is the \emph{bundle norm} derived from~$h$. 

\smallskip 
Suppose 
\hb{V=\Vh} is~a \emph{metric vector bundle}, that is, $V$~is endowed 
with a bundle metric~$h$. Then $V_{\coV p}$~is an \hbox{$n$-dimensional} 
Hilbert space with inner product~$h(p)$. Hence 
\hb{V_{\coV p}^*=(V_{\coV p}',h^*{(p)})}, where $h^*(p)$~is 
the adjoint inner product on~$V_{\coV p}'$, equals~$V_{\coV p}'$ as a 
Banach space. The dual bundle 
\hb{V^*=\bigcup_{p\in M}V_{\coV p}^*} is endowed with the adjoint 
bundle metric~$h^*$ satisfying 
\hb{h^*\sn(V^*\oplus V^*)_p=h^*(p)} for  
\hb{p\in M}, where 
\hb{{}\oplus{}}~is the Whitney sum. 

\smallskip 
The (\emph{bundle}) \emph{duality pairing} 
\hb{\pw_V} is the smooth section of 
\hb{V\otimes V^*} defined by 
\hb{\pw_V(p)=\pw_{V_{\coV p}}} for 
\hb{p\in M}. It follows 
$$  
|\dl v^*,v\dr_V|\leq|v^*|_{h^*}\,|v|_h 
\qb \vsv\in\Ga(M,V^*\oplus V). 
$$ 
We denote by 
\hb{h_\flat(p)\sco V_{\coV p}\ra V_{\coV p}^*} the Riesz isomorphism for 
$\bigl(V_{\coV p},h(p)\bigr)$ and by~$h^\sh(p)$ its inverse. 
This defines the $C^\iy(M)$-conjugate linear (\emph{bundle}) 
\emph{Riesz isomorphism} 
\hb{h_\flat\sco V\ra V^*} and its inverse 
\hb{h^\sh\sco V^*\ra V}, given by 
\hb{h_\flat\sn V_{\coV p}=h_\flat(p)} and 
\hb{h^\sh\sn V_{\coV p}^*=h^\sh(p)}, respectively, for 
\hb{p\in M}. Thus 
$$ 
\dl h_\flat v,w\dr_V=h\wv 
\qb \vw\in\Ga(M,V\oplus V). 
$$
The canonical identification of~$V_{\coV p}^{**}$ with~$V_{\coV p}$ implies 
$$ 
V^{**}=V 
\qb \dl v,v^*\dr_{V^*}=\dl v^*,v\dr_V 
\qa \vvs\in\Ga(M,V\oplus V^*). 
$$ 

\smallskip 
We fix an \hbox{$n$-dimensional} Hilbert space 
\hb{E=\bigl(E,\prsn_E\bigr)}, a~\emph{model fiber for}~$V$. We also fix a 
basis $(e_1,\ldots,e_n)$ of~$E$ and denote by $(\ve^1,\ldots,\ve^n)$ the 
dual basis. Of course, without loss of generality we could set 
\hb{E=\BK^n}. However, for notational simplicity it is more convenient to 
use coordinate-free settings. 

\smallskip 
Let $U$ be open in~$M$. A~\emph{local chart for~$V$ over}~$U$ is a map 
$$ 
\ka\slt\vp\sco V_U\ra\ka(U)\times E 
\qb v_p\mt\bigl(\ka(p),\vp(p)v_p\bigr) 
\qa v_p\in V_{\coV p} 
\qb p\in U, 
$$  
such that 
\hb{(\ka,\ka\slt\vp)\sco(U,V_U)\ra\bigl(\ka(U),\ka(U)\times E\bigr)} 
is a bundle isomorphism, where $\ka(U)$~is open in the closed half-space 
\hb{\BH^m:=\BR^+\times\BR^{m-1}} of~$\BR^m$ (and 
\hb{\BR^0:=\{0\}}). In particular, $\ka$~is a local chart for~$M$. 

\smallskip 
Suppose 
\hb{\ka\slt\vp} and 
\hb{\wt{\ka}\slt\wt{\vp}} are local charts of $V$ over $U$ and~$\wt{U}$, 
respectively. Then the \emph{coordinate change} 
$$ 
(\wt{\ka}\slt\wt{\vp})\circ(\ka\circ\vp)^{-1}
\sco \ka(U\cap\wt{U})\times E\ra\wt{\ka}(U\cap\wt{U})\times E 
$$ 
is given by 
\hb{(x,\xi)\mt\bigl(\wt{\ka}\circ\ka^{-1}(x),\vp_{\ka\wt{\ka}}(x)\xi\bigl)}, 
where 
$$ 
\vp_{\ka\wt{\ka}}\in C^\iy\bigl(\ka(U\cap\wt{U}),\Laut(E)\bigr) 
$$ 
is the corresponding \emph{bundle transition map}. It follows 
\beq\Label{V.php} 
\vp_{\wt{\ka}\wh{\ka}}\vp_{\ka\wt{\ka}}=\vp_{\ka\wh{\ka}} 
\qb \vp_{\ka\ka}=1_E, 
\eeq 
$1_E$~being the identity in~$\cL(E)$. We set 
$$ 
\vp^{-\top}(p):=\bigl(\vp^{-1}(p)\bigr)'\in\Lis(V_{\coV p}^*,E^*) 
\qa p\in U. 
\npb 
$$ 
Then 
\hb{\ka\slt\vp^{-\top}\sco V_U^*\ra\ka(U)\times E^*} 
is the local chart for~$V^*$ over~$U$ \emph{dual} to 
\hb{\ka\slt\vp}.

\smallskip 
In the following, we use standard notation for the 
pull-back and push-forward of functions, that is, 
\hb{\ka^*f=f\circ\ka} and 
\hb{\ka_*f=f\circ\ka^{-1}}. The \emph{push-forward by} 
\hb{\ka\slt\vp} is the vector space isomorphism 
$$ 
(\ka\slt\vp)_*\sco\Ga\UV \ra E^{\ka(U)} 
\qb v\mt\bigl(x\mt\vp\bigl(\ka^{-1}(x)\bigr)v\bigl(\ka^{-1}(x)\bigr)\bigr). 
$$ 
Its inverse is the \emph{pull-back}, defined by 
$$ 
(\ka\slt\vp)^*\sco E^{\ka(U)}\ra\Ga\UV 
\qb \xi\mt\bigl(p\mt\bigl(\vp(p)\bigr)^{-1}\xi\bigl(\ka(p)\bigr)\bigr). 
$$ 
It follows that 
\hb{(\ka\slt\vp)_*} is a vector space isomorphism from $C^\iy\UV$ onto 
$C^\iy\bigl(\ka(U),E\bigr)$, and 
\beq\Label{V.pp} 
(\wt{\ka}\slt\wt{\vp})_*(\ka\slt\vp)^*\xi 
=\vp_{\ka\wt{\ka}}\bigl(\xi\circ(\wt{\ka}\circ\ka^{-1})\bigr) 
\qa \xi\in E^{\wt{\ka}(U_{\coU\ka}\cap U_{\coU\wt{\ka}})}. 
\eeq 
Furthermore, 
\beq\Label{V.pdu} 
\ka_*\bigl(\dl v^*,v\dr_V\bigr) 
=\big\dl(\ka\slt\vp^{-\top})_*v^*,(\ka\slt\vp)_*v\big\dr_E 
\qa (v^*,v)\in\Ga(U,V^*\oplus V). 
\eeq 
In addition, 
\beq\Label{V.pfv} 
(\ka\slt\vp)_*(fv)=(\ka_*f)(\ka\slt\vp)_*v 
\qa f\in\BK^U 
\qb v\in\Ga\UV. 
\eeq 

\smallskip 
We define the \emph{coordinate frame} $(b_1,\ldots,b_n)$ for~$V$ over~$U$ 
\emph{associated with} 
\hb{\ka\slt\vp} by 
$$ 
b_\nu:=(\ka\slt\vp)^*e_\nu 
\qa 1\leq\nu\leq n. 
$$ 
Then 
$$ 
\ba^\nu:=(\ka\slt\vp^{-\top})^*\ve^\nu 
\qa 1\leq\nu\leq n, 
$$ 
defines the \emph{dual coordinate frame} for~$V^*$ over~$U$. 
In fact, it follows from \Eqref{V.pdu} that 
$$ 
\dl\ba^\mu,b_\nu\dr_V=\ka^*\bigl(\dl\ve^\mu,e_\nu\dr_E\bigr)=\da_\nu^\mu 
\qa 1\leq\mu,\nu\leq n. 
$$ 
Let $(\wt{b}_1,\ldots,\wt{b}_n)$ be the coordinate frame for~$V$ 
over~$\wt{U}$ associated with 
\hb{\wt{\ka}\slt\wt{\vp}}. Then \Eqref{V.pp} and \Eqref{V.pdu} imply 
$$ 
\ka_*\dl\ba^\mu,\wt{b}_\nu\dr_V 
=\big\dl\ve^\mu,(\ka\slt\vp)_*(\wt{\ka}\slt\wt{\vp})^*e_\nu\big\dr_E 
=\dl\ve^\mu,\vp_{\wt{\ka}\ka}e_\nu\dr_E 
=:(\vp_{\wt{\ka}\ka})_\nu^\mu 
\in C^\iy\bigl(\ka(U\cap\wt{U})\bigr). 
$$ 
Hence we infer from 
\hb{\wt{b}_\nu=\dl\ba^\mu,\wt{b}_\nu\dr_Vb_\mu} on 
\hb{U\cap\wt{U}} and \Eqref{V.pfv} that 
\beq\Label{V.pbb} 
(\ka\slt\vp)_*\wt{b}_\nu=(\vp_{\wt{\ka}\ka})_\nu^\mu e_\mu 
\qa 1\leq\nu\leq n. 
\eeq 

\smallskip 
The push-forward of the bundle metric~$h$ is the bundle metric 
\hb{(\ka\slt\vp)_*h} on 
\hb{\ka(U)\times E} defined by 
\beq\Label{V.puh} 
(\ka\slt\vp)_*h(\xi,\eta) 
:=\ka_*\bigl(h\bigl((\ka\slt\vp)^*\xi,(\ka\slt\vp)^*\eta\bigr)\bigr) 
\qa \xi,\eta\in E^{\ka(U)}.
\eeq 
Since $h$~is a smooth section of 
\hb{V^*\otimes V^*} it has a local representation with respect to the 
dual coordinate frame:
\beq\Label{V.hbb} 
h=h_{\mu\nu}\ba^\mu\otimes\ba^\nu 
\qa h_{\mu\nu}=h(b_\mu,b_\nu)\in C^\iy(U). 
\eeq 
In the following, we endow~$\BK^{r\times s}$ with the Hilbert-Schmidt norm 
by identifying it with~$\cL(\BK^s,\BK^r)$ by means of the standard bases. 
Then we call 
\hb{[h]:=[h_{\mu\nu}]\in C^\iy(U,\BK^{n\times n})} 
\emph{representation matrix} of~$h$ with respect to the local coordinate 
frame $(b_1,\ldots,b_n)$. Let $\wt{[h]}$ be the representation matrix 
of~$h$ with respect to the local coordinate frame associated with 
\hb{\wt{\ka}\slt\wt{\vp}}. It follows from \Eqref{V.pbb} that 
\beq\Label{V.hlc} 
\ka_*\wt{[h]} 
=[\vp_{\wt{\ka}\ka}]^\top\ka_*[h]\ol{[\vp_{\wt{\ka}\ka}]} 
\text{ on }\ka(U\cap\wt{U}),  
\eeq 
where $[\vp_{\wt{\ka}\ka}]$ is the representation matrix of 
\hb{\vp_{\wt{\ka}\ka}\in C^\iy\bigl(U,\cL(E)\bigr)} with respect to 
$(e_1,\ldots,e_n)$ and $a^\top$~is the transposed of the matrix~$a$. 

\smallskip 
It should also be noted that \Eqref{V.puh} implies 
\beq\Label{V.pN} 
\ka_*(|v|_h)=|(\ka\slt\vp)_*v|_{(\ka\slt\vp)_*h} 
\qa v\in\Ga\UV. 
\eeq 
Let $[h^*]$ be the representation matrix of~$h^*$ with respect to the dual 
coordinate frame on~$U$. Denote by~$[h^{\mu\nu}]$ the inverse of~$[h]$. 
It is a consequence of 
\hb{\dl b_\nu,h_\flat b_\mu\dr_{V^*}=\dl h_\flat b_\mu,b_\nu\dr_V 
   =h(b_\nu,b_\mu)=h_{\nu\mu}} that 
$$ 
h_\flat b_\mu 
=\dl b_\nu,h_\flat b_\mu\dr_{V^*}\ba^\nu 
=h_{\nu\mu}\ba^\nu 
=\ol{h_{\mu\nu}}\ba^\nu. 
$$ 
Hence 
\hb{h^\sh\ba^\nu=h^{\nu\rho}b_\rho}. This implies 
\hb{h^{*\mu\nu}=h^*(\ba^\mu,\ba^\nu)=h(h^\sh\ba^\nu,h^\sh\ba^\mu) 
   =h^{\nu\rho}\ol{h^{\mu\sa}}h_{\rho\sa}=\ol{h^{\mu\nu}}}, that is, 
\beq\Label{V.hM} 
[h]^{-1}=\ol{[h^*]}. 
\eeq 

\smallskip 
Let 
\hb{V_{\coV i}=\Vihi} be a metric vector bundle of rank~$n_i$ over~$M$, 
where 
\hb{i=1,2}. Assume $U$~is open in~$M$ and 
\hb{\ka\slt\vp_i} is a local chart for~$V_{\coV i}$ over~$U$. Denote by 
$(b_1^i,\ldots,b_{n_i}^i)$  the 
coordinate frame for $V_{\coV i}$ over~$U$ associated with 
\hb{\ka\slt\vp_i} and by 
$(\ba_i^1,\ldots,\ba_i^{n_i})$ its dual frame. Suppose 
\hb{a\in\Ga\bigl(U,\Hom\VeVz\bigr)}. Then 
\beq\Label{V.A} 
a=a_{\nu_1}^{\nu_2}b_{\nu_2}^2\otimes\ba_1^{\nu_1} 
\qa a_{\nu_1}^{\nu_2}=\dl\ba_2^{\nu_2},ab_{\nu_1}^1\dr_{V_2}\in\BK^U. 
\eeq 
Hence, given 
\hb{u_i=u_i^{\nu_i}b_{\nu_i}^i\in\Ga\UVi}, it follows from \Eqref{V.hbb} 
that 
$$ 
h_2(au_1,u_2) 
=a_{\nu_1}^{\nu_2}h_{2,\nu_2\wt{\nu}_2}u_1^{\nu_1}\ol{u_2^{\wt{\nu}_2}}. 
$$ 
For the adjoint section 
\hb{a^*=a_{\nu_2}^{*\nu_1}b_{\nu_1}^1\otimes\ba_2^{\nu_2}
   \in\Ga\bigl(U,\Hom\VzVe\bigr)} we find analogously 
$$ 
h_1(u_1,a^*u_2)=\ol{a_{\wt{\nu}_2}^{*\wt{\nu}_1}}h_{1,\nu_1\wt{\nu}_1} 
u_1^{\nu_1}\ol{u_2^{\wt{\nu}_2}}. 
$$ 
From 
\hb{h_2(au_1,u_2)=h_1(u_1,a^*u_2)} for all $u_i$ in $\Ga\UVi$ we thus get 
\hb{a_{\nu_1}^{\nu_2}h_{2,\nu_2\wt{\nu}_2} 
   =\ol{a_{\wt{\nu}_2}^{*\wt{\nu}_1}}h_{1,\nu_1\wt{\nu}_1}}. Hence it 
follows from \Eqref{V.hM} 
\beq\Label{V.Aad} 
a_{\nu_2}^{*\nu_1} 
=h_1^{*\nu_1\wt{\nu}_1} 
\,\ol{a_{\wt{\nu}_1}^{\wt{\nu}_2}}\,\ol{h_{2,\wt{\nu}_2\nu_2}} 
\qa 1\leq\nu_i\leq n_i.  
\eeq 

\smallskip 
The following well-known basic examples of vector bundles are included for 
later reference and to fix notation.
\begin{examples}\LabelT{exa-V.ex}
\hh{(a)}
(Trivial bundles)\quad  
Consider the trivial vector bundle 
\hb{V=\Vh:=\bigl(M\times E,\prsn_E\bigr)} with the usual identification 
of the inner product of~$E$ with the bundle metric 
\hb{M\times E}. For any local chart~$\ka$ of~$M$, the 
\emph{trivial bundle chart over}~$\ka$ is given by 
\hb{\ka\slt 1_E}. Thus 
\hb{(\ka\slt 1_E)_*v=\ka_*v} for 
\hb{v\in\Ga\bigl(\dom(\ka),M\times E)=E^{\dom(\ka)}}. 

\smallskip
\hh{(b)} 
(Tangent bundles)\quad 
Let 
\hb{M=\Mg} be an \hbox{$m$-dimensional} Riemannian manifold. 
Throughout this paper we denote by~$TM$ the tangent bundle if 
\hb{\BK=\BR} and the complexified tangent  bundle if 
\hb{\BK=\BC}. Then~$g$, respectively its complexification, 
is a bundle metric on~$TM$ (also denoted by~$g$ if 
\hb{\BK=\BC}). Thus 
$$ 
T^*M:=(TM)^*=(T^*M,g^*) 
\npb 
$$ 
is the (complexified, if 
\hb{\BK=\BC}) cotangent bundle of~$M$. 

\smallskip 
We use~$\BK^m$ as the model fiber for~$TM$ and choose for 
$(e_1,\ldots,e_m)$ the standard basis 
\hb{e_j^i=\da_j^i}, 
\ \hb{1\leq i,j\leq m}. Furthermore, 
\hb{\prsn=\prsn_{\BK^m}} is the Euclidean (Hermitean) 
inner product on~$\BK^m$ and 
\hb{\vsdot=\vsdot_{\BK^m}} the corresponding norm. We identify~$(\BK^m)^*$ 
with~$\BK^m$ by means of the duality pairing  
\beq\Label{V.Kdu} 
\dl \eta,\xi\dr=\dl\eta,\xi\dr_{\BK^m}:=\eta_i\xi^j 
\qa \eta=\eta_i\ve^i 
\qb \xi=\xi^je_j, 
\npb 
\eeq 
so that 
\hb{\ve^i=e_i} for 
\hb{1\leq i\leq m}. 

\smallskip 
Suppose $\ka$~is a local chart for~$M$ and set 
\hb{U:=\dom(\ka)}. Denote by 
\hb{T\ka\sco T_UM=(TM)_U\ra\ka(U)\times\BK^m} the (complexified, if 
\hb{\BK=\BC}) tangent map of~$\ka$. Then 
\hb{\ka\slt T\ka} is a local chart for~$TM$ over~$U$, the 
\emph{canonical chart} for~$TM$ over~$\ka$. It is completely determined 
by~$\ka$. For this reason 
\hb{(\ka\slt T\ka)_*v} is denoted, as usual, by~$\ka_*v$ for 
\hb{v\in\Ga(U,TM)}. Then the push-forward   
\hb{\bigl(\ka\slt(T\ka)^{-\top}\bigr)_*w} of a covector field 
\hb{w\in\Ga(U,T^*M)} is the usual push-forward of~$w$, denoted 
by~$\ka_*w$ also. 

\smallskip 
Note that the bundle transition map for the coordinate change 
\hb{(\wt{\ka}\slt T\wt{\ka})\circ(\ka\slt T\ka)^{-1}} equals 
\hb{\pl_x(\wt{\ka}\circ\ka^{-1})}, where $\pl_x$~denotes the (Fr\'echet) 
derivative (on~$\BR^m$). 

\smallskip 
The coordinate frame for~$TM$ on~$U$ associated with~$\ka$, that is, with 
\hb{\ka\slt T\ka}, equals 
\hb{(\pl/\pl x^1,\ldots,\pl/\pl x^m)}. Its dual frame is 
\hb{(dx^1,\ldots,dx^m)}. The representation matrix of~$g$ with respect 
to this frame is the \emph{fundamental matrix} 
\hb{[g_{ij}]\in C^\iy\UKmm} of~$M$ on~$U$. 

\smallskip 
For abbreviation, we set 
\hb{\cT M:=C^\iy\MTM} and 
\hb{\cT^*M:=C^\iy\MTsM}. Then~$\cT M$, respectively~$\cT^*M$, is 
the $C^\iy(M)$ module of all (complexified, if 
\hb{\BK=\BC}) smooth vector, respectively covector, fields 
on~$M$.\hfill$\Box$\vskip.5\baselineskip 
\end{examples} 
Let 
\hb{V_{\coV i}=\Vihi} be a metric vector bundle over~$M$ for 
\hb{i=1,2}. Then the dual 
\hb{(V_1\otimes V_2)^*} of the tensor product 
\hb{V_1\otimes V_2} is identified with 
\hb{V_1^*\otimes V_2^*} by means of the duality pairing 
\hb{\pw_{V_1\otimes V_2}} defined by 
\beq\Label{V.Tdu}
\dl v_1^*\otimes v_2^*,v_1\otimes v_2\dr_{V_1\otimes V_2} 
:=\dl v_1^*,v_1\dr_{V_1}\dl v_2^*,v_2\dr_{V_2} 
\qa (v_i^*,v_i)\in\Ga(M,V_{\coV i}^* \oplus V_{\coV i}). 
\eeq 
By 
\hb{h_1\otimes h_2} we denote the bundle metric for  
\hb{V_1\otimes V_2}, given by 
\beq\Label{V.Tbp} 
h_1\otimes h_2(v_1\otimes v_2,w_1\otimes w_2) 
:=h_1(v_1,w_1)h_2(v_2,w_2) 
\qa (v_i,w_i)\in\Ga(M,V_{\coV i}\oplus V_{\coV i}).  
\npb 
\eeq 
We always equip 
\hb{V_1\otimes V_2} with this metric. 

\smallskip 
Suppose that $\ka$~is a local chart for~$M$ and 
\hb{\ka\slt\vp_i} is a local chart for~$V_{\coV i}$ over~$\dom(\ka)$. Then 
\hb{\ka\slt(\vp_1\otimes\vp_2)} denotes the local chart for 
\hb{V_1\otimes V_2} over~$\dom(\ka)$ induced by 
\hb{\ka\slt\vp_i}, 
\ \hb{i=1,2}, that is, 
\beq\Label{V.Tpf} 
\bigl(\ka\slt(\vp_1\otimes\vp_2)\bigr)_*(v_1\otimes v_2) 
=(\ka\slt\vp_1)_*v_1\otimes(\ka\slt\vp_2)_*v_2 
\qa (v_1,v_2)\in\Ga(M,V_1\oplus V_2). 
\npb 
\eeq 
It is obvious how these concepts generalize to tensor products of more than 
two vector bundles over~$M$. 

\smallskip 
A~\emph{connection} on~$V$ is a map 
$$ 
\na\sco\cT M\times C^\iy\MV\ra C^\iy\MV 
\qa \Xv\mt\naX v 
$$ 
which is $C^\iy(M)$ linear in the first argument, additive in its second, 
and satisfies the `product rule' 
\beq\Label{V.Nfv} 
\naX(fv)=(Xf)v+f\naX v 
\qa X\in\cT M 
\qb v\in C^\iy\MV 
\qb f\in C^\iy(M), 
\eeq 
where 
\hb{Xf:=df(X)=\dl df,X\dr:=\dl df,X\dr_{TM}}. Equivalently, $\na$~is 
considered as a \hbox{$\BK$ linear} map,  
$$ 
\na\sco C^\iy\MV\ra\cT^*M\otimes C^\iy\MV, 
$$ 
called \emph{covariant derivative}, defined by 
\beq\Label{V.NXv} 
\dl\na v,X\otimes v^*\dr_{TM\otimes V^*} 
=\dl v^*,\naX v\dr_V 
\qa v^*\in C^\iy\MVs 
\qb v\in C^\iy\MV 
\qb X\in\cT M, 
\eeq 
and satisfying the product rule. Here and in similar situations, $TM$~is 
identified with the `real' subbundle of the complexification 
\hb{TM+\imi TM} if 
\hb{\BK=\BC}. (In other words: We consider `real derivatives' of 
complex-valued sections.)

\smallskip 
A~connection is \emph{metric} if it satisfies 
\beq\Label{V.Xh} 
Xh\vw=h(\naX v,w)+h(v,\naX w) 
\qa X\in\cT M 
\qb v,w\in C^\iy\MV. 
\eeq 
Let $\na$  be a metric connection on~$V$. Then we define a connection 
on~$V^*$, again denoted by~$\na$, by 
\beq\Label{V.Ddu}
\dl\naX v^*,v\dr_V:=X\dl v^*,v\dr_V-\dl v^*,\naX v\dr_V 
\eeq 
for 
\hb{v^*\in C^\iy\MVs}, 
\hb{v\in C^\iy\MV}, and  
\hb{X\in\cT M}. It follows for 
\hb{v,w\in C^\iy\MV} and 
\hb{X\in\cT M} that, due to \Eqref{V.Xh}, 
$$ 
\bal 
Xh\vw 
&=X\dl h_\flat w,v\dr_V 
 =\bigl\dl\naX(h_\flat w),v\bigr\dr_V 
 +\dl h_\flat w,\naX v\dr_V\\
&=\bigl\dl\naX(h_\flat w),v\bigr\dr_V 
 +h(\naX v,w) 
 =\bigl\dl\naX(h_\flat w),v\bigr\dr_V 
 +Xh\vw-h(v,\naX w)\\ 
&=\bigl\dl\naX(h_\flat w),v\bigr\dr_V 
 +Xh\vw
 -\bigl\dl h_\flat(\naX w),v\bigr\dr_V. 
\eal 
$$ 
This and 
\hb{h^\sh=(h_\flat)^{-1}} imply 
$$ 
\na\circ h_\flat=h_\flat\circ\na 
\qb h^\sh\circ\na=\na\circ h^\sh. 
$$ 
Consequently, 
$$ 
\bal 
Xh^*\vsws 
&=Xh(h^\sh w^*,h^\sh v^*) 
 =h(h^\sh\naX w^*,h^\sh v^*)+h(h^\sh w^*,h^\sh\naX v^*)\\ 
&=h^*(\naX v^*,w^*)+h^*(v^*,\naX w^*) 
\eal 
\npb 
$$ 
for 
\hb{v^*,w^*\in C^\iy\MVs}. This shows that $\na$~is a 
metric connection on~$(V^*,h^*)$. 

\smallskip 
Let $\Vihi$ be a metric vector bundle over~$M$ for 
\hb{i=1,2}. Suppose $\na_{\cona i}$~is a metric connection 
on~$V_{\coV i}$. Then 
\beq\Label{V.Ntp} 
\naX(v_1\otimes v_2) 
:=\na_{\cona1X}v_1\otimes v_2+v_1\otimes\na_{\cona2X}v_2 
\qa v_i\in C^\iy(M,V_{\coV i}) 
\qb X\in\cT M, 
\eeq 
defines a  metric connection 
\hb{\na=\na(\na_{\cona1},\na_{\cona2})} on 
\hb{V_1\otimes V_2}, the connection \emph{induced} by $\na_{\cona1}$ 
and~$\na_{\cona2}$. In the particular case where either 
\hb{V_2=V_1} or 
\hb{V_2=V_1^*} and 
\hb{\na_{\cona2}=\na_{\cona1}}, we write again~$\na_{\cona1}$ for 
$\na(\na_{\cona1},\na_{\cona1})$. 
 
\smallskip 
Let $\na$ be a connection on~$V$. Suppose 
\hb{\ka\slt\vp} is a local chart for~$V$ over~$U$. The 
\emph{Christoffel symbols}~$\Chrz$, 
\ \hb{1\leq i\leq m}, 
\ \hb{1\leq\mu,\nu\leq n}, of~$\na$ with respect to 
\hb{\ka\slt\vp} are defined by 
\beq\Label{V.Nb} 
\na_{\pl/\pl x^i}b_\mu=\Chrz b_\nu. 
\eeq 
Here and in similar situations, it is understood that 
Latin indices run from~$1$ to~$m$ and Greek ones from~$1$ to~$n$. 
It follows 
\beq\Label{V.Nv} 
\na v=\Bigl(\frac{\pl v^\nu}{\pl x^i}+\Chrz v^\mu\Bigr)\,dx^i\otimes b_\nu 
\qa v=v^\nu b_\nu\in C^\iy\UV. 
\eeq 
Let $V_1$ and~$V_2$ be metric vector bundles over~$M$ with metric 
connections $\na_{\cona1}$ and~$\na_{\cona2}$, respectively. For a smooth 
section~$a$ of $\Hom\VeVz$ we define 
\beq\Label{V.VaV}  
(\na_{\cona12}a)u:=\na_{\cona2}(au)-a\na_{\cona1}u 
\qa u\in C^\iy\MVe. 
\eeq 
Then $\na_{\cona12}$~is a metric connection on $\Hom\VeVz$, the one 
\emph{induced} by $\na_{\cona1}$ and~$\na_{\cona2}$, where $\Hom\VeVz$ is 
endowed with the (fiber-wise defined) Hilbert-Schmidt inner product. It is 
verified that this definition is consistent with \Eqref{V.A} and 
\Eqref{V.Ntp}. Hence we also write $\na(\na_{\cona1},\na_{\cona2})$ 
for~$\na_{\cona12}$. 
\section{Uniform Regularity}\LabelT{sec-U}%
Let $M$ be an \hbox{$m$-dimensional} manifold. We set 
\hb{Q:=(-1,1)\is\BR}. If $\ka$~is a local chart for~$M$, then we 
write~$U_{\coU\ka}$ for the corresponding coordinate patch~$\dom(\ka)$.
A~local chart~$\ka$ is \emph{normalized} if
\hb{\ka(U_{\coU\ka})=Q^m} whenever
\hb{U_{\coU\ka}\is\ci{M}}, the interior of~$M$, whereas
\hb{\ka(U_{\coU\ka})=Q^m\cap\BH^m} if
\hb{U_{\coU\ka}\cap\pl M\neq\es}. We put
\hb{Q_\ka^m:=\ka(U_{\coU\ka})} if $\ka$~is normalized.

\smallskip
An atlas~$\gK$ for~$M$ has \emph{finite multiplicity} if there exists
\hb{k\in\BN} such that any intersection of more than $k$ coordinate
patches is empty. In this case 
$$ 
\gN(\ka):=\{\,\wt{\ka}\in\gK 
\ ;\ U_{\coU\wt{\ka}}\cap U_{\coU\ka}\neq\es\,\}
$$
has cardinality~%
\hb{\leq k} for each 
\hb{\ka\in\gK}. An atlas is \emph{uniformly shrinkable} if it consists of
normalized charts and there exists
\hb{r\in(0,1)} such that
\hb{\big\{\,\ka^{-1}(rQ_\ka^m)\ ;\ \ka\in\gK\,\big\}} is a cover of~$M$.

\smallskip
Given an open subset~$X$ of $\BR^m$ or~$\BH^m$ and a Banach space~$\cX$
over~$\BK$, we write
\hb{\Vsdot_{k,\iy}} for the usual norm of $BC^k\XcX$, the Banach space of
all
\hb{u\in C^k\XcX} such that $|\pa u|_\cX$~is uniformly bounded for
\hb{\al\in\BN^m} with
\hb{|\al|\leq k} (see Section~\ref{sec-PH}). 

\smallskip
By~$c$ we denote constants
\hb{\geq1} whose numerical value may vary from occurrence to occurrence;
but $c$~is always independent of the free variables in a given formula,
unless an explicit dependence is indicated.

\smallskip
Let $S$ be a nonempty set. On~$\BR^S$ we introduce an equivalence relation~%
\hb{{}\sim{}} by setting
\hb{f\sim g} iff there exists
\hb{c\geq1} such that
\hb{f/c\leq g\leq cf}. Inequalities between bundle metrics have to be 
understood in the sense of quadratic forms. 

\smallskip 
An atlas~$\gK$ for~$M$ is \emph{uniformly regular} if 
\begin{equation}\Label{U.K} 
\bal
\rm{(i)}  \qquad    &\gK\text{ is uniformly shrinkable and has finite
                     multiplicity.}\\
\rm{(ii)} \qquad    &\|\wt{\ka}\circ\ka^{-1}\|_{k,\iy}\leq c(k),
                     \ \ \ka,\wt{\ka}\in\gK,
                     \ \ k\in\BN.\\
\eal 
\npb 
\end{equation} 
In (ii) and in similar situations it is understood that only 
\hb{\ka,\wt{\ka}\in\gK} with 
\hb{U_{\coU\ka}\cap U_{\coU\wt{\ka}}\neq\es} are being considered. 
Two uniformly regular atlases $\gK$ and~$\wt{\gK}$ are \emph{equivalent}, 
\hb{\gK\approx\wt{\gK}}, if 
\begin{equation}\Label{U.Keq} 
\bal
\rm{(i)}  \qquad    &\card\{\,\wt{\ka}\in\wt{\gK}
                     \ ;\ U_{\coU\wt{\ka}}\cap U_{\coU\ka}\neq\es\,\}\leq c,
                     \ \ \ka\in\gK.\\
\rm{(ii)}\qquad    &\|\wt{\ka}\circ\ka^{-1}\|_{k,\iy}\leq c(k),
                     \ \ \ka\in\gK,
                     \ \ \wt{\ka}\in\wt{\gK},
                     \ \ k\in\BN.
\eal
\end{equation} 

\smallskip 
Let $V$ be a vector bundle of rank~$n$ over~$M$ with model fiber~$E$. 
Suppose $\gK$~is an atlas for~$M$ and 
\hb{\ka\slt\vp} is for each 
\hb{\ka\in\gK} a~local chart for~$V$ over~$U_{\coU\ka}$. Then 
\hb{\gK\slt\Phi:=\{\,\ka\slt\vp\ ;\ \ka\in\gK\,\}} is an 
\emph{atlas for}~$V$ \emph{over}~$\gK$. It is \emph{uniformly regular} if 
\begin{equation}\Label{U.Ph}
\bal
\rm{(i)} \qquad    &\gK\text{ is uniformly regular;}\\
\rm{(ii)}\qquad    &\|\vp_{\ka\wt{\ka}}\|_{k,\iy}\leq c(k),
                     \ \ \ \ka\slt\vp,  
                     \wt{\ka}\slt\wt{\vp}\in\gK\slt\Phi,
                     \ \ k\in\BN, 
\eal
\end{equation} 
where $\vp_{\ka\wt{\ka}}$~is the bundle transition map corresponding to the 
coordinate change 
\hb{(\wt{\ka}\slt\wt{\vp})\circ(\ka\slt\vp)^{-1}}. Two atlases 
\hb{\gK\slt\Phi} and 
\hb{\wt{\gK}\slt\wt{\Phi}} for~$V$ over $\gK$ and~$\wt{\gK}$, respectively, 
are equivalent, 
\hb{\gK\slt\Phi\approx\wt{\gK}\slt\wt{\Phi}}, if 
\begin{equation}\Label{U.peq}
\bal
\rm{(i)} \qquad    &\gK\approx\wt{\gK};\\
\rm{(ii)}\qquad    &\|\vp_{\ka\wt{\ka}}\|_{k,\iy}\leq c(k),
                     \ \ \ka\slt\vp\in\gK\slt\Phi,
                     \ \ \wt{\ka}\slt\wt{\vp}\in\wt{\gK}\slt\wt{\Phi},
                     \ \ k\in\BN. 
\eal
\end{equation} 

\smallskip 
Suppose $h$~is a bundle metric for~$V$. Let 
\hb{\gK\slt\Phi} be a uniformly regular atlas for~$V$ over~$\gK$. Then 
$h$~is \emph{uniformly regular over} 
\hb{\gK\slt\Phi} if 
\begin{equation}\Label{U.h}
\bal
\rm{(i)} \qquad    &(\ka\slt\vp)_*h\sim\prsn_E, 
                    \ \ \ka\slt\vp\in\gK\slt\Phi;\\ 
\rm{(ii)}\qquad    &\|(\ka\slt\vp)_*h\|_{k,\iy}\leq c(k),
                     \ \ \ka\slt\vp\in\gK\slt\Phi,
                     \ \ k\in\BN. 
\eal 
\end{equation} 
Let 
\hb{[h]_{\ka\slt\vp}=[h_{\mu\nu}]_{\ka\slt\vp}} be the 
representation matrix of~$h$ with respect to the local coordinate frame 
associated with 
\hb{\ka\slt\vp}. Then it follows from \Eqref{V.hbb} that 
\beq\Label{U.kh} 
\ka_*\bigl([h]_{\ka\slt\vp}\bigr)=[\ka_*h_{\mu\nu}] 
=\bigl[(\ka\slt\vp)_*h\bigr]. 
\eeq 
Hence \Eqref{U.h}(i) is equivalent to 
$$ 
|\za|^2/c\leq\ka_*h_{\mu\nu}(x)\za^\mu\ol{\za^\nu}\leq c\,|\za|^2, 
\qa x\in Q_\ka^m 
\qb \za\in\BK^n 
\qb \ka\slt\vp\in\gK\slt\Phi. 
$$ 
If 
\hb{\gK\slt\Phi\approx\wt{\gK}\slt\wt{\Phi}} 
and $h$~is uniformly regular over~$\gK$, then we see from \Eqref{V.php} and 
\Eqref{V.hlc} that $h$~is uniformly regular over~$\wt{\gK}$. 

\smallskip 
Assume $\na$~is a connection on~$V$. Let 
\hb{\gK\slt\Phi} be an atlas for~$V$ over~$\gK$. For 
\hb{\ka\slt\vp\in\gK\slt\Phi} we denote by 
\hb{\Chrz[\ka\slt\vp]} the Christoffel symbols of~$\na$ with respect 
to the coordinate frame for~$V$ over~$U_{\coU\ka}$ induced by 
\hb{\ka\slt\vp}. Then $\na$~is \emph{uniformly regular over} 
\hb{\gK\slt\Phi} if 
\[
\bal
\rm{(i)} \qquad    &\gK\slt\Phi\text{ is uniformly regular;}\\
\rm{(ii)}\qquad    &\big\|\ka_*\bigl(\Chrz[\ka\slt\vp]\bigr)\big\|_{k,\iy}
                    \leq c(k),
                    \ \ 1\leq i\leq m, 
                    \ \ 1\leq\mu,\nu\leq n, 
                    \ \ \ka\slt\vp\in\gK\slt\Phi,
                    \ \ k\in\BN.  
\eal 
\] 
Suppose $\na$~is uniformly regular over 
\hb{\gK\slt\Phi} and 
\hb{\wt{\gK}\slt\wt{\Phi}\approx\gK\slt\Phi}. Then it follows from 
\Eqref{V.pbb}, \,\Eqref{V.Nv}, \,\Eqref{U.Keq}, and \Eqref{U.peq} 
that $\na$~is uniformly regular over 
\hb{\wt{\gK}\slt\wt{\Phi}}. 

\smallskip 
A~\emph{uniformly regular structure} for~$M$ is a maximal family of 
equivalent uniformly regular atlases for it. We say $M$~is a 
\hh{uniformly regular manifold} if it is endowed with a uniformly 
regular structure. In this case it is understood that 
each uniformly regular atlas under consideration belongs to this 
uniformly regular structure. 

\smallskip 
Let $M$ be uniformly regular and $V$ a~vector bundle over~$M$. 
A~\emph{uniformly regular bundle structure} for~$V$ is a 
maximal family of equivalent uniformly regular atlases for~$V$. Then $V$~is 
a \emph{uniformly regular vector bundle over}~$M$, if it is equipped 
with a uniformly regular bundle structure. Again it is understood 
that in this case each atlas for~$V$ belongs to the given uniformly regular 
bundle structure. A~\emph{uniformly regular metric vector bundle} is a 
uniformly regular vector bundle endowed with a uniformly regular 
bundle metric. By~a \hh{fully uniformly regular vector bundle} 
\hb{V=(V,h_V,\na_{\cona V})} over~$M$ we mean a uniformly regular vector 
bundle~$V$ over~$M$ equipped with a uniformly regular bundle metric~$h_V$ 
and a uniformly regular metric connection~$\na_{\cona V}$. 

\smallskip 
As earlier, it is the main purpose of the following examples to fix 
notation and to prepare the setting for further investigations. 
\begin{examples}\LabelT{exa-U.ex}
\hh{(a)}
(Trivial bundles)\quad  
Let 
\hb{E=\bigl(E,\prsn_E\bigr)} be an \hbox{$n$--dimensional} Hilbert space. 
Suppose $M$~is a uniformly regular manifold. It is obvious from 
Example~\ref{exa-V.ex}(a) that the trivial bundle 
\hb{M\times E} is uniformly regular over~$M$ and 
\hb{\prsn_E} is a uniformly regular bundle metric. 

\smallskip 
We consider~$E$ as a manifold of dimension~$n$ if 
\hb{\BK=\BR}, and of dimension~$2n$ if 
\hb{\BK=\BC} (using the standard identification of 
\hb{\BC=\BR+\imi\BR} with~$\BR^2$) whose 
smooth structure is induced by the trivial chart~$1_E$. We identify~$TE$ 
canonically with 
\hb{E\times E}. Then 
\hb{Tv\sco TM\ra TE=E\times E}, the tangential of 
\hb{v\in C^\iy\ME}, is well-defined. We set 
$$ 
d_X\,v:=\pro_2\circ Tv(X) 
\qa X\in\cT M 
\qb v\in C^\iy\ME. 
$$ 
Then 
$$ 
d\sco\cT M\times C^\iy\ME\ra C^\iy\ME 
\qb \Xv\mt d_X\,v 
\npb 
$$ 
is a connection on 
\hb{M\times E}, the \hbox{$E$-\emph{valued}} \emph{differential} on~$M$. 

\smallskip 
Let $(e_1,\ldots,e_n)$ be a basis for~$E$ and use the same symbol for 
the constant frame 
\hb{p\mt(e_1,\ldots,e_n)} of 
\hb{M\times E}. Then it follows that 
$$ 
df=df^\nu e_\nu 
\qa f=f^\nu e_\nu\in C^\iy\ME. 
\npb 
$$ 
Thus, since all Christoffel symbols are identically zero,  
$d$~is trivially uniformly regular. 

\smallskip
\hh{(b)} 
(Subbundles)\quad 
Let $V$ be a vector bundle of rank~$n$ over a manifold~$M$, 
endowed with a bundle metric~$h$ and a metric connection~$\na$. 
Suppose $W$~is a subbundle of rank~$\ell$. Denote by 
\hb{\ia\sco W\hr V} the canonical injection. Let 
\hb{h_W:=\ia^*h} be the pull-back metric on~$W$.  We write~$P$ for the  
the orthogonal projection onto~$W$ in~$V$. Then 
\hb{P\in C^\iy\bigl(M,\Hom\VV\bigr)} and it is verified that 
$$ 
\na_{\cona W}\sco\cT M\times C^\iy\MW\ra C^\iy\MW 
\qb \Xw\mt P\naX\bigl(\ia(w)\bigr) 
\npb 
$$ 
is a metric connection on~$(W,h_W)$, the one induced by~$\na$. 

\smallskip 
Let $E$ be a model fiber of~$V$ and $(e_1,\ldots,e_n)$ a~basis for it. 
Suppose $V$~is uniformly regular and there exists an atlas 
\hb{\gK\slt\Phi} for~$V$ such that 
\hb{(\ka\slt\vp)^*(e_1,\ldots,e_\ell)} is for each 
\hb{\ka\in\gK} a~frame for~$W$ over~$U_{\coU\ka}$. Then it is checked that 
\hb{W=(W,h_W,\na_{\cona W})} is a fully uniformly regular vector bundle 
over~$M$. 

\smallskip 
Suppose 
\hb{V_{\coV i}=(V_{\coV i},h_i,\na_{\cona i})}, 
\ \hb{i=1,2}, are fully uniformly regular vector bundles over~$M$. Set 
$$ 
(h_1\oplus h_2)(v_1\oplus v_2,\ \wt{v}_1\oplus\wt{v}_2) 
:=h_1(v_1,\wt{v}_1)+h_2(v_2,\wt{v}_2) 
\qa (v_i,\wt{v}_i)\in\Ga(M,V_{\coV i}\oplus V_{\coV i}), 
$$ 
and 
$$ 
(\na_{\cona1}\oplus\na_{\cona2})(v_1\oplus v_2) 
:=\na_{\cona1}v_1\oplus\na_{\cona2}v_2 
\qa (v_1,v_2)\in C^\iy(M,V_1\oplus V_2). 
$$ 
Then 
\hb{(V_1\oplus V_2,\ h_1\oplus h_2,\ \na_{\cona1}\oplus\na_{\cona2})} is a 
fully uniformly regular vector bundle over~$M$. Furthermore, 
$V_{\coV i}$~is for 
\hb{i=1,2} a~fully uniformly regular subbundle of~$V$. 

\smallskip
\hh{(c)} 
(Riemannian manifolds)\quad 
Let 
\hb{M=\Mg} be an \hbox{$m$-dimensional} Riemannian manifold. We denote by 
\hb{g_m=(dx^1)^2+\cdots+(dx^m)^2} the Euclidean metric on~$\BR^m$ and use 
the same symbol for its complexification 
as well as for the restriction thereof to open subsets of $\BR^m$ 
and~$\BH^m$. Then $M$~is a \hh{uniformly regular Riemannian manifold}, 
if $TM$~is uniformly regular and $g$~is a uniformly regular bundle metric 
on~$TM$. It follows from Example~\ref{exa-V.ex}(b) that $M$~is a uniformly 
regular Riemannian manifold iff 
\begin{equation}\Label{U.Rr}
\bal
\rm{(i)}  \qquad    &M\text{ is uniformly regular};\\
\rm{(ii)} \qquad    &\ka_*g\sim g_m,
                     \ \ \ka\in\gK;\\
\rm{(iii)}\qquad    &\|\ka_*g\|_{k,\iy}\leq c(k),
                     \ \ \ka\in\gK,
                     \ \ k\in\BN,  
\eal 
\npb 
\end{equation} 
for some uniformly regular atlas~$\gK$ for~$M$. Of course, 
\hb{\ka_*g:=(\ka\slt T\ka)_*g} in conformity with standard usage. 

\smallskip 
We denote by~$\na_{\cona g}$ the (complexified, if 
\hb{\BK=\BC}) Levi-Civita connection for~$M$, that is, for~$TM$. Its 
Christoffel symbols with respect to the coordinate frame 
\hb{(\pl/\pl x^1,\ldots,\pl/\pl x^m)} over~$U_{\coU\ka}$ admit the 
representation 
\beq\Label{U.Ch} 
2\Ga_{\kern-1pt ij}^k 
=g^{k\ell}(\pl_ig_{\ell j}+\pl_jg_{\ell i}-2\pl_\ell g_{ij}), 
\eeq 
where 
\hb{\pl_i:=\pl/\pl x^i}. From this and \Eqref{U.Rr}(ii) and~(iii) 
it follows that $\na_{\cona g}$~is uniformly regular if $\Mg$~is a 
uniformly regular Riemannian manifold. In addition, $\na_{\cona g}$~is 
metric and 
\hb{\Ga_{\kern-1pt ij}^k=\Ga_{\kern-1pt ji}^k}. 

\smallskip
\hh{(d)}\quad 
Every compact Riemannian manifold is a uniformly regular 
Riemannian manifold. 

\smallskip
\hh{(e)}\quad 
It has been shown in Example~2.1(c) of~\cite{Ama12b} that  
\hb{\BR^m=\Rmgm} and 
\hb{\BH^m=\Hmgm} are uniformly regular Riemannian manifolds. 

\smallskip
\hh{(f)} 
(Homomorphism bundles)\quad 
For 
\hb{i=1,2} let $\Vihi$ be a uniformly regular metric vector bundle of 
rank~$n_i$ over~$M$. We denote by $(V_{12},h_{12})$ the homomorphism bundle 
\hb{V_{12}:=\Hom\VeVz} endowed with the Hilbert-Schmidt bundle metric 
\hb{h_{12}=\prsn_{HS}}. 

\smallskip 
Assume 
\hb{\gK\slt\Phi_i} is a uniformly regular atlas for~$V_{\coV i}$, and 
$E_i$~is a model fiber for~$V_{\coV i}$ with basis 
$(e_1^i,\ldots,e_{n_i}^i)$ and dual basis 
$(\ve_i^1,\ldots,\ve_i^{n_i})$. For 
\hb{\ka\slt\vp_i\in\gK\slt\Phi_i} we define a bundle isomorphism 
$$ 
(\ka,\ka\slt\vp_i)\sco\bigl(U_{\coU\ka},(V_{12})_{U_{\coU\ka}}\bigr) 
\ra \bigl(\ka(U_{\coU\ka}),\ka(U_{\coU\ka})\times\cL\EeEz\bigr) 
$$ 
by setting 
\hb{(\ka\slt\vp_{12})a_p:=\bigl(\ka(p),\vp_{12}(p)a_p\bigr)} for 
\hb{p\in U_{\coU\ka}} and 
\hb{a_p\in(V_{12})_p}, where 
$$ 
\vp_{12}(p)a_p(x):=\vp_2(p)a_p\vp_1^{-1}(x) 
\qa x=\ka(p). 
$$ 
It follows 
$$ 
(\wt{\ka}\slt\wt{\vp}_{12})_*(\ka\slt\vp_{12})^*b 
=(\wt{\ka}\slt\wt{\vp}_2)_*(\ka\slt\vp_2)^*b(\ka\slt\vp_1)_* 
 (\wt{\ka}\slt\wt{\vp}_1)^* 
\qa b\in\cL\EeEz, 
$$ 
if 
\hb{\wt{\ka}\slt\wt{\vp}_i} belongs to a uniformly regular atlas 
for~$V_{\coV i}$. From this we deduce that 
$$ 
\gK_{12}:=\{\,\ka\slt\vp_{12} 
\ ;\ \ka\slt\vp_i\in\gK\slt\Phi_i,\ i=1,2\,\} 
$$ 
is a uniformly regular atlas for~$V_{12}$ and that any two such 
atlases are equivalent. Hence $V_{12}$~is a uniformly regular 
vector bundle over~$M$. 

\smallskip 
The coordinate frame of $V_{12}$ over~$U_{\coU\ka}$ associated with 
\hb{\ka\slt\vp_{12}} is given by 
\beq\Label{U.f12} 
\{\,b_{\nu_2}^2\otimes\ba_1^{\nu_1} 
\ ;\ 1\leq\nu_i\leq n_i,\ i=1,2\,\}, 
\eeq 
where 
$(b_1^i,\ldots,b_{n_i}^i)$ is the coordinate frame of~$V_{\coV i}$ 
over~$U_{\coU\ka}$ associated with 
\hb{\ka\slt\vp_i} and 
$(\ba_i^1,\ldots,\ba_i^{n_i})$ is its dual frame. By \Eqref{V.Aad} and 
\Eqref{U.f12} we find 
\beq\Label{U.h12} 
[h_{12}]=\bigl[h_1^{*\nu_1\wt{\nu}_1}\,\ol{h_{2,\wt{\nu}_2\nu_2}}\,\bigr]. 
\eeq 
From this, \Eqref{V.hbb}, \,\Eqref{U.h}, and \Eqref{U.kh} we deduce 
$$ 
(\ka\slt\vp_{12})_*h_{12}(a,a) 
=\ka_*h_1^{*\nu_1\wt{\nu}_1}\ka_*h_{2,\nu_2\wt{\nu}_2}a_{\nu_1}^{\nu_2} 
 \ol{a_{\wt{\nu}_1}^{\wt{\nu}_2}} 
\sim\sum_{\nu_2}\ka_*h_1^{*\nu_1\wt{\nu}_1}a_{\nu_1}^{\nu_2} 
 \ol{a_{\wt{\nu}_1}^{\nu_2}} 
\sim\sum_{\nu_1,\nu_2}a_{\nu_1}^{\nu_2} 
 \ol{a_{\nu_1}^{\nu_2}}=(a,a)_{HS}  
$$ 
for 
\hb{a\in\cL\EeEz}, as well as 
\hb{\|(\ka\slt\vp_{12})_*h_{12}\|_{k,\iy}\leq c(k)} for 
\hb{\ka\slt\vp_{12}\in\gK\slt\Phi_{12}} and 
\hb{k\in\BN}. Hence $(V_{12},h_{12})$ is a uniformly regular metric 
vector bundle over~$M$. 

\smallskip 
Suppose $\na_{\cona i}$~is a uniformly regular metric connection 
on~$V_{\coV i}$. Then it is a consequence of the consistency of 
\Eqref{V.VaV} with \Eqref{V.Ntp} that $\na_{\cona12}$~is a uniformly regular 
metric connection on~$V_{12}$. 

\smallskip
\hh{(g)} 
(Tensor products)\quad 
Let $\Vihi$, 
\ \hb{i=1,2}, be uniformly regular metric vector bundles over~$M$. Then 
it follows from \hbox{\Eqref{V.Tdu}--\Eqref{V.Tpf}}  
that 
\hb{(V_1\otimes V_2,\ h_1\otimes h_2)} is 
a uniformly regular metric vector bundle over~$M$. If $\na_{\cona i}$~is a 
uniformly regular metric connection on~$V_{\coV i}$, then we see from 
\Eqref{V.Ntp} that $\na(\na_{\cona1},\na_{\cona2})$ is a 
uniformly regular metric connection on 
\hb{V_1\otimes V_2}.\hfill$\Box$
\end{examples} 
\section{Singular Manifolds}\LabelT{sec-S}%
Let 
\hb{M=\Mg} be an \hbox{$m$-dimensional} Riemannian manifold. Suppose 
\hb{\rho\in C^\iy\bigl(M,(0,\iy)\bigr)}. Then $\rgK$ is~a \emph{singularity 
datum} for~$M$ if 
\begin{equation}\Label{S.sd} 
\kern-4pt                                  
\bal
\rm{(i)}   \qquad    &(M,g/\rho^2)\text{ is a uniformly regular Riemannian 
                      manifold.}\\
\rm{(ii)}  \qquad    &\gK\text{ is a uniformly regular atlas for $M$ 
                      which is orientation preserving if $M$ is oriented.}\\
\rm{(iii)} \qquad    &\|\ka_*\rho\|_{k,\iy}\leq c(k)\rho_\ka, 
                      \ \ \ka\in\gK,
                      \ \ k\in\BN,
                      \text{\ where }
                      \rho_\ka:=\ka_*\rho(0)=\rho\bigl(\ka^{-1}(0)\bigr).\\
\rm{(iv)}  \qquad    &\rho_\ka/c\leq\rho(p)\leq c\rho_\ka, 
                      \ \ p\in U_{\coU\ka},
                      \ \ \ka\in\gK.
\eal
\end{equation}  
Two singularity data $\rgK$ and $\wrgK$ are \emph{equivalent}, 
\hb{\rgK\approx\wrgK}, if 
\beq\Label{S.eq} 
\rho\sim\wt{\rho} 
\quad\text{and}\quad 
\gK\approx\wt{\gK}. 
\eeq 
Note that \Eqref{S.sd}(iv) and \Eqref{S.eq} imply 
\beq\Label{S.err} 
1/c\leq\rho_\ka/\rho_{\wt{\ka}}\leq c 
\qa \ka\in\gK 
\qb \wt{\ka}\in\wt{\gK} 
\qb U_{\coU\ka}\cap U_{\coU\wt{\ka}}\neq\es. 
\eeq 

\smallskip 
A~\emph{singularity structure},~$\gS(M)$, for~$M$ is a maximal family of
equivalent singularity data. A~\emph{singularity function} for~$M$ is a
function
\hb{\rho\in C^\iy\bigl(M,(0,\iy)\bigr)} such that there exists an
atlas~$\gK$ with
\hb{\rgK\in\gS(M)}. The set of all singularity functions is the
\emph{singularity type},~$\gT(M)$, of~$M$. By~a \hh{singular manifold} we
mean a Riemannian manifold~$M$ endowed with a singularity
structure~$\gS(M)$. Then $M$~is said to be \emph{singular of
type}~$\gT(M)$. If
\hb{\rho\in\gT(M)}, then it is convenient to set
\hb{[\![\rho]\!]:=\gT(M)}. 

\smallskip 
Let $M$ be singular of type~%
\hb{[\![\rho]\!]}. Then $M$~is a uniformly regular Riemannian manifold iff 
\hb{\rho\sim\mf{1}}. If 
\hb{\rho\notsim\mf{1}}, then either 
\hb{\inf\rho=0} or 
\hb{\sup\rho=\iy}, or both. Hence $M$~is not compact but has singular ends. 
It follows from~\Eqref{S.sd} that the diameter of the coordinate 
patches converges either to zero or to infinity near the singular ends in 
a manner controlled by the singularity type~$\gT(M)$.

\smallskip 
We refer to \cite{Ama12b} and~\cite{Ama12c} for examples 
of singular manifolds which are not uniformly regular Riemannian manifolds. 

\smallskip 
Throughout the rest of this paper we assume 
\bspezeqn\Label{S.ass} 
\bal 
\frame{
\begin{minipage}{336pt}
\bspezeq
\bal 
{}      
&M=\Mg\text{ is an $m$-dimensional singular manifold}.\\
&W=(W,h_W,D)\text{ is a fully uniformly regular vector bundle 
 of rank $n$ over }M.\\
&\sa,\tau\in\BN.\\
\noalign{\vskip2.5\jot}
\eal
\espezeq 
\end{minipage}} 
\eal
\espezeqn 
It follows from the preceding section that the uniform regularity of $W$, 
$h_W$, and~$D$ is independent of the particular choice of the singularity 
datum~$(\rho,\gK)$.  

\smallskip 
Henceforth, $TM$ and~$T^*M$ have to be interpreted as the complexified 
tangent and cotangent bundles, respectively, if 
\hb{\BK=\BC}. Accordingly, $\pw_{TM}$, $g$, and~$\na_{\cona g}$ are then 
the complexified duality pairing, Riemannian metric, and 
Levi-Civita connection, respectively. 

\smallskip 
As usual, 
\hb{T_\tau^\sa M=TM^{\otimes\sa}\otimes T^*M^{\otimes\tau}} is the 
$(\sa,\tau)$-tensor bundle, that is, the vector bundle of all 
\hbox{$\BK$-valued} tensors on~$M$ being contravariant of order~$\sa$ and 
covariant of order~$\tau$. In particular, 
\hb{T_0^1M=TM}, 
\ \hb{T_1^0M=T^*M}, and 
\hb{T_0^0M=M\times\BK}. Then 
$$  
V=V_\tau^\sa(W)=T_\tau^\sa\MW:=T_\tau^\sa M\otimes W 
\npb 
$$ 
is the vector bundle of \hbox{$W$-\emph{valued}} 
$(\sa,\tau)$-\emph{tensors on}~$M$. 

\smallskip 
If 
\hb{W=M\times E} with an \hbox{$n$-dimensional} Hilbert space~$E$, then we 
write~$T_\tau^\sa\ME$ for 
\hb{T_\tau^\sa(M,M\times E)} and call its elements 
\hbox{$E$-valued} $(\sa,\tau)$-tensors. Furthermore, $T_\tau^\sa\MBK$ is 
naturally identified with~$T_\tau^\sa M$. For abbreviation, we set 
$$ 
\cT_\tau^\sa\MW:=C^\iy\bigl(M,T_\tau^\sa\MW\bigr). 
\npb 
$$ 
It is the $C^\iy(M)$ module of smooth \hbox{$W$-\emph{valued}} 
$(\sa,\tau)$-\emph{tensor fields on}~$M$. 

\smallskip 
The canonical identification of~$(T_\tau^\sa M)^*$ 
with~$T_\sa^\tau M$ leads to  
\hb{T_\tau^\sa\MW^* =T_\sa^\tau\MWs} with respect to the (bundle) duality 
pairing 
$$  
\pw_V:=\pw_{T_\tau^\sa M}\otimes\pw_W. 
$$ 
We endow~$V$ with the bundle metric 
\beq\Label{S.h} 
h:=\prsn_\sa^\tau\otimes h_W, 
\npb 
\eeq 
where 
\hb{\prsn_\sa^\tau:=g^{\otimes\sa}\otimes g^{*\otimes\tau}} 
is the bundle metric on~$T_\tau^\sa M$ induced by~$g$ (denoted by 
\hb{\prsn_g} in Section~3 of~\cite{Ama12b}). 

\smallskip 
Finally, we equip~$V$ with the metric connection 
$$ 
\na:=\na(\na_{\cona g},D) 
$$ 
induced by the Levi-Civita connection of~$M$ and connection~$D$ of~$W$. 
In summary, in addition to \Eqref{S.ass}, 
\bspezeq 
\frame{
\begin{minipage}{255pt}
\bspezeq 
\bal 
V=(V,h,\na):=\bigl(T_\tau^\sa\MW, 
\ \prsn_\sa^\tau\otimes h_W,\ \na(\na_{\cona g},D)\bigr)\\
\noalign{\vskip2\jot} 
\eal 
\espezeq 
\end{minipage}} 
\espezeq 
is a standing assumption. In particular, $\na$~is a \hbox{$\BK$-linear} map 
from $\cT_\tau^\sa\MW$ into $\cT_{\tau+1}^\sa\MW$. We set 
\hb{\na^0:=\id} and 
\hb{\na^{k+1}:=\na\circ\na^k} for 
\hb{k\in\BN}. Note 
\hb{\na u=Du} for  
\hb{u\in\cT_0^0\MW= C^\iy\MW}. 
\section{Local Representations}\LabelT{sec-L}%
Although $W$~is a fully uniformly regular vector bundle over~$M$ this is 
not true for~$V$, due to the fact that $h$~involves the singular Riemannian 
metric~$g$. For this reason we have to study carefully the dependence of 
various local representations on the singularity datum. This is done in 
the present section. 

\smallskip 
For a subset~$S$ of~$M$ and a normalized atlas~$\gK$ we let 
\hb{\gK_S:=\{\,\ka\in\gK\ ;\ U_{\coU\ka}\cap S\neq\es\;\}}; hence 
\hb{\gK_\es=\es}. Then, given 
\hb{\ka\in\gK}, 
\beq\Label{L.Xk} 
\BX_\ka:=
\left\{
\bal
{}      
&\BR^m  &&\quad \text{if }\ka\in\gK\ssm\gK_{\pl M},\\
&\BH^m  &&\quad \text{otherwise},
\eal
\right.
\eeq 
considered as an \hbox{$m$-dimensional} uniformly regular Riemannian 
manifold with the Euclidean metric. Furthermore, $Q_\ka^m$~is an open 
Riemannian submanifold of~$\BX_\ka$. 

\smallskip 
Let $F$ be a finite-dimensional Hilbert space. Then, using standard 
identifications, 
$$ 
T_\tau^\sa(Q_\ka^m,F) 
=(\BK^m)^{\otimes\sa}\otimes\bigl((\BK^m)^*\bigr)^{\otimes\tau}\otimes F. 
$$ 
Of course, we identify~$(\BK^m)^*$ with~$\BK^m$ by means of \Eqref{V.Kdu}, 
but continue to denote it by~$(\BK^m)^*$ for clarity. We endow 
$T_\tau^\sa(Q_\ka^m,F)$ with the inner product 
\beq\Label{L.Tp} 
\prsn_{T_\tau^\sa(Q_\ka^m,F)}
:=\prsn_{\BK^m}^{\otimes\sa}\otimes\prsn_{(\BK^m)^*}^{\otimes\tau} 
\otimes\prsn_F. 
\eeq 
For 
\hb{\nu\in\BN^\times} we set 
\hb{\BJ_\nu:=\{1,\ldots,m\}^\nu} and denote its general point by 
\hb{(i)=(i_1,\ldots,i_\nu)}. The standard basis 
\hb{(\breve{e}_1,\ldots,\breve{e}_m)} of~$\BK^m$, that is, 
\hb{\breve{e}_j^i=\da_j^i}, and its dual basis 
\hb{(\breve{\ve}^1,\ldots,\breve{\ve}^m)} induce the \emph{standard basis} 
$$ 
\bigl\{\,\breve{e}_{(i)}\otimes\breve{\ve}^{(j)}
\ ;\ (i)\in\BJ_\sa,\ (j)\in\BJ_\tau\,\bigr\} 
$$ 
\emph{of}~$T_\tau^\sa Q_\ka^m$, where 
\hb{\breve{e}_{(i)}=\breve{e}_{i_1}\otimes\cdots\otimes\breve{e}_{i_\sa}} 
and 
\hb{\breve{\ve}^{(j)} 
   =\breve{\ve}^{j_1}\otimes\cdots\otimes\breve{\ve}^{j_\tau}}. 
Then 
$$ 
a\in T_\tau^\sa(Q_\ka^m,F) 
=\cL\bigl( \bigl((\BK^m)^*\bigr)^{\otimes\sa} 
\otimes(\BK^m)^{\otimes\tau},F\bigr) 
$$ 
has the representation matrix 
\hb{\big[a_{(j)}^{(i)}\big]\in F^{m^\sa\times m^\tau}}. We 
endow~$F^{m^\sa\times m^\tau}$ with the inner product 
$$ 
\bigl(\big[a_{(j)}^{(i)}\big] 
\bsn\big[b_{(\wt{\jmath})}^{(\wt{\imath})}\big]\bigr)_{HS,F} 
:=\sum_{(i)\in\BJ_\sa,\ (j)\in\BJ_\tau} 
\big(a_{(j)}^{(i)}\bsn b_{(j)}^{(i)}\big)_F 
$$ 
which coincides with the Hilbert-Schmidt inner product if 
\hb{F=\BK}. For abbreviation, we set 
\beq\Label{L.EstF} 
E=E_\tau^\sa=E_\tau^\sa(F):=F^{m^\sa\times m^\tau} 
\qb \prsn_E:=\prsn_{HS,F}. 
\eeq 
It follows from \Eqref{L.Tp} that 
\hb{a\mt\big[a_{(j)}^{(i)}\big]} defines an isometric isomorphism by which 
$$  
\text{we identify } 
\bigl(T_\tau^\sa(Q_\ka^m,F),\prsn_{T_\tau^\sa(Q_\ka^m,F)}\bigr) 
\text{ with } 
\bigl(E,\prsn_E\bigr). 
$$ 

\smallskip 
We assume 
\beq\Label{L.K} 
\bal
\bt\quad
&\rgK\text{ is a singularity datum for }M;\\
\bt\quad
&\gK\slt\Phi\text{ is a uniformly regular atlas for $W$ over }\gK;\\
\bt\quad
&F=\bigl(F,\prsn_F\bigr)\text{ is a model fiber for $W$ 
 with basis }(e_1,\ldots,e_n).\\
\eal
\eeq 
Suppose 
\hb{\ka\slt\vp\in\gK\slt\Phi} and 
\hb{\ka=(x^1,\ldots,x^m)}. Then 
$$ 
\ka\slt\vp_\tau^\sa\sco V_{U_{\coU\ka}}\ra Q_\ka^m\times E 
\qb v_p\mt\bigl(\ka(p),\vp_\tau^\sa(p)v_p\bigr) 
\qb v_p\in V_{\coV p} 
\qb p
\in U_{\coU\ka}, 
$$ 
the \emph{local chart for~$V$ over~$U_{\coU\ka}$ induced by} 
\hb{\ka\slt\vp}, is defined by 
$$ 
\vp_\tau^\sa(p)v_p 
:=(T_p\ka)X_p^1\otimes\cdots\otimes(T_p\ka)X_p^\sa 
\otimes(T_p\ka)^{-\top}\al_{1,p} 
\otimes\cdots\otimes(T_p\ka)^{-\top}\al_{\tau,p} 
\otimes\vp(p)w_p 
$$  
for 
\hb{v_p=X_p^1\otimes\cdots\otimes X_p^\sa\otimes\al_{1,p} 
   \otimes\cdots\otimes\al_{\tau,p}\otimes w_p\in T_\tau^\sa\MW_p} with 
\hb{X_p^i\in T_pM}, 
\ \hb{\al_{j,p}\in T_p^*M}, and $w_p$ belonging to~$W_{\coW p}$. 

\smallskip 
Set 
$$ 
\frac\pl{\pl x^{(i)}} 
:=\frac\pl{\pl x^{i_1}}\otimes\cdots\otimes\frac\pl{\pl x^{i_\sa}} 
\qb dx^{(j)}:=dx^{j_1}\otimes\cdots\otimes dx^{j_\tau} 
\qa (i)\in\BJ_\sa 
\qb (j)\in\BJ_\tau. 
$$ 
Furthermore, let 
$(b_1,\ldots,b_n)$ be the coordinate frame for~$W$ over~$U_{\coU\ka}$ 
associated with 
\hb{\ka\slt\vp} and $(\ba^1,\ldots,\ba^n)$ its dual frame. Then 
\beq\Label{L.bas} 
\Bigl\{\,
\frac\pl{\pl x^{(i)}}\otimes dx^{(j)}\otimes b_\nu 
\ ;\ (i)\in\BJ_\sa,\ (j)\in\BJ_\tau,\ 1\leq\nu\leq n\,\Bigr\} 
\eeq 
is the coordinate frame 
for~$V$ over~$U_{\coU\ka}$ associated with 
\hb{\ka\slt\vp_\tau^\sa}. Hence 
\hb{v\in\Ga(U_{\coU\ka},V)} has the local representation 
$$ 
v=v_{(j)}^{(i),\nu}\frac\pl{\pl x^{(i)}}\otimes dx^{(j)}\otimes b_\nu 
$$ 
and 
$$ 
\vp_\tau^\sa v(x)=\big[v_{(j)}^{(i),\nu}\bigl(\ka^{-1}(x)\bigr)e_\nu\big] 
\in F^{m^\sa\times m^\tau}=E 
\qa x\in Q_\ka^m. 
$$ 
Assume 
\hb{\wt{\ka}\slt\wt{\vp}\in\gK\slt\Phi}. Then 
\hb{(\wt{\ka}\slt\wt{\vp}_\tau^\sa)\circ(\ka\slt\vp_\tau^\sa)^{-1} 
   =\bigl(\wt{\ka}\circ\ka^{-1},\ (\vp_\tau^\sa)_{\ka\wt{\ka}}\bigr)}, 
where 
\beq\Label{L.st} 
\bigl((\vp_\tau^\sa)_{\ka\wt{\ka}}\xi\bigr)_{(j)}^{(i),\nu} 
=A_{(\wt{\imath})}^{(i)}B_{(j)}^{(\wt{\jmath})} 
(\vp_{\ka\wt{\ka}})_{\wt{\nu}}^\nu 
\xi_{(\wt{\jmath})}^{(\wt{\imath}),\wt{\nu}} 
\qa \xi\in E, 
\npb 
\eeq 
with 
\hb{A_{(\wt{\imath})}^{(i)} 
   =A_{\wt{\imath}_1}^{i_1}\cdots A_{\wt{\imath}_\sa}^{i_\sa}} and 
\hb{B_{(j)}^{(\wt{\jmath})} 
   =B_{j_1}^{\wt{\jmath}_1}\cdots B_{j_\tau}^{\wt{\jmath}_\tau}}, and  
\beq\Label{L.AB} 
A_{\wt{\imath}}^i 
=\frac{\pl(\wt{\ka}\circ\ka^{-1})^i}{\pl x^{\wt{\imath}}} 
\qb B_j^{\wt{\jmath}} 
=\frac{\pl(\ka\circ\wt{\ka}^{-1})^{\wt{\jmath}} }{\pl y^j} 
\circ(\wt{\ka}\circ\ka^{-1})  
\eeq   
for 
\hb{1\leq i,\wt{\imath},j,\wt{\jmath}\leq n} and 
\hb{y=\wt{\ka}\circ\ka^{-1}(x)}. Hence \Eqref{U.K}, \,\Eqref{U.Ph}, and 
assumption~\Eqref{S.ass} imply that 
$$ 
\gK\slt\Phi_\tau^\sa:=\{\,\ka\slt\vp_\tau^\sa 
\ ;\ \ka\slt\vp\in\gK\slt\Phi\,\} 
$$ 
is a uniformly regular atlas for~$V$ over~$\gK$. From \Eqref{U.Keq} and 
\Eqref{U.peq} we also infer that 
$$  
\gK\slt\Phi\approx\wt{\gK}\slt\wt{\Phi} 
\quad\Ra\quad 
\gK\slt\Phi_\tau^\sa\approx\wt{\gK}\slt\wt{\Phi}_\tau^\sa.  
$$  

\smallskip 
The local chart 
\hb{\ka\slt\vp_\tau^\sa} is completely determined by 
\hb{\ka\slt\vp}. For this reason, and to simplify notation, we denote the 
push-forward and pull-back by 
\hb{\ka\slt\vp_\tau^\sa} simply by 
\hb{(\ka\slt\vp)_*} and 
\hb{(\ka\slt\vp)^*}, respectively. This is consistent with the  
use of~$\ka_*$ for the push-forward of vector fields by 
\hb{\ka\slt\vp} (see Example~\ref{exa-V.ex}(b)).  

\smallskip 
We set 
$$ 
g_{(i)(k)}^{(j)(\ell)}:=g_{i_1k_1}\cdots g_{i_\sa k_\sa} 
g^{j_1\ell_1}\cdots g^{j_\tau\ell_\tau} 
$$ 
with $(i)$,~$(k)$ running through~$\BJ_\sa$  and 
$(j)$,~$(\ell)$ through~$\BJ_\tau$ . Then \Eqref{S.h} 
and \Eqref{V.hM} imply 
$$  
h\uv=g_{(i)(k)}^{(j)(\ell)}u_{(j)}^{(i),\nu}\ol{v_{(\ell)}^{(k),\mu}} 
h_W(b_\nu,b_\mu) 
\qa u,v\in\Ga(U_{\coU\ka},V). 
$$  
Hence, setting 
\hb{u_\ka:=(\ka\slt\vp)_*u} etc., we get from \Eqref{V.puh} 
\beq\Label{L.hl} 
(\ka\slt\vp)_*h(u_\ka,v_\ka) 
=\ka_*g_{(i)(k)}^{(j)(\ell)} 
\ka_*u_{(j)}^{(i),\nu}\ka_*\ol{v_{(\ell)}^{(k),\mu}} 
\ka_*h_{W_{\nu\mu}}. 
\eeq 
Lemma~3.1 of~\cite{Ama12b} guarantees 
\beq\Label{L.kgg} 
\ka_*g\sim\rho_\ka^2g_m
\qb \ka_*g^*\sim\rho_\ka^{-2}g_m 
\qa \ka\in\gK, 
\eeq  
and 
\beq\Label{L.kgk} 
\rho_\ka^{-2}\,\|\ka_*g\|_{k,\iy}+\rho_\ka^2\,\|\ka_*g^*\|_{k,\iy} 
\leq c(k) 
\qa \ka\in\gK 
\qb k\in\BN. 
\eeq 
From \Eqref{V.pN}, the uniform regularity of~$h_W$ over 
\hb{\gK\slt\Phi}, \ \Eqref{L.hl},  and \Eqref{L.kgg} we deduce 
\beq\Label{L.Nh} 
\ka_*(|u|_h)=|(\ka\slt\vp)_*u|_{(\ka\slt\vp)_*h} 
\sim\rho_\ka^{\sa-\tau}\,|(\ka\slt\vp)_*u|_{E_\tau^\sa} 
\qa \ka\slt\vp\in\gK\slt\Phi 
\qb u\in\Ga\MV. 
\eeq 

\smallskip 
Suppose 
\hb{u\in\cT_\tau^\sa\MV} has the local representation 
$$ 
u=u_{(j)}^{(i),\nu}\frac\pl{\pl x^{(i)}}\otimes dx^{(j)}\otimes b_\nu. 
$$ 
Then it follows from \Eqref{V.Nfv}, \,\Eqref{V.NXv}, \,\Eqref{V.Ddu}, 
\,\Eqref{V.Ntp}, and \Eqref{V.Nb}, denoting by~$D_{k\mu}^\nu$ the 
Christoffel symbols of~$D$, that 
\beq\Label{L.Nu} 
\bal 
\na u 
&=\frac{\pl u_{(j)}^{(i),\nu}}{\pl x^k}\frac\pl{\pl x^{(i)}}
 \otimes dx^{(j)}\otimes dx^k\otimes b_\nu\\ 
&\ph{{}={}} 
 +\sum_{s=1}^\sa u_{(j)}^{(i_1,\ldots,i_s,\ldots,i_\sa),\nu} 
 \,\Ga_{ki_s}^\ell\frac\pl{\pl x^{(i_1,\ldots,\ell,\ldots,i_\sa)}} 
 \otimes dx^{(j)}\otimes dx^k\otimes b_\nu\\ 
&\ph{{}={}} 
 -\sum_{t=1}^\tau u_{(j_1,\ldots,j_t,\ldots,j_\tau)}^{(i),\nu}
 \,\Ga_{k\ell}^{j_t}\frac\pl{\pl x^{(i)}} 
 \otimes dx^{(j_1,\ldots,\ell,\ldots,j_\tau)}\otimes dx^k\otimes b_\nu\\ 
&\ph{{}={}} 
 +u_{(j)}^{(i),\mu}D_{k\mu}^\nu \frac\pl{\pl x^{(i)}}
 \otimes dx^{(j)}\otimes dx^k\otimes b_\nu, 
\eal 
\npb 
\eeq 
with $\ell$~being at position~$s$ in 
$(i_1,\ldots,\ell,\ldots,i_\sa)$ and position~$t$ in 
$(j_1,\ldots,\ell,\ldots,j_\tau)$. 

\smallskip 
We endow the trivial bundle 
\hb{Q_\ka^m\times E_\tau^\sa} with the Euclidean connection, denoted 
by~$\pl_x$ and being naturally identified with the Fr\'echet derivative. 
Thus, given 
\hb{v\in C^\iy(Q_\ka^m,E_\tau^\sa)}, 
$$ 
\pl_x^\ell v\in C^\iy\bigl(Q_\ka^m,\cL^\ell(\BR^m;E_\tau^\sa)\bigr) 
\qa \ell\in\BN^\times, 
$$ 
where $\cL^\ell(\BR^m;E_\tau^\sa)$ is the space of \hbox{$\ell$-linear} 
maps from~$\BR^m$ into~$E_\tau^\sa$. If 
\hb{v=\big[v_{(j)}^{(i)}\big]\sco Q_\ka^m\ra F^{m^\sa\times m^\tau}}, 
then, setting 
\hb{\pl_{(k)}:=\pl_{k_\ell}\circ\cdots\circ\pl_{k_1}} for 
\hb{(k)\in\BJ_\ell} with 
\hb{\pl_i=\pl/\pl x^i}, 
\beq\Label{L.Dlv} 
\pl_x^\ell v 
=\big[\pl_{(k)}v_{(j)}^{(i)}\big] 
\sco Q_\ka^m\ra F^{m^\sa\times m^{\tau+\ell}}. 
\eeq 
Hence, using the latter interpretation, 
\beq\Label{L.Dl} 
\pl_x^\ell\in 
\cL^\ell\bigl(C^\iy(Q_\ka^m,E_\tau^\sa),
C^\iy(Q_\ka^m,E_{\tau+\ell}^\sa)\bigr) 
\qa \ell\in\BN, 
\npb 
\eeq 
where 
\hb{\pl_x^0:=\id}. 

\smallskip 
We define the push-forward 
$$ 
(\ka\slt\vp)_*\na^\ell\sco C^\iy(Q_\ka^m,E_\tau^\sa) 
\ra C^\iy(Q_\ka^m,E_{\tau+\ell}^\sa) 
$$ 
of~$\na^\ell$ by 
\hb{\ka\slt\vp} by 
$$ 
(\ka\slt\vp)_*\na^\ell 
:=(\ka\slt\vp)_*\circ\na^\ell\circ(\ka\slt\vp)^* 
$$ 
for 
\hb{\ell\in\BN}. Then 
\hb{(\ka\slt\vp)_*\na} is a metric connection on 
\hb{\bigl(T_\tau^\sa(Q_\ka^m,F),(\ka\slt\vp)_*h\bigr)} and 
$$ 
(\ka\slt\vp)_*\na^{\ell+1} 
=\bigl((\ka\slt\vp)_*\na\bigr)\circ(\ka\slt\vp)_*\na^\ell 
\qa \ell\in\BN.  
$$ 

\smallskip 
Suppose 
\hb{r\in\BN^\times} and
\hb{u\in C^r\MV}. Set 
\hb{v:=(\ka\slt\vp)_*u\in C^r(Q_\ka^m,E_\tau^\sa)}. Then we infer 
from \Eqref{L.Nu} by induction, and from \Eqref{L.Dlv} and \Eqref{L.Dl} 
that there exist 
$$ 
a_\ell 
\in C^\iy\bigl(Q_\ka^m,\cL(E_{\tau+\ell}^\sa,E_{\tau+r}^\sa)\bigr) 
\qa 0\leq\ell\leq r-1, 
$$ 
such that 
\beq\Label{L.Dr} 
(\ka\slt\vp)_*\na^rv=\pl_x^rv+\sum_{\ell=0}^{r-1}a_\ell\pl_x^\ell v. 
\eeq 
More precisely, the entries of the matrix representation of~$a_\ell$ 
are polynomials in the derivatives of order at most 
\hb{r-\ell-1} of the Christoffel symbols of $\na_{\cona g}$ and~$D$. 
Hence assumption~\Eqref{S.ass} implies 
\beq\Label{L.abd} 
\|a_\ell\|_{k,\iy}\leq c(k) 
\qa 0\leq\ell\leq r-1 
\qb \ka\slt\vp\in\gK\slt\Phi,  
\eeq 
due to \Eqref{U.Ch}, \,\Eqref{L.kgg}, and \Eqref{L.kgk}. By solving 
system~\Eqref{L.Dr} for 
\hb{0\leq\ell\leq r} `from the bottom' we find 
\beq\Label{L.dr} 
\pl_x^rv 
=(\ka\slt\vp)_*\na^rv
+\sum_{\ell=0}^{r-1}\wt{a}_\ell(\ka\slt\vp)_*\na^\ell v, 
\eeq 
where 
\hb{\wt{a}_\ell\in 
   C^\iy\bigl(Q_\ka^m,\cL(E_{\tau+\ell}^\sa,E_{\tau+r}^\sa)\bigr)} satisfy 
\beq\Label{L.adb} 
\|\wt{a}_\ell\|_{k,\iy}\leq c(k) 
\qa 0\leq\ell\leq r-1 
\qb \ka\slt\vp\in\gK\slt\Phi. 
\eeq 
From \hbox{\Eqref{L.Dr}--\Eqref{L.adb}} 
we infer that, given 
\hb{r\in\BN^\times}, 
\beq\Label{L.Nd} 
\sum_{i=0}^r\big|(\ka\slt\vp)_*\na^i\bigl((\ka\slt\vp)_*u\bigr)
 \big|_{E_{\tau+i}^\sa} 
\sim\sum_{|\al|\leq r}\big|\pa_x\bigl((\ka\slt\vp)_*u 
\bigr)\big|_{E_\tau^\sa} 
\npb 
\eeq 
for 
\hb{\ka\slt\vp\in\gK\slt\Phi}  and 
\hb{u\in C^r\MV}. 
\section{Isotropic Bessel Potential and Besov Spaces}\LabelT{sec-J}%
Weighted (isotropic) function spaces on singular manifolds have been studied 
in detail in~\cite{Ama12b}, where, however, only scalar-valued tensor 
fields are considered. In this and the next section we recall the basic 
definitions and notation on which we shall build in the anisotropic case, 
and describe the needed extensions to the case of vector-bundle-valued 
tensor fields. 

\smallskip 
We denote by 
\hb{\ci\cD:=\ci\cD(V):=\cD\ciMV}, respectively 
\hb{\cD:=\cD(V):=\cD\MV}, the \hbox{LF-space} of smooth sections of~$V$ 
which are compactly supported in~$\ci M$, respectively~$M$. Then 
\hb{\ci\cD'=\ci\cD'(V):=\ci\cD(V')_{w^*}'} is the dual of~$\ci\cD(V')$ 
endowed with the \hbox{$w^*$-topology}, the space of \emph{distribution 
sections on}~$\ci M$, whereby 
\hb{V'=T_\sa^\tau\MWa}. As usual, we identify 
\hb{v\in L_{1,\loc}\ciMV} with the distribution section 
\hb{\bigl(u\mt\dl u,v\dr_M\bigr)\in\ci\cD'}, where 
$$ 
\dl u,v\dr_M:=\int_M\dl u,v\dr_V\,dV_g 
\qa u\in\cD\ciMVs 
\qb v\in L_{1,\loc}\ciMV, 
$$ 
and $dV_g$~is the volume measure of~$M$. Hence 
$$  
\ci\cD\hr\cD\sdh L_{1,\loc}\MV 
\overset{u\mt u\sn\ci M}{\lllllora}L_{1,\loc}\ciMV\hr\ci\cD', 
\npb 
$$ 
where 
\hb{{}\hr{}}~means `continuous' and 
\hb{{}\sdh{}}~`continuous and dense' embedding. 

\smallskip 
In addition to \Eqref{S.ass} we suppose throughout 
\bspezeqn\Label{J.rpl} 
\bal 
\frame{ 
\begin{minipage}{168pt}
\bspezeq 
\bal 
{}      
&\rho\in\gT(M)  
\qb 1<p<\iy  
\qb \lda\in\BR.\\ 
\noalign{\vskip2.5\jot}
\eal 
\espezeq 
\end{minipage}} 
\eal 
\espezeqn 
Assume 
\hb{k\in\BN}. The \hh{weighted Sobolev space} 
$$ 
W_{\coW p}^{k,\lda}=W_{\coW p}^{k,\lda}(V)=W_{\coW p}^{k,\lda}(V;\rho) 
$$ 
of \hbox{$W$-valued} $(\sa,\tau)$-tensor fields on~$M$ is the 
completion of~$\cD$ in~$L_{1,\loc}(V)$ with respect to the norm 
$$ 
u\mt\|u\|_{k,p;\lda} 
:=\Bigl(\sum_{i=0}^k\big\|\rho^{\lda+\tau-\sa+i} 
\,|\na^iu|_h\,\big\|_p^p\Bigr)^{1/p}. 
$$ 
It is independent of the particular choice of~$\rho$ in the sense that 
\hb{W_{\coW p}^{k,\lda}(V;\rho')\doteq W_{\coW p}^{k,\lda}(V;\rho)} for 
\hb{\rho'\in[\![\rho]\!]}, where 
\hb{{}\doteq{}}~means `equal except for equivalent norms'. 

\smallskip 
For simplicity, we do not indicate the dependence of these norms, and of 
related ones to be introduced below, on~$(\sa,\tau)$. 
This has to be kept in mind.  

\smallskip 
Note that 
$$ 
W_{\coW p}^{0,\lda}=L_p^\lda=L_p^\lda(V) 
:=\bigl(\bigl\{\,u\in L_{p,\loc} 
\ ;\ \|u\|_{p;\lda}<\iy\,\bigr\},\ \Vsdot_{p;\lda}\bigr),  
\npb 
$$ 
where 
\hb{\Vsdot_{p;\lda}:=\Vsdot_{0,p;\lda}}. Also observe 
\hb{W_{\coW p}^{k,\lda}\sdh W_{\coW p}^{\ell,\lda}} for 
\hb{k>\ell}. 

\smallskip 
Given 
\hb{0<\ta<1}, we write~%
\hb{\pe_\ta} for the complex, and 
\hb{\pr_{\ta,q}}, 
\ \hb{1\leq q\leq\iy}, for the real interpolation functor 
of exponent~$\ta$ (see \cite[Section~I.2]{Ama95a} for definitions and a 
summary of the basic facts of interpolation theory of which we make free 
use). Then, given 
\hb{k\in\BN}, 
$$ 
H_p^{s,\lda}=H_p^{s,\lda}(V):=
\left\{
\bal
{}      
[ &W_{\coW p}^{k,\lda},W_{\coW p}^{k+1,\lda}]_{s-k},  
        &\quad k<s  &<k+1,\\
    &W_{\coW p}^{k,\lda}, 
        &\quad   s  &=k, 
\eal
\right.
$$ 
and 
$$ 
B_p^{s,\lda}=B_p^{s,\lda}(V):=
\left\{
\bal 
{}      
&(W_{\coW p}^{k,\lda},W_{\coW p}^{k+1,\lda})_{s-k,p}, 
        &\quad k<s&<k+1,\\
&(W_{\coW p}^{k,\lda},W_{\coW p}^{k+2,\lda})_{1/2,p}, 
        &\quad   s&=k. 
\eal
\right.
$$ 
In favor of a unified treatment, throughout the rest of this paper 
\bspezeq 
\bal 
\frame{ 
\begin{minipage}{154pt} 
\bspezeq 
\bal 
{}      
&\gF\in\{H,B\}  
\qb \gF_p^{s,\lda}:=\gF_p^{s,\lda}(V).\\ 
\noalign{\vskip2.5\jot}
\eal 
\espezeq 
\end{minipage}} 
\eal  
\espezeq 

\smallskip 
We denote by~$\ci\gF_p^{s,\lda}$ the closure of~$\ci\cD$ 
in~$\gF_p^{s,\lda}$ for 
\hb{s>0} and set  
$$ 
\gF_p^{-s,\lda}(V):=\bigl(\ci\gF_{p'}^{s,-\lda}(V')\bigr)' 
\qa s>0, 
$$ 
with respect to the duality pairing induced by~%
\hb{\pw_M}. We also set 
$$ 
B_p^{0,\lda} 
:=(W_{\coW p}^{-1,\lda},W_{\coW p}^{1,\lda})_{1/2,p}. 
$$ 
This defines the \hh{weighted Bessel potential space scale} 
\hb{[\,H_p^{s,\lda}\ ;\ s\in\BR\,]} and the 
\hh{weighted Besov space scale} 
\hb{[\,B_p^{s,\lda}\ ;\ s\in\BR\,]}. 

\smallskip 
It follows (see the next section) that $\gF_p^{s,\lda}$~is for 
\hb{s\in\BR} a~reflexive Banach space, and  
$$ 
\cD\sdh\gF_p^{s,\lda}\sdh\gF_p^{t,\lda}\sdh\cD' 
\qa -\iy<t<s<\iy. 
$$ 
Denoting, for \emph{any} 
\hb{s\in\BR}, by~$\ci\gF_p^{s,\lda}$ the closure of~$\cD$ 
in~$\gF_p^{s,\lda}$, 
$$  
\ci\gF_p^{s,\lda}=\gF_p^{s,\lda} 
\qa s<1/p. 
$$ 
Thus, by reflexivity, 
$$  
\ci\gF_p^{s,\lda}(V)=\bigl(\ci\gF_{p'}^{-s,-\lda}(V')\bigr)' 
\qa s\in\BR, 
\npb 
$$ 
with respect to~%
\hb{\pw_M}. 

\smallskip 
If 
\hb{\rho\sim\mf{1}}, then all these spaces are independent of~$\lda$. 
Furthermore, $\gF_p^{s,\lda}$~reduces to the non-weighted (standard) Bessel 
potential space~$H_p^s(V)$ and Besov space~$B_p^s(V)$, respectively. 
Assume, in addition, 
\hb{M=\BX\in\{\BR^m,\BH^m\}} with 
\hb{g=g_m}, 
\ \hb{V=\BX\times E}, and 
\hb{D=d_F}. Then $H_p^s\BXE$ is the classical (\hbox{$E$-valued}) 
Bessel potential space and $B_p^s\BXE$ the standard (\hbox{$E$-valued}) 
Besov space~$B_{p,p}^s\BXE$. In the scalar case these spaces are well 
investigated (cf.\ H.~Triebel~\cite{Tri78a}, for example). Thus noting 
\hb{\gF_p^s\BXE\simeq(\gF_p^s)^d} with 
\hb{d=\dim(E)}, we can make free use of their properties which 
we shall do without further reference. 
\section{The Isotropic Retraction Theorem}\LabelT{sec-R}%
Let $E_\al$ be a locally convex space for each~$\al$ in a countable index 
set. Then
\hb{\mf{E}:=\prod_\al E_\al} is endowed with the product topology. Now 
suppose that each~$E_\al$ is a Banach space. Then we denote for 
\hb{1\leq q\leq\iy} by~$\ell_q(\mf{E})$ the linear subspace
of~$\mf{E}$ consisting of all
\hb{\mf{x}=(x_\al)} such that
$$
\|\mf{x}\|_{\ell_q(\mf{E})}:=
\left\{
\bal
{\textstyle \bigl(\sum_\al} &\|x_\al\|_{E_\al}^q\bigr)^{1/q},
                            &\quad 1\leq q  &<\iy,\\
{\textstyle \sup_\al}       &\|x_\al\|_{E_\al},
                            &\quad       q  &=\iy,
\eal
\right.
$$
is finite. Then $\ell_q(\mf{E})$ is a Banach space with norm~%
\hb{\Vsdot_{\ell_q(\mf{E})}}, and
\beq\LabelT{R.ll}
\ell_p(\mf{E})\hr\ell_q(\mf{E})
\qa 1\leq p<q\leq\iy.
\eeq 
We also set
\hb{c_c(\mf{E}):=\bigoplus_\al E_\al}, where 
\hb{{}\bigoplus{}}~denotes the locally convex direct sum. Thus 
\hb{\bigoplus_\al E_\al} consists of all finitely supported sequences 
in~$\mf{E}$ equipped with the finest locally convex topology for which all 
injections 
\hb{E_\ba\ra\bigoplus_\al E_\al} are continuous. It follows 
\beq\LabelT{R.cl}
c_c(\mf{E})\hr\ell_q(\mf{E})
\qb 1\leq q\leq\iy
\qa  c_c(\mf{E})\sdh\ell_q(\mf{E})
\qb q<\iy. 
\npb 
\eeq 
Furthermore, $c_0(\mf{E})$~is the closure of~$c_c(\mf{E})$
in~$\ell_\iy(\mf{E})$.

\smallskip
If each $E_\al$ is reflexive, then $\ell_p(\mf{E})$ is reflexive as well,
and
\hb{\ell_p(\mf{E})'=\ell_{p'}(\mf{E}')} with respect to the duality pairing
\hb{\mfpw:=\sum_\al\pw_\al}. Of course,
\hb{\mf{E}':=\prod_\al E_\al'}, and
\hb{\pw_\al}~is the \hbox{$E_\al$-duality} pairing. 

\smallskip 
Let  assumption~\Eqref{L.K} be satisfied. A~\emph{localization 
system subordinate to}~$\gK$ is a family 
\hb{\bigl\{\,(\pi_\ka,\chi_\ka)\ ;\ \ka\in\gK\,\bigr\}} such that 
\begin{equation}\Label{R.LS} 
\bal
\rm{(i)}  \qquad    &\pi_\ka\in\cD\bigl(U_{\coU\ka},[0,1]\bigr)\text{ and }
                     \{\,\pi_\ka^2\ ;\ \ka\in\gK\,\}
                     \text{ is a partition of unity on }M\\
                    &\text{subordinate to the covering }
                     \{\,U_{\coU\ka}\ ;\ \ka\in\gK\,\};\\
\rm{(ii)} \qquad    &\chi_\ka=\ka^*\chi\text{ with }
                     \chi\in\cD\bigl(Q^m,[0,1]\bigr)
                     \text{ and $\chi\sn\supp(\ka_*\pi_\ka)=\mf{1}$ 
                     for }\ka\in\gK;\\
\rm{(iii)}\qquad    &\|\ka_*\pi_\ka\|_{k,\iy}+\|\ka_*\chi_\ka\|_{k,\iy}
                     \leq c(k),
                     \ \ \ka\in\gK,
                     \ \ k\in\BN.
\eal 
\npb 
\end{equation} 
Lemma~3.2 of~\cite{Ama12b} guarantees the existence of such a 
localization system. 

\smallskip 
In addition to \Eqref{L.K} we assume 
$$ 
\bigl\{\,(\pi_\ka,\chi_\ka)\ ;\ \ka\in\gK\,\bigr\} 
\text{ is a localization system subordinate to }\gK. 
$$ 
For abbreviation, we put 
for 
\hb{s\in\BR} 
$$ 
W_{\coW p,\ka}^s:=W_{\coW p}^s\BXkE 
\qb \gF_{p,\ka}^s:=\gF_p^s\BXkE 
\qa \ka\in\gK, 
$$ 
where 
\hb{E=E_\tau^\sa(F)}. Hence 
\hb{\mf{W}_{\coW p}^s=\prod_\ka W_{\coW p,\ka}^s} is well-defined, as 
is~$\mf{\gF}_p^s$. We set 
$$ 
\cD_\ka:=\cD\BXkE 
\qb \ci\cD_\ka:=\cD\ciBXkE, 
$$ 
as well as 
$$ 
\mf{\cD}=\mf{\cD}\BXE:=\bigoplus_\ka\cD_\ka 
\qb \ci{\mf{\cD}}=\mf{\cD}\ciBXE:=\bigoplus_\ka\ci\cD_\ka. 
\npb 
$$ 
It should be noted that, due to \Eqref{L.Xk}, in $\mf{W}_{\coW p}^s$, 
$\mf{\gF}_p^s$, $\mf{\cD}$, and~$\ci{\mf{\cD}}$ there occur at most two 
distinct function spaces. 

\smallskip 
Given 
\hb{\ka\slt\vp\in\gK\slt\Phi}, we put for 
\hb{1\leq q\leq\iy} 
$$ 
\vp_{q,\ka}^\lda u:=\rho_\ka^{\lda+m/q}(\ka\slt\vp)_*(\pi_\ka u) 
\qa u\in C(V), 
$$ 
and 
$$ 
\psi_{q,\ka}^\lda v:=\rho_\ka^{-\lda-m/q}\pi_\ka(\ka\slt\vp)^*v 
\qa v\in C\BXkE.  
$$ 
Here and in similar situations it is understood that a
partially defined and compactly supported section of a vector bundle is
extended over the whole base manifold by identifying it with the zero
section outside its original domain. In addition, 
$$ 
\vp_q^\lda u:=(\vp_{q,\ka}^\lda u)\in\prod_\ka C\BXkE
\qa u\in C(V), 
$$ 
and 
$$ 
\psi_q^\lda\mf{v}:=\sum_\ka\psi_{q,\ka}^\lda v_\ka 
\qa \mf{v}=(v_\ka)\in\prod_\ka C\BXkE. 
$$ 

\smallskip 
A~\emph{retraction} from a locally convex space~$\cX$ onto a locally convex 
space~$\cY$ is a map 
\hb{R\in\cL\cXcY} possessing a right inverse 
\hb{R^c\in\cL\cYcX}, a~coretraction. 

\smallskip 
If no confusion seems likely, we use the same symbol for a 
continuous linear map and its restriction to a linear subspace of its domain, 
respectively for a unique continuous linear extension of it. Furthermore, 
in a diagram arrows always represent continuous linear maps. 

\smallskip 
The following theorem shows that $\psi_p^\lda$~is a retraction 
from~$\mf{\cD}$ onto~$\cD$, and that $\vp_p^\lda$~is a coretraction. 
Moreover, $\psi_p^\lda$~has a unique continuous linear extension to a 
retraction from~$\ell_p(\mf{\gF}_p^s)$ onto~$\gF_p^{s,\lda}$, and  
$\vp_p^\lda$~extends uniquely to a coretraction. This holds for any 
choice of 
\hb{s\in\BR} and 
\hb{p\in(1,\iy)}. Thus $\psi_p^\lda$~is a \emph{universal} retraction 
from~$\ell_p(\gF_p^s)$ onto~$\gF_p^{s,\lda}$ in the sense that it is 
completely determined by its restriction to~$\mf{\cD}$. The same holds if 
$\cD$ and~$\gF_p^{s,\lda}$ are replaced by  $\ci\cD$ 
and~$\ci\gF_p^{s,\lda}$, respectively. 
\begin{theorem}\LabelT{thm-R.R} 
Suppose 
\hb{s\in\BR}. Then the diagrams 
 \bspezeq 
 \bal
 \begin{picture}(177,88)(-80,-41)        
 \put(0,40){\makebox(0,0)[b]{\small{$d$}}}
 \put(0,5){\makebox(0,0)[b]{\small{$d$}}}
 \put(0,-30){\makebox(0,0)[b]{\small{$d$}}}
 \put(-66,35){\makebox(0,0)[c]{\small{$\cD$}}}
 \put(-65,-35){\makebox(0,0)[c]{\small{$\cD$}}}
 \put(-28,0){\makebox(0,0)[c]{\small{$\mf{\cD}$}}}
 \put(40,0){\makebox(0,0)[c]{\small{$\ell_p(\mf{\gF}_p^s)$}}}
 \put(90,35){\makebox(0,0)[c]{\small{$\gF_p^{s,\lda}$}}}
 \put(90,-35){\makebox(0,0)[c]{\small{$\gF_p^{s,\lda}$}}}
 \put(-29,17.5){\makebox(0,0)[r]{\small{$\vp_p^\lda$}}}
 \put(50,17.5){\makebox(0,0)[l]{\small{$\vp_p^\lda$}}}
 \put(-29,-17.5){\makebox(0,0)[r]{\small{$\psi_p^\lda$}}}
 \put(50,-17.5){\makebox(0,0)[l]{\small{$\psi_p^\lda$}}}
 \put(-71,0){\makebox(0,0)[r]{\small{$\id$}}}
 \put(90,0){\makebox(0,0)[l]{\small{$\id$}}}
 \put(-54,35){\vector(1,0){130}}
 \put(-54,37){\oval(4,4)[l]}
 \put(-18,0){\vector(1,0){40}}
 \put(-18,2){\oval(4,4)[l]}
 \put(-54,-35){\vector(1,0){130}}
 \put(-54,-33){\oval(4,4)[l]}
 \put(-66,25){\vector(0,-1){50}}
 \put(85,25){\vector(0,-1){50}}
 \put(-56,25){\vector(1,-1){20}}
 \put(-36,-5){\vector(-1,-1){20}}
 \put(75,25){\vector(-1,-1){20}}
 \put(55,-5){\vector(1,-1){20}}
 \end{picture}
\qquad\qquad 
 \begin{picture}(177,88)(-80,-41)        
 \put(0,40){\makebox(0,0)[b]{\small{$d$}}}
 \put(0,5){\makebox(0,0)[b]{\small{$d$}}}
 \put(0,-30){\makebox(0,0)[b]{\small{$d$}}}
 \put(-66,35){\makebox(0,0)[c]{\small{$\ci\cD$}}}
 \put(-65,-35){\makebox(0,0)[c]{\small{$\ci\cD$}}}
 \put(-28,0){\makebox(0,0)[c]{\small{$\ci{\mf{\cD}}$}}}
 \put(40,0){\makebox(0,0)[c]{\small{$\ell_p(\cimfgF_p^s)$}}}
 \put(90,35){\makebox(0,0)[c]{\small{$\ci\gF_p^{s,\lda}$}}}
 \put(90,-35){\makebox(0,0)[c]{\small{$\ci\gF_p^{s,\lda}$}}}
 \put(-29,17.5){\makebox(0,0)[r]{\small{$\vp_p^\lda$}}}
 \put(50,17.5){\makebox(0,0)[l]{\small{$\vp_p^\lda$}}}
 \put(-29,-17.5){\makebox(0,0)[r]{\small{$\psi_p^\lda$}}}
 \put(50,-17.5){\makebox(0,0)[l]{\small{$\psi_p^\lda$}}}
 \put(-71,0){\makebox(0,0)[r]{\small{$\id$}}}
 \put(90,0){\makebox(0,0)[l]{\small{$\id$}}}
 \put(-54,35){\vector(1,0){130}}
 \put(-54,37){\oval(4,4)[l]}
 \put(-18,0){\vector(1,0){40}}
 \put(-18,2){\oval(4,4)[l]}
 \put(-54,-35){\vector(1,0){130}}
 \put(-54,-33){\oval(4,4)[l]}
 \put(-66,25){\vector(0,-1){50}}
 \put(85,25){\vector(0,-1){50}}
 \put(-56,25){\vector(1,-1){20}}
 \put(-36,-5){\vector(-1,-1){20}}
 \put(75,25){\vector(-1,-1){20}}
 \put(55,-5){\vector(1,-1){20}}
 \end{picture} 
 \eal 
 \npb 
 \espezeq 
are commuting, where 
\hb{s>0} in the second case. 
\end{theorem} 
\begin{proof} 
(1) Suppose 
\hb{W=M\times\BK} so that 
\hb{V=T_\tau^\sa\MW=T_\tau^\sa M}. Also suppose 
\hb{k\in\BN}. Then Theorem~6.1 of~\cite{Ama12b} guarantees that 
\beq\Label{R.K} 
\psi_p^\lda\text{ is a retraction from $\mf{\cD}$ onto $\cD$ 
and from $\ell_p(\mf{W}_{\coW p}^k)$ onto $W_{\coW p}^{k,\lda}$, 
and $\vp_p^\lda$ is a coretraction.}
\eeq 
Furthermore, set 
\beq\Label{R.ph0} 
\ci\vp_{p,\ka}^\lda:=\rho_\ka^{-m}\sqrt{\ka_*g}\kern1pt \vp_{p,\ka}^\lda 
\qb \ci\psi_{p,\ka}^\lda:=\rho_\ka^m(\sqrt{\ka_*g})^{-1}\psi_{p,\ka}^\lda 
\eeq  
and 
$$ 
\ci\vp_p^\lda u:=(\ci\vp_{p,\ka}^\lda u) 
\qb \ci\psi_p^\lda\mf{v}:=\sum_\ka\ci\psi_{p,\ka}^\lda v_\ka 
\qa u\in\ci\cD 
\qb \mf{v}\in\ci{\mf{\cD}}. 
$$ 
Then it follows from Theorem~11.1 of~\cite{Ama12b} that $\ci\psi_p^\lda$~is 
a retraction from from~$\ci{\mf{\cD}}$ onto~$\ci\cD$ and 
from~$\ell_p(\cimfW_{\coW p}^k)$ onto~$\ci W_{\coW p}^{k,\lda}$, and 
$\ci\vp_p^\lda$~is a coretraction. 

\smallskip 
From step~(2) of the proof of the latter theorem we know 
\beq\Label{R.rho} 
\rho_\ka^{-m}\sqrt{\ka_*g}\sim\mf{1} 
\qb \|\rho_\ka^{-m}\sqrt{\ka_*g}\|_{k,\iy} 
+\|\rho_\ka^m(\sqrt{\ka_*g})^{-1}\|_{k,\iy}\leq c(k)  
\qa \ka\in\gK 
\qb k\in\BN. 
\eeq 
This implies that we can replace $\ci\vp_{p,\ka}$ and~$\ci\psi_{p,\ka}$ 
in \cite[Theorem~11.1]{Ama12b} by $\vp_{p,\ka}$ and~$\psi_{p,\ka}$, 
respectively. Consequently, 
\beq\Label{R.K0} 
\psi_p^\lda
\text{ is a retraction from $\ci{\mf{\cD}}$ onto $\ci\cD$ 
and from $\ell_p(\cimfW_{\coW p}^k)$ onto $\ci W_{\coW p}^{k,\lda}$, 
and $\vp_p^\lda$ is a coretraction.} 
\eeq 

\smallskip
(2) Let now 
\hb{W=(W,h_W,D)} be an arbitrary fully uniformly regular vector bundle 
over~$M$. Then \Eqref{L.Nd} is the analogue of Lemma~3.1(iv) 
of~\cite{Ama12b}. Furthermore, \Eqref{L.Nh} implies the analogue of 
\cite[part~(v) of Lemma~3.1]{Ama12b}. If 
\hb{W=M\times\BK}, then the proofs of \Eqref{R.K} and \Eqref{R.K0} are 
solely based on Lemma~3.1 of~\cite{Ama12b}. Hence, due to the preceding 
observations, they apply without change to the general case as well. Thus 
\Eqref{R.K} and \Eqref{R.K0} hold if $W$~is an arbitrary fully uniformly 
regular vector bundle over~$M$. 

\smallskip
(3) The assertions of the theorem are now deduced from \Eqref{R.K} and 
\Eqref{R.K0} by interpolation and duality as in~\cite{Ama12b}. 
\end{proof} 
Let $X$ and~$Y$ be Banach spaces, 
\hb{R\sco X\ra Y} a~retraction, and 
\hb{R^c\sco Y\ra X} a~coretraction. Then 
$$ 
\|y\|_Y=\|RR^cy\|_Y\leq\|R\|\,\|R^cy\|_X\leq\|R\|\,\|R^c\|\,\|y\|_Y 
\qa y\in Y. 
$$ 
Hence 
\beq\Label{R.rN} 
\Vsdot_Y\sim\|R^c\sdot\|_X. 
\eeq 
From this and Theorem~\ref{thm-R.R} it follows that 
\beq\Label{R.peq} 
u\mt\|\vp_p^\lda u\|_{\ell_p(\mf{\gF}_p^s)} 
\npb 
\eeq 
is a norm for~$\gF_p^{s,\lda}$. Furthermore, another choice of 
\hb{\gK\slt\Phi} and the localization  system leads to an equivalent norm. 

\smallskip 
For 
\hb{\ka\in\gK} and 
\hb{\wt{\ka}\in\gN(\ka)} we define a linear map 
\beq\Label{R.S} 
S_{\wt{\ka}\ka}\sco E^{\BX_{\wt{\ka}}}\ra E^{\BX_{\ka}} 
\qb v\mt(\ka\slt\vp)_*(\wt{\ka}\slt\wt{\vp})^*(\chi v). 
\eeq 

\smallskip 
The following lemma will be repeatedly useful. 
\begin{lemma}\LabelT{lem-R.S} 
Suppose 
\hb{s\in\BR^+} with 
\hb{s>0} if\/ 
\hb{\gF=B}. Then 
$$ 
S_{\wt{\ka}\ka}\in\cL(\gF_{p,\wt{\ka}}^s,\gF_{p,\ka}^s) 
\qb \|S_{\wt{\ka}\ka}\|\leq c 
\qa \wt{\ka}\in\gN(\ka)
\qb \ka\in\gK. 
$$ 
\end{lemma} 
\begin{proof} 
Note that, by~\Eqref{V.pp} and our convention on 
\hb{(\ka\slt\vp)_*}, 
$$ 
S_{\wt{\ka}\ka}v 
=(\vp_\tau^\sa)_{\wt{\ka}\ka} 
\bigl((\chi v)\circ(\ka\circ\wt{\ka}^{-1})\bigr).  
$$ 
Hence it follows from \Eqref{U.K}, \,\Eqref{U.Ph}, \,\Eqref{L.st}, 
\,\Eqref{L.AB}, \,\Eqref{R.LS}, and the product rule and Leibniz' formula 
that the assertion is true if 
\hb{s\in\BN} and 
\hb{\gF=H}, since 
\hb{H_{p,\ka}^s\doteq W_{\coW p,\ka}^s} for 
\hb{s\in\BN}. Now we obtain the statement for general~$s$ by interpolation. 
\end{proof} 
It follows from Theorem~\ref{thm-R.R} and the preceding consideration that 
\beq\Label{R.ext} 
\bal  
{}      
&\text{\emph{all results proved in \cite{Ama12b} for the 
 Banach space scales }}[\,\gF_p^{s,\lda}\ ;\ s\in\BR\,]\\ 
&\text{\emph{of scalar-valued $(\sa,\tau)$-tensor fields are likewise true 
for $W$-valued $(\sa,\tau)$-tensor fields,}}
\eal 
\eeq 
using obvious adaptions. Thus, in 
particular, the properties of~$\gF_p^{s,\lda}$ listed in 
Section~\ref{sec-J} are valid. Henceforth, we use \Eqref{R.ext} without 
further ado and simply refer to~\cite{Ama12b}. 
\section{Anisotropic Bessel Potential and Besov Spaces}\LabelT{sec-A}%
Given subsets $X$ and~$Y$ of a Hausdorff topological space, we write 
\hb{X\isis Y} if $\ol{X}$~is compact and contained in the interior of~$Y$. 

\smallskip 
Let $I$ be an interval with nonempty interior and $\cX$~a locally 
convex space. Suppose $\cQ$~is a family of seminorms for~$\cX$ 
generating its topology. Then $C^\iy\IcX$ is a locally convex space 
with respect to the topology induced by the family of seminorms 
$$ 
u\mt\sup_{t\in K}q\bigl(\pl^ku(t)\bigr) 
\qa k\in\BN 
\qb K\isis I 
\qb q\in\cQ. 
\npb 
$$ 
This topology is independent of the particular choice of~$\cQ$. 

\smallskip 
For 
\hb{K\isis I} we denote by $\cD_K\IcX$ the linear subspace of 
$C^\iy\IcX$ consisting of those functions which are supported in~$K$. We 
provide $\cD_K\IcX$ with the topology induced by~$C^\iy\IcX$. Then 
$\cD\IcX$, the vector space of smooth compactly supported 
\hbox{$\cX$-valued} functions, is endowed with the inductive topology with 
respect to the spaces $\cD_K\IcX$ with 
\hb{K\isis I}. If 
\hb{K\isis K'\isis I}, then $\cD_{K'}\IcX$ induces on $\cD_K\IcX$ its 
original topology. Note, however, that in general $\cD\IcX$ is not an 
\hbox{LF-space} since $\cD_K\IcX$ may not be a Fr\'echet space. 
Given a locally convex space~$\cY$, a~linear map 
\hb{T\sco\cD\IcX\ra\cY} is continuous iff its restriction to every 
subspace $\cD_K\IcX$ is continuous (e.g.,~Section~6 of 
H.H.~Schaefer~\cite{Schae71a}). 

\smallskip 
From now on it is assumed, in addition to \Eqref{S.ass} and \Eqref{J.rpl}, 
that 
\bspezeq  
\frame{
\begin{minipage}{165pt}
\bspezeq 
\bal 
r\in\BN^\times 
\qb \mu\in\BR  
\qb J\in\{\BR,\BR^+\}.\\
\noalign{\vskip2\jot} 
\eal 
\espezeq 
\end{minipage}} 
\espezeq  
We set 
$$ 
\bt\quad 
1/\vec r:=(1,1/r)\in\BR^2 
\qb \vec\om:=(\lda,\mu), 
\npb 
$$ 
so that 
\hb{s/\vec r=(s,s/r)} for 
\hb{s\in\BR}. 

\smallskip 
Suppose 
\hb{k\in\BN}. The \hh{anisotropic weighted Sobolev space} of time-dependent 
\hbox{$W$-valued} $(\sa,\tau)$-tensor fields on~$M$, 
\beq\Label{A.W} 
\bal 
{}      
&W_{\coW p}^{kr/\vec r,\vec\om}=W_{\coW p}^{kr/\vec r,\vec\om}\JV\text{, is 
 the linear subspace of $L_p(J,W_{\coW p}^{kr,\lda})$}\\ 
&\text{consisting of all $u$ satisfying $\pl^ku\in L_p(J,L_p^{\lda+k\mu})$, 
 endowed with the norm}\\ 
&\qquad\quad  
 \|u\|_{kr/\vec r,p;\vec\om} 
 :=\bigl(\|u\|_{L_p(J,W_{\coW p}^{kr,\lda})}^p 
 +\|\pl^ku\|_{L_p(J,L_p^{\lda+k\mu})}\bigr)^{1/p}.
\eal 
\npb 
\eeq
Thus 
\hb{W_{\coW p}^{0/\vec r,\vec\om}\doteq L_p(J,L_p^\lda)}.\po  
{\samepage 
\begin{theorem}\LabelT{thm-A.B}
\begin{itemize} 
\item[{\rm(i)}] 
$W_{\coW p}^{kr/\vec r,\vec\om}$ is a reflexive Banach space.\po 
\item[{\rm(ii)}] 
${}$
\hb{\|u\|_{kr/\vec r,p;\vec\om}^\sim 
 :=\bigl(\|u\|_{L_p(J,W_{\coW p}^{kr,\lda})}^p 
 +\sum_{j=0}^k\|\pl^ju\|_{L_p(J,W_{\coW p}^{(k-j)r,\lda+j\mu})}^p 
 \bigr)^{1/p}} 
is an equivalent norm. 
\item[{\rm(iii)}] 
${}$
\hb{\cD\JcD\sdh W_{\coW p}^{kr/\vec r,\vec\om}}.\po 
\end{itemize} 
\end{theorem} } 
\begin{proof} 
It follows from Theorem~\ref{thm-RA.R} below that 
$W_{\coW p}^{kr/\vec r,\vec\om}$~is isomorphic to a closed linear subspace 
of a reflexive Banach space, hence it is complete and reflexive. 
Proofs for parts (ii) and~(iii) are given in the next section. 
\end{proof} 
Observe 
\beq\Label{A.Nkr} 
\|u\|_{kr/\vec r,p;\vec\om}^\sim 
=\Bigl(\int_J\sum_{i+jr\leq kr} 
\big\|\rho^{\lda+i+j\mu+\tau-\sa} 
\,|\na^i\pl^ju|_h\,\big\|_p^p\,dt\Bigr)^{1/p}
\eeq  
and 
$$  
W_{\coW p}^{kr/\vec r,(\lda,0)} 
\doteq L_p(J,W_{\coW p}^{kr,\lda})\cap W_{\coW p}^k(J,L_p^\lda). 
$$ 
Note that Theorem~\ref{thm-A.B}(ii) and \Eqref{A.Nkr} show that 
definition~\Eqref{A.W} coincides, except for equivalent norms, with 
\Eqref{I.Nkr}. Also note that the reflexivity of~$L_p^\lda$ implies 
$$ 
W_{\coW p}^{0/\vec r,\vec\om} 
=L_p(J,L_p^\lda)=\bigl(L_{p'}(J,L_{p'}^{-\lda}(V'))\bigr)' 
$$ 
with respect to the duality pairing defined by 
$$ 
\dl u,v\dr_{M\times J}:=\int_J\bigl\dl u(t),v(t)\bigr\dr_M\,dt. 
$$ 

\smallskip 
Given 
\hb{0<\ta<1}, we set 
$$ 
\pr_\ta:=
\left\{
\bal
{}      
&\pe_\ta        &&\quad \text{if }\gF=H,\\
&\pr_{\ta,p}    &&\quad \text{if }\gF=B. 
\eal
\right.
$$ 
For 
\hb{s>0} we define `fractional order' spaces by 
\beq\Label{A.def} 
\gF_p^{s/\vec r,\vec\om}=\gF_p^{s/\vec r,\vec\om}\JV:=
\left\{
\bal
{}      
&(W_{\coW p}^{kr/\vec r,\vec\om},W_{\coW p}^{(k+1)r/\vec r,\vec\om} 
    )_{(s-kr)/r},
    &\quad kr<{} &s<(k+1)r,\\
&(W_{\coW p}^{kr/\vec r,\vec\om},W_{\coW p}^{(k+2)r/\vec r,\vec\om} 
    )_{1/2},
    &\quad       &s=(k+1)r.
\eal
\right.
\eeq 
We denote by 
\beq\Label{A.def0} 
\ci\gF_p^{s/\vec r,\vec\om}=\ci\gF_p^{s/\vec r,\vec\om}\JV 
\text{ the closure of $\cD(\ci J,\ci\cD)$ in }\gF_p^{s/\vec r,\vec\om}. 
\eeq 
Then negative order spaces are introduced by duality, that is, 
\beq\Label{A.def1} 
\gF_p^{-s/\vec r,\vec\om}=\gF_p^{-s/\vec r,\vec\om}\JV 
:=\bigl(\ci\gF_{p'}^{s/\vec r,-\vec\om}\JVs\bigr)' 
\qa s>0, 
\eeq 
with respect to the duality pairing induced by~%
\hb{\pw_{M\times J}}. We also set 
\hb{s(p):=1/2p} and 
\beq\Label{A.Bdef} 
H_p^{0/\vec r,\vec\om}:=L_p(J,L_p^\lda) 
\qb B_p^{0/\vec r,\vec\om} 
 :=(H_p^{-s(p)/\vec r,\vec\om},H_p^{s(p)/\vec r,\vec\om})_{1/2,p}. 
\eeq 
This defines the \hh{weighted anisotropic Bessel potential space scale} 
\hb{[\,H_p^{s/\vec r,\vec\om}\ ;\ s\in\BR\,]} and the 
\hh{weighted an\-i\-so\-trop\-ic Besov space scale} 
\hb{[\,B_p^{s/\vec r,\vec\om}\ ;\ s\in\BR\,]}. 

\smallskip 
The proof of the following theorem, which describes the interrelations 
between these two scales and gives first interpolation results, is given in 
the next section. Henceforth, 
\hb{\xi_\ta:=(1-\ta)\xi_0+\ta\xi_1} for 
\hb{\xi_0,\xi_1\in\BR} and 
\hb{0\leq\ta\leq1}.\po  
{\samepage 
\begin{theorem}\LabelT{thm-A.WHB} 
\begin{itemize} 
\item[{\rm(i)}] 
${}$
\hb{H_p^{kr/\vec r,\vec\om}\doteq W_{\coW p}^{kr/\vec r,\vec\om}}, 
\ \hb{k\in\BN}.\po 
\item[{\rm(ii)}] 
${}$
\hb{B_2^{s/\vec r,\vec\om}\doteq H_2^{s/\vec r,\vec\om}}, 
\ \hb{s\in\BR}.\po 
\item[{\rm(iii)}] 
${}$
\hb{(B_p^{s_0/\vec r,\vec\om},B_p^{s_1/\vec r,\vec\om})_{\ta,p} 
   \doteq B_p^{s_\ta/\vec r,\vec\om}}, 
\ \hb{0\leq s_0<s_1}, 
\ \hb{0<\ta<1}. 
\item[{\rm(iv)}] 
${}$
\hb{[\gF_p^{s_0/\vec r,\vec\om},\gF_p^{s_1/\vec r,\vec\om}]_\ta 
   \doteq\gF_p^{s_\ta/\vec r,\vec\om}}, 
\ \hb{0\leq s_0<s_1}, 
\ \hb{0<\ta<1}.\po 
\end{itemize} 
\end{theorem} }
Next we prove, among other things, an elementary embedding theorem for 
anisotropic weighted Bessel potential and Besov spaces.\po  
{\samepage 
\begin{theorem}\LabelT{thm-A.HBH} 
\begin{itemize} 
\item[{\rm(i)}] 
${}$
Suppose 
\hb{-\iy<s_0<s<s_1<\iy}. Then 
\beq\Label{A.HBH} 
\cD\JcD\sdh H_p^{s_1/\vec r,\vec\om}\sdh B_p^{s/\vec r,\vec\om} 
\sdh H_p^{s_0/\vec r,\vec\om}. 
\eeq 
\item[{\rm(ii)}] 
${}$
Assume 
\hb{s<1/p} if\/ 
\hb{\pl M\neq\es}, and 
\hb{s<r(1+1/p)} if\/ 
\hb{\pl M=\es} and 
\hb{J=\BR^+}. Then 
\hb{\ci\gF_p^{s/\vec r,\vec\om}=\gF_p^{s/\vec r,\vec\om}}.\po 
\end{itemize} 
\end{theorem} }
\begin{proofc}{of (i) for $s\neq0$} 
Using reiteration theorems, well-known density properties, and 
relations between the real and complex interpolation functor 
(e.g.,~\cite[formula~(I.2.5.2)]{Ama95a} and Theorem~\ref{thm-A.B}(iii)), 
we see that \Eqref{A.HBH} is true if 
\hb{s_0\geq0}. 

\smallskip 
Since $\cD(\ci J,\ci\cD)$ is dense in 
\hb{H_p^{0/\vec r,\vec\om}=L_p(J,L_p^\lda)} it follows 
$$ 
\ci H_p^{s_1/\vec r,\vec\om}\sdh\ci B_p^{s/\vec r,\vec\om} 
\sdh\ci H_p^{s_0/\vec r,\vec\om}\sdh L_p(J,L_p^\lda) 
\qa s_0\geq0. 
$$ 
Hence the definition of the negative order spaces implies that 
\Eqref{A.HBH} holds if \hb{s_1\leq0}, where the density of these embeddings 
follows by reflexivity. This implies assertion~(i) if 
\hb{s\neq0}. 
The proofs for the case 
\hb{s=0} and for assertion~(ii) are given in the next section. 
\end{proofc}\po 
{\samepage 
\begin{corollary}\LabelT{cor-A.FF} 
Suppose 
\hb{s\in\BR}. 
\begin{itemize} 
\item[{\rm(i)}] 
$\gF_p^{s/\vec r,\vec\om}$ is a reflexive Banach space.\po 
\item[{\rm(ii)}] 
If 
\hb{s>0}, then 
\hb{\ci\gF_p^{s/\vec r,\vec\om} 
   =\bigl(\gF_{p'}^{-s/\vec r,-\vec\om}\JVs\bigr)'} with respect to~%
\hb{\pw_{M\times J}}. 
\item[{\rm(iii)}] 
${}$
${}$
\hb{\gF_p^{s_1/\vec r,\vec\om}\sdh\gF_p^{s_0/\vec r,\vec\om}} if 
\hb{s_1>s_0}.\po 
\end{itemize} 
\end{corollary} }
\begin{proof} 
Assume  
\hb{s>0}. Then assertion~(i) follows from the reflexivity 
of~$W_{\coW p}^{kr/\vec r,\vec\om}$ for 
\hb{k\in\BN} and the duality properties of the real and complex 
interpolation functors. Hence $\ci\gF_p^{s/\vec r,\vec\om}\JVs$, being a 
closed linear subspace of a reflexive Banach space, is reflexive. Thus 
$\gF_p^{-s/\vec r,\vec\om}$ is reflexive since it is the dual of a 
reflexive Banach space. We have already seen that 
$H_p^{0/\vec r,\vec\om}$~is reflexive. The reflexivity 
of~$B_p^{0/\vec r,\vec\om}$ follows by interpolation as well. 
This proves~(i) for every 
\hb{s\in\BR}. 

\smallskip 
Assertion~(ii) is a consequence of (i) and \Eqref{A.def1}. Claim~(iii) is 
immediate by \Eqref{A.HBH}. 
\end{proof} 
If $M$~is uniformly regular, that is,  
\hb{\gT(M)=[\![1]\!]}, then $\gF_p^{s/\vec r,\vec\om}$~is independent 
of~$\vec\om$. These non-weighted spaces are denoted 
by~$\gF_p^{s/\vec r}$, of course. If 
\hb{W=M\times\BK}, then we write 
\hb{\gF_p^{s/\vec r,\vec\om}(M\times J)} for 
$\gF_p^{s/\vec r,\vec\om}\JVnn$. Since 
\hb{V_0^0=T_0^0M} is in this case the trivial vector bundle 
\hb{M\times\BK}, whose sections are the \hbox{$\BK$-valued} functions 
on~$M$, this notation is consistent with usual identification 
of~$L_p\bigl(J,L_p^\lda(M)\bigr)$ with 
\hb{L_p^\lda(M\times J)} via the identification of~$u(t)$ 
with~$u(\cdot,t)$. 
\section{The Anisotropic Retraction Theorem}\LabelT{sec-RA}%
Let 
\hb{\{\,E_\al\ ;\ \al\in\sA\,\}} be a countable family of Banach spaces. 
We set 
\hb{L_p(J,\mf{E}):=\prod_\al L_p(J,E_\al)}. Fubini's theorem implies 
\beq\Label{RA.lL} 
\ell_p\bigl(L_p(J,\mf{E})\bigr)=L_p\bigl(J,\ell_p(\mf{E})\bigr), 
\eeq 
using obvious identifications. We also set 
\hb{(\mf{E},\mf{F})_\ta:=\prod_\al(E_\al,F_\al)_\ta} for 
\hb{0<\ta<1} if each $(E_\al,F_\al)$ is an interpolation couple. 

\smallskip 
We presuppose as standing hypothesis 
\bspezeq   
\bal 
\frame{
\begin{minipage}{276pt}
\bspezeq 
\bal 
{}      
&\rgK\text{ is a singularity datum for }M.\\
&\gK\slt\Phi\text{ is a uniformly regular atlas for $W$ over }\gK.\\
&F=\bigl(F,\prsn_F\bigr)\text{ is a model fiber for $W$ 
 with basis }(e_1,\ldots,e_n).\\
&\bigl\{\,(\pi_\ka,\chi_\ka)\ ;\ \ka\in\gK\,\bigr\}\text{ is a 
 localization system subordinate to }\gK.\\
\noalign{\vskip2.5\jot}
\eal
\espezeq 
\end{minipage}} 
\eal 
\npb 
\espezeq  
On the basis of \Eqref{R.peq} we can provide localized versions of the norms 
\hb{\Vsdot_{kr/\vec r,p;\vec\om}} and~%
\hb{\Vsdot_{kr/\vec r,p;\vec\om}^\sim}. 
\begin{theorem}\LabelT{thm-RA.N} 
Suppose 
\hb{k\in\BN}. Set 
$$ 
\Vvsdot_{kr/\vec r,p;\vec\om} 
:=\bigl(\|\vp_p^\lda u\|_{\ell_p(L_p(J,\mf{W}_{\coW p}^{kr}))}^p 
 +\|\vp_p^{\lda+k\mu}(\pl^k u)\|_{\ell_p(L_p(J,\mf{L}_p))}^p\bigr)^{1/p} 
$$ 
and 
$$ 
\Vvsdot_{kr/\vec r,p;\vec\om}^\sim 
:=\Bigl(\|\vp_p^\lda u\|_{\ell_p(L_p(J,\mf{W}_{\coW p}^{kr}))}^p 
 +\sum_{j=0}^k\|\vp_p^{\lda+j\mu}(\pl^j u)\|
  _{\ell_p(L_p(J,\mf{W}_{\coW p}^{(k-j)r}))}^p\Bigr)^{1/p}. 
\npb 
$$ 
Then 
\hb{\Vvsdot_{kr/\vec r,p;\vec\om}\sim\Vsdot_{kr/\vec r,p;\vec\om}} and 
\hb{\Vvsdot_{kr/\vec r,p;\vec\om}^\sim\sim\Vsdot_{kr/\vec r,p;\vec\om}^\sim}. 
\end{theorem} 
\begin{proof} 
This follows from \Eqref{R.peq} and \Eqref{RA.lL}. 
\end{proof} 
It is worthwhile to note 
$$ 
\vd u\vd_{kr/\vec r,p;\vec\om}^\sim 
=\Bigl(\sum_\ka\int_J\sum_{|\al|+jr\leq kr} 
\bigl(\rho_\ka^{\lda+|\al|+j\mu+m/q} 
\,\|\pa_x\pl^j(\ka\slt\vp)_*(\pi_\ka u)\|_{p;E}\bigr)^p\,dt\Bigr)^{1/p}. 
$$ 
Together with Theorems \ref{thm-A.B}(ii) and~\ref{thm-RA.N} this gives a 
rather explicit and practically useful local characterization of 
anisotropic Sobolev spaces. 

\smallskip 
For abbreviation, we set 
$$ 
\bt\quad 
\BY_\ka:=\BX_\ka\times J 
\qb \ka\in\gK. 
$$ 
Hence 
\hb{\ci\BY_\ka=\ci\BX_\ka\times\ci J} is the interior of~$\BY_\ka$ in 
\hb{\BR^{m+1}=\BR^m\times\BR}. We also put 
$$ 
\mf{\cD}\BYE:=\bigoplus_\ka\cD\BYkE 
\qb \mf{\cD}\ciBYE:=\bigoplus_\ka\cD\ciBYE 
$$ 
and 
$$ 
W_{\coW p,\ka}^{kr/\vec r}:=W_{\coW p}^{kr/\vec r}\BYkE  
\qb \gF_{p,\ka}^{s/\vec r}:=\gF_p^{s/\vec r}\BYkE  
\qa s\in\BR 
\qb k\in\BN. 
$$ 
More precisely, the `local' spaces $W_{\coW p,\ka}^{kr/\vec r}$ 
and~$\gF_{p,\ka}^{s/\vec r}$ are special instances of 
$W_{\coW p}^{kr/\vec r,\vec\om}$ 
and~$\gF_p^{s/\vec r,\vec\om}$, respectively, namely with 
\hb{M=(\BX_\ka,g_m)}, 
\ \hb{\rho=\mf{1}}, 
\ \hb{W=\BX_\ka\times F}, and 
\hb{D=d_F}. 

\smallskip 
It is of fundamental importance that these spaces 
coincide with the anisotropic Sobolev, Bessel potential, and Besov spaces 
studied by means of Fourier analytical techniques in detail in 
H.~Amann~\cite{Ama09a}, therein denoted by 
$W_{\coW p}^{k\nu/\mf{\nu}}\BYkE$, 
\ $H_p^{s/\mf{\nu}}\BYkE$, and 
$B_p^{s/\mf{\nu}}\BYkE$, respectively, where 
\hb{\nu:=r} and 
\hb{\mf{\nu}:=(1,r)}. For abbreviation, we set 
\hb{W_{\coW p,\ka}^{k\nu/\mf{\nu}}:=W_{\coW p}^{k\nu/\mf{\nu}}\BYkE} 
and 
\hb{\gF_{p,\ka}^{s/\mf{\nu}}:=\gF_p^{s/\mf{\nu}}\BYkE}. 
Furthermore, we write $\wt{W}_{\coW p,\ka}^{kr/\vec r}$ for 
\hb{L_p(J,W_{\coW p,\ka}^{kr})\cap W_{\coW p}^k(J,L_{p,\ka})} endowed with 
the norm~%
\hb{\Vsdot_{kr/\vec r,p}^\sim}.\po  
{\samepage 
\begin{lemma}\LabelT{lem-RA.WW} 
\begin{itemize} 
\item[{\rm(i)}] 
If\/ 
\hb{k\in\BN}, then 
\hb{W_{\coW p,\ka}^{kr/\vec r}\doteq W_{\coW p,\ka}^{k\nu/\mf{\nu}} 
   \doteq\wt{W}_{\coW p,\ka}^{kr/\vec r}} for 
\hb{\ka\in\gK}. 
\item[{\rm(ii)}] 
If 
\hb{s\in\BR}, then 
\hb{\gF_{p,\ka}^{s/\vec r}\doteq\gF_{p,\ka}^{s/\mf{\nu}}} for 
\hb{\ka\in\gK}.\po 
\end{itemize} 
\end{lemma} }
\begin{proof} 
(1) If 
\hb{J=\BR^+} and 
\hb{\ka\in\gK_{\pl M}}, then $\BY_\ka$~is isomorphic 
to the closed \hbox{$2$-corner} 
\hb{\BR^+\times\BR^+\times\BR^{m-1}} (in the sense of Section~4.3 
of~\cite{Ama09a}) by a permutation isomorphism. Otherwise, $\BY_\ka$~equals 
either the half-space~$\BH^{m+1}$ (except for a possible permutation) 
or~$\BR^{m+1}$. 

\smallskip 
(2) If  
\hb{\BY_\ka=\BR^{m+1}}, then (i)~follows from Theorem~2.3.8 
of~\cite{Ama09a} and the definition of~$W_{\coW p,\ka}^{k\nu/\mf{\nu}}$ 
in the first paragraph of \cite[Section~3.5]{Ama09a}. If 
\hb{\BY_\ka\neq\BR^{m+1}}, then we obtain claim~(i) by invoking 
\cite[Theorem~4.4.3(i)]{Ama09a}. 

\smallskip 
(3) Suppose 
\hb{\BY_\ka=\BR^{m+1}}. Then statement~(ii) follows from 
\cite[Theorem~3.7.1]{Ama09a}. Let 
\hb{\BY_\ka\neq\BR^{m+1}} and \hb{s\neq0} if 
\hb{\gF=B}. Then we get this claim by employing, in addition, 
\cite[Theorems 4.4.1 and~4.4.4]{Ama09a}. If 
\hb{\BY_\ka\neq\BR^{m+1}}, 
\ \hb{\gF=B}, and 
\hb{s=0}, then we have to use \cite[Theorem~4.7.1(ii) and 
Corollary~4.11.2]{Ama09a} in addition. 
\end{proof} 
Due to this lemma we can apply the results of~\cite{Ama09a} to the local 
spaces~$\gF_{p,\ka}^{s/\vec r}$. This will be done in the following usually 
without referring to Lemma~\ref{lem-RA.WW}. 

\smallskip
Let $\cX$ be a locally convex space and 
\hb{1\leq q\leq\iy}. 
For 
\hb{\ka\in\gK} we consider the linear map 
\hb{\Ta_{q,\ka}^\mu\sco\cX^J\ra\cX^J} 
defined by 
\beq\Label{RA.tdef} 
\Ta_{q,\ka}^\mu u(t):=\rho_\ka^{\mu/q}u(\rho_\ka^\mu t) 
\qa u\in\cX^J 
\qb t\in J. 
\eeq 
Note 
\beq\Label{RA.tab} 
\Ta_{q,\ka}^\mu\circ \Ta_{q,\ka}^{-\mu}=\Ta_{q,\ka}^0=\id 
\eeq 
and 
\beq\Label{RA.tc} 
\Ta_{q,\ka}^\mu\bigl(C\JcX\bigr)\is C\JcX. 
\eeq 
Moreover, 
\beq\Label{RA.tk} 
\pl^k\circ\Ta_{q,\ka}^\mu=\rho_\ka^{k\mu}\Ta_{q,\ka}^\mu\circ\pl 
\qa k\in\BN, 
\eeq 
and, if $\cX$~is a Banach space, 
\beq\Label{RA.tLq} 
\|\Ta_{q,\ka}^\mu u\|_{L_q\JcX}=\|u\|_{L_q\JcX}. 
\eeq 

\smallskip 
We put 
\beq\Label{RA.ph} 
\vp_{q,\ka}^{\vec\om}u:=\Ta_{q,\ka}^\mu\circ\vp_{q,\ka}^\lda u 
\qb \vp_q^{\vec\om}u:=(\vp_{q,\ka}^{\vec\om}u) 
\qa u\in C\bigl(J,C(V)\bigr), 
\eeq 
and 
\beq\Label{RA.ps} 
\psi_{q,\ka}^{\vec\om}v_\ka:=\Ta_{q,\ka}^{-\mu}\circ\psi_{q,\ka}^\lda v_\ka  
\qb \psi_q^{\vec\om}\mf{v}:=\sum_\ka\psi_{q,\ka}^{\vec\om}v_\ka 
\qa \mf{v}=(v_\ka)\in\bigoplus_\ka C\BYkE.  
\eeq 

\smallskip 
After these preparations we can prove the following analogue to 
Theorem~\ref{thm-R.R}. Not only will it play a fundamental role in this 
paper but also be decisive for the study of parabolic equations on singular 
manifolds. 
\begin{theorem}\LabelT{thm-RA.R} 
Suppose  
\hb{s\in\BR}. Then the diagrams 
 \bspezeq 
 \bal 
 \begin{picture}(210,88)(-105,-41)        
 \put(0,40){\makebox(0,0)[b]{\small{$d$}}}
 \put(0,5){\makebox(0,0)[b]{\small{$d$}}}
 \put(0,-30){\makebox(0,0)[b]{\small{$d$}}}
 \put(-90,35){\makebox(0,0)[c]{\small{$\cD\JcD$}}}
 \put(-90,-35){\makebox(0,0)[c]{\small{$\cD\JcD$}}}
 \put(-40,0){\makebox(0,0)[c]{\small{$\mf{\cD}\BYE$}}}
 \put(44,0){\makebox(0,0)[c]{\small{$\ell_p(\mf{\gF}_p^{s/\vec r})$}}}
 \put(92,35){\makebox(0,0)[c]{\small{$\gF_p^{s/\vec r,\vec\om}$}}}
 \put(92,-35){\makebox(0,0)[c]{\small{$\gF_p^{s/\vec r,\vec\om}$}}}
 \put(-53,17.5){\makebox(0,0)[r]{\small{$\vp_p^{\vec\om}$}}}
 \put(58,17.5){\makebox(0,0)[l]{\small{$\vp_p^{\vec\om}$}}}
 \put(-54,-17.5){\makebox(0,0)[r]{\small{$\psi_p^{\vec\om}$}}}
 \put(55,-17.5){\makebox(0,0)[l]{\small{$\psi_p^{\vec\om}$}}}
 \put(-95,0){\makebox(0,0)[r]{\small{$\id$}}}
 \put(97,0){\makebox(0,0)[l]{\small{$\id$}}}
 \put(-68,35){\vector(1,0){140}}
 \put(-68,37){\oval(4,4)[l]}
 \put(-18,0){\vector(1,0){40}}
 \put(-18,2){\oval(4,4)[l]}
 \put(-68,-35){\vector(1,0){140}}
 \put(-68,-33){\oval(4,4)[l]}
 \put(-90,25){\vector(0,-1){50}}
 \put(92,25){\vector(0,-1){50}}
 \put(-80,25){\vector(1,-1){20}}
 \put(-60,-5){\vector(-1,-1){20}}
 \put(82,25){\vector(-1,-1){20}}
 \put(62,-5){\vector(1,-1){20}}
 \end{picture} 
 \qquad 
 \begin{picture}(210,88)(-105,-41)        
 \put(0,40){\makebox(0,0)[b]{\small{$d$}}}
 \put(0,5){\makebox(0,0)[b]{\small{$d$}}}
 \put(0,-30){\makebox(0,0)[b]{\small{$d$}}}
 \put(-90,35){\makebox(0,0)[c]{\small{$\cD\ciJcicD$}}}
 \put(-90,-35){\makebox(0,0)[c]{\small{$\cD\ciJcicD$}}}
 \put(-40,0){\makebox(0,0)[c]{\small{$\mf{\cD}\ciBYE$}}}
 \put(44,0){\makebox(0,0)[c]{\small{$\ell_p(\cimfgF_p^{s/\vec r})$}}}
 \put(92,35){\makebox(0,0)[c]{\small{$\ci\gF_p^{s/\vec r,\vec\om}$}}}
 \put(92,-35){\makebox(0,0)[c]{\small{$\ci\gF_p^{s/\vec r,\vec\om}$}}}
 \put(-53,17.5){\makebox(0,0)[r]{\small{$\vp_p^{\vec\om}$}}}
 \put(58,17.5){\makebox(0,0)[l]{\small{$\vp_p^{\vec\om}$}}}
 \put(-54,-17.5){\makebox(0,0)[r]{\small{$\psi_p^{\vec\om}$}}}
 \put(55,-17.5){\makebox(0,0)[l]{\small{$\psi_p^{\vec\om}$}}}
 \put(-95,0){\makebox(0,0)[r]{\small{$\id$}}}
 \put(97,0){\makebox(0,0)[l]{\small{$\id$}}}
 \put(-68,35){\vector(1,0){140}}
 \put(-68,37){\oval(4,4)[l]}
 \put(-18,0){\vector(1,0){40}}
 \put(-18,2){\oval(4,4)[l]}
 \put(-68,-35){\vector(1,0){140}}
 \put(-68,-33){\oval(4,4)[l]}
 \put(-90,25){\vector(0,-1){50}}
 \put(92,25){\vector(0,-1){50}}
 \put(-80,25){\vector(1,-1){20}}
 \put(-60,-5){\vector(-1,-1){20}}
 \put(82,25){\vector(-1,-1){20}}
 \put(62,-5){\vector(1,-1){20}}
 \end{picture} 
 \eal 
 \npb 
 \espezeq    
are commuting, where 
\hb{s>0} in the second case.  
\end{theorem} 
\begin{proof} 
(1) It is not difficult to see that 
\hb{\cD\JcDka=\cD\BYkE} by means of the identification 
\hb{u(t)=u(\cdot,t)} for 
\hb{t\in J} (see Corollary~1 in Section~40 of F.~Treves~\cite{Tre67a}, 
for example). Consequently, 
$$  
\mf{\cD}\BYE=\bigoplus_\ka\cD\JcDka. 
$$ 
Similarly, 
\hb{\cD\ciJcicDka=\cD\ciBYkE}, and thus 
$$ 
\mf{\cD}\ciBYE=\bigoplus_\ka\cD\ciJcicDka. 
$$  
Using this, \Eqref{RA.tc}, and \Eqref{RA.tk}, obvious modifications of the 
proof of Theorem~5.1 in~\cite{Ama12b} show that the assertions encoded in 
the respective left triangles of the diagrams are true. 

\smallskip 
(2) Suppose 
\hb{k\in\BN}. From \Eqref{RA.tLq} we get 
$$ 
\|\vp_{p,\ka}^{\vec\om}u\|_{L_p(J,W_{\coW p,\ka}^{kr})} 
=\|\vp_{p,\ka}^\lda u\|_{L_p(J,W_{\coW p,\ka}^{kr})}. 
$$ 
Hence, using \Eqref{RA.lL} 
$$ 
\|\vp_p^{\vec\om}u\|_{\ell_p(L_p(J,\mf{W}_{\coW p}^{kr}))} 
=\|\vp_p^\lda u\|_{L_p(J,\ell_p(\mf{W}_{\coW p}^{kr}))}. 
$$ 
From this and Theorem~\ref{thm-R.R} we deduce 
$$ 
\|\vp_p^{\vec\om}u\|_{\ell_p(L_p(J,\mf{W}_{\coW p}^{kr}))} 
\leq c\,\|u\|_{L_p(J,W_{\coW p}^{kr,\lda})},  
$$ 
that is, 
\beq\Label{RA.phc0} 
\vp_p^{\vec\om}\in\cL\bigl(L_p(J,W_{\coW p}^{kr,\lda}), 
\ell_p(L_p(J,\mf{W}_{\coW p}^{kr}))\bigr). 
\eeq 
By means of \Eqref{RA.tk} and~\Eqref{RA.tLq} we obtain 
\beq\Label{RA.phj} 
\|\pl^j\vp_{p,\ka}^{\vec\om}u\|_{L_p(J,W_{\coW p,\ka}^{k-j})} 
=\|\vp_{p,\ka}^{\lda+j\mu}(\pl^ju)\|_{L_p} 
\qa 0\leq j\leq k. 
\eeq 
Consequently, invoking \Eqref{RA.lL} and Theorem~\ref{thm-R.R} once more, 
$$ 
\|\pl^k\vp_p^{\vec\om}u\|_{\ell_p(L_p(J,\mf{L}_p))} 
\leq c\,\|\pl^ku\|_{L_p(J,L_p^{\lda+k\mu})}. 
$$ 
This, together with \Eqref{RA.phc0}, implies 
\beq\Label{RA.phc} 
\vp_p^{\vec\om} 
\in\cL\bigl(W_{\coW p}^{kr/\vec r,\vec\om}, 
\ell_p(\mf{W}_{\coW p}^{kr/\vec r})\bigr). 
\eeq 

\smallskip 
(3) Note that 
\beq\Label{RA.as} 
\vp_{p,\ka}^\lda\psi_{p,\wt{\ka}}^\lda=a_{\wt{\ka}\ka}S_{\wt{\ka}\ka}, 
\eeq 
where 
$$ 
a_{\wt{\ka}\ka}:=(\rho_\ka/\rho_{\wt{\ka}})^{\lda+m/p}(\ka_*\pi_\ka) 
S_{\wt{\ka}\ka}(\wt{\ka}_*\pi_{\wt{\ka}}). 
$$ 
Lemma~\ref{lem-R.S}, estimate~\Eqref{S.err}, and \Eqref{R.LS}(iii) imply 
$$ 
a_{\wt{\ka}\ka}\in BC^k(\BX_\ka) 
\qb \|a_{\wt{\ka}\ka}\|_{k,\iy}\leq c(k) 
\qa \wt{\ka}\in\gN(\ka) 
\qb \ka\in\gK 
\qb k\in\BN. 
$$ 
Hence we infer from \Eqref{RA.as} and Lemma~\ref{lem-R.S} 
$$  
\vp_{p,\ka}^\lda\psi_{p,\wt{\ka}}^\lda 
\in\cL(W_{\coW p,\wt{\ka}}^k,W_{\coW p,\ka}^k) 
\qb \|\vp_{p,\ka}^\lda\psi_{p,\wt{\ka}}^\lda\|\leq c(k) 
$$ 
for 
\hb{\wt{\ka}\in\gN(\ka)}, 
\ \hb{\ka\in\gK}, and  
\hb{k\in\BN}. By this and \Eqref{RA.tLq} we find 
\beq\Label{RA.phpt} 
\|\vp_{p,\ka}^\lda 
\psi_{p,\wt{\ka}}^{\vec\om}v\|_{L_p(J,W_{\coW p,\ka}^k)}  
=\|\vp_{p,\ka}^\lda\psi_{p,\wt{\ka}}^\lda v\|_{L_p(J,W_{\coW p,\ka}^k)}  
\leq c\,\|v\|_{L_p(J,W_{\coW p,\wt{\ka}}^k)}  
\eeq 
for 
\hb{\wt{\ka}\in\gN(\ka)}, 
\ \hb{\ka\in\gK}, and  
\hb{k\in\BN}. Similarly, using \Eqref{RA.tk}, 
\beq\Label{RA.phptk} 
\big\|\vp_{p,\ka}^{\lda+k\mu} 
\bigl(\pl^k(\psi_{p,\wt{\ka}}^{\vec\om}v)\big)\big\|_{L_p(J,L_{p,\ka})}  
=\|\vp_{p,\ka}^{\lda+k\mu}\psi_{p,\wt{\ka}}^{(\lda+k\mu,\mu)}
(\pl^kv)\|_{L_p(J,L_{p,\ka})}  
\leq c\,\|\pl^kv\|_{L_p(J,L_{p,\wt{\ka}})} 
\npb 
\eeq 
for 
\hb{\wt{\ka}\in\gN(\ka)}, 
\ \hb{\ka\in\gK}, and  
\hb{k\in\BN}. 

\smallskip 
Observe 
\beq\Label{RA.ppN} 
\vp_{p,\ka}^\lda\psi_p^{\vec\om}\mf{v} 
=\sum_{\wt{\ka}\in\gN(\ka)}\vp_{p,\ka}^\lda 
\psi_{p,\wt{\ka}}^{\vec\om}v_{\wt{\ka}}. 
\eeq 
From \hbox{\Eqref{RA.phpt}--\Eqref{RA.ppN}} 
and the finite multiplicity of~$\gK$ we infer  
\beq\Label{RA.phpv} 
\|\vp_p^\lda(\psi_p^{\vec\om}\mf{v})\|_{\ell_p(L_p(J,\mf{W}_{\coW p}^{kr}))} 
\leq c\,\|\mf{v}\|_{\ell_p(L_p(J,\mf{W}_{\coW p}^{kr}))} 
\eeq 
and 
$$ 
\big\|\vp_p^{\lda+k\mu} 
\bigl(\pl^k(\psi_p^{\vec\om}\mf{v})\big)\big\|_{\ell_p(L_p(J,\mf{L}_p))}  
\leq c\,\|\pl^k\mf{v}\|_{\ell_p(L_p(J,\mf{L}_p))}.  
$$ 
Hence Theorem~\ref{thm-RA.N} implies 
$$ 
\|\psi_p^{\vec\om}\mf{v}\|_{kr/\vec r,p:\vec\om}  
\leq c\,\|\mf{v}\|_{\ell_p(\mf{W}_{\coW p}^{kr/\vec r})},  
$$ 
that is, 
\beq\Label{RA.psc} 
\psi_p^{\vec\om} 
\in\cL\bigl(\ell_p(\mf{W}_{\coW p}^{kr/\vec r}), 
W_{\coW p}^{kr/\vec r,\vec\om}\bigr). 
\eeq 
It follows from 
\hb{\psi_p^\lda\vp_p^\lda=\id} that 
\hb{\psi_p^{\vec\om}\vp_p^{\vec\om}=\id}. Thus we see from \Eqref{RA.phc} 
and \Eqref{RA.psc} that the diagram 
\bspezeqn\Label{RA.Wre} 
\bal
 \begin{picture}(143,59)(-67,-6)        
\put(15,10){\vector(2,1){40}}
\put(-45,30){\vector(2,-1){40}}
\put(-30,40){\vector(1,0){60}}
\put(5,0){\makebox(0,0){\small$\ell_p(\mf{W}_{\coW p}^{kr/\vec r})$}}
\put(-50,40){\makebox(0,0){\small$W_{\coW p}^{kr/\vec r,\vec\om}$}}
\put(60,40){\makebox(0,0){\small$W_{\coW p}^{kr/\vec r,\vec\om}$}}
\put(0,45){\makebox(0,0)[b]{\small$\id$}}
\put(-30,15){\makebox(0,0)[r]{\small$\vp_p^{\vec\om}$}}
\put(37.5,15){\makebox(0,0)[l]{\small$\psi_p^{\vec\om}$}}
\end{picture}
\eal 
\npb 
\espezeqn 
is commuting. 

\smallskip 
(4) It is a consequence of Lemma~\ref{lem-RA.WW}(i) 
and \cite[Theorems 2.3.2(i) and~4.4.1]{Ama09a} that 
\hb{\cD\JcDka=\cD\BYkE} is dense in~$W_{\coW p,\ka}^{kr/\vec r}$. 
This implies 
$$ 
\mf{\cD}\BYE\sdh\bigoplus_\ka W_{\coW p,\ka}^{kr/\vec r} 
=c_c(\mf{W}_{\coW p}^{kr/\vec r}). 
$$ 
Hence, by \Eqref{R.cl}, 
\beq\Label{RA.DlW} 
\mf{\cD}\BYE\sdh\ell_p(\mf{W}_{\coW p}^{kr/\vec r}). 
\eeq 
Thus we deduce from step~(1) and \Eqref{RA.Wre} that 
 \bspezeq 
  \bal 
 \begin{picture}(221,88)(-105,-41)        
 \put(0,5){\makebox(0,0)[b]{\small{$d$}}}
 \put(-90,35){\makebox(0,0)[c]{\small{$\cD\JcD$}}}
 \put(-90,-35){\makebox(0,0)[c]{\small{$\cD\JcD$}}}
 \put(-40,0){\makebox(0,0)[c]{\small{$\mf{\cD}\BYE$}}}
 \put(48,0){\makebox(0,0)[c]{\small{$\ell_p(\mf{W}_{\coW p}^{kr/\vec r})$}}}
 \put(101,35){\makebox(0,0)[c]{\small{$W_{\coW p}^{kr/\vec r,\vec\om}$}}}
 \put(101,-35){\makebox(0,0)[c]{\small{$W_{\coW p}^{kr/\vec r,\vec\om}$}}}
 \put(-53,17.5){\makebox(0,0)[r]{\small{$\vp_p^{\vec\om}$}}}
 \put(67,17.5){\makebox(0,0)[l]{\small{$\vp_p^{\vec\om}$}}}
 \put(-54,-17.5){\makebox(0,0)[r]{\small{$\psi_p^{\vec\om}$}}}
 \put(65,-17.5){\makebox(0,0)[l]{\small{$\psi_p^{\vec\om}$}}}
 \put(-95,0){\makebox(0,0)[r]{\small{$\id$}}}
 \put(106,0){\makebox(0,0)[l]{\small{$\id$}}}
 \put(-68,35){\vector(1,0){149}}
 \put(-68,37){\oval(4,4)[l]}
 \put(-18,0){\vector(1,0){40}}
 \put(-18,2){\oval(4,4)[l]}
 \put(-68,-35){\vector(1,0){149}}
 \put(-68,-33){\oval(4,4)[l]}
 \put(-90,25){\vector(0,-1){50}}
 \put(101,25){\vector(0,-1){50}}
 \put(-80,25){\vector(1,-1){20}}
 \put(-60,-5){\vector(-1,-1){20}}
 \put(91,25){\vector(-1,-1){20}}
 \put(71,-5){\vector(1,-1){20}}
 \end{picture} 
 \eal 
 \espezeq 
is a commuting diagram. From this and \cite[Lemma~4.1.6]{Ama09a} we obtain 
\beq\Label{RA.DW} 
\cD\JcD\sdh W_{\coW p}^{kr/\vec r,\vec\om}.  
\eeq 

\smallskip 
(5) Suppose 
\hb{k\in\BN} and 
\hb{kr<s\leq(k+1)r}. If 
\hb{s<(k+1)r}, set 
\hb{\ta:=(s-kr)/r} and 
\hb{\ell:=k+1}. Otherwise, 
\hb{\ta:=1/2} and 
\hb{\ell:=k+2}. Then we infer from \Eqref{RA.Wre} and \Eqref{A.def} by 
interpolation that $\psi_p^{\vec\om}$~is a retraction from 
\beq\Label{RA.llt} 
\bigl(\ell_p(\mf{W}_{\coW p}^{kr/\vec r}), 
\ell_p(\mf{W}_{\coW p}^{\ell r/\vec r})\bigr)_\ta 
\eeq 
onto~$\gF_p^{s/\vec r,\vec\om}$. By Theorem~1.18.1 in 
H.~Triebel~\cite{Tri78a}, \ \Eqref{RA.llt} equals
\hb{\ell_p\bigl((\mf{W}_p^{kr/\vec r},
   \mf{W}_{\coW p}^{\ell r/\vec r})_\ta\bigr)}, 
except for equivalent norms. 
   
\smallskip 
It follows from Lemma~\ref{lem-RA.WW} and \cite[Theorem~3.7.1(iv), 
formula~(3.3.12), and Theorems 3.5.2 and~4.4.1]{Ama09a} that 
\hb{(W_{p,\ka}^{kr/\vec r},W_{\coW p,\ka}^{\ell r/\vec r})_\ta
   \doteq\gF_{p,\ka}^{s/\vec r}}. This shows that the right triangle of the 
first diagram is commuting if 
\hb{s>0}. Furthermore, the density properties of the interpolation 
functor~%
\hb{\pr_\ta}, \ \Eqref{RA.DlW}, and \Eqref{RA.DW} imply that the `horizontal 
embeddings' of the first diagram of the assertion are dense if 
\hb{s>0}. This proves the first assertion for 
\hb{s>0}. 
  
\smallskip 
(6) 
It is a consequence of what has just been shown and step~(1) that 
the second part of the statement is true. 

\smallskip 
(7) 
Let $X$ be a reflexive Banach space. Then 
$$ 
\dl v,\Ta_{p,\ka}^\mu u\dr_{L_p\JX}=\dl\Ta_{p',\ka}^{-\mu}v,u\dr_{L_p\JX} 
\qa u\in L_p\JX 
\qb v\in L_{p'}\JXs=\bigl(L_p\JX\bigr)'. 
$$ 
We define $\ci\vp_p^{\vec\om}$ 
and~$\ci\psi_p^{\vec\om}$ by replacing $\vp_{p,\ka}^\lda$ 
and~$\psi_{p,\ka}^\lda$ in \Eqref{RA.ph} and \Eqref{RA.ps} 
by $\ci\vp_{p,\ka}^\lda$ and~$\ci\psi_{p,\ka}^\lda$, defined in 
\Eqref{R.ph0}, respectively. From 
this we infer (cf.~the proof of Theorem~5.1 in~\cite{Ama12b}) 
\beq\Label{RA.psdu} 
\dl\ci\psi_{p'}^{-\vec\om}\mf{v},u\dr_{M\times J} 
=\mfdl\mf{v},\vp_p^{\vec\om}u\mfdr 
\qa \mf{v}\in\mf{\cD}\ciBYE 
\qb u\in\cD\JcD, 
\eeq 
and 
\beq\Label{RA.phdu} 
\mfdl\ci\vp_{p'}^{-\vec\om}v,\mf{u}\mfdr 
=\dl v,\psi_p^{\vec\om}\mf{u}\dr_{M\times J} 
\qa v\in\cD\ciJcicD 
\qb \mf{u}\in\mf{\cD}\BYE. 
\eeq 
Moreover, \Eqref{A.def1} implies for 
\hb{s>0} 
$$ 
\ell_p(\mf{\gF}_p^{-s/\vec r}) 
=\bigl(\ell_{p'}\bigl(\mf{\gF}_{p'}^{s/\vec r}\BYEs\bigr)\bigr)'. 
$$ 
It follows from \Eqref{R.rho} that $\ci\vp_p^{\vec\om}$ 
and~$\ci\psi_p^{\vec\om}$ possess the same mapping properties as 
$\vp_p^{\vec\om}$ and~$\psi_p^{\vec\om}$, respectively. 
Hence we deduce from \Eqref{RA.psdu} and \Eqref{RA.phdu} that, given 
\hb{s>0}, 
$$ 
\|\vp_p^{\vec\om}u\|_{\ell_p(\mf{\gF}_p^{-s/\vec r})} 
\leq c\,\|u\|_{\gF_p^{-s/\vec r,\vec\om}} 
\qa u\in\cD\JcD,  
$$ 
and 
\beq\Label{RA.psD} 
\|\psi_p^{\vec\om}\mf{u}\|_{\gF_p^{-s/\vec r,\vec\om}} 
\leq c\,\|\mf{u}\|_{\ell_p(\mf{\gF}_p^{-s/\vec r})} 
\qa \mf{u}\in\mf{\cD}\BYE.  
\eeq 
We infer from \Eqref{RA.DW}, Theorem~\ref{thm-A.HBH}(i), and reflexivity 
that $\cD\JcD$ is dense in~$\gF_p^{-s/\vec r,\vec\om}$. Hence 
\beq\Label{RA.phc1} 
\vp_p^{\vec\om}\in\cL\bigl(\gF_p^{-s/\vec r,\vec\om}, 
\ell_p(\mf{\gF}_p^{-s/\vec r})\bigr). 
\eeq 
Since, as above, 
\hb{\cD\JcDka=\cD\BYkE} is dense in~$\gF_{p,\ka}^{-s/\vec r}$ we see, 
by the arguments used to prove \Eqref{RA.DlW}, 
\beq\Label{RA.Dlp} 
\mf{\cD}\BYE\sdh\ell_p(\mf{\gF}_p^{-s/\vec r}). 
\eeq 
Thus \Eqref{RA.psD} implies 
\beq\Label{RA.psc1} 
\psi_p^{\vec\om} 
\in\cL\bigl(\ell_p(\mf{\gF}_p^{-s/\vec r}),\gF_p^{-s/\vec r,\vec\om}\bigr). 
\npb 
\eeq 
From \hbox{\Eqref{RA.phc1}--\Eqref{RA.psc1}}
and step~(1) it now follows that the first statement is true if 
\hb{s<0}. 

\smallskip 
(8) Suppose 
\hb{s=0}. If 
\hb{\gF=H}, then assertion~(i) is contained in \Eqref{RA.Wre} (for 
\hb{k=0}). If 
\hb{\gF=B}, then we deduce from Lemma~\ref{lem-RA.WW}(ii) and 
\cite[Theorems 3.7.1, \ 4.4.1, \ 4.7.1(ii), and Corollary~4.11.2]{Ama09a} 
that 
$$ 
(H_{p,\ka}^{-s(p)/\vec r},H_{p,\ka}^{s(p)/\vec r})_{1/2,p} 
\doteq B_{p,\ka}^{0/\vec r} 
\qa \ka\in\gK. 
$$ 
Thus, as in step~(5), 
$$ 
\bigl(\ell_p(\mf{H}_p^{-s(p)/\vec r}),\ell_p(\mf{H}_p^{s(p)/\vec r})
\bigr)_{1/2,p} 
\doteq\ell_p(\mf{B}_p^{0/\vec r}). 
$$ 
Since we have already shown that $\psi_p^{\vec\om}$~is a retraction 
from~$\ell_p(\mf{H}_p^{\pm s(p)/\vec r})$ onto~$H_p^{\pm s(p)/\vec r}$, 
it follows from definition~\Eqref{A.Bdef} that it is a retraction 
from~$\ell_p(\mf{B}_p^{0/\vec r})$ onto~$B_p^{0/\vec r,\vec\om}$.
This proves the theorem. 
\end{proof}
Now we can supply the proofs left out in Section~\ref{sec-A}. First note 
that assertion~(iii) of Theorem~\ref{thm-A.B} has been shown 
in~\Eqref{RA.DW}.
\begin{proofc}{of part~(ii) of Theorem~\ref{thm-A.B}}   
It is a consequence of Lemma~\ref{lem-RA.WW}(i) that 
$$
\ell_p(\mf{W}_{\coW p}^{kr/\vec r}) 
\doteq\ell_p(\wtmfW_{\coW p}^{kr/\vec r}). 
$$ 
Hence, due to \Eqref{R.rN} and \Eqref{RA.Wre}, 
$$ 
\|\vp_p^{\vec\om}\cdot\|_{\ell_p(\wtmfW_p^{\kern2pt kr/\vec r})} 
\sim\Vsdot_{kr/\vec r,p;\vec\om}.
$$ 
Using \Eqref{RA.phc0} and \Eqref{RA.phj} one verifies 
$$ 
\|\vp_p^{\vec\om}\cdot\|_{\ell_p(\wtmfW_p^{\kern2pt kr/\vec r})} 
\sim\Vvsdot_{kr/\vec r,p;\vec\om}^\sim. 
\npb 
$$ 
Now the assertion follows from Theorem~\ref{thm-RA.N}. 
\end{proofc} 
\begin{proofc}{of Theorem~\ref{thm-A.WHB}}   
(1) Lemma~\ref{lem-RA.WW} and \cite[Theorems 3.7.1 and 4.4.3(i)]{Ama09a} 
imply 
\hb{H_{p,\ka}^{kr/\vec r}\doteq W_{\coW p,\ka}^{kr/\vec r}} for 
\hb{\ka\in\gK} and 
\hb{k\in\BN}. Hence 
$$ 
\ell_p(\mf{H}_p^{kr/\vec r})\doteq\ell_p(\mf{W}_{\coW p}^{kr/\vec r}) 
\qa k\in\BN, 
\npb 
$$ 
and assertion~(i) is a consequence of Theorem~\ref{thm-RA.R}.

\smallskip 
(2) In order to prove~(ii) it suffices, due to Theorem~\ref{thm-RA.R} 
and Lemma~\ref{lem-RA.WW}, to show 
\hb{H_2^{s/\mf{\nu}}\BYkE\doteq B_2^{s/\vec r}\BYkE}. By the results 
of Section~4.4 of~\cite{Ama09a} we can assume 
\hb{\BY_\ka=\BR^{m+1}}. 

\smallskip 
Suppose 
\hb{s>0} and write 
\hb{H_2^s:=H_2^s\RmE}, etc. Then \cite[Theorem~3.7.2]{Ama09a} asserts 
$$ 
H_2^{s/\mf{\nu}}=L_p(\BR,H_2^s)\cap H_2^{s/\nu}(\BR,L_2). 
$$ 
From Theorem~3.6.7 of~\cite{Ama09a} we get 
$$ 
B_2^{s/\mf{\nu}}=L_2(\BR,B_2^s)\cap B_2^{s/\nu}(\BR,L_2). 
$$ 
By Theorem~2.12 in~\cite{Tri78a} we know that 
\hb{H_2^s\doteq B_2^s}. Remark~7 and Proposition~2(1) in 
H.-J. Schmei{\ss}er and W.~Sickel~\cite{SchmS01a} guarantee 
\hb{H_2^{s/\nu}(\BR,L_2)\doteq B_2^{s/\nu}(\BR,L_2)}. This proves 
\hb{H_2^{s/\mf{\nu}}\doteq B_2^{s/\mf{\nu}}} for 
\hb{s>0}. The case 
\hb{s<0} follows by duality. 

\smallskip 
From Lemma~\ref{lem-RA.WW}(ii) and \cite[(3.4.1) and Theorem~3.7.1]{Ama09a} 
we get 
\hb{[\gF_2^{-s(p)/\vec r},\gF_2^{s(p)/\vec r}]_{1/2}\doteq\gF_2^{0/\vec r}}. 
Thus, by what we already know, 
$$ 
B_2^{0/\vec r} 
\doteq[B_2^{-s(p)/\vec r},B_2^{s(p)/\vec r}]_{1/2} 
\doteq[H_2^{-s(p)/\vec r},H_2^{s(p)/\vec r}]_2 
\doteq H_2^{0/\vec r}. 
\npb 
$$ 
This settles the case 
\hb{s=0} also. 

\smallskip 
(3) By \cite[\ (3.3.12), \ (3.4.1), and Theorems 3.7.1(iv) 
and 4.4.1]{Ama09a} we know that assertions (iii) and~(iv) hold for the 
local spaces~$\gF_{p,\ka}^{s/\vec r}$. Thus we get (iii) and~(iv) in the 
general case by the arguments of step~(5) of the proof of 
Theorem~\ref{thm-RA.R}. 
\end{proofc} 
\begin{proofc}{of Theorem~\ref{thm-A.HBH}(i) for \hb{s=0}}  Since 
\Eqref{A.HBH} has already been established for 
\hb{s\in\BR\ssm\{0\}} it remains to show that 
$$ 
H_p^{s_1/\vec r,\vec\om}\sdh B_p^{0/\vec r,\vec\om} 
\sdh H_p^{s_0/\vec r,\vec\om} 
$$ 
if 
\hb{-1+1/p<s_0<0<s_1<1/p}. By \cite[Theorems 3.7.1(iii), \ 4.4.1, 
\ 4.7.1(ii), and Corollary~4.11.2]{Ama09a} 
$$ 
H_{p,\ka}^{s_1/\mf{\nu}} \sdh B_{p,\ka}^{0/\mf{\nu}} 
\sdh H_{p,\ka}^{s_0/\mf{\nu}}. 
$$ 
From this and Lemma~\ref{lem-RA.WW} we deduce 
$$ 
\ell_p(\mf{H}_p^{s_1/\vec r})\sdh\ell_p(\mf{B}_p^{0/\vec r}) 
\sdh\ell_p(\mf{H}_p^{s_0/\vec r}). 
\npb 
$$ 
Now the claim follows from Theorem~\ref{thm-RA.R}.
\end{proofc} 
\begin{proofc}{of Theorem~\ref{thm-A.HBH}(ii)} 
If 
\hb{J=\BR} and 
\hb{\pl M=\es}, then the claim is obvious by \Eqref{A.def0}, 
\ \hb{\cD(\ci J,\ci\cD)=\cD\JcD}, and~(i). Otherwise, we get from 
\cite[Theorem~4.7.1 and Corollary~4.11.2]{Ama09a}, due to the stated 
restrictions for~$s$,  that 
\hb{\ci\gF_{p,\ka}^{s/\vec r}=\gF_{p,\ka}^{s/\vec r}}. Here we also used 
the fact that 
$$ 
\cD\JcDka\sdh\gF_\ka^{0/\vec r}\sdh\gF_\ka^{-t/\vec r} 
\qa t>0 
\qb \ka\in\gK. 
\npb 
$$ 
Hence 
\hb{\ell_p(\cimfgF_p^{s/\vec r})=\ell_p(\mf{\gF}_p^{s/\vec r})} and 
the claim follows from (the right triangles of the diagrams of) 
Theorem~\ref{thm-RA.R}.
\end{proofc} 
\section{Renorming of Besov Spaces}\LabelT{sec-N}%
Let $\cX$ be a Banach space and 
\hb{\BX\in\{\BR^m,\BH^m\}}. For 
\hb{u\sco\BX\ra \cX} and 
\hb{h\in\BH^m\ssm\{0\}} we put 
$$ 
\tri_h u:=u(\cdot+h)-u 
\qb \tri_h^{k+1}u:=\tri_h\tri_h^ku 
\qa k\in\BN 
\qb \tri_h^0u:=u. 
$$ 

\smallskip 
Given 
\hb{k\leq s<k+1} with 
\hb{s>0}, 
$$ 
[u]_{s,p;\cX} 
:=\Bigl(\int_\BX\Bigl(\frac{\|\tri_h^{k+1}u\|_{p;\cX}}{|h|^s}\Bigr)^p 
\,\frac{dh}{|h|^m}\Bigr)^{1/p},
$$ 
where 
\hb{\Vsdot_{p;\cX}:=\Vsdot_{L_p\BXcX}}. We set for 
\hb{s>0}
$$ 
\Vsdot_{s,p;\cX}^*:=\bigl(\Vsdot_{p;\cX}^p+\esdot_{s,p;\cX}^p\bigr)^{1/p}. 
$$ 
Suppose 
\hb{k\leq s<k+1} with 
\hb{k\in\BN} and 
\hb{s>0}. Then  
$$ 
\|u\|_{k,p;\cX} 
:=\Bigl(\sum_{|\al|\leq k}\|\pa_xu\|_{p;\cX}^p\Bigr)^{1/p}
$$ 
is the norm of the \emph{\hbox{$\cX$-valued} Sobolev 
space}~$W_{\coW p}^k\BXcX$ and 
$$ 
\|u\|_{s,p;\cX}^{**} 
:=\left\{
\bal 
{}      
&\Bigl(\|u\|_{k,p;\cX}^p+\sum_{|\al|=k}[\pa_xu]_{s-k,p;\cX}^p\Bigr)^{1/p}, 
&\quad k<s   &<k+1,\\
&\Bigl(\|u\|_{k-1,p;\cX}^p+\sum_{|\al|=k-1}[\pa_xu]_{1,p;\cX}^p\Bigr)^{1/p}, 
&\quad   s   &=k\in\BN^\times.
\eal
\right.
$$ 
Then, given 
\hb{s>0}, 
$$ 
B_p^s\BXcX:=\bigl(\bigl\{\,u\in L_p\BXcX\ ;\ [u]_{s,p;\cX}<\iy\,\bigr\}, 
\ \Vsdot_{s,p;\cX}^*\bigr) 
$$  
is a Banach space, an \emph{\hbox{$\cX$-valued} Besov space}, 
\beq\Label{N.eq} 
\Vsdot_{s,p;\cX}^*\sim\Vsdot_{s,p;\cX}^{**}, 
\eeq 
and 
\hb{\cD\BXcX\sdh B_p^s\BXcX}. These facts can be derived by modifying the 
corresponding well-known scalar-valued results (e.g., 
H.-J.~Schmei{\ss}er~\cite{Schm87a} or H.~Amann~\cite{Ama97b}). 

\smallskip 
Now we choose 
\hb{\BX=J}. Note that 
$$  
\tri_h^k\circ\Ta_{q,\ka}^\mu=\Ta_{q,\ka}^\mu\circ\tri_{\rho_\ka^\mu h}^k 
\qa 1\leq q\leq\iy. 
$$ 
Hence \Eqref{RA.tLq} implies 
\beq\Label{N.Dts} 
[\Ta_{p,\ka}^\mu u]_{s,p;\cX}=\rho_\ka^{\mu s}\,[u]_{s,p;\cX}. 
\eeq 
Suppose 
\hb{s>0}. Then 
\beq\Label{N.N1} 
\|u\|_{s/\vec r,p;\vec\om}^* 
:=\bigl(\|u\|_{p;B_p^{s,\lda}}^p+[u]_{s/r,p;L_p^{\lda+s\mu/r}}^p\bigr)^{1/p}  
\eeq 
and, if 
\hb{kr<s\leq(k+1)r} with 
\hb{k\in\BN}, 
\beq\Label{N.N2} 
\|u\|_{s/\vec r,p;\vec\om}^{**}  
:=\Bigl(\|u\|_{p;B_p^{s,\lda}}^p 
+\sum_{j\leq k}\|\pl^ju\|_{p;W_{\coW p}^{(k-j)r,\lda+\mu j}}^p 
+[\pl^ku]_{(s-kr)/r,p;L_p^{\lda+s\mu/r}}^p\Bigr)^{1/p}. 
\eeq 
Besides of these norms we introduce localized versions of them by 
\beq\Label{N.N3} 
\vd u\vd_{s/\vec r,p;\vec\om}^* 
:=\bigl(\|\vp_p^\lda u\|_{p;\ell_p(\mf{B}_p^s)}^p 
+[\vp_p^{\lda+s\mu/r}u]_{s/r,p;\ell_p(\mf{L}_p)}^p\bigr)^{1/p} 
\eeq 
and, if 
\hb{kr<s\leq(k+1)r}, 
\beq\Label{N.N4} 
\bal 
\vd u\vd_{s/\vec r,p;\vec\om}^{**} 
:=\Bigl(\|\vp_p^\lda u\|_{p;\ell_p(\mf{B}_p^s)}^p 
&+\sum_{j\leq k}\|\pl^j\vp_p^{\lda+j\mu}u\|_{p;\ell_p(\mf{W}_p^{(k-j)r})}^p 
 \\
&+[\pl^k\vp_p^{\lda+s\mu/r}u]_{(s-kr)/r,p;\ell_p(\mf{L}_p)}^p\Bigr)^{1/p}.  
\eal 
\eeq 
\begin{theorem}\LabelT{thm-N.B} 
Suppose 
\hb{s>0}. Then 
\hbox{\Eqref{N.N1}--\Eqref{N.N4}} 
are equivalent norms for~$B_p^{s/\vec r,\vec\om}$. 
\end{theorem} 
\begin{proof} 
(1) 
It follows from \Eqref{RA.tLq} that 
\beq\Label{N.lom} 
\|\vp_p^\lda u\|_{p;\ell_p(\mf{B}_p^s)} 
=\|\vp_p^{\vec\om}u\|_{p;\ell_p(\mf{B}_p^s)}. 
\eeq 
Using \Eqref{N.Dts} we get 
\beq\Label{N.0l} 
[\vp_{p,\ka}^{\vec\om}u]_{s/r,p;L_{p,\ka}} 
=[\vp_{p,\ka}^{\lda+s\mu/r}u]_{s/r,p;L_{p,\ka}}. 
\eeq 
Thus, by Fubini's theorem, 
$$ 
[\vp_p^{\vec\om}u]_{s/r,p;\ell_p(\mf{L}_p)} 
=[\vp_p^{\lda+s\mu/r}u]_{s/r,p;\ell_p(\mf{L}_p)}. 
$$ 
From this and \Eqref{N.lom} we obtain 
\beq\Label{N.ph1} 
\vd u\vd_{s/\vec r,p;\vec\om}^* 
=\bigl(\|\vp_p^{\vec\om}u\|_{p;\ell_p(\mf{B}_p^s)}^p 
+[\vp_p^{\vec\om}u]_{s/r,p;\ell_p(\mf{L}_p)}^p\bigr)^{1/p}. 
\eeq 
Similarly, invoking \Eqref{RA.tk} as well, 
$$ 
\vd u\vd_{s/\vec r,p;\vec\om}^{**} 
=\Bigl(\|\vp_p^{\vec\om}u\|_{p;\ell_p(\mf{B}_p^s)}^p 
+\sum_{j\leq k}\|\pl^j\vp_p^{\vec\om}u\|
 _{p;\ell_p(\mf{W}_{\coW p}^{(k-j)r})}^p 
+[\pl^k\vp_p^{\vec\om}u]_{(s-kr)/r,p;\ell_p(\mf{L}_p)}^p\Bigr)^{1/p} 
\npb 
$$ 
if 
\hb{kr<s\leq(k+1)r}. 

\smallskip 
(2) 
Lemma~\ref{lem-RA.WW} and \cite[Theorems 3.6.3 and~4.4.3]{Ama09a} imply 
$$ 
B_{p,\ka}^{s/\vec r} 
\doteq L_p(J,B_{p,\ka}^s)\cap B_p^{s/r}(J,L_{p,\ka}) 
\qa \ka\in\gK. 
$$ 
Hence 
$$ 
\Vsdot_{B_{p,\ka}^{s/\vec r}} 
\sim\Vsdot_{p;B_{p,\ka}^s}+\esdot_{s/r,p;L_{p,\ka}},   
$$ 
due to 
\hb{B_{p,\ka}^s\hr L_{p,\ka}}. From this, \Eqref{N.ph1}, and Fubini's 
theorem we deduce 
$$ 
\Vvsdot_{s/\vec r,p;\vec\om}^* 
\sim\|\vp_p^{\vec\om}\cdot\|_{\ell_p(\mf{B}_p^{s/\vec r})}. 
$$ 
Thus \Eqref{R.rN} and Theorem~\ref{thm-RA.R} guarantee that \Eqref{N.N3} 
is a norm for~$B_p^{s/\vec r,\vec\om}$. Similarly, using \Eqref{N.eq}, 
we see that \Eqref{N.N4} is a norm for~$B_p^{s/\vec r,\vec\om}$. 

\smallskip 
(3) 
We set 
\hb{\al:=\lda+s\mu/r} and 
\hb{\ba:=\al+\tau-\sa}. Then we deduce from 
\Eqref{S.sd}(iv), \,\Eqref{L.Nh}, \,\Eqref{R.rho}, and 
\cite[Lem\-ma~3.1(iii)]{Ama12b}  
$$ 
\bal
{}
[\vp_{p,\ka}^\al u]_{s/r,p;L_{p,\ka}}^p 
&=\int_0^\iy\int_J\int_{\BX_\ka} 
 \bigl(\rho_\ka^{\al+m/p}
 \,\big|\tri_\xi^{k+1}\bigl((\ka\slt\vp)_*(\pi_\ka u)
 \bigr)\big|_E\bigr)^p 
 \,dV_{\coV g_m}\,dt\,\frac{d\xi}{\xi^{1+ps/r}}\\ 
&\sim\int_0^\iy\int_J\int_{\BX_\ka}\ka_* 
 \bigl((\rho^\ba\pi_\ka\,|\tri_\xi^{k+1}u|_h)^p 
 \,dV_{\coV g}\bigr)\,dt\,\frac{d\xi}{\xi^{1+ps/r}}\\ 
&=\int_0^\iy\int_J\int_{U_{\coU\ka}} 
 (\rho^\ba\pi_\ka\,|\tri_\xi^{k+1}u|_h)^p 
 \,dV_{\coV g}\,dt\,\frac{d\xi}{\xi^{1+ps/r}} 
\eal 
$$ 
for 
\hb{u\in\cD\JcD}. We insert 
\hb{\mf{1}=\sum_{\wt{\ka}}\pi_{\wt{\ka}}^2} in the inner integral, sum over 
\hb{\ka\in\gK}, and interchange the order of summation. Then 
$$ 
[\vp_p^\al u]_{s/r,p;\ell_p(\mf{L}_p)}^p 
\sim\sum_{\wt{\ka}}\sum_{\ka\in\gN(\wt{\ka})} 
\int_0^\iy\int_J\int_{U_{\coU\wt{\ka}}}\pi_{\wt{\ka}}^2 
(\rho^\ba\pi_\ka\,|\tri_\xi^{k+1}u|_h)^p 
\,dV_{\coV g}\,dt\,\frac{d\xi}{\xi^{1+ps/r}}.  
$$ 
Using \Eqref{R.LS}(iii) and the finite multiplicity of~$\gK$ 
we see that the last term can be bounded above by 
$$ 
\bal 
c\sum_{\wt{\ka}}
 \int_0^\iy\int_J\int_{U_{\coU\wt{\ka}}}\pi_{\wt{\ka}}^2 
 (\rho^\ba\,|\tri_\xi^{k+1}u|_h)^p 
 \,dV_{\coV g}\,dt\,\frac{d\xi}{\xi^{1+ps/r}}\\
{}= c\int_0^\iy\int_J\int_M 
 (\rho^\ba\,|\tri_\xi^{k+1}u|_h)^p 
 \,dV_{\coV g}\,dt\,\frac{d\xi}{\xi^{1+ps/r}}
&=c\,[u]_{s/r,p;L_p^\al}^p. 
\eal 
$$ 
Hence, recalling \Eqref{N.0l}, 
\beq\Label{N.sLp} 
[\vp_p^{\vec\om}u]_{s/r,p;\ell_p(\mf{L}_p)} 
\leq c\,[u]_{s/r,p;L_p^{\lda+s\mu/r}} 
\qa u\in\cD\JcD. 
\eeq 

\smallskip 
(4) 
It is a consequence of Theorem~\ref{thm-R.R} that 
\hb{\vp_p^\lda\in\cL\bigl(B_p^{s,\lda},\ell_p(\mf{B}_p^s)\bigr)}. 
This implies, due to \Eqref{N.lom}, 
\beq\Label{N.Llp} 
\|\vp_p^{\vec\om}u\|_{p;\ell_p(\mf{B}_p^s)} 
=\|\vp_p^\lda u\|_{p;\ell_p(\mf{B}_p^s)} 
\leq c\,\|u\|_{p;B_p^{s,\lda}} 
\qa u\in\cD\JcD. 
\eeq 
Thus we obtain from \Eqref{N.ph1}, \,\Eqref{N.sLp}, and \Eqref{N.Llp} 
\beq\Label{N.phs} 
\vd u\vd_{s/\vec r,p;\vec\om}^* 
\leq c\,\|u\|_{s/\vec r,p;\vec\om}^* 
\qa u\in\cD\JcD. 
\eeq 
We denote by $\starB_p^{s/\vec r,\vec\om}$ the completion 
of $\cD\JcD$ in~$L_p(J,L_p^\lda)$ with respect to the norm~%
\hb{\Vsdot_{s/\vec r,p;\vec\om}^*}. Then \Eqref{N.phs} and step~(2) imply 
$$ 
\starB_p^{s/\vec r,\vec\om}\hr B_p^{s/\vec r,\vec\om}. 
$$ 

\smallskip 
(5) 
Observing 
\hb{\psi_{p,\ka}^{\vec\om}=\chi_\ka\psi_{p,\ka}^{\vec\om}} and 
\hb{0\leq\chi_\ka\leq1}, the finite multiplicity of~$\gK$ implies 
$$ 
\bal 
|\tri_\xi^{k+1}\psi_p^{\vec\om}\mf{v}|_h 
&=\Bigl|\sum_\ka\tri_\xi^{k+1}\psi_{p,\ka}^{\vec\om}v_\ka\Bigr|_h 
 \leq\Bigl(\sum_\ka|\tri_\xi^{k+1}\psi_{p,\ka}^{\vec\om}v_\ka|_h^p 
 \Bigr)^{1/p}\Bigl(\sum_\ka\chi_\ka\Bigr)^{1/p'}\\ 
&\leq c\Bigl(\sum_\ka|\tri_\xi^{k+1}\psi_{p,\ka}^{\vec\om}v_\ka|_h^p 
 \Bigr)^{1/p} 
\eal 
$$ 
for 
\hb{\mf{v}\in\mf{\cD}\BYE}. Hence, reasoning as in step~(3), 
$$ 
\bal 
{}
[\psi_p^{\vec\om}\mf{v}]_{s/r,p;L_p^{\lda+s\mu/r}}^p 
&\leq c\int_0^\iy\int_J\int_M 
 \rho^{\ba p}\sum_\ka\,|\tri_\xi^{k+1}\psi_{p,\ka}^{\vec\om}v_\ka|_h^p 
 \,dV_{\coV g}\,dt\,\frac{d\xi}{\xi^{1+ps/r}}\\ 
&\leq c\int_0^\iy\int_J\int_{\BX_\ka} 
 \sum_\ka\,|\tri_\xi^{k+1}(\pi_\ka v_\ka)|_{g_m}^p 
 \,dV_{\coV g_m}\,dt\,\frac{d\xi}{\xi^{1+ps/r}}\\
&\leq c\sum_\ka[\pi_\ka v_\ka]_{s/r,p;L_{p,\ka}}^p  
 \leq c\sum_\ka[v_\ka]_{s/r,p;L_{p,\ka}}^p\\  
&\leq c\sum_\ka\|v_\ka\|_{B_p^{s/r}(J,L_{p,\ka})}^p  
\eal 
\npb 
$$ 
for 
\hb{\mf{v}\in\mf{\cD}\BYE}.

\smallskip 
(6) 
Theorem~\ref{thm-R.R} and \Eqref{R.peq} guarantee that 
\hb{\|\vp_p^\lda\cdot\|_{\ell_p(\mf{B}_p^s)}} is an equivalent norm 
for~$B_p^{s,\lda}$. This implies
\beq\Label{N.psv} 
\|\psi_p^{\vec\om}\mf{v}\|_{p;B_p^{s,\lda}} 
\leq c\,\|\vp_p^\lda\psi_p^{\vec\om}\mf{v}\|_{p;\ell_p(\mf{B}_p^s)} 
\qa \mf{v}\in\mf{\cD}\BYE.   
\eeq 
From \Eqref{RA.phpv} we infer by interpolation, using the arguments of 
step~(5) of the proof of Theorem~\ref{thm-RA.R}, that 
$$ 
\|\vp_p^\lda\psi_p^{\vec\om}\mf{v}\|_{\ell_p(L_p(J,\mf{B}_p^s))} 
\leq c\,\|\mf{v}\|_{\ell_p(L_p(J,\mf{B}_p^s))} 
\qa \mf{v}\in\mf{\cD}\BYE.   
$$ 
Hence \Eqref{N.psv} and \Eqref{RA.lL} imply 
$$ 
\|\psi_p^{\vec\om}\mf{v}\|_{p;B_p^{s,\lda}} 
\leq c\,\|\mf{v}\|_{p;\ell_p(\mf{B}_p^s)} 
\qa \mf{v}\in\mf{\cD}\BYE.   
$$ 
By combining this with the result of step~(5) we find, employing 
\Eqref{RA.lL} once more, 
$$ 
\|\psi_p^{\vec\om}\mf{v}\|_{s/\vec r,p;\vec\om}^* 
\leq c\,\|\mf{v}\|_{\ell_p(\mf{B}_p^{s/\vec r})} 
\qa \mf{v}\in\mf{\cD}\BYE.   
$$ 
Thus, by Theorem~\ref{thm-RA.R}, 
$$ 
\|u\|_{s/\vec r,p;\vec\om}^* 
=\|\psi_p^{\vec\om}(\vp_p^{\vec\om}u)\|_{s/\vec r,p;\vec\om}^* 
\leq c\,\|\vp_p^{\vec\om}u\|_{\ell_p(\mf{B}_p^{s/\vec r})}  
=c\,\vd u\vd_{s/\vec r,p;\vec\om}^*   
\qa u\in\cD\JcD,   
$$ 
the last estimate being a consequence of \Eqref{N.ph1}. Since, by 
step~(2), \ \Eqref{N.N3} is a norm for~$B_p^{s/\vec r,\vec\om}$, we get 
$$ 
\|u\|_{s/\vec r,p;\vec\om}^*\leq c\,\|u\|_{B_p^{s/\vec r,\vec\om}} 
\qa u\in\cD\JcD.   
\npb 
$$ 
This implies 
\hb{B_p^{s/\vec r,\vec\om}\hr\starB_p^{s/\vec r,\vec\om}}. From this 
and step~(4) it follows that \Eqref{N.N1} is a norm 
for~$B_p^{s/\vec r,\vec\om}$. 

\smallskip 
(7) 
The proof of the fact that \Eqref{N.N2} is a norm 
for~$B_p^{s/\vec r,\vec\om}$ is similar. 
\end{proof} 
\begin{corollary}\LabelT{cor-N.B} 
If 
\hb{s>0}, then 
\hb{B_p^{s/\vec r,(\lda,0)} 
    \doteq L_p(J,B_p^{s,\lda})\cap B_p^{s/r}(J,L_p^\lda)}. 
\end{corollary} 
\section{H\"older Spaces in Euclidean Settings}\LabelT{sec-PH}%
In~\cite{Ama12b} it has been shown that isotropic weighted H\"older spaces 
are important point-wise multiplier spaces for weighted isotropic Bessel 
potential and Besov spaces. In Section~\ref{sec-P} we shall show that 
similar results hold in the anisotropic case. For this reason we introduce 
and study anisotropic weighted H\"older spaces and establish the fundamental 
retraction theorem which allows for local characterizations. 
In order to achieve this we have to have a good understanding of 
H\"older spaces of Banach-space-valued functions on $\BR^m$ and~$\BH^m$. In 
this section we derive those properties of such spaces which are needed to 
study weighted H\"older spaces on~$M$. 

\smallskip 
Let $\cX$ be a Banach space. Suppose 
\hb{\BX\in\{\BR^m,\BH^m\}} and 
\hb{X\in\{\BX,\ \BX\times J\}}. Then 
\hb{B=B\XcX}~is the Banach space of all bounded \hbox{$\cX$-valued} 
functions on~$X$ endowed with the supremum norm 
\hb{\Vsdot_\iy=\Vsdot_{0,\iy}}. 

\smallskip 
Throughout this section, 
\hb{k,k_0,k_1\in\BN}. Then 
$$ 
BC^k=BC^k\XcX 
:=\bigl(\bigl\{\,u\in C^k\XcX\ ;\ \pa_xu\in B\XcX,\ |\al|\leq k\,\bigr\}, 
\ \Vsdot_{k,\iy}\bigr), 
$$ 
where 
$$ 
\|u\|_{k,\iy}:=\max_{|\al|\leq k}\|\pa_xu\|_\iy, 
$$ 
is a Banach space. As usual, 
\hb{BC=BC^0}. We write 
\hb{\Vsdot_{k,\iy;\cX}} for 
\hb{\Vsdot_{k,\iy}} if it seems to be necessary to indicate the image 
space. Similar conventions apply to the other norms and seminorms introduced 
below. 

\smallskip 
Note that 
$$ 
\BUC^k
=\bigl\{\,u\in BC^k\ ;\ \pa_xu\text{ is uniformly continuous for } 
|\al|\leq k\,\bigr\} 
$$ 
is a closed linear subspace of~$BC^k$. The mean value theorem implies the 
first embedding of 
\beq\Label{PH.BUB} 
BC^{k+1}\hr\BUC^k\hr BC^k. 
\eeq 
Hence 
\beq\Label{PH.Binft} 
BC^\iy:={\textstyle\bigcap_k}BC^k={\textstyle\bigcap_k}\BUC^k. 
\eeq 
It is a Fr\'echet space with the natural projective topology. Thus 
$$ 
BC^\iy\hr\BUC^k 
\qa k\in\BN. 
$$ 
In fact, this embedding is dense. For this we recall that~a \emph{mollifier 
on}~$\BR^d$ is a family 
\hb{\{\,w_\eta\ ;\ \eta>0\,\}} of nonnegative compactly supported smooth 
functions on~$\BR^d$ such that 
\hb{w_\eta(x)=\eta^{-d}w_1(x/\eta)} for 
\hb{x\in\BR^d} and 
\hb{\int w_1\,dx=1}. Then, denoting by 
\hb{w_\eta*u} convolution, 
\beq\Label{PH.con} 
w_\eta*u\in BC^\iy\BRdcX 
\qa u\in BC\BRdcX, 
\eeq 
and 
\beq\Label{PH.app} 
\lim_{\eta\ra0}w_\eta*u=u\text{ in }\BUC^k\BRdcX 
\qa u\in\BUC^k\BRdcX, 
\eeq 
(cf.~\cite[Theorem~X.7.11]{AmE08a}, for example, whose proof carries 
literally over to \hbox{$\cX$-valued} spaces). From this we get 
\beq\Label{PH.BBU} 
BC^\iy\sdh\BUC^k 
\eeq 
if 
\hb{\BX=\BR^m} and 
\hb{J=\BR}. In the other cases it follows by an additional extension and 
restriction argument based on the extension map~(4.1.7) of~\cite{Ama09a} 
(also cf.~Section~4.3 therein). 

\smallskip 
From now on 
\hb{X=\BX}. For 
\hb{k\leq s<k+1}, 
\ \hb{0<\da\leq\iy}, and 
\hb{u\sco\BX\ra\cX} we put 
$$ 
[u]_{s,\iy}^\da 
:=\sup_{h\in(0,\da)^m}\frac{\|\tri_h^{k+1}u\|_{\iy;\cX}}{|h|^s} 
\qb \esdot_{s,\iy}:=\esdot_{s,\iy}^\iy. 
$$ 
Furthermore, 
$$ 
\Vsdot_{s,\iy}^*:=\Vsdot_\iy+\esdot_{s,\iy} 
\qa s>0. 
$$ 
Note that 
\hb{h\in(0,\iy)^m\ssm(0,\da)^m} implies 
\hb{\da\leq|h|_\iy\leq|h|\leq\sqrt{m}\,|h|_\iy}. Hence 
\beq\Label{PH.del1} 
\esdot_{\ta,\iy}\leq\esdot_{\ta,\iy}^\da+4\da^{-\ta}\ \Vsdot_\iy 
\qa 0<\ta\leq1 
\qb 0<\da<\iy. 
\eeq 
If 
\hb{0<\ta_0<\ta\leq1}, then 
\beq\Label{PH.lHd} 
\esdot_{\ta_0,\iy}^\da\leq\sqrt{m}\,\da^{\ta-\ta_0}\esdot_{\ta,\iy}^\da  
\qa 0<\da<\iy. 
\eeq 
Consequently, 
$$ 
\esdot_{\ta_0,\iy} 
\leq\sqrt{m}\,\esdot_{\ta,\iy}^1+4\,\Vsdot_\iy 
\leq\sqrt{m}\,\esdot_{\ta,\iy}+4\,\Vsdot_\iy. 
$$ 
This implies 
\beq\Label{PH.s01} 
\Vsdot_{\ta_0,\iy}^*\leq c(m)\,\Vsdot_{\ta,\iy}^* 
\qa 0<\ta_0<\ta\leq1. 
\eeq 

\smallskip 
Suppose 
\hb{u\in BC^k} and denote by~$D$ the Fr\'echet derivative. Then, by the mean 
value theorem, 
$$ 
\tri_h^ku(x) 
=\int_0^1\cdots\int_0^1D^ku\bigl(x+(t_1+\cdots+t_k)h\bigr)[h]^k 
\,dt_1\cdots dt_k, 
$$ 
where 
\hb{[h]^k:=(h,\ldots,h)\in\BX^k}. From this we get 
\beq\Label{PH.BCk} 
[u]_{\ta,\iy}^\da\leq m^{k/2}\da^{k-\ta}\,\|u\|_{k,\iy} 
\qa 0<\ta\leq1 
\qb \ta<k 
\qb \da>0 
\qb u\in BC^k. 
\eeq 
Thus, by \Eqref{PH.del1}, 
\beq\Label{PH.BC1} 
\Vsdot_{\ta,\iy}^*\leq c(m)\,\Vsdot_{1,\iy} 
\qa 0<\ta<1. 
\eeq 

\smallskip 
We also set for 
\hb{k<s\leq k+1} 
$$ 
\|u\|_{s,\iy}^{**} 
:=\|u\|_{k,\iy}+\max_{|\al|=k}[\pa_xu]_{s-k,\iy}. 
$$ 
If 
\hb{k<s<k+1}, then 
\hb{\Vsdot_{s,\iy}:=\Vsdot_{s,\iy}^{**}} and 
$$ 
BC^s=BC^s\BXcX:=\bigl(\bigl\{\,u\in BC^k 
\ ;\ \max_{|\al|=k}[\pa_xu]_{s-k,\iy}<\iy\,\bigr\}, 
\ \Vsdot_{s,k}\bigr) 
\qa k<s<k+1,  
\npb 
$$ 
is a \emph{H\"older space} of order~$s$. 

\smallskip 
Given 
\hb{h=(h^1,\ldots,h^m)\in\BX}, we set 
\hb{h_j:=(0,\ldots,0,h^j,\ldots,h^m)} for 
\hb{1\leq j\leq m}, and 
\hb{h_{m+1}:=0}. Then 
$$ 
\tri_hu(x) 
=\sum_{j=1}^m\bigl(u(x+h_j)-u(x+h_{j+1})\bigr). 
$$ 
From this we infer for 
\hb{0<\ta<1} and 
\hb{h^j\neq0} for 
\hb{1\leq j\leq m} 
$$ 
\bal 
\frac{\|\tri_hu\|_\iy}{|h|^\ta} 
&\leq\sum_{j=1}^m\frac{\|u(\cdot+h^je_j)-u\|_\iy}{|h^j|^\ta}  
 \leq\sum_{j=1}^m\sup_{h^j\neq0}\frac{\|u(\cdot+h^je_j)-u\|_\iy} 
 {|h^j|^\ta}\\   
&=\sum_{j=1}^m\sup_{h^j>0}\frac{\|u(\cdot+h^je_j)-u\|_\iy}{(h^j)^\ta}
 \leq m\,[u]_{\ta,\iy}. 
\eal 
$$ 
Consequently, 
\beq\Label{PH.eqsn} 
[u]_{\ta,\iy}\leq\sup_{h\neq0}\frac{\|\tri_hu\|_\iy}{|h|^\ta} 
\leq m\,[u]_{\ta,\iy} 
\qa 0<\ta<1. 
\eeq 
This shows that $BC^s$~coincides, except for equivalent norms, with the 
usual H\"older space of order~$s$ if 
\hb{s\in\BR^+\ssm\BN}. From \Eqref{PH.eqsn} we read off that the last 
embedding of 
\beq\Label{PH.BCUC} 
BC^{k+1}\hr BC^s\hr BC^{s_0}\hr\BUC^k 
\qa k<s_0<s<k+1,  
\npb 
\eeq 
is valid. The other two follow from \Eqref{PH.s01} and \Eqref{PH.BC1}. 

\smallskip 
We introduce the \hh{Besov-H\"older space scale} 
\hb{[\,B_\iy^s\ ;\ s>0\,]} by 
$$ 
B_\iy^s 
:=\left\{
\bal
{}      
&(\BUC^k,\BUC^{k+1})_{s-k,\iy},
    &\quad  k<s &<k+1,\\ 
&(\BUC^k,\BUC^{k+2})_{1/2,\iy}, 
    &\quad    s &=k+1. 
\eal 
\right. 
\po 
$$ 
{\samepage 
\begin{theorem}\LabelT{thm-PH.B} 
\begin{itemize} 
\item[{\rm(i)}] 
${}$
\hb{\Vsdot_{s,\iy}^*} and\/ 
\hb{\Vsdot_{s,\iy}^{**}} are norms for~$B_\iy^s$.\po 
\item[{\rm(ii)}] 
${}$
\hb{B_\iy^s\doteq(\BUC^{k_0},\BUC^{k_1})_{(s-k_0)/(k_1-k_0),\iy}} for 
\hb{k_0<s<k_1}.  
\item[{\rm(iii)}] 
If\/ 
\hb{0<s_0<s_1} and 
\hb{0<\ta<1}, then 
\hb{(B_\iy^{s_0},B_\iy^{s_1})_{\ta,\iy} 
   \doteq B_\iy^{s_\ta}\doteq[B_\iy^{s_0},B_\iy^{s_1}]_\ta}.\po  
\end{itemize} 
\end{theorem} }
\begin{proof} 
(1) 
For 
\hb{s>0} we denote by 
\hb{B_{\iy,\iy}^s=B_{\iy,\iy}^s\BXcX} the `standard' Besov space modeled 
on~$L_\iy$ for whose precise definition we refer to~\cite{Ama09a} (choosing 
the trivial weight vector therein). 

\smallskip 
It is a consequence of \cite[\ (3.3.12), (3.5.2), and 
Theorem~4.4.1]{Ama09a} that 
$$ 
B_{\iy,\iy}^s 
\doteq(\BUC^{k_0},\BUC^{k_1})_{(s-k_0)/(k_1-k_0),\iy} 
\qa k_0<s<k_1. 
$$ 
This implies 
\beq\Label{PH.BB} 
B_\iy^s\doteq B_{\iy,\iy}^s 
\npb 
\eeq 
and, consequently, statement~(ii). 

\smallskip 
(2) 
The first part of~(iii) follows by reiteration from~(ii).\\  
For 
\hb{\xi\in\BR^m} we set 
\hb{\Lda(\xi):=(1+|\xi|^2)^{1/2}}. Given 
\hb{s\in\BR}, we put 
\hb{\mf{\Lda}^{\coLda s}:=\cF^{-1}\Lda^{\coLda s}\cF}, where 
\hb{\cF=\cF_m} is the Fourier transform on~$\BR^m$. 

\smallskip 
Suppose 
\hb{\BX=\BR^m}. It follows from \cite[Theorem~3.4.1]{Ama09a} and 
\Eqref{PH.BB} that 
\beq\Label{PH.lam} 
\mf{\Lda}^{\coLda s}\in\Lis(B_\iy^{t+s},B_\iy^t) 
\qb (\mf{\Lda}^{\coLda s})^{-1}=\mf{\Lda}^{\coLda -s} 
\qa t,\ s+t>0. 
\eeq 
We set 
\hb{A:=-\mf{\Lda}^{\coLda s_1-s_0}}, considered as a linear operator 
in~$B_\iy^{s_0}$ with domain~$B_\iy^{s_1}$. Then 
\cite[Proposition~1.5.2 and Theorem~3.4.2]{Ama09a} guarantee the existence 
of 
\hb{\vp\in(\pi/2,\pi)} such that the sector 
\hb{S_\vp:=\{\,z\in\BC\ ;\ |\arg z|\leq\vp\,\}\cup\{0\}} belongs to the 
resolvent set of~$A$ and 
\hb{\|(\lda-A)^{-1}\|\leq c/|\lda|} for 
\hb{\lda\in S_\vp}. Furthermore, by 
\cite[Proposition~1.5.4 and Theorem~3.4.2]{Ama09a} we find that 
\hb{A^z\in\cL(B_\iy^{s_0})} and there exists 
\hb{\ga>0} such that 
\hb{\|A^z\|\leq c\kern1pt e^{\ga\,|\Im z|}} for 
\hb{\Re z\leq0}. Now Seeley's theorem, more precisely: the proof in 
R.~Seeley~\cite{See71a}, and \Eqref{PH.lam} imply 
\hb{[B_\iy^{s_0},B_\iy^{s_1}]_\ta\doteq B_\iy^{s_\ta}}. This proves the 
second part of~(iii) if 
\hb{\BX=\BR^m}. The case 
\hb{\BX=\BH^m} is then covered by \cite[Theorem~4.4.1]{Ama09a}. 

\smallskip 
(3) By \cite[Theorems 3.3.2, 3.5.2, and~4.4.1]{Ama09a} we get 
\hb{B_{\iy,\iy}^s\hr\BUC}. Using this and the arguments of the proof of 
\cite[Theorem~4.4.3(i)]{Ama09a} we infer from 
\cite[Theorem~3.6.1]{Ama09a} that 
\hb{\Vsdot_{B_{\iy,\iy}^s}\sim\Vsdot_{s,\iy}^*}. By appealing to 
\cite[Theorem~1.13.1]{Tri78a} in the proof of \cite[Theorem~3.6.1]{Ama09a} 
we obtain similarly 
\hb{\Vsdot_{B_{\iy,\iy}^s}\sim\Vsdot_{s,\iy}^{**}}, making also use of 
\Eqref{PH.BCUC} in the usual extension-restriction argument. Due to 
\Eqref{PH.BB} this proves~(i). 
\end{proof}\po 
{\samepage 
\begin{corollary}\LabelT{cor-PH.B} 
\begin{itemize} 
\item[{\rm(i)}] 
${}$
\hb{B_\iy^s\doteq BC^s} for 
\hb{s\in\BR^+\ssm\BN}.  
\item[{\rm(ii)}] 
${}$
\hb{\BUC^k\hr B_\iy^k} and 
\hb{\BUC^k\neq B_\iy^k}.\po 
\end{itemize} 
\end{corollary} }
\begin{proof} 
(i)~is implied by part~(i) of the theorem.  

\smallskip 
(ii) 
The first claim is a consequence of \cite[Theorem~3.5.2]{Ama09a}. It 
follows from Example~IV.4.3.1 in E.~Stein~\cite{Ste70a} that the `Zygmund 
space'~$B_\iy^1$ contains functions which are not uniformly Lipschitz 
continuous. This proves the second statement. 
\end{proof} 
By \Eqref{PH.BCUC} we see that 
$$  
BC^{s_1}\hr BC^{s_0} 
\qa 0\leq s_0<s_1. 
$$ 
However, these embeddings are not dense. Since dense embeddings are of great 
importance in the theory of elliptic and parabolic differential equations 
we introduce the smaller subscale of `little' H\"older spaces which enjoy 
the desired property. 

\smallskip 
Suppose 
\hb{s\in\BR^+}. The \hh{little H\"older space}  
$$ 
bc^s=bc^s\BXcX\text{ is the closure of $BC^\iy$ in }BC^s. 
$$ 
Similarly, the \hh{little Besov-H\"older space scale}  
\hb{[\,b_\iy\ ;\ s>0\,]} is defined by 
\beq\Label{PH.bdef} 
b_\iy^s\text{ is the closure of $BC^\iy$ in }B_\iy^s. 
\npb 
\eeq 
These spaces possess intrinsic characterizations.\po 
{\samepage 
\begin{theorem}\LabelT{thm-PH.bc} 
\begin{itemize} 
\item[{\rm(i)}] 
${}$
\hb{bc^k=\BUC^k}.\po 
\item[{\rm(ii)}] 
${}$
\hb{b_\iy^s\doteq bc^s} for 
\hb{s\in\BR^+\ssm\BN}.\po 
\item[{\rm(iii)}] 
Suppose 
\hb{k<s\leq k+1}. Then 
\hb{u\in B_\iy^s} belongs to~$b_\iy^s$ iff 
\beq\Label{PH.del2} 
\lim_{\da\ra0}[\pa_xu]_{s-k,\iy}^\da=0 
\qa |\al|=k. 
\eeq 
\item[{\rm(iv)}] 
${}$
\hb{BC^s\sdh b_\iy^{s_0}} \,for 
\hb{0<s_0<s}.\po  
\end{itemize} 
\end{theorem} }
\begin{proof} 
(1) 
Assertion~(i) is a consequence of \Eqref{PH.BBU}. Statement~(ii) follows 
from Corollary~\ref{cor-PH.B}(i). 

\smallskip 
(2) 
Suppose 
\hb{k<s\leq k+1}. We denote by~$\wt{b}_\iy^s$ the linear subspace 
of~$B_\iy^s$ of all~$u$ satisfying \Eqref{PH.del2}. Then we infer from 
\Eqref{PH.BCk} that 
\beq\Label{PH.BUCb} 
BC^\iy\hr\BUC^{k+1}\hr\wt{b}_\iy^s. 
\eeq 
Let 
\hb{u\in b_\iy^s} and 
\hb{\ve>0}. Then \Eqref{PH.BCUC} implies the existence of 
\hb{v\in\BUC^{k+2}} with 
\hb{\|u-v\|_{s,\iy}^{**}<\ve/2}. By \Eqref{PH.BUCb} we can find 
\hb{\da_\ve>0} such that 
\hb{[\pa_xv]_{s-k,\iy}^{\da_\ve}\leq\ve/2} for 
\hb{|\al|=k} and 
\hb{0<\da\leq\da_\ve}. Hence 
$$ 
[\pa_xu]_{s-k,\iy}^\da\leq\bigl[\pa_x(u-v)\bigr]_{s-k,\iy} 
+[\pa_xv]_{s-k,\iy}^{\da_\ve}\leq\|u-v\|_{s,\iy}^{**}+\ve/2<\ve 
\npb 
$$ 
for 
\hb{|\al|=k} and 
\hb{\da\leq\da_\ve}. This proves 
\hb{b_\iy^s\is\wt{b}_\iy^s}. 

\smallskip 
(3) 
Suppose 
\hb{\BX=\BR^m} and 
\hb{u\in\wt{b}_\iy^s}. We claim that 
\hb{w_\eta*u} converges in~$B_\iy^s$ towards~$u$ as 
\hb{\eta\ra0}. Using \Eqref{PH.app} and 
\hb{\pa_x(w_\eta*u)=w_\eta*\pa_xu} we can assume 
\hb{0<s\leq1} and then have to show 
\beq\Label{PH.wet} 
[w_\eta*u-u]_{s,\iy}\ra0 
\quad\text{as }\eta\ra0. 
\eeq 
Note 
$$ 
w_\eta*u(x)-u(x) 
=\int\bigl(u(x-y)-u(x)\bigr)w_\eta(y)\,dy. 
$$ 
From this we infer 
\beq\Label{PH.wdel} 
[w_\eta*u-u]_{s,\iy}^\da\leq2\,[u]_{s,\iy}^\da 
\qa \da>0. 
\eeq 
Fix 
\hb{\da_\ve>0} such that 
\hb{[u]_{s,\iy}^\da<\ve/4}. Then we get from \Eqref{PH.del1} and 
\Eqref{PH.wdel} that there exists 
\hb{\eta_\ve>0} such that 
$$ 
[w_\eta*u-u]_{s,\iy} 
\leq\ve/2+4\da_\ve^{-s}\,\|w_\eta*u-u\|_\iy\leq\ve 
\npb 
$$ 
for 
\hb{\eta\leq\eta_\ve}, due to 
\hb{B_\iy^s\hr\BUC} and \Eqref{PH.app}. This proves \Eqref{PH.wet}. Thus 
\hb{\wt{b}_\iy^s\is b_\iy^s}. 

\smallskip 
(4) 
If 
\hb{\BX=\BH^m}, then we get 
\hb{\wt{b}_\iy^s\is b_\iy^s} from~(3) and a standard extension and 
restriction argument based on the extension operator~(4.1.7) 
of~\cite{Ama09a}. Together with the result of step~(2) this proves 
claim~(iii). The last assertion follows from \Eqref{PH.BCUC} and 
\Eqref{PH.lHd}. 
\end{proof} 
It should be remarked that assertion~(iii) is basically known (see, for 
example, Proposition~0.2.1 in A.~Lunardi~\cite{Lun95a}, where the case 
\hb{m=1} is considered). The proof is included here for further reference. 

\smallskip 
Little Besov-H\"older spaces can be characterized by interpolation as well. 
For this we recall that, given Banach spaces 
\hb{\cX_1\sdh\cX_0}, the \emph{continuous interpolation space} 
$\cXncXe_{\ta,\iy}^0$ of exponent 
\hb{\ta\in(0,1)} is the closure of~$\cX_1$ in $\cXncXe_{\ta,\iy}$. This 
defines an interpolation functor of exponent~$\ta$ in the category of 
densely injected Banach couples, the \emph{continuous interpolation 
functor}. It possesses the reiteration property 
(cf.~\cite[Section~I.2]{Ama95a} for more details and, in particular, 
G.~Dore and A.Favini~\cite{DoF87a}).\po 
{\samepage 
\begin{theorem}\LabelT{thm-PH.bcint} 
\begin{itemize} 
\item[{\rm(i)}] 
Suppose 
\hb{k_0<s<k_1} with 
\hb{s\notin\BN}. Then 
\hb{(bc^{k_0},bc^{k_1})_{(s-k_0)/(k_1-k_0),\iy}^0\doteq b_\iy^s}. 
\item[{\rm(ii)}] 
If\/ 
\hb{0<s_0<s_1} and 
\hb{0<\ta<1}, then 
\hb{(b_\iy^{s_0},b_\iy^{s_1})_{\ta,\iy}^0 
   \doteq b_\iy^{s_\ta}\doteq[b_\iy^{s_0},b_\iy^{s_1}]_\ta}.\po 
\end{itemize} 
\end{theorem} }
\begin{proof} 
(1) 
The validity of~(i) and the first part of~(ii) follow from 
Theorem~\ref{thm-PH.B}(ii) and~(iii) and Theorem~\ref{thm-PH.bc}(i). 

\smallskip 
(2) 
We deduce from \Eqref{PH.Binft}, \,\Eqref{PH.BCUC}, and 
Corollary~\ref{cor-PH.B} that 
\hb{BC^\iy=\bigcap_{s>0}B_\iy^s}. From this and \Eqref{PH.lam} we infer 
\hb{\mf{\Lda}^{\coLda s}\in\Laut(BC^\iy)}. Hence, using the definition of 
the little Besov-H\"older spaces and once more \Eqref{PH.lam} and 
Corollary~\ref{cor-PH.B}, we find 
$$ 
\mf{\Lda}^{\coLda s}\in\Lis(b_\iy^{t+s},b_\iy^t) 
\qb (\mf{\Lda}^{\coLda s})^{-1}=\mf{\Lda}^{\coLda-s} 
\qa t,\ t+s>0. 
$$ 
Thus the relevant arguments of part~(2) of the proof of 
Theorem~\ref{thm-PH.B} apply literally to give the second part of~(ii). This 
is due to the fact that the Fourier multiplier Theorem 
\cite[Theorem~3.4.2]{Ama09a} holds for $b_\iy^s$ also (see 
\cite[Theorem~6.2]{Ama97b}). 
\end{proof} 
Now we turn to anisotropic spaces. We set 
$$ 
BC^{kr/\vec r} 
:=\bigl(\bigl\{\,u\in C(\BX\times J,\cX) 
\ ;\ \pa_x\pl^ju\in BC(\BX\times J,\cX), 
\ |\al|+jr\leq kr\,\bigr\}, 
\ \Vsdot_{kr/\vec r}\bigr), 
$$ 
where 
$$ 
\|u\|_{kr/\vec r} 
:=\max_{|\al|+jr\leq kr}\|\pa_x\pl^ju\|_\iy. 
$$  
This space is complete and contains 
$$ 
\BUC^{kr/\vec r} 
:=\bigl\{\,u\in BC^{kr/\vec r}\ ;\ \pa_x\pl^ju\in\BUC(\BX\times J,\cX), 
\ |\al|+jr\leq kr\,\bigr\}   
\npb 
$$ 
as a closed linear subspace. 
\begin{proposition}\LabelT{pro-PH.BUC} 
\hb{\BUC^{kr/\vec r}=\bigcap_{j=0}^k\BUC^j(J,\BUC^{(k-j)r})}. 
\end{proposition} 
\begin{proof} 
(1) 
Due to 
\hb{u(x,t)-u(y,s)=u(x,t)-u(y,t)+u(y,t)-u(y,s)} for 
\hb{(x,t),(y,s)\in\BX\times J}, the claim is immediate for 
\hb{k=0}. 

\smallskip 
(2) 
Suppose 
\hb{k\in\BN^\times} and 
\hb{u\in\BUC^{kr/\vec r}}. Suppose also 
\hb{0\leq j\leq k-1} and 
\hb{|\al|\leq(k-j)r}. Then, by the mean value theorem, 
$$ 
\pa_x\pl^ju(x,t+h)-\pa_x\pl^ju(x,t)-h\pa_x\pl^{j+1}u(x,t) 
=h\int_0^1\big(\pa_x\pl^{j+1}u(x,t+\tau h) 
-\pa_x\pl^{j+1}u(x,t)\big)\,d\tau 
$$ 
for 
\hb{x\in\BX} and 
\hb{t,h\in J}. Thus, given 
\hb{\ve>0}, the uniform continuity of~$\pa_x\pl^{j+1}u$ implies the 
existence of 
\hb{\da>0} such that 
$$ 
\bal 
\big\|h^{-1}\big(\pa_x\pl^ju(\cdot,t+h) 
 -\pa_x\pl^ju(\cdot,t)\big)
&-\pa_x\pl^{j+1}u(\cdot,t)\big\|_{\iy;\cX}\\ 
 \leq\max_{0\leq\tau\leq1}\|\pa_x\pl^{j+1}u(\cdot,t+\tau h) 
&-\pa_x\pl^{j+1}u(\cdot,t)\|_\iy 
 \leq\ve 
\eal 
$$ 
for 
\hb{h\in J\ssm\{0\}} with 
\hb{|h|\leq\da}. Hence the map 
\hb{\big(t\mt\pa_x\pl^ju(\cdot,t)\big)\sco J\ra B\BXcX} is differentiable 
and its derivative equals 
\hb{t\mt\pa_x\pl^{j+1}u(\cdot,t)}. From this and step~(1) we infer 
\hb{u\in\BUC^j(J,\BUC^{(k-j)r})} for 
\hb{0\leq j\leq k}. This implies 
\hb{\BUC^{kr/\vec r}\hr\bigcap_{j=0}^k\BUC^j(J,\BUC^{(k-j)r})}. The 
converse embedding is an obvious consequence of step~(1). 
\end{proof} 
It is an immediate consequence of this lemma that 
$$ 
\BUC^{kr/\vec r}\hr\BUC(J,\BUC^{kr})\cap\BUC^k(J,\BUC). 
$$ 
It follows from Remark~1.13.4.2 in~\cite{Tri78a}, for instance, that 
$\BUC^{kr/\vec r}$~is a proper subspace of the intersection space on 
the right hand side. 

\smallskip 
We infer from \Eqref{PH.BUB} that 
\hb{BC^{(k+1)r/\vec r}\hr\BUC^{kr/\vec r}\hr BC^{kr/\vec r}}. 
Consequently, 
\beq\Label{PH.Binf} 
BC^{\iy/\vec r} 
:={\textstyle\bigcap_k}BC^{kr/\vec r} 
={\textstyle\bigcap_k}\BUC^{kr/\vec r} 
=BC^\iy(\BX\times J,\cX). 
\eeq 

\smallskip 
For 
\hb{s>0} we set 
\beq\Label{PH.Nsrst} 
\bal 
\|u\|_{s/\vec r,\iy}^*
&:=\sup_t\|u(\cdot,t)\|_{s,\iy}^* 
 +\sup_x\big[u(x,\cdot)\big]_{s/r,\iy}\\ 
&\ph{:}=\|u\|_\iy 
 +\sup_t\big[u(\cdot,t)\big]_{s,\iy}
 +\sup_x\big[u(x,\cdot)\big]_{s/r,\iy}. 
\eal 
\eeq  
Suppose 
\hb{0<s\leq r}. Then 
$$ 
\|u\|_{s/\vec r,\iy}^{**}
:=\sup_t\|u(\cdot,t)\|_{s,\iy}^{**}+\sup_x\big[u(x,\cdot)\big]_{s/r,\iy}. 
$$ 
If 
\hb{kr<s\leq(k+1)r} with 
\hb{k\in\BN^\times}, then 
\beq\Label{PH.Nsr2s} 
\|u\|_{s/\vec r,\iy}^{**} 
:=\max_{|\al|+jr\leq kr}\|\pa_x\pl^ju\|_{(s-kr)/\vec r,\iy}^{**}. 
\eeq 

\smallskip 
The \hh{anisotropic Besov-H\"older space scale} 
\hb{[\,B_\iy^{s/\vec r}\ ;\ s>0\,]} is defined by 
$$ 
B_\iy^{s/\vec r} 
:=\left\{
\bal
{}      
&(\BUC^{kr/\vec r},\BUC^{(k+1)r/\vec r})_{(s-kr)/r,\iy},
    &\quad  kr<s &<(k+1)r,\\ 
&(\BUC^{kr/\vec r},\BUC^{(k+2)r/\vec r})_{1/2,\iy}, 
    &\quad     s &=(k+1)r. 
\eal 
\right. 
\npb 
$$  
The next theorem is the anisotropic analogue of Theorem~\ref{thm-PH.B}.\po 
{\samepage 
\begin{theorem}\LabelT{thm-PH.BJ} 
\begin{itemize} 
\item[{\rm(i)}] 
${}$
\hb{\Vsdot_{s/\vec r,\iy}^*} and 
\hb{\Vsdot_{s/\vec r,\iy}^{**}} are norms for~$B_\iy^{s/\vec r}$.\po 
\item[{\rm(ii)}] 
Suppose 
\hb{k_0r<s<k_1r}. Then 
\hb{(\BUC^{k_0r/\vec r},\BUC^{k_1r/\vec r})_{(s-k_0r)/(k_1-k_0)r,\iy} 
   \doteq B_\iy^{s/\vec r}}.\po  
\item[{\rm(iii)}] 
If\/ 
\hb{0<s_0<s_1} and 
\hb{0<\ta<1}, then 
\hb{(B_\iy^{s_0/\vec r},B_\iy^{s_1/\vec r})_{\ta,\iy} 
   \doteq B_\iy^{s_\ta/\vec r} 
   \doteq[B_\iy^{s_0/\vec r},B_\iy^{s_1/\vec r}]_\ta}. 
\item[{\rm(iv)}] 
${}$
\hb{\pa_x\pl^j\in\cL(B_\iy^{(s+|\al|+jr)/\vec r},B_\iy^{s/\vec r})} for 
\hb{\al\in\BN^m} and 
\hb{j\in\BN}.\po 
\end{itemize} 
\end{theorem} }
\begin{proof} 
(1) 
We infer from \cite[\ (3.3.12), \ (3.5.2), and Theorem~4.4.1]{Ama09a} 
that 
\beq\Label{PH.BBJ} 
B_\iy^{s/\vec r}=B_{\iy,\iy}^{s/\mf{\nu}}  
\npb 
\eeq 
and that (ii)~is true. 

\smallskip 
(2) 
The first part of~(iii) follows from~(ii) by reiteration. 

\smallskip 
(3) 
For 
\hb{(\xi,\tau)\in\BR^m\times\BR} we set 
\hb{\wt{\Lda}(\xi,\tau):=(1+|\xi|^{2r}+\tau^2)^{1/2r}}. Then 
\hb{\wt{\mf{\Lda}}\vph{\Lda}^{\coLda s} 
   :=\cF_{m+1}^{-1}\wt{\Lda}\vph{\Lda}^{\coLda s}\cF_{m+1}} for 
\hb{s\in\BR}. From \cite[Theorem~3.4.1]{Ama09a} and \Eqref{PH.BBJ} we get 
$$ 
\wt{\mf{\Lda}}\vph{\Lda}^{\coLda s} 
\in\Lis(B_\iy^{(t+s)/\vec r},B_\iy^{t/\vec r}) 
\qb (\wt{\mf{\Lda}}\vph{\Lda}^{\coLda s})^{-1} 
    =\wt{\mf{\Lda}}\vph{\Lda}^{\coLda-s}  
\qa t,\ t+s>0, 
$$ 
provided 
\hb{\BX=\BR^m} and 
\hb{J=\BR}. Now we obtain the second part of~(iii) by obvious 
modifications of the relevant sections of part~(2) of the proof of 
Theorem~\ref{thm-PH.B}. 

\smallskip 
(4) 
Taking \cite[Section~4.4]{Ama09a} into account, we get from Theorems 
3.3.2 and~3.5.2 therein that 
\hb{B_\iy^{s/\vec r}\hr\BUC}. Suppose 
\hb{kr<s\leq(k+1)r}. By \cite[Theorem~3.6.1]{Ama09a} 
$$ 
\|u\|_{B_\iy^{s/\vec r}} 
\sim\|u\|_\iy 
+\sup_{h\in(0,\iy)^m}\frac{\|\tri_{(h,0)}^{[s]+1}u\|_\iy}{|h|^s} 
+\sup_{h>0}\frac{\|\tri_{(0,h)}^{[s/r]+1}u\|_\iy}{h^{s/r}}, 
$$ 
where $[t]$~is the largest integer less than or equal to 
\hb{t\in\BR}. Since 
\hb{u\in BC} it follows 
$$ 
\|\tri_{(h,0)}^{[s]+1}u\|_\iy=\sup_t\|\tri_h^{[s]+1}u(\cdot,t)\|_\iy 
\qb \|\tri_{(0,h)}^{[s/r]+1}u\|_\iy 
    =\sup_x\|\tri_h^{[s/r]+1}u(x,\cdot)\|_\iy.  
\npb 
$$ 
Thus 
\hb{\Vsdot_{B_\iy^{s/\vec r}}\sim\Vsdot_{s/\vec r,\iy}^*}. 

\smallskip 
(5) 
Suppose 
\hb{\BX=\BR^m} and 
\hb{J=\BR}. Then (iv)~follows by straightforward modifications of the proof 
of \cite[Lemma~2.3.7]{Ama09a} by invoking the Fourier multiplier 
Theorem~3.4.2 therein. Similarly as in the proof of 
\cite[Theorem~2.3.8]{Ama09a}, we see 
that, given 
\hb{0<s\leq r} and 
\hb{k\in\BN}, 
\beq\Label{PH.Bsj} 
\Vsdot_{B_\iy}^{(s+kr)/\vec r} 
\sim\max_{|\al|+jr\leq kr}\|\pa_x\pl^j\cdot\|_{B_\iy^{s/\vec r}} 
\eeq 
(cf.~\cite[Corollary~2.3.4]{Ama09a}). In the general case we now obtain the 
validity of~(iv) and \Eqref{PH.Bsj} by extension and restriction, taking 
\hb{B_\iy^{s/\vec r}\hr\BUC} into account. 

\smallskip 
(6) 
Suppose 
\hb{0<s\leq r}. Then 
\hb{\Vsdot_{s/\vec r,\iy}^*\sim\Vsdot_{s/\vec r,\iy}^{**}} follows from 
Theorem~\ref{thm-PH.B}(i). By combining this with \Eqref{PH.Bsj} we see that 
the latter equivalence holds for every 
\hb{k\in\BN}. This proves the theorem.\po 
\end{proof} 
{\samepage 
\begin{corollary}\LabelT{cor-PH.BJ} 
\begin{itemize} 
\item[{\rm(i)}] 
${}$
\hb{B_\iy^{s/\vec r}\doteq B(J,B_\iy^s)\cap B_\iy^{s/r}\JB}.\po 
\item[{\rm(ii)}] 
Set 
$$ 
\|u\|_{s/\vec r,\iy}^\sim 
:=\sup_t\|u(\cdot,t)\|_{s,\iy}^{**}+\sup_x\|u(x,\cdot)\|_{s/r}^{**} 
\qa s>0. 
$$ 
Then 
\hb{\Vsdot_{s/\vec r,\iy}^\sim} is a norm for~$B_\iy^s$.\po 
\end{itemize} 
\end{corollary} }
\begin{proof} 
(i) is implied by Theorem~\ref{thm-PH.BJ}(i), 
\ \hb{B_\iy^s\hr\BUC^k} if  
\hb{k<s\leq k+1}, and Proposition~\ref{pro-PH.BUC}. \ (ii)~follows from~(i) 
and Theorem~\ref{thm-PH.B}(i). 
\end{proof} 
We define \hh{anisotropic H\"older spaces} by 
\hb{BC^{s/\vec r}:=B_\iy^{s/\vec r}} for 
\hb{s\in\BR^+\ssm r\BN}. By means of the mean value theorem and using the 
norm~%
\hb{\Vsdot_{s/\vec r,\iy}^\sim}, for example, we find, similarly as in the 
isotropic case, that 
$$ 
BC^{s/\vec r}\hr BC^{s_0/\vec r} 
\qa 0\leq s_0<s. 
$$ 
In order to obtain scales of spaces enjoying dense 
embeddings we define \hh{anisotropic little H\"older spaces} by 
\beq\Label{PH.defbc} 
bc^{s/\vec r}\text{ is the closure of $BC^{\iy/\vec r}$ in }BC^{s/\vec r} 
\qa s\in\BR^+. 
\eeq 
Similarly, the \hh{anisotropic little Besov-H\"older space} 
$$ 
b_\iy^{s/\vec r}\text{ is the closure of $BC^{\iy/\vec r}$ 
in }B_\iy^{s/\vec r}  
\qa s>0. 
$$ 
These spaces possess intrinsic characterizations as well. To allow 
for a simple formulation we denote by~$[s]_-$ the largest integer strictly 
less than~$s$.\po 
{\samepage 
\begin{theorem}\LabelT{thm-PH.bJ} 
\begin{itemize} 
\item[{\rm(i)}] 
${}$
\hb{bc^{kr/\vec r}=\BUC^{kr/\vec r}}.\po 
\item[{\rm(ii)}] 
${}$
\hb{bc^{s/\vec r}=b_\iy^{s/\vec r}} if 
\hb{s\in\BR^+\ssm\BN}.\po 
\item[{\rm(iii)}] 
${}$
\hb{u\in b_\iy^{s/\vec r}} iff 
\hb{u\in B_\iy^{s/\vec r}} and 
\beq\Label{PH.Sdx} 
\sup_t\max_{|\al|=[s]_-}\bigl[\pa_xu(\cdot,t)\bigr]_{s-[s]_-,\iy}^\da 
+\sup_x\bigl[\pl^{[s/r]_-}u(x,\cdot)\bigr]_{s/r-[s/r]_-,\iy}^\da\ra0 
\npb 
\eeq 
as 
\hb{\da\ra0}. 
\item[{\rm(iv)}] 
${}$
\hb{BC^{s/\vec r}\sdh b_\iy^{s_0/\vec r}} for 
\hb{0<s_0<s}.\po  
\end{itemize} 
\end{theorem} }
\begin{proof} 
As in previous proofs it suffices to consider the case 
\hb{\BX=\BR^m} and 
\hb{J=\BR}. 

\smallskip 
(1) 
We know from \Eqref{PH.Binf} that 
\hb{BC^{\iy/\vec r}\hr\BUC^{kr/\vec r}}. Let 
\hb{\{\,w_\eta\ ;\ \eta>0\,\}} be a mollifier on~$\BR^{m+1}$. If 
$u$~belongs to~$\BUC^{kr/\vec r}$, then it follows from \Eqref{PH.app} and 
\hb{\pa_x\pl^j(w_\eta*u)=w_\eta*(\pa_x\pl^ju)} that 
\hb{w_\eta*u\ra u} in~$BC^{kr/\vec r}$ as 
\hb{\eta\ra0}. This proves assertion~(i). Claim~(ii) is trivial. 

\smallskip 
(2) 
Let 
\hb{kr\leq i<s\leq i+1\leq(k+1)r} with 
\hb{i\in\BN}. Suppose 
\hb{u\in b_\iy^{s/\vec r}} and 
\hb{\ve>0}. Then we can find $v$ belonging to 
\hb{\in\BUC^{(k+2)r/\vec r}\hr B_\iy^{s/\vec r}} such that 
\hb{\|u-v\|_{s/\vec r,\iy}^{**}<\ve/2}. By Proposition~\ref{pro-PH.BUC} 
we know 
$$ 
\BUC^{(k+2)r/\vec r}\hr\BUC(J,\BUC^{(k+2)r})\cap\BUC^{k+2}(J,\BUC). 
$$ 
Hence it follows from \Eqref{PH.BCk} that 
$$ 
\sup_t\bigl[\pa_xv(\cdot,t)\bigr]_{s-i,\iy}^\da 
\leq c\kern1pt\da\,\|v\|_{(k+2)r/\vec r,\iy} 
\qa 0<\da\leq1 
\qb |\al|=i. 
$$ 
Similarly, 
$$ 
\sup_x\bigl[\pl^kv(x,\cdot)\bigr]_{s/r-k,\iy}^\da 
\leq c\kern1pt\da\,\|v\|_{(k+2)r/\vec r,\iy} 
\qa 0<\da\leq1. 
$$ 
Thus we find 
\hb{\da_\ve>0} such that 
$$ 
\sup_t\max_{|\al|=i}\bigl[\pa_xv(\cdot,t)\bigr]_{s-i,\iy}^\da 
+\sup_x\bigl[\pl^kv(x,\cdot)\bigr]_{s/r-k,\iy}^\da 
<\ve/2 
\qa 0<\da\leq\da_\ve. 
$$ 
Consequently, 
$$ 
\sup_t\max_{|\al|=i}\bigl[\pa_xu(\cdot,t)\bigr]_{s-i,\iy}^\da 
\leq\sup_t\max_{|\al|=i}\bigl[\pa_x(u-v)(\cdot,t)\bigr]_{s-i,\iy} 
+\sup_t\max_{|\al|=i}\bigl[\pa_xv(\cdot,t)\bigr]_{s-i,\iy}^\da 
\leq\ve  
\npb 
$$ 
for 
\hb{0<\da\leq\da_\ve}. This shows that the first term in \Eqref{PH.Sdx} 
converges to zero. Analogously, we see that this is true for the second 
summand. 

\smallskip 
(3) 
Suppose 
\hb{0<s\leq1} and 
\hb{u\in B_\iy^{s/\vec r}} satisfies \Eqref{PH.Sdx}. By \Eqref{PH.con} it 
suffices to show that 
\beq\Label{PH.wetu} 
\|w_\eta*u-u\|_{s/\vec r,\iy}^\sim\ra0 
\quad\text{as }\eta\ra0. 
\eeq 
It follows from 
\hb{\tri_{(h,0)}^{[s]+1}(w_\eta*u)=w_\eta*(\tri_{(h,0)}^{[s]+1}u)} that 
$$ 
\|\tri_{(h,0)}^{[s]+1}(w_\eta*u)(x,t)\| 
\leq\sup_t\|\tri_h^{[s]+1}u(\cdot,t)\|_\iy 
\qa (x,t)\in\BX\times J. 
$$ 
Consequently, 
$$ 
\sup_t\bigl[w_\eta*u(\cdot,t)\bigr]_{s,\iy}^\da 
\leq\sup_t\bigl[u(\cdot,t)\bigr]_{s,\iy}^\da 
\qa 0<\da<\iy. 
$$ 
Let 
\hb{\ve>0} and fix 
\hb{\da_\ve>0} with 
\hb{\sup_t\bigl[u(\cdot,t)\bigr]_{s,\iy}^{\da_\ve}<\ve/4}. Then 
$$ 
\sup_t\bigl[(w_\eta*u-u)(\cdot,t)\bigr]_{s,\iy}^{\da_\ve} 
\leq2\sup_t\bigl[u(\cdot,t)\bigr]_{s,\iy}^{\da_\ve} 
<\ve/2. 
$$ 
Thus we infer from \Eqref{PH.del1} that 
$$ 
\sup_t\bigl[(w_\eta*u-u)(\cdot,t)\bigr]_{s,\iy} 
\leq\ve/2+4\da_\ve^{-s}\sup_t\|(w_\eta*u-u)(\cdot,t)\|_\iy. 
$$ 
Since 
\hb{u\in\BUC(\BX\times J,\cX)} it follows from \Eqref{PH.app} that 
$$ 
\sup_t\|(w_\eta*u-u)(t)\|_\iy 
=\|w_\eta*u-u\|_{B(\BX\times J,\cX)}\ra0 
\quad\text{as }\eta\ra0. 
$$ 
Hence 
$$ 
\sup_t\bigl[(w_\eta*u-u)(\cdot,t)\bigr]_{s,\iy}\ra0 
\quad\text{as }\eta\ra0. 
$$ 
Similarly, 
$$ 
\sup_x\bigl[(w_\eta*u-u)(x,\cdot)\bigr]_{s/r,\iy}\ra0 
\quad\text{as }\eta\ra0. 
\npb 
$$ 
This proves \Eqref{PH.wetu}, thus, due to step~(2), assertion~(iii) for 
\hb{0<s\leq1}.  

\smallskip 
(4) 
To prove (iv) assume 
\hb{kr\leq i<s\leq i+1\leq(k+1)r} and 
\hb{u\in B_\iy^{s/\vec r}} satisfies \Eqref{PH.Sdx}. Then it follows from 
\hb{\pa_x\pl^j(w_\eta*u)=w_\eta*(\pa_x\pl^ju)} for 
\hb{|\al|+jr<s} and step~(3) that 
\hb{w_\eta*u\ra u} in~$B_\iy^{s/\vec r}$ as 
\hb{\eta\ra0}. Hence 
\hb{u\in b_\iy^{s/\vec r}}, which shows that claim~(iii) is always true. 

\smallskip 
(5) 
The proof of~(iv) is obtained by employing \Eqref{PH.lHd}, \,\Eqref{PH.BCk}, 
and Corollary~\ref{cor-PH.BJ}(ii). 
\end{proof} 
Anisotropic little H\"older spaces can be characterized by interpolation, 
similarly as their isotropic relatives.\po 
{\samepage 
\begin{theorem}\LabelT{thm-PH.bcintJ} 
\begin{itemize} 
\item[{\rm(i)}] 
${}$
\hb{b_\iy^{s/\vec r} 
\doteq(bc^{k_0r/\vec r},bc^{k_1r/\vec r})_{(s/r-k_0)/(k_1-k_0),\iy}^0} 
for 
\hb{k_0r<s<k_1r}.\po 
\item[{\rm(ii)}] 
If\/ 
\hb{0<s_0<s_1} and 
\hb{0<\ta<1}, then 
\hb{(b_\iy^{s_0/\vec r},b_\iy^{s_1/\vec r})_{\ta,\iy}^0  
   \doteq b_\iy^{s_\ta/\vec r} 
   \doteq[b_\iy^{s_0/\vec r},b_\iy^{s_1/\vec r}]_\ta}. 
\item[{\rm(iii)}] 
${}$
\hb{\pa_x\pl^j\in\cL(b_\iy^{(s+|\al|+jr)/\vec r},b_\iy^{s/\vec r})} for 
\hb{\al\in\BN^m} and 
\hb{j\in\BN}.\po 
\end{itemize} 
\end{theorem} }
\begin{proof} 
(1) 
The first assertion as well as the first part of~(ii) follow from part~(i) 
of Theorem~\ref{thm-PH.bJ}. Part two of~(ii) and the first claim are 
implied by part~(iii) of Theorem~\ref{thm-PH.BJ}. 

\smallskip 
(2) 
The last part of statement~(ii) is obtained by replacing $BC^\iy$ 
and~$\mf{\Lda}^{\coLda s}$ in step~(2) of the proof of 
Theorem~\ref{thm-PH.bcint} by $BC^{\iy/\vec r}$ 
and~$\wt{\mf{\Lda}}^{\coLda s}$, respectively. 

\smallskip 
(3) 
Theorem~\ref{thm-PH.BJ}(iv) implies 
\hb{\pa_x\pl^j\in\cL(BC^{\iy/\vec r})}. Thus, using the definition 
of~$b_\iy^{s/\vec r}$ and once more the latter theorem, we obtain~(iii). 
\end{proof} 
In the next section we need to employ H\"older spaces with a particular 
choice of~$\cX$ which we discuss now. For this we remind the reader of the 
notations and conventions introduced at the beginning of 
Section~\ref{sec-R}. 

\smallskip 
Let 
\hb{\{\,F_\ba\ ;\ \ba\in\sB\,\}} be a countable family of Banach spaces. 
Then it is obvious that 
\beq\Label{PH.f} 
\mf{f}\sco\mf{F}^\BX\ra\prod_\ba F_\ba^\BX 
\qb u\mt\mf{f}u:=(\pro_\ba\circ u) 
\eeq 
is a linear bijection. Since $\mf{F}$~carries the product topology 
\hb{u\in\mf{F}^\BX} is continuously differentiable iff 
$$ 
u_\ba:=\pro_\ba\circ u\in C^1(\BX,F_\ba) 
\qa \ba\in\sB. 
$$ 
Then 
\hb{\pl_ju=(\pl_ju_\ba)}, that is, 
\beq\Label{PH.fdj} 
\mf{f}\circ\pa_x=\pa_x\circ\mf{f} 
\qa \al\in\BN^m. 
\eeq 
Setting 
\hb{\mf{C}^k\BXBF:=\prod_\ba C^k(\BX,F_\ba)} etc., it follows 
\beq\Label{PH.fLis} 
\mf{f}\in\Lis\bigl(C^k(\BX,\mf{F}),\mf{C}^k\BXBF\bigr). 
\eeq 
Furthermore, 
\beq\Label{PH.fBC1} 
\mf{f}\in\cL\bigl(BC^k\bigl(\BX,\ell_\iy(\mf{F})\bigr), 
\ell_\iy\bigl(\mf{BC}^k\BXBF\bigr)\bigr). 
\eeq 
Suppose 
\hb{u\in BC^1\bigl(\BX,\ell_\iy(\mf{F})\bigr)}. Then, given 
\hb{x\in\BX}, 
$$ 
\sup_{\ba\in\sB}\big\|t^{-1}\bigl(u_\ba(x+te_j)-u_\ba(x)\bigr)-\pl_ju_\ba(x) 
\big\|_{F_\ba} 
=\big\|t^{-1}\bigl(u(x+te_j)-u(x)\bigr)-\pl_ju(x) 
\big\|_{\ell_\iy(\mf{F})} 
\ra0 
$$ 
as 
\hb{t\ra0}, with 
\hb{t>0} if 
\hb{\BX=\BH^m} and 
\hb{j=1}. From this we see that $\mf{f}$~maps 
\hb{u\in BC^k\bigl(\BX,\ell_\iy(\mf{F})\bigr)} into the linear subspace of 
$\ell_\iy\bigl(\mf{BC}^k\BXBF\bigr)$ consisting of all 
\hb{\mf{v}=(v_\ba)} for which $v_\ba$~is \hbox{$k$-times} continuously 
differentiable, uniformly with respect to 
\hb{\ba\in\sB}. Thus \Eqref{PH.fBC1} is not surjective if 
\hb{k\geq1}. 

\smallskip 
We denote by 
$$ 
\ell_{\iy,\unif}\bigl(\mf{bc}^k\BXBF\bigr)
$$ 
the linear subspace of $\ell_\iy\bigl(\mf{BC}^k\BXBF\bigr)$ of all 
\hb{\mf{v}=(v_\ba)} such that $\pa v_\ba$~is uniformly continuous on~$\BX$ 
for 
\hb{|\al|\leq k}, uniformly with respect to 
\hb{\ba\in\sB}. 
\begin{lemma}\LabelT{lem-PH.f} 
$\mf{f}$~is an isomorphism 
\beq\Label{PH.fBUC} 
\text{from } 
bc^k\bigl(\BX,\ell_\iy(\mf{F})\bigr) 
\text{ onto } 
\ell_{\iy,\unif}\bigl(\mf{bc}^k\BXBF\bigr) 
\eeq 
and 
\beq\Label{PH.fBC2} 
\text{from } 
B_\iy^s\bigl(\BX,\ell_\iy(\mf{F})\bigr) 
\text{ onto } 
\ell_\iy\bigl(\mf{B}_\iy^s\BXBF\bigr) 
\qa s>0. 
\eeq 
\end{lemma} 
\begin{proof} 
(1) 
Suppose 
\hb{u\in bc^k\bigl(\BX,\ell_\iy(\mf{F})\bigr)}. Then, by the above, it is 
obvious that 
\hb{\mf{f}u\in\ell_{\iy,\unif}\bigl(\mf{bc}^k\BXBF\bigr)}.\\ 
Conversely, assume 
\hb{\mf{u}=(u_\ba)\in\ell_{\iy,\unif}\bigl(\mf{bc}^k\BXBF\bigr)}. Set 
\hb{u:=\mf{f}^{-1}\mf{u}}, which is defined due to \Eqref{PH.fLis}. Then 
$$ 
\|u(x)\|_{\ell_\iy(\mf{F})} 
=\sup_\ba\|u_\ba(x)\|_{F_\ba} 
\qa x\in\BX, 
$$ 
and 
$$ 
\|u(x)-u(y)\|_{\ell_\iy(\mf{F})} 
=\sup_\ba\|u_\ba(x)-u_\ba(y)\|_{F_\ba} 
\qa x,y\in\BX, 
\npb 
$$ 
show 
\hb{u\in bc\bigl(\BX,\ell_\iy(\mf{F})\bigr)}. Hence we infer from 
\Eqref{PH.fdj} that 
\hb{\pa_xu\in bc\bigl(\BX,\ell_\iy(\mf{F})\bigr)} for 
\hb{|\al|\leq k}. 

\smallskip 
(2) 
Let 
\hb{k\geq1} and 
\hb{1\leq j\leq m}. Then, by the mean value theorem, 
$$ 
t^{-1}\bigl(u_\ba(x+te_j)-u_\ba(x)\bigr)-\pl_ju_\ba(x) 
=\int_0^1\big(\pl_ju_\ba(x+ste_j)-\pl_ju_\ba(x)\big)\,ds, 
$$ 
where 
\hb{t>0} if 
\hb{j=1} and 
\hb{\BX=\BH^m}. Hence 
$$ 
\bal 
\big\|t^{-1}\bigl(u_\ba(x+te_j)-u_\ba(x)\bigr)-\pl_ju_\ba(x)\big\|_{F_\ba} 
&\leq\sup_{|t|\leq\da}\sup_{x\in\BX}\|\pl_ju_\ba(x+te_j)-\pl_ju_\ba(x)\| 
 _{F_\ba}\\ 
&\leq\sup_{|t|\leq\da}\|\pl_j\mf{u}(\cdot+te_j)-\pl_j\mf{u}\| 
 _{\ell_\iy(\mf{BC}\BXBF)}  
\eal 
$$ 
for 
\hb{|t|\leq\da}, 
\ \hb{x\in\BX}, and 
\hb{\ba\in\sB}. Thus 
$$ 
\big\|t^{-1}\bigl(u(\cdot+te_j)-u\bigr)-\pl_ju 
\big\|_{B(\BX,\ell_\iy(\mf{F}))} 
\leq\sup_{|t|\leq\da}\|\pl_j\mf{u}(\cdot+te_j)-\pl_j\mf{u} 
\| _{\ell_\iy(\mf{BC}\BXBF)}  
$$ 
for 
\hb{|t|\leq\da}. This implies that $u$~is differentiable in the topology of 
$BC\bigl(\BX,\ell_\iy(\mf{F})\bigr)$. From this, step~(1), and by 
induction we infer 
$$ 
\mf{f}^{-1}\in\cL\bigl(\ell_{\iy,\unif}\bigl(\mf{bc}^k\BXBF\bigr), 
bc^k\bigl(\BX,\ell_\iy(\mf{F})\bigr)\bigr). 
\npb   
$$ 
This proves \Eqref{PH.fBUC}. 

\smallskip 
(3) 
Suppose 
\hb{0<s\leq1} and set 
\hb{i:=[s]_-}. It is convenient to write 
\hb{h\gg0} iff 
\hb{h\in(0,\iy)^m}. Given $u$ belonging to 
$B_\iy^s\bigl(\BX,\ell_\iy(\mf{F})\bigr)$, we deduce from 
\hb{\tri_hu_\ba=\pro_\ba(\tri_hu)} that 
\beq\Label{PH.HSl} 
\bal 
\sup_\ba\bigl[\pro_\ba(\mf{f}u)\bigr]_{s,\iy;F_\ba} 
=\sup_\ba[u_\ba]_{s,\iy;F_\ba} 
&=\sup_\ba\sup_{h\gg0}\sup_x\frac{\|\tri_h^{i+1}u_\ba(x)\|_{F_\ba}}{|h|^s}\\ 
&=\sup_{h\gg0}\sup_x\frac{\|\tri_h^{i+1}u(x)\|_{\ell_\iy(\mf{F})}}{|h|^s} 
 =\sup_{h\gg0}\frac{\|\tri_h^{i+1}u\|_{\iy;\ell_\iy(\mf{F})}}{|h|^s}\\ 
&=[u]_{s,\iy;\ell_\iy(\mf{F})}. 
\eal 
\eeq 
From \Eqref{PH.fBC1} and \Eqref{PH.HSl} we infer 
\beq\Label{PH.fBCs} 
\mf{f}\in\cL\bigl(B_\iy^s\bigl(\BX,\ell_\iy(\mf{F})\bigr), 
\ell_\iy\bigl(\mf{B}_\iy^s\BXBF\bigr)\bigr).  
\npb 
\eeq 
Now it follows from \Eqref{PH.fdj} that \Eqref{PH.fBCs} holds for any 
\hb{s>0}. 

\smallskip 
It is obvious from \Eqref{PH.eqsn} and \Eqref{PH.fdj} that, given 
\hb{k<s\leq k+1}, 
$$ 
\ell_\iy\bigl(\mf{B}_\iy^s\BXBF\bigr) 
\hr\ell_{\iy,\unif}\bigl(\mf{bc}^k\BXBF\bigr). 
$$ 
From this, \Eqref{PH.HSl}, and~\Eqref{PH.fBUC} we get that $\mf{f}$~is 
onto $\ell_\iy\bigl(\mf{B}_\iy^s\BXBF\bigr)$. Due to \Eqref{PH.fLis} this 
proves \Eqref{PH.fBC2}. 
\end{proof} 
We denote for 
\hb{k<s\leq k+1} by 
$$ 
\ell_{\iy,\unif}\bigl(\mf{b}_\iy^s\BXBF\bigr)
$$ 
the linear subspace of $\ell_{\iy,\unif}\bigl(\mf{bc}^k\BXBF\bigr)$ 
of all 
\hb{\mf{v}=(v_\ba)} such that 
\hb{\lim_{\da\ra0}\max_{|\al|=k}[\pa_xv_\ba]_{s-k,\iy;F_\ba}^\da=0}, 
uniformly with respect to 
\hb{\ba\in\sB}. 
\begin{lemma}\LabelT{lem-PH.fbc} 
\hb{\mf{f}\in\Lis\bigl(b_\iy^s\bigl(\BX,\ell_\iy(\mf{F})\bigr),
   \ell_{\iy,\unif}\bigl(\mf{b}_\iy^s\BXBF\bigr)\bigr)}. 
\end{lemma} 
\begin{proof} 
The proof of \Eqref{PH.HSl} shows that, given 
\hb{k<s\leq k+1}, 
$$ 
\sup_\ba\bigl[\pro_\ba(\mf{f}u)\bigr]_{s,\iy;F_\ba}^\da 
=[u]_{s,\iy;\ell_\iy(\mf{F})}^\da 
\qa \da>0.  
\npb 
$$ 
Thus the claim follows by the arguments of step~(2) of the proof 
of Lemma~\ref{lem-PH.f} and from Theorem~\ref{thm-PH.bc}. 
\end{proof} 
Now we extend~$\mf{f}$ point-wise over~$J$: 
$$ 
\wt{\mf{f}}\sco\mf{F}^{\BX\times J}\ra\prod_\ba F_\ba^{\BX\times J} 
\qb u\mt\wt{\mf{f}}u:=\bigl(t\mt\mf{f}u(\cdot,t)\bigr). 
\npb 
$$ 
As above, 
\hb{\mf{B}_\iy^{s/\vec r}(\BX\times J,\BF) 
   :=\prod_\ba B_\iy^{s/\vec r}(\BX\times J,F_\ba)} for 
\hb{s>0}. Analogous definitions apply to 
\hb{\mf{b}_\iy^{s/\vec r}(\BX\times J,\BF)}. 

\smallskip 
Clearly, 
$$ 
\ell_{\iy,\unif}\bigl(\mf{bc}^{kr/\vec r}(\BX\times J,\BF)\bigr)
$$ 
is the closed subspace of 
$\ell_\iy\bigl(\mf{BC}^{kr/\vec r}(\BX\times J,\BF)\bigr)$ of all 
\hb{\mf{u}=(u_\ba)} for which 
\hb{\pa_x\pl^ju_\ba\in\BUC(\BX\times J,F_\ba)} for 
\hb{|\al|+jr\leq kr}, uniformly with respect to 
\hb{\ba\in\sB}. 

\smallskip 
Suppose 
\hb{kr<s\leq(k+1)r}. We denote by 
$$ 
\ell_{\iy,\unif}\bigl(\mf{b}_\iy^{s/\vec r}(\BX\times J,\BF)\bigr)
$$ 
the set of all 
\hb{\mf{u}=(u_\ba)\in\ell_\iy\bigl(\mf{B}_\iy^{s/\vec r} 
   (\BX\times J,\BF)\bigr)} satisfying 
$$ 
\sup_\ba\sup_t\max_{|\al|=[s]_-} 
\bigl[\pa_xu_\ba(\cdot,t)\bigr]_{s-[s]_-,\iy;F_\ba}^\da 
+\sup_\ba\sup_x  
\bigl[\pl^{[s/r]_-}u_\ba(x,\cdot)\bigr]_{s/r-[s/r]_-,\iy;F_\ba}^\da\ra0  
\npb 
$$ 
as 
\hb{\da\ra0}. 

\smallskip 
Now we can prove the following anisotropic analogue of Lemmas 
\ref{lem-PH.f} and \ref{lem-PH.fbc}. 
\begin{lemma}\LabelT{lem-PH.fJ} 
$\wt{\mf{f}}$~is an isomorphism 
$$ 
\text{from } 
bc^{kr/\vec r}\bigl(\BX\times J,\ell_\iy(\mf{F})\bigr) 
\text{ onto } 
\ell_{\iy,\unif}\bigl(\mf{bc}^{kr/\vec r}(\BX\times J,\BF)\bigr) 
$$ 
and 
$$ 
\text{from } 
B_\iy^{s/\vec r}\bigl(\BX\times J,\ell_\iy(\mf{F})\bigr) 
\text{ onto } 
\ell_\iy\bigl(\mf{B}_\iy^{s/\vec r}(\BX\times J,\BF)\bigr) 
$$ 
as well as 
$$ 
\text{from } 
b_\iy^{s/\vec r}\bigl(\BX\times J,\ell_\iy(\mf{F})\bigr) 
\text{ onto } 
\ell_{\iy,\unif}\bigl(\mf{b}_\iy^{s/\vec r}(\BX\times J,\BF)\bigr). 
$$ 
\end{lemma} 
\begin{proof} 
Note 
\hb{\pl^j\circ\wt{\mf{f}}=\wt{\mf{f}}\circ\pl^j}. Hence the first assertion 
follows from \Eqref{PH.fBUC}. The remaining statements are verified by 
obvious modifications of the relevant parts of the proofs of Lemmas 
\ref{lem-PH.f} and~\ref{lem-PH.fbc}, taking Corollary~\ref{cor-PH.BJ}(ii) 
and Theorem~\ref{thm-PH.bJ} into account. 
\end{proof} 
\section{Weighted H\"older Spaces }\LabelT{sec-H}%
Having investigated H\"older spaces on $\BR^m$ and~$\BH^m$ in the 
preceding section we now return to the setting of singular manifolds. First 
we introduce isotropic weighted H\"older spaces and study some of their 
properties. Afterwards we study to anisotropic H\"older spaces of 
time-dependent \hbox{$W$-valued} 
\hb{(\sa,\tau)}-tensor fields on~$M$. Making use of the results of 
Section~\ref{sec-PH} we can give coordinate-free invariant definitions 
of these spaces. 

\smallskip 
By 
\hb{B^{0,\lda}=B^{0,\lda}(V)} we mean the weighted 
Banach space of all sections~$u$ of~$V$ satisfying 
$$ 
\|u\|_{\iy;\lda}=\|u\|_{0,\iy;\lda}  
:=\big\|\rho^{\lda+\tau-\sa}\,|u|_h\,\big\|_\iy<\iy, 
\npb 
$$ 
endowed with the norm~%
\hb{\Vsdot_{\iy;\lda}}, and 
\hb{B:=B^{0,0}}. 

\smallskip 
For 
\hb{k\in\BN} 
$$
BC^{k,\lda}=BC^{k,\lda}(V)
:=\bigl(\bigl\{\,u\in C^k\MV\ ;\ \|u\|_{k,\iy;\lda}<\iy\,\bigr\},
\ \Vsdot_{k,\iy;\lda}\bigr),
$$
where
$$
\|u\|_{k,\iy;\lda}
:=\max_{0\leq i\leq k}\big\|\rho^{\lda+\tau-\sa+i}\,|\na^iu|_h\big\|_\iy.
$$
The topologies of $B^{0,\lda}$ and $BC^{k,\lda}$ are independent of the
particular choice of
\hb{\rho\in\gT(M)}. Consequently, this is also true for all other spaces of 
this section as follows from their definition which involves the topology 
of~$BC^{k,\lda}$ for 
\hb{k\in\BN} only. It is a consequence of Theorem~\ref{thm-H.Ri} below that 
$BC^{k,\lda}$~is a Banach space. 

\smallskip 
We set 
$$ 
BC^{\iy,\lda}=BC^{\iy,\lda}(V) 
:={\textstyle\bigcap_k}BC^{k,\lda}, 
$$ 
endowed with the obvious projective topology. Then 
$$ 
bc^{k,\lda}=bc^{k,\lda}(V)\text{ is the closure of 
$BC^{\iy,\lda}$ in }BC^{k,\lda} 
\qa k\in\BN. 
$$ 
The \hh{weighted Besov-H\"older space scale} 
\hb{[\,B_\iy^{s,\lda}\ ;\ s>0\,]} is defined by 
\beq\Label{H.Bdef1} 
B_\iy^{s,\lda}=B_\iy^{s,\lda}(V) 
:=\left\{
\bal
{}      
&(bc^{k,\lda},bc^{k+1,\lda})_{s-k,\iy},
    &\quad  k<s &<k+1,\\ 
&(bc^{k,\lda},bc^{k+2,\lda})_{1/2,\iy}, 
    &\quad  s   &=k+1. 
\eal 
\right. 
\npb 
\eeq 
It is a scale of Banach spaces. 

\smallskip 
The following fundamental retraction theorem allows to characterize 
Besov-H\"older spaces locally. 
\begin{theorem}\LabelT{thm-H.Ri} 
Suppose 
\hb{k\in\BN} and 
\hb{s>0}. Then $\psi_\iy^\lda$~is a retraction from~$\ell_\iy(\mf{BC}^k)$  
onto~$BC^{k,\lda}$  and from $\ell_\iy(\mf{B}_\iy^s)$  
onto~$B_\iy^{s,\lda}$, and $\vp_\iy^\lda$~is a coretraction. 
\end{theorem} 
\begin{proof} 
(1) 
The first claim is settled by Theorem~6.3 of~\cite{Ama12b}. 

\smallskip 
(2) 
Suppose 
\hb{k\in\BN}. It is obvious by the definition of~$\mf{bc}^k$, step~(1), 
\ \Eqref{S.sd}, and \Eqref{R.LS} 
that 
\beq\Label{H.psi} 
\psi_\iy^\lda\text{ is a retraction from $\ell_{\iy,\unif}(\mf{bc}^k)$ 
onto $bc^{k,\lda}$, and $\vp_\iy^\lda$ is a coretraction.} 
\eeq 

\smallskip 
(3) 
If 
\hb{\pl M=\es}, then we put 
\hb{\BM:=\BR^m}, 
\ \hb{\sB:=\gK}, and 
\hb{F_\ka:=E_\ka:=E} for 
\hb{\ka\in\gK}. Then, defining~$\mf{f}$ by \Eqref{PH.f} with this choice 
of~$F_\ba$ and 
\hb{\BX:=\BR^m}, Lemma~\ref{lem-PH.f} implies 
\beq\Label{H.fis} 
\mf{f}\in\Lis\bigl(bc^k\bigl(\BM,\ell_\iy(\mf{E})\bigr), 
\ell_{\iy,\unif}(\mf{bc}^k)\bigr). 
\eeq 

\smallskip 
(4) 
Suppose 
\hb{\pl M\neq\es}. Then we set 
\hb{\gK_0:=\gK\ssm\gK_{\pl M}} and 
\hb{\gK_1:=\gK_{\pl M}}. With 
\hb{E_\ka:=E} for 
\hb{\ka\in\gK} we put 
\hb{\mf{E}_i:=\prod_{\ka\in\gK_i}E_\ka} and define~$\mf{f}_{\cof i}$ by 
setting 
\hb{\sB=\gK_i} and 
\hb{F_\ka=E_\ka}. Then, letting 
\hb{\BX_0:=\BR^m} and 
\hb{\BX_1:=\BH^m}, we infer from Lemma~\ref{lem-PH.f} 
\beq\Label{H.fi} 
\mf{f}_{\cof i}\in\Lis\bigl(bc^k\bigl(\BX_i,\ell_\iy(\mf{E}_i)\bigr), 
\ell_{\iy,\unif}\bigl(\mf{bc}^k(\BX_i,\BE_i)\bigr)\bigr),  
\npb 
\eeq 
with 
\hb{\mf{bc}^k\BXiBEi:=\prod_{\ka\in\gK_i}bc^k\BXiEka}. 

\smallskip 
For 
\hb{\mf{bc}^k=\prod_{\ka\in\gK}bc_\ka^k} we use the natural identification 
\hb{\mf{bc}^k=\mf{bc}^k\BXnBEn\oplus\mf{bc}^k\BXeBEe}. It induces a 
topological direct sum decomposition 
\beq\Label{H.ls} 
\ell_{\iy,\unif}(\mf{bc}^k) 
=\ell_{\iy,\unif}\bigl(\mf{bc}^k\BXnBEn\bigr) 
\oplus\ell_{\iy,\unif}\bigl(\mf{bc}^k\BXeBEe\bigr), 
\npb 
\eeq 
where on the right side we use the maximum of the norms of the two summands. 

\smallskip 
Denoting by~%
\hb{{}\sqcup{}} the disjoint union, we set 
\hb{\BM:=\BR^m\sqcup\BH^m} and 
$$ 
bc^k\bigl(\BM,\ell_\iy(\mf{E})\bigr) 
:=bc^k\bigl(\BX_0,\ell_\iy(\mf{E}_0)\bigr) 
\oplus bc^k\bigl(\BX_1,\ell_\iy(\mf{E}_1)\bigr). 
$$ 
It follows from \Eqref{H.fi} and \Eqref{H.ls} that 
\beq\Label{H.ffis} 
\mf{f}:=f_0\circ\pro_0+f_1\circ\pro_1  
\in\Lis\bigl(bc^k\bigl(\BM,\ell_\iy(\mf{E})\bigr), 
\ell_{\iy,\unif}(\mf{bc}^k)\bigr).   
\eeq 

\smallskip 
(5) 
Returning to the general case, where $\pl M$~may or may not be empty, 
we set 
$$  
\Phi_\iy^\lda:=\mf{f}^{-1}\circ\vp_\iy^\lda 
\qb \Psi_\iy^\lda:=\psi_\iy^\lda\circ\mf{f}. 
$$  
We deduce from \Eqref{H.psi}, \,\Eqref{H.fis}, and \Eqref{H.ffis} that 
\beq\Label{H.Psi} 
\Psi_\iy^\lda\text{ is a retraction from 
$bc^k\bigl(\BM,\ell_\iy(\mf{E})\bigr)$ onto $bc^{k,\lda}$, and 
$\Phi_\iy^\lda$ is a coretraction.} 
\eeq 
As a consequence of this, Theorem~\ref{thm-PH.B}(ii), 
definition~\Eqref{H.Bdef1}, and general properties of interpolation functors 
(cf.~\cite{Ama95a}, Proposition~I.2.3.3) we find 
\beq\Label{H.Psis} 
\Psi_\iy^\lda\text{ is a retraction from 
$B_\iy^s\bigl(\BM,\ell_\iy(\mf{E})\bigr)$ onto $B_\iy^{s,\lda}$, and 
$\Phi_\iy^\lda$ is a coretraction.} 
\eeq 
Since 
\beq\Label{H.PPh} 
\psi_\iy^\lda=\Psi_\iy^\lda\circ\mf{f}^{-1} 
\qb \vp_\iy^\lda=\mf{f}\circ\Phi_\iy^\lda 
\npb 
\eeq 
we get the second assertion from \Eqref{H.Psis} and 
Lemma~\ref{lem-PH.f}.\po  
\end{proof} 
{\samepage 
\begin{corollary}\LabelT{cor-H.Ri} 
\begin{itemize} 
\item[{\rm(i)}] 
${}$
\hb{u\mt\vd u\vd_{k,\iy;\lda}
   :=\sup_{\ka\slt\vp}\rho_\ka^\lda 
   \,\|(\ka\slt\vp)_*(\pi_\ka u)\|_{k,\iy;E}} is a norm 
for~$BC^{k,\lda}$.\po 
\item[{\rm(ii)}] 
Suppose 
\hb{s>0}. Then 
$$ 
u\mt\vd u\vd_{s,\iy;\lda}^* 
:=\sup_{\ka\slt\vp}\rho_\ka^\lda 
\,\|(\ka\slt\vp)_*(\pi_\ka u)\|_{s,\iy;E}^* 
$$ 
and 
$$ 
u\mt\vd u\vd_{s,\iy;\lda}^{**} 
:=\sup_{\ka\slt\vp}\rho_\ka^\lda 
\,\|(\ka\slt\vp)_*(\pi_\ka u)\|_{s,\iy;E}^{**} 
\npb 
$$ 
are norms for~$B_\iy^{s,\lda}$.\po 
\item[{\rm(iii)}] 
Assume 
\hb{k_0<s<k_1} with 
\hb{k_0,k_1\in\BN}. Then 
\hb{(bc^{k_0,\lda},bc^{k_1,\lda})_{(s-k_0)/(k_1-k_0),\iy}
   \doteq B_\iy^{s,\lda}}.  
\item[{\rm(iv)}] 
If\/ 
\hb{0<s_0<s_1} and 
\hb{0<\ta<1}, then 
\hb{(B_\iy^{s_0,\lda},B_\iy^{s_1,\lda})_{\ta,\iy}
   \doteq B_\iy^{s_\ta,\lda} 
   \doteq[B_\iy^{s_0,\lda},B_\iy^{s_1,\lda}]_\ta}.\po  
\end{itemize} 
\end{corollary} }
\begin{proof} 
(i) and~(ii) are implied by \Eqref{R.peq} and Theorem~\ref{thm-PH.B}(i). 
Assertions (iii) and~(iv) follow from \Eqref{H.Psi} and \Eqref{H.Psis} 
and parts (ii) and~(iii) of Theorem~\ref{thm-PH.B}, respectively,  and 
\Eqref{H.PPh} and Lemma~\ref{lem-PH.f}. 
\end{proof} 
\hh{Weighted H\"older spaces} are defined by 
\hb{BC^{s,\lda}:=B_\iy^{s,\lda}} for 
\hb{s\in\BR^+\ssm\BN}. This is in agreement with 
Theorem~\ref{thm-PH.B}(ii).

\smallskip 
Parts (i) and~(ii) of Corollary~\ref{cor-H.Ri} show that the present 
definition of weighted H\"older spaces is equivalent to the one used 
in~\cite{Ama12b}. It should be noted that Corollary~\ref{cor-H.Ri}(iii) 
gives a positive answer to the conjecture of Remark~8.2 of~\cite{Ama12b}, 
provided $BC^{k,\lda}$ and~$BC^{k+1,\lda}$ are replaced by 
$bc^{k,\lda}$ and~$bc^{k+1,\lda}$, respectively. 

\smallskip 
We define \hh{weighted little H\"older spaces} by 
$$ 
bc^{s,\lda}\text{ is the closure of $BC^{\iy,\lda}$ in }BC^{s,\lda} 
\qa s\geq0. 
$$ 
Similarly, the \hh{weighted little Besov-H\"older space scale} 
\hb{[\,b_\iy^{s,\lda}\ ;\ s>0\,]} is obtained by 
\beq\Label{H.bldef} 
b_\iy^{s,\lda}\text{ is the closure of $BC^{\iy,\lda}$ in }B_\iy^{s,\lda}. 
\eeq 
\begin{theorem}\LabelT{thm-H.Rbci} 
$\psi_\iy^\lda$~is a retraction from $\ell_{\iy,\unif}(\mf{b}_\iy^s)$ 
onto~$b_\iy^{s,\lda}$, and $\vp_\iy^\lda$~is a coretraction. 
\end{theorem} 
\begin{proof} 
We infer from \Eqref{PH.Binf} that 
\hb{BC^\iy\bigl(\BM,\ell_\iy(\mf{E})\bigr) 
   =\bigcap_kbc^k\bigl(\BM,\ell_\iy(\mf{E})\bigr)}. Hence we get 
from \Eqref{H.Psi} that 
$\Psi_\iy^\lda$~is a retraction from 
$BC^\iy\bigl(\BM,\ell_\iy(\mf{E})\bigr)$ onto $BC^{\iy,\lda}$, and 
$\Phi_\iy^\lda$~is a coretraction. Due to this and definitions 
\Eqref{PH.bdef} and \Eqref{H.bldef} we deduce from \Eqref{H.Psi} and 
\Eqref{H.Psis} that 
\beq\Label{H.Psbci} 
\Psi_\iy^\lda\text{ is a retraction from 
$b_\iy^s\bigl(\BM,\ell_\iy(\mf{E})\bigr)$ onto $b_\iy^{s,\lda}$, and 
$\Phi_\iy^\lda$ is a coretraction}.  
\npb 
\eeq 
Now the assertion follows from Lemma~\ref{lem-PH.fbc} and \Eqref{H.PPh}. 
\end{proof}\po  
{\samepage 
\begin{corollary}\LabelT{cor-H.Rbci} 
\begin{itemize} 
\item[{\rm(i)}] 
Suppose 
\hb{k_0<s<k_1} with 
\hb{k_0,k_1\in\BN}. Then 
\hb{(bc^{k_0,\lda},bc^{k_1,\lda})_{(s-k_0)/(k_1-k_0),\iy}^0 
   \doteq b_\iy^{s,\lda}}. 
\item[{\rm(ii)}] 
If\/ 
\hb{0<s_0<s_1} and\/ 
\hb{0<\ta<1}, then 
\hb{(b_\iy^{s_0,\lda},b_\iy^{s_1,\lda})_{\ta,\iy}^0 
   \doteq b_\iy^{s_\ta,\lda} 
   \doteq[b_\iy^{s_0,\lda},b_\iy^{s_1,\lda}]_\ta}.\po  
\end{itemize} 
\end{corollary} }
\begin{proof} 
These predications are derived from \Eqref{H.Psbci} and 
Theorem~\ref{thm-PH.bcint}. 
\end{proof} 
Now we turn to weighted anisotropic spaces. We set 
\beq\Label{H.krN0} 
BC^{0/\vec r,\vec\om}=BC^{0/\vec r,(\lda,0)} 
:=\bigl(\bigl\{\,u\in C\bigl(J,C(V)\bigr) 
\ ;\ \|u\|_{\iy;B^{0,\lda}}<\iy\,\bigr\}, 
\ \Vsdot_{\iy;B^{0,\lda}}\bigr) 
\eeq 
and, for 
\hb{k\in\BN^\times}, 
\beq\Label{H.krN1} 
BC^{kr/\vec r,\vec\om} 
:=\bigl\{\,u\in C\bigl(J,C(V)\bigr) 
\ ;\ \na^i\pl^ju\in BC^{0/\vec r,(\lda+i+j\mu,0)}, 
\ i+jr\leq kr\,\bigr\}, 
\eeq 
endowed with the norm 
\beq\Label{H.krN} 
u\mt\|u\|_{kr/\vec r,\iy;\vec\om} 
:=\max_{i+jr\leq kr}\|\na^i\pl^ju\|_{\iy;\lda+i+j\mu}. 
\npb 
\eeq 
It is a consequence of Theorem~\ref{thm-H.Rk} below that 
$BC^{kr/\vec r,\vec\om}$ is a Banach space. 

\smallskip 
Similarly as in the isotropic case, anisotropic H\"older spaces can be 
characterized by means of local coordinates. For this we prepare the 
following analogue of Lemma~\ref{lem-R.S}. 
\begin{lemma}\LabelT{lem-H.S} 
Suppose 
\hb{k\in\BN} and  
\ \hb{s>0}. Then 
$$ 
S_{\wt{\ka}\ka}\in\cL(BC_{\wt{\ka}}^k,BC_\ka^k) 
\cap\cL(bc_{\wt{\ka}}^k,bc_\ka^k) 
\cap\cL(B_{\iy,\wt{\ka}}^s,B_{\iy,\ka}^s) 
\cap\cL(b_{\iy,\wt{\ka}}^s,b_{\iy,\ka}^s) 
\npb 
$$ 
and 
\hb{\|S_{\wt{\ka}\ka}\|\leq c} for 
\hb{\wt{\ka}\in\gN(\ka)} and 
\hb{\ka\in\gK}. 
\end{lemma} 
\begin{proof} 
As in the proof of Lemma~\ref{lem-R.S} we see that the statement applies 
for the spaces $BC^k$ and~$bc^k$. Now we get the remaining assertions by 
interpolation, due to Theorems \ref{thm-PH.B}(ii) and 
\ref{thm-PH.bcint}(i). 
\end{proof} 
\begin{theorem}\LabelT{thm-H.Rk} 
$\psi_\iy^{\vec\om}$~is a retraction from~$\ell_\iy(\mf{BC}^{kr/\vec r})$ 
onto~$BC^{kr/\vec r,\vec\om}$, and $\vp_\iy^{\vec\om}$~is a coretraction.  
\end{theorem} 
\begin{proof} 
(1) 
From \Eqref{RA.tLq} and Theorem~\ref{thm-H.Ri} we get 
$$ 
\|\vp_{\iy,\ka}^{\vec\om}u\|_{\iy;BC_\ka^{kr}} 
=\|\vp_{\iy,\ka}^\lda u\|_{\iy;BC_\ka^{kr}} 
\leq c\,\|u\|_{\iy;BC^{kr,\lda}} 
\qa \ka\slt\vp\in\gK\slt\Phi.  
$$  
Similarly, by invoking \Eqref{RA.tk} as well, 
$$ 
\|\pl^j\vp_{\iy,\ka}^{\vec\om}u\|_{\iy;BC_\ka^{(k-j)r}} 
=\|\vp_{\iy,\ka}^{\lda+jr}\pl^ju\|_{\iy;BC_\ka^{(k-j)r}} 
\leq c\,\|\pl^ju\|_{\iy;BC^{(k-j)r,\lda+jr}} 
$$  
for 
\hb{0\leq j\leq k}  and 
\hb{\ka\slt\vp\in\gK\slt\Phi}. From this and definition~\Eqref{H.krN} we 
infer 
$$ 
\|\vp_\iy^{\vec\om}u\|_{\ell_\iy(\mf{BC}^{kr/\vec r})} 
\leq c\,\|u\|_{kr/\vec r,\iy;\vec\om}.  
$$ 

\smallskip 
(2) 
Given 
\hb{\ka\in\gK} and 
\hb{\wt{\ka}\in\gN(\ka)}, 
\beq\Label{H.rr} 
\vp_{\iy,\ka}^\lda\circ\psi_{\iy,\wt{\ka}}^\lda 
=a_{\wt{\ka}\ka}S_{\wt{\ka}\ka} 
\eeq 
with 
$$ 
a_{\wt{\ka}\ka} 
:=(\rho_\ka/\rho_{\wt{\ka}})^\lda(\ka_*\pi_\ka) 
S_{\wt{\ka}\ka}(\wt{\ka}_*\pi_{\wt{\ka}}).  
$$ 
It is obvious that the scalar-valued \hbox{$BC^k$-spaces} form 
continuous multiplication algebras. Hence \Eqref{S.err}, \,\Eqref{R.LS}, 
and Lemma~\ref{lem-H.S} imply  
\beq\Label{H.a} 
\|a_{\wt{\ka}\ka}\|_{BC^{kr}}\leq c 
\qa \wt{\ka}\in\gN(\ka) 
\qa \ka\in\gK. 
\eeq 
Thus we deduce from \Eqref{H.rr}, \,\Eqref{H.a}, and Lemma~\ref{lem-H.S} 
that 
$$ 
\|\vp_{\iy,\ka}^\lda 
\circ\psi_{\iy,\wt{\ka}}^{\vec\om}v_{\wt{\ka}}\|_{\iy;BC_\ka^{kr}} 
=\|\vp_{\iy,\ka}^\lda 
\circ\psi_{\iy,\wt{\ka}}^\lda v_{\wt{\ka}}\|_{\iy;BC_\ka^{kr}} 
\leq c\,\|v_{\wt{\ka}}\|_{\iy;BC_{\wt{\ka}}^{kr}} 
$$   
for 
\hb{\wt{\ka}\in\gN(\ka)} and 
\hb{\ka\slt\vp,\,\wt{\ka}\slt\wt{\vp}\in\gK\slt\Phi}. 
By this and the finite multiplicity of~$\gK$ we obtain 
$$ 
\bal 
\|\vp_{\iy,\ka}^\lda\circ\psi_\iy^{\vec\om}\mf{v}\|_{\iy;BC_\ka^{kr}} 
& =\Big\|\sum_{\wt{\ka}\in\gN(\ka)}\vp_{\iy,\ka}^\lda 
 \circ\psi_{\iy,\wt{\ka}}^{\vec\om}v_{\wt{\ka}}\Big\|_{\iy;BC_\ka^{kr}}\\ 
&\leq c\,\max_{\wt{\ka}\in\gN(\ka)} 
 \|\vp_{\iy,\ka}^\lda 
 \circ\psi_{\iy,\wt{\ka}}^{\vec\om}v_{\wt{\ka}}\|_{\iy;BC_\ka^{kr}}
 \leq c\,\|\mf{v}\|_{\ell_\iy(B(J,\mf{BC}^{kr}))} 
\eal 
\npb 
$$  
for 
\hb{\ka\slt\vp\in\gK\slt\Phi}. 
   
\smallskip 
Note 
$$ 
\|\vp_{\iy,\ka}^{\lda+j\mu} 
\circ\pl^j  
\circ\psi_{\iy,\wt{\ka}}^{\vec\om}v_{\wt{\ka}}\|_{\iy;BC_\ka^{(k-j)r}} 
=\|\vp_{\iy,\ka}^{\lda+j\mu} 
\circ\psi_{\iy,\wt{\ka}}^{\lda+j\mu}(\pl^jv_{\wt{\ka}})\| 
_{\iy;BC_\ka^{(k-j)r}} 
$$ 
for 
\hb{0\leq j\leq k} and 
\hb{\ka\slt\vp,\,\wt{\ka}\slt\wt{\vp}\in\gK\slt\Phi}. Thus, as above, 
$$  
\|\vp_\iy^{\lda+j\mu}\circ\pl^j\circ\psi_\iy^{\vec\om}\mf{v}\| 
_{\ell_\iy(\mf{BC}^{(k-j)r})} 
\leq c\,\|\pl^j\mf{v}\|_{\ell_\iy(B(J,\mf{BC}^{(k-j)r}))} 
\qa 0\leq j\leq k.  
$$ 
Now we deduce from Corollary~\ref{cor-H.Ri}(i) 
$$ 
\|\psi_\iy^{\vec\om}\mf{v}\|_{kr/\vec r,\iy;\vec\om} 
\leq c\,\|\mf{v}\|_{\ell_\iy(\mf{BC}^{kr/\vec r})}. 
\npb 
$$ 
Since $\vp_\iy^{\vec\om}$~is a right inverse for~$\psi_\iy^{\vec\om}$ the 
theorem is proved. 
\end{proof} 
Next we introduce a linear subspace of $BC^{kr/\vec r,\vec\om}$ by 
$$ 
bc^{kr/\vec r,\vec\om}\text{ is the set of all $u$ in 
$BC^{kr/\vec r,\vec\om}$ with } 
\vp_\iy^{\vec\om}u\in\ell_{\iy,\unif}(\mf{bc}^{kr/\vec r}). 
$$ 
Due to the fact that $\ell_{\iy,\unif}(\mf{bc}^{kr/\vec r})$ is a closed 
linear subspace of $\ell_\iy(\mf{BC}^{kr/\vec r})$ it follows from the 
continuity of~$\vp_\iy^{\vec\om}$ that $bc^{kr/\vec r,\vec\om}$ is a closed 
linear subspace of $BC^{kr/\vec r,\vec\om}$. 

\smallskip 
The next theorem shows, in particular, that $bc^{kr/\vec r,\vec\om}$ is 
independent of the particular choice of 
\hb{\gK\slt\Phi} and the localization system used in the 
preceding definition. For this we set 
$$ 
BC^{\iy/\vec r,\vec\om} 
:={\textstyle\bigcap_k}BC^{kr/\vec r,\vec\om}, 
\npb 
$$ 
equipped with the natural projective topology.\po 
{\samepage 
\begin{theorem}\LabelT{thm-H.Rbck} 
\begin{itemize} 
\item[{\rm(i)}] 
$\psi_\iy^{\vec\om}$~is a retraction from 
$\ell_{\iy,\unif}(\mf{bc}^{kr/\vec r})$ onto $bc^{kr/\vec r,\vec\om}$, 
and $\vp_\iy^{\vec\om}$~is a coretraction. 
\item[{\rm(ii)}] 
 $bc^{kr/\vec r,\vec\om}$ is the closure of $BC^{\iy/\vec r,\vec\om}$ in 
 $BC^{kr/\vec r,\vec\om}$.\po 
\end{itemize} 
\end{theorem} }
\begin{proof} 
(1) 
Suppose 
\hb{\vp_\iy^{\vec\om}u=0} for some 
\hb{u\in BC^{kr/\vec r,\vec\om}}. Then it follows from \Eqref{RA.tab} that 
\hb{(\ka\slt\vp)_*(\pi_\ka u)=0} for 
\hb{\ka\slt\vp\in\gK\slt\Phi}. Hence 
\hb{\pi_\ka u=0} for 
\hb{\ka\in\gK}, and consequently 
\hb{\pi_\ka^2u=0} for 
\hb{\ka\in\gK}. This implies 
\hb{u=\sum_\ka\pi_\ka^2u=0}. Thus $\vp_\iy^{\vec\om}$~is injective. 

\smallskip 
(2) 
We denote by~$\cY$ the image space of $BC^{kr/\vec r,\vec\om}$ 
under~$\vp_\iy^{\vec\om}$. Theorem~\ref{thm-H.Rk} and 
\cite[Lemma~4.1.5]{Ama09a} imply 
\beq\Label{H.lS} 
\ell_\iy(\mf{BC}^{kr/\vec r}) 
=\cY\oplus\ker(\psi_\iy^{\vec\om}) 
\qa \psi_\iy^{\vec\om}\in\Lis(\cY,BC^{kr/\vec r,\vec\om}).  
\eeq 
Thus, by step~(1) (see Remarks~2.2.1 of~\cite{Ama09a}), 
$$ 
\vp_\iy^{\vec\om}\in\Lis(BC^{kr/\vec r,\vec\om},\cY) 
\qb (\vp_\iy^{\vec\om})^{-1}=\psi_\iy^{\vec\om}\sn\cY. 
$$  
Since 
\hb{\cX:=\cY\cap\ell_{\iy,\unif}(\mf{bc}^{kr/\vec r})} is a closed linear 
subspace of~$\cY$ we thus get 
\beq\Label{H.pfis} 
\vp_\iy^{\vec\om}\in\Lis(bc^{kr/\vec r,\vec\om},\cX)  
\qb (\vp_\iy^{\vec\om}\sn bc^{kr/\vec r,\vec\om})^{-1} 
=\psi_\iy^{\vec\om}\sn\cX. 
\eeq 
Due to \Eqref{H.lS} we can write 
\hb{\mf{w}\in\ell_{\iy,\unif}(\mf{bc}^{kr/\vec r})} in the form 
\hb{\mf{w}=\mf{u}+\mf{v}} with 
\hb{\mf{u}\in\cX} and 
\hb{\mf{v}\in\ker(\psi_\iy^{\vec\om})}. From this and \Eqref{H.pfis} it 
follows 
\hb{\psi_\iy^{\vec\om}\bigl(\ell_{\iy,\unif}(\mf{bc}^{kr/\vec r})\bigr) 
   \is bc^{kr/\vec r,\vec\om}}. Hence 
\hb{\psi_\iy^{\vec\om}\in\cL\bigl(\ell_{\iy,\unif}(\mf{bc}^{kr/\vec r}), 
   bc^{kr/\vec r,\vec\om}\bigr)} and 
\hb{\psi_\iy^{\vec\om}\circ\vp_\iy^{\vec\om}u=u} for 
\hb{u\in bc^{kr/\vec r,\vec\om}}. This proves~(i). 

\smallskip 
(3) 
Using obvious adaptions of the notations of the proof of 
Theorem~\ref{thm-H.Ri} we deduce from Lemma~\ref{lem-PH.fJ} 
\beq\Label{H.fJ} 
\wt{\mf{f}} 
\in\Lis\bigl(bc^{kr/\vec r}\bigl(\BM\times J,\ell_\iy(\mf{E})\bigr), 
\ell_{\iy,\unif}(\mf{bc}^{kr/\vec r})\bigr). 
\eeq 
We set 
$$ 
\Phi_\iy^{\vec\om}:=\wt{\mf{f}}\vph{f}^{-1}\circ\vp_\iy^{\vec\om} 
\qb \Psi_\iy^{\vec\om}:=\psi_\iy^{\vec\om}\circ\wt{\mf{f}}. 
$$  
Then we infer from~(i) and \Eqref{H.fJ} that 
\beq\Label{H.PsJ} 
\Psi_\iy^{\vec\om}\text{ is a retraction from 
$bc^{kr/\vec r}\bigl(\BM\times J,\ell_\iy(\mf{E})\bigr)$ onto 
$bc^{kr/\vec r,\vec\om}$, and $\Phi_\iy^{\vec\om}$ is a coretraction.} 
\eeq 
Definition~\Eqref{PH.defbc} guarantees 
$$ 
BC^\iy\bigl(\BM\times J,\ell_\iy(\mf{E})\bigr) 
\sdh bc^{kr/\vec r}\bigl(\BM\times J,\ell_\iy(\mf{E})\bigr). 
$$ 
It is an easy consequence of the mean value theorem that 
\hb{\ell_\iy(\mf{BC}^{(k+1)r/\vec r}) 
   \hr\ell_{\iy,\unif}(\mf{bc}^{kr/\vec r})}. From these embeddings, 
Theorem~\ref{thm-H.Rk}, and~(i) we infer that the first of the injections 
$$ 
BC^{(k+1)r/\vec r,\vec\om} 
\hr bc^{kr/\vec r,\vec\om} 
\hr BC^{kr/\vec r,\vec\om} 
$$ 
is valid. Thus 
$$ 
BC^{\iy/\vec r,\vec\om} 
={\textstyle\bigcap_k}BC^{kr/\vec r,\vec\om} 
={\textstyle\bigcap_k}bc^{kr/\vec r,\vec\om}. 
$$ 
Now it follows from \Eqref{PH.Binf} and \Eqref{H.PsJ} that 
$$ 
\Psi_\iy^{\vec\om}\text{ is a retraction from 
$BC^\iy\bigl(\BM\times J,\ell_\iy(\mf{E})\bigr)$ 
onto }BC^{\iy/\vec r,\vec\om}. 
\npb 
$$  
Assertion~(ii) is implied by \Eqref{H.PsJ} and~\cite[Lemma~4.1.6]{Ama09a}. 
\end{proof} 
We define the \hh{weighted anisotropic Besov-H\"older space scale} 
\hb{[\,B_\iy^{s/\vec r,\vec\om}\ ;\ s>0\,]} by 
\beq\Label{H.Bdef2} 
B_\iy^{s/\vec r,\vec\om}=B_\iy^{s/\vec r,\vec\om}\JV 
:=\left\{
\bal
{}      
&(bc^{kr/\vec r,\vec\om},bc^{(k+1)r/\vec r,\vec\om})_{(s-kr)/r,\iy},
    &\quad  kr<s &<(k+1)r,\\ 
&(bc^{kr/\vec r,\vec\om},bc^{(k+2)r/\vec r,\vec\om})_{1/2,\iy}, 
    &\quad     s &=(k+1)r. 
\eal 
\right. 
\eeq 
These spaces allow for a retraction-coretraction theorem as well which 
provides representations via local coordinates. 
\begin{theorem}\LabelT{thm-H.R} 
$\psi_\iy^{\vec\om}$~is a retraction from 
$\ell_\iy(\mf{B}_\iy^{s/\vec r})$ onto 
$B_\iy^{s/\vec r,\vec\om}$, and $\vp_\iy^{\vec\om}$ is a coretraction. 
\end{theorem} 
\begin{proof} 
We infer from \Eqref{H.PsJ}, Theorem~\ref{thm-PH.BJ}(ii), and 
definition~\Eqref{H.Bdef2} that 
\beq\Label{H.PsB} 
\Psi_\iy^{\vec\om}\text{ is a retraction from 
$B_\iy^{s/\vec r}\bigl(\BM\times J,\ell_\iy(\mf{E})\bigr)$ onto 
$B_\iy^{s/\vec r,\vec\om}$, and $\Phi_\iy^{\vec\om}$ is a coretraction.}
\eeq 
Thus the assertion follows from 
\beq\Label{H.phpsP} 
\vp_\iy^{\vec\om}=\wt{\mf{f}}\circ\Phi_\iy^{\vec\om} 
\qb \psi_\iy^{\vec\om}=\Psi_\iy^{\vec\om}\circ\wt{\mf{f}}\vph{f}^{-1}, 
\npb 
\eeq 
and Lemma~\ref{lem-PH.fJ}. 
\end{proof}\po  
{\samepage 
\begin{corollary}\LabelT{cor-H.R} 
\begin{itemize} 
\item[{\rm(i)}] 
Suppose 
\hb{k_0r<s<k_1r} with 
\hb{k_0,k_1\in\BN}. Then 
\hb{(bc^{k_0r/\vec r,\vec\om},bc^{k_1r/\vec r,\vec\om}) 
   _{(s-k_0)/(k_1-k_0),\iy} 
   \doteq B_\iy^{s/\vec r,\vec\om}}. 
\item[{\rm(ii)}] 
If\/ 
\hb{0<s_0<s_1} and\/ 
\hb{0<\ta<1}, then 
\hb{(B_\iy^{s_0/\vec r,\vec\om},B_\iy^{s_1/\vec r,\vec\om})_{\ta,\iy} 
   \doteq B_\iy^{s_\ta/\vec r,\vec\om} 
   \doteq[B_\iy^{s_0/\vec r,\vec\om},B_\iy^{s_1/\vec r,\vec\om}]_\ta}.\po 
\end{itemize} 
\end{corollary} }
\begin{proof} 
This is implied by \Eqref{H.PsJ}, \,\Eqref{H.PsB}, and 
Theorem~\ref{thm-PH.BJ}.
\end{proof} 
\hh{Weighted anisotropic H\"older spaces} are defined by setting 
\hb{BC^{s/\vec r,\vec\om}:=B_\iy^{s/\vec r,\vec\om}} for 
\hb{s\in\BR\ssm\BN}. Then we introduce \hh{weighted anisotropic little 
H\"older spaces} by  
$$ 
bc^{s/\vec r,\vec\om}=bc^{s/\vec r,\vec\om}\JV\text{ is the closure of 
$BC^{\iy/\vec r,\vec\om}$ in }BC^{s/\vec r,\vec\om} 
\npb 
$$ 
for 
\hb{s\geq0}. Note that this is consistent with 
Theorem~\ref{thm-H.Rbck}(ii). 

\smallskip 
Lastly, we get the \hh{weighted anisotropic little Besov-H\"older space 
scale} 
\hb{[\,b_\iy^{s/\vec r,\vec\om}\ ;\ s>0\,]} by 
\beq\Label{H.bdef} 
b_\iy^{s/\vec r,\vec\om}\text{ is the closure of 
$BC^{\iy/\vec r,\vec\om}$ in }B_\iy^{s/\vec r,\vec\om}. 
\eeq 
\begin{theorem}\LabelT{thm-H.Rbc} 
{\rm(i)} 
$\psi_\iy^{\vec\om}$~is a retraction from 
$\ell_{\iy,\unif}(\mf{b}_\iy^{s/\vec r})$ onto 
$b_\iy^{s/\vec r,\vec\om}$, and $\vp_\iy^{\vec\om}$ is a coretraction. 
 
\smallskip 
{\rm(ii)} 
Suppose 
\hb{k_0r<s<k_1r} with 
\hb{k_0,k_1\in\BN}. Then 
$$ 
(bc^{k_0r/\vec r,\vec\om},bc^{k_1r/\vec r,\vec\om}) 
_{(s/r-k_0)/(k_1-k_0),\iy}^0 
\doteq bc^{s/\vec r,\vec\om}. 
$$ 
 
\smallskip 
{\rm(iii)} 
If\/  
\hb{0<s_0<s_1} and 
\hb{0<\ta<1}, then 
$$ 
(bc^{s_0/\vec r,\vec\om},bc^{s_1/\vec r,\vec\om})_{\ta,\iy}^0 
\doteq bc^{s_\ta/\vec r,\vec\om} 
\doteq[bc^{s_0/\vec r,\vec\om},bc^{s_1/\vec r,\vec\om}]_\ta.   
$$ 
\end{theorem} 
\begin{proof} 
Assertion~(ii) and the first part of~(iii) follow from 
Corollary~\ref{cor-H.R}(i) and definition~\Eqref{H.bdef}. 
From~(ii), Theorem~\ref{thm-PH.bcintJ}(i), and \Eqref{H.PsJ} 
it follows that 
\beq\Label{H.Psbc} 
\Psi_\iy^{\vec\om}\text{ is a retraction from 
$bc^{s/\vec r}\bigl(\BM\times J,\ell_\iy(\mf{E})\bigr)$ onto 
$bc^{s/\vec r,\vec\om}$, and $\Phi_\iy^{\vec\om}$ is a coretraction.} 
\eeq 
Due to this the second part of~(iii) is now implied by 
Theorem~\ref{thm-PH.bcintJ}(ii). Statement~(i) is a consequence of 
\Eqref{H.Psbc}, \,\Eqref{H.phpsP}, and Lemma~\ref{lem-PH.fJ}. 
\end{proof} 
\section{Point-Wise Multipliers}\LabelT{sec-P}%
In connection with differential and pseudodifferential operators there occur 
naturally `products' of tensor fields possessing different regularity of 
the factors, so called `point-wise products' or `multiplications'. Although 
there is no problem in establishing mapping properties of differential 
operators say, if the coefficients are smooth, this is a much more difficult 
task if one is interested in operators with little regularity of the 
coefficients. Since such low regularity coefficients are of great importance 
in practice we derive in this and the next section point-wise multiplier 
theorems which are (almost) optimal. 

Let~$\cX_j$, 
\ \hb{j=0,1,2}, be Banach spaces. A~\emph{multiplication 
\hb{\cX_0\times\cX_1\ra\cX_2} from 
\hb{\cX_0\times\cX_1} into}~$\cX_2$ is an element of 
$\cL(\cX_0,\cX_1;\cX_2)$, the Banach space of continuous bilinear maps from 
\hb{\cX_0\times\cX_1} into~$\cX_2$. 

\smallskip 
Before considering multiplications in tensor bundles we first investigate 
point-wise products in Euclidean settings. Let 
\hb{E_i=(E_i,\vsdot_i)}, 
\ \hb{i=0,1,2}, be finite-dimensional Banach spaces, 
\hb{\BX\in\{\BR^m,\BH^m\}}, and 
\hb{\BY:=\BX\times J}.\po 
{\samepage 
\begin{theorem}\LabelT{thm-P.b} 
Suppose 
\hb{\sfb\in\cL(E_0,E_1;E_2)} and 
$$ 
\sfm\sco E_0^\BY\times E_1^\BY\ra E_2^\BY 
\qb (u_0,u_1)\mt\sfb(u_0,u_1) 
\npb 
$$ 
is its point-wise extension. Then\po  
\begin{itemize} 
\item[{\rm(i)}] 
${}$
\hb{\sfm\in\cL\bigl(\cB^{s/\vec r}\BYEn,\cB^{s/\vec r}\BYEe; 
\cB^{s/\vec r}\BYEz\bigr)} if either 
\hb{s\in r\BN} and 
\hb{\cB\in\{BC,bc\}}, or 
\hb{s>0}\\ 
and 
\hb{\cB\in\{B_\iy,b_\iy\}}.\po 
\item[{\rm(ii)}] 
${}$
\hb{\sfm\in\cL\bigl(BC^{s/\vec r}\BYEn,W_{\coW p}^{s/\vec r}\BYEe; 
W_{\coW p}^{s/\vec r}\BYEz\bigr)}, 
\ \hb{s\in r\BN}.\po  
\item[{\rm(iii)}] 
${}$
\hb{\sfm\in\cL\bigl(B_\iy^{s_0/\vec r}\BYEn,\gF_p^{s/\vec r}\BYEe; 
\gF_p^{s/\vec r}\BYEz\bigr)}, 
\ \hb{0<s<s_0}.\po 
\end{itemize} 
In either case the map 
\hb{\sfb\mt\sfm} is linear and continuous. 
\end{theorem} }
\begin{proof} 
(1) 
Assertion~(i) for 
\hb{s\in r\BN} and 
\hb{\cB\in\{BC,bc\}} as well as assertion~(ii) follow from the product 
rule. 

\smallskip 
(2)  
Suppose 
\hb{u_i\in E_i^\BY}, 
\ \hb{i=0,1}, and 
\hb{0<\ta<1}. Then 
\beq\Label{P.del} 
\tri_\xi\bigl(\sfm(u_0,u_1)\bigr) 
=\sfm\bigl(\tri_\xi u_0,u_1(\cdot+\xi)\bigr)+\sfm(u_0,\tri_\xi u_1) 
\qa \xi\in\BY. 
\eeq 
From this we infer, letting 
\hb{\xi=(h,0)} with 
\hb{h\in(0,\da)^m}, 
$$ 
\sup_t\bigl[\sfm(u_0,u_1)(\cdot,t)\bigr]_{\ta,\iy}^\da 
\leq c\bigl(\sup_t\bigl[u_0(\cdot,t)\bigr]_{\ta,\iy}^\da\,\|u_1\|_\iy 
 +\sup_t\bigl[u_1(\cdot,t)\bigr]_{\ta,\iy}^\da\,\|u_0\|_\iy\bigr)  
$$ 
for 
\hb{0<\da\leq\iy}. Similarly, 
$$ 
\sup_x\bigl[\sfm(u_0,u_1)(x,\cdot)\bigr]_{\ta,\iy}^\da 
\leq c\bigl(\sup_x\bigl[u_0(x,\cdot)\bigr]_{\ta,\iy}^\da\,\|u_1\|_\iy 
 +\sup_x\bigl[u_1(x,\cdot)\bigr]_{\ta,\iy}^\da\,\|u_0\|_\iy\bigr).  
$$ 
By step~(1), \,\Eqref{PH.Nsrst}, and \Eqref{PH.Nsr2s} we infer that (i)~is 
true if 
\hb{s\in\BR^+\ssm\BN}. Now we fill in the gaps 
\hb{s\in\BN} by means of Theorems \ref{thm-PH.BJ}(iii) 
and~\ref{thm-PH.bcintJ}(ii) and bilinear complex interpolation 
(cf.~J.~Bergh and J.~L\"ofstr\"om \cite[Theorem~4.4.1]{BeL76a}). This 
proves~(i) for 
\hb{s>0} and 
\hb{\cB\in\{B_\iy,b_\iy\}}. 

\smallskip 
(3) 
Assume 
\hb{s\in r\BN} and 
\hb{s_0>s}. By Theorem~\ref{thm-PH.bJ} 
\hb{B_\iy^{s_0/\vec r}\BYEn\hr b_\iy^{s/\vec r}\BYEn 
   \hr BC_\iy^{s/\vec r}\BYEn}. Hence we deduce from~(ii) 
$$ 
\sfm\in\cL\bigl(B_\iy^{s_0/\vec r}\BYEn,W_{\coW p}^{s/\vec r}\BYEe; 
W_{\coW p}^{s/\vec r}\BYEz\bigr) 
\qa s\in r\BN 
\qb s<s_0. 
$$ 
Using this, Theorem~\ref{thm-A.WHB}(iv), and once more bilinear complex 
interpolation we obtain 
$$ 
\sfm\in\cL\bigl(B_\iy^{s_0/\vec r}\BYEn,H_p^{s/\vec r}\BYEe; 
H_p^{s/\vec r}\BYEz\bigr) 
\qa 0<s<s_0. 
$$ 
   
\smallskip 
(4) 
We assume 
\hb{kr<s<(k+1)r} with 
\hb{k\in\BN}. It is well-known that 
$$ 
B_\iy^{s_0}\BXEn\times B_p^s\BXEe\ra B_p^s\BXEz 
\qb (v_0,v_1)\mt b(v_0,v_1) 
$$ 
is a multiplication (see Remark~4.2(b) in H.~Amann~\cite{Ama91c}, where 
$B_\iy^{s_0}$~is denoted by~$\BUC^{s_0}$, Th.~Runst and W.~Sickel 
\cite[Theorem~4.7.1]{RuS96a}, or V.G. Maz'ya and T.O. 
Shaposhnikova~\cite{MaS85a}, and H.~Triebel~\cite{Tri83a}), depending 
linearly and continuously on~$\sfb$. From this we infer 
\beq\Label{P.s0} 
\|\sfm(u_0,u_1)\|_{p;B_p^s\BXEz} 
\leq c\,\|u_0\|_{\iy;B_\iy^{s_0}\BXEn}\,\|u_1\|_{p;B_p^s\BXEe}. 
\eeq 
By the product rule and~(ii) 
\beq\Label{P.s1} 
\bal 
\big\|\pl^\ell\bigl(\sfm(u_0,u_1)\bigr)\big\|_{p;L_p\BXEz} 
&\leq c\sum_{j=0}^\ell\|\pl^ju_0\|_{\iy;B\BXEn} 
 \,\|\pl^{\ell-j}u_1\|_{p;L_p\BXEe}\\ 
&\leq c\,\|u_0\|_{k,\iy;B\BXEn}\,\|u_1\|_{k,p;L_p\BXEe}\\ 
&\leq c\,\|u_0\|_{s_0/r,\iy;B\BXEn}^{**}\,\|u_1\|_{s/r,p;L_p\BXEe} 
\eal 
\npb 
\eeq 
for 
\hb{0\leq\ell\leq k}. 

\smallskip 
We deduce from \Eqref{P.del} that, given 
\hb{\ta\in(0,1)}, 
$$ 
\bigl[\sfm(u_0,u_1)\bigr]_{\ta,p;L_p\BXEz} 
\leq c\bigl([u_0]_{\ta,\iy;B\BXEn}\,\|u_1\|_{p;L_p\BXEe} 
+\|u_0\|_{\iy;B\BXEn}\,[u_1]_{\ta,p;L_p\BXEe}\bigr). 
$$ 
Hence 
$$ 
\bal 
\bigl[\sfm(\pl^ju_0,\pl^{k-j}u_1)\bigr]_{(s-kr)/r,p;L_p\BXEz} 
&\leq c\,\|\pl^ju_0\|_{(s-kr)/r,\iy;B\BXEn}^*\,\|\pl^{k-j}u_1\| 
 _{(s-kr)/r,p;L_p\BXEz}^*\\ 
&\leq c\,\|u_0\|_{s_0/r,\iy;B\BXEn}^{**}\,\|u_1\|_{s/r,p;L_p\BXEz}^{**}, 
\eal 
$$ 
where we used
\hb{B_\iy^{s_0/r}\bigl(J,B\BXEn\bigr)\hr B_\iy^{s/r}\bigl(J,B\BXEn\bigr)} 
in the last estimate. Thus 
\beq\Label{P.s2} 
\bigl[\pl^k\bigl(\sfm(u_0,u_1)\bigr]_{(s-kr)/r,p;L_p\BXEz} 
\leq c\,\|u_0\|_{s_0/r,\iy;B\BXEn}^{**}\,\|u_1\|_{s/r,p;L_p\BXEe}^{**}. 
\eeq  
By Corollary~\ref{cor-N.B} 
$$ 
B_p^{s/\vec r}\BYEz 
\doteq L_p\bigl(J,B_p^s\BXEz\bigr) 
\cap B_p^{s/r}\bigl(J,L_p\BXEz\bigr). 
$$ 
Thus we infer from \Eqref{N.eq} and 
\hbox{\Eqref{P.s0}--\Eqref{P.s2}} 
that 
$$ 
\sfm\in\cL\bigl(B_\iy^{s_0/\vec r}\BYEn,B_p^{s/\vec r}\BYEe; 
B_p^{s/\vec r}\BYEz\bigr) 
\qa s\notin r\BN. 
$$ 
Now we fill in the gaps at 
\hb{s\in r\BN} once more by bilinear complex interpolation, which is 
possible due to Theorems \ref{thm-A.WHB}(iv) and~\ref{thm-PH.BJ}(iii). 

\smallskip 
Since the last part of the statement is obvious from the above 
considerations, the theorem is proved. 
\end{proof} 
It should be remarked that J.~Johnson~\cite{Joh95a} has undertaken a 
detailed study of point-wise multiplication in anisotropic Besov and 
Triebel-Lizorkin spaces on~$\BR^n$. However, it does not seem to be  
possible to derive Theorem~\ref{thm-P.b} from his results. 

\smallskip 
Next we extend the preceding theorem to $(x,t)$-dependent bilinear 
operators. 
\begin{theorem}\LabelT{thm-P.bt} 
Suppose 
\hb{\sfb\in\cL(E_0,E_1;E_2)^\BY} and set 
$$ 
\sfm\sco E_0^\BY\times E_1^\BY\ra E_2^\BY 
\qb (u_0,u_1)\mt\bigl((x,t)\mt\sfb(x,t)\bigl(u_0(x,t),u_1(x,t)\bigr)\bigr). 
$$ 
Then assertions \hbox{\rm(i)--(iii)} of the preceding theorem are valid in 
this case also, provided $\sfb$~possesses the same regularity as~$u_0$. 
\end{theorem} 
\begin{proof} 
Consider the multiplication 
$$ 
\sfb_0\sco\cL(E_0,E_1;E_2)\times E_0\ra\cL\EeEz 
\qb (\sfb,e_0)\mt\sfb(e_0,\cdot) 
$$ 
and let $\sfm_0$ be its point-wise extension. By applying 
Theorem~\ref{thm-P.b}(i) we obtain 
\beq\Label{P.m0} 
\sfm_0 
\in\cL\bigl(\cB^{s/\vec r}\bigl(\BY,\cL(E_0,E_1;E_2)\bigr), 
\cB^{s/\vec r}\BYEn;\cB^{s/\vec r}\bigl(\BY,\cL\EeEz\bigr)\bigr),  
\npb 
\eeq 
where either 
\hb{s\in r\BN} and 
\hb{\cB\in\{BC,bc\}}, or 
\hb{s>0} and 
\hb{\cB\in\{B_\iy,b_\iy\}}. 

\smallskip 
Next we introduce the multiplication 
$$ 
\cL\EeEz\times E_1\ra E_2 
\qb (A,e_1)\mt Ae_1 
$$ 
and its point-wise extension~$\sfm_1$. Then we infer from 
Theorem~\ref{thm-P.b} 
\beq\Label{P.m1} 
\sfm_1 
\in\cL\bigl(\cB^{s/\vec r}\bigl(\BY,\cL\EeEz\bigr), 
\cB^{s/\vec r}\BYEe;\cB^{s/\vec r}\BYEz\bigr),  
\npb 
\eeq 
if either 
\hb{s\in r\BN} and 
\hb{\cB\in\{BC,bc\}}, or 
\hb{s>0} and 
\hb{\cB\in\{B_\iy,b_\iy\}}, 
\beq\Label{P.m2} 
\sfm_1 
\in\cL\bigl(BC^{s/\vec r}\bigl(\BY,\cL\EeEz\bigr), 
W_{\coW p}^{s/\vec r}\BYEe;W_{\coW p}^{s/\vec r}\BYEz\bigr) 
\qa s\in r\BN,   
\eeq 
and 
\beq\Label{P.m3} 
\sfm_1 
\in\cL\bigl(B_\iy^{s_0/\vec r}\bigl(\BY,\cL\EeEz\bigr), 
\gF_p^{s/\vec r}\BYEe;\gF_p^{s/\vec r}\BYEz\bigr) 
\qa 0<s<s_0.    
\eeq 
Note 
$$ 
\sfm(u_0,u_1)=\sfm_1\bigl(\sfm_0(\sfb,u_0),u_1\bigr) 
\qa (u_0,u_1)\in E_0^\BY\times E_1^\BY. 
\npb 
$$ 
Thus the statement is a consequence of \hbox{\Eqref{P.m0}--\Eqref{P.m3}}. 
\end{proof} 
In order to study point-wise multiplications on manifolds we prepare a 
technical lemma which is a relative of Lemma~\ref{lem-H.S}. For this 
we set 
\beq\Label{P.RTS} 
T_{\wt{\ka}\ka}u(t):=u\bigl((\rho_\ka/\rho_{\wt{\ka}})^\mu t\bigr) 
\qb t\in J  
\qa R_{\wt{\ka}\ka}:=T_{\wt{\ka}\ka}\circ S_{\wt{\ka}\ka} 
\qb \ka,\wt{\ka}\in\gK. 
\eeq 
Note 
\beq\Label{P.thta}  
\Ta_{q,\ka}^\mu 
=(\rho_\ka/\rho_{\wt{\ka}})^{\mu/q}T_{\wt{\ka}\ka}\Ta_{q,\wt{\ka}}^\mu  
\qa \ka,\wt{\ka}\in\gK 
\qb 1\leq q\leq\iy. 
\eeq  
We also put 
$$ 
\wh{\vp}_{q,\ka}^{\kern1pt\vec\om} 
:=\rho_\ka^{\lda+m/q}\Ta_{q,\ka}^\mu (\ka\slt\vp)_*(\chi_\ka\cdot) 
\qa \ka\slt\vp\in\gK\slt\Phi. 
$$  
Then, using 
\hb{u=\sum_{\wt{\ka}}\pi_{\wt{\ka}}^2u}, 
\beq\Label{P.TS} 
\wh{\vp}_{q,\ka}^{\kern1pt\vec\om} 
=\sum_{\wt{\ka}\in\gN(\ka)}a_{\wt{\ka}\ka} 
R_{\wt{\ka}\ka}\vp_{q,\wt{\ka}}^{\vec\om}, 
\eeq 
where 
\beq\Label{P.a} 
a_{\wt{\ka}\ka} 
:=(\rho_\ka/\rho_{\wt{\ka}})^{\lda+(m+\mu)/q} 
\chi S_{\wt{\ka}\ka}(\wt{\ka}_*\pi_{\wt{\ka}}).  
\eeq 
Hence, given 
\hb{q\in[1,\iy]}, we deduce from \Eqref{S.err}, Lemma~\ref{lem-H.S}, and 
\Eqref{R.LS}(iii) that 
\beq\Label{P.abd} 
a_{\wt{\ka}\ka}\in BC^k(\BX_\ka) 
\qb \|a_{\wt{\ka}\ka}\|_{k,\iy}\leq c(k) 
\qa \wt{\ka}\in\gN(\ka) 
\qb \ka\in\gK 
\qb k\in\BN. 
\eeq 
\begin{lemma}\LabelT{lem-P.T} 
Suppose 
\hb{k\in\BN} and 
\hb{s>0}. Let 
\hb{\gG_\ka\in\{W_{\coW p,\ka}^{kr/\vec r},BC_\ka^{kr/\vec r}, 
   bc_\ka^{kr/\vec r},\gF_{p,\ka}^{s/\vec r},B_{\iy,\ka}^{s/\vec r}, 
   b_{\iy,\ka}^{s/\vec r}\}}. Then 
\beq\Label{P.Rkk} 
R_{\wt{\ka}\ka}\in\cL(\gG_{\wt{\ka}},\gG_\ka) 
\qb \|R_{\wt{\ka}\ka}\|\leq c 
\qa \wt{\ka}\in\gN(\ka) 
\qb \ka\slt\vp,\,\wt{\ka}\slt\wt{\vp}\in\gK\slt\Phi. 
\eeq 
\end{lemma} 
\begin{proof} 
It is immediate from \Eqref{RA.tk}, \,\Eqref{RA.tLq}, \,\Eqref{S.err}, and 
Lemma~\ref{lem-H.S} that 
$$ 
R_{\wt{\ka}\ka}\in\cL(W_{\coW p,\wt{\ka}}^{kr/\vec r}, 
W_{\coW p,\ka}^{kr/\vec r}) 
\cap\cL(BC_{\wt{\ka}}^{kr/\vec r},BC_\ka^{kr/\vec r}) 
\cap\cL(bc_{\wt{\ka}}^{kr/\vec r},bc_\ka^{kr/\vec r})  
\npb 
$$ 
and that the uniform estimates of \Eqref{P.Rkk} are satisfied. Now the 
remaining statements follow by interpolation. 
\end{proof} 
Assume 
\hb{V_{\coV j}=(V_{\coV j},h_j)}, 
\,\hb{j=0,1,2}, are metric vector bundles. By a \emph{bundle multiplication} 
from 
\hb{V_0\oplus V_1} into~$V_2$, denoted by 
$$ 
\sfm\sco V_0\oplus V_1\ra V_2 
\qb (v_0,v_1)\mt\sfm(v_0,v_1), 
$$ 
we mean a smooth section~$\gsm$ of 
\hb{\Hom(V_0\otimes V_1,V_2)} such that 
\hb{\sfm(v_0,v_1):=\gsm(v_0\otimes v_1)} and  
$$ 
|\sfm(v_0,v_1)|_{h_2}\leq c\,|v_0|_{h_0}\,|v_1|_{h_1} 
\qa v_i\in\Ga\MVi 
\qb i=0,1. 
$$ 
\begin{examples}\LabelT{exa-P.ex} 
\hh{(a)} 
The duality pairing 
$$ 
\pw_{V_1}\sco V_1^*\oplus V_1\ra M\times\BK 
\qb (v^*,v)\mt\dl v^*,v\dr_{V_1} 
\npb 
$$ 
is a bundle multiplication. 

\smallskip 
\hh{(b)}\quad 
Assume 
\hb{\sa_i,\tau_i\in\BN} for 
\hb{i=0,1}. Then the tensor product 
$$ 
\otimes\sco T_{\tau_0}^{\sa_0}M\oplus T_{\tau_1}^{\sa_1}M  
\ra T_{\tau_0+\tau_1}^{\sa_0+\sa_1}M 
\qb (a,b)\mt a\otimes b 
$$ 
is a bundle multiplication where 
\hb{(X^{\otimes\sa_0}\otimes X^{*\otimes\tau_0}) 
   \otimes(X^{\otimes\sa_1}\otimes X^{*\otimes\tau_1})
   :=X^{\otimes(\sa_0+\sa_1)}\otimes X^{*(\tau_0+\tau_1)}}, where we set 
\hb{X^{\otimes\sa}=X^1\otimes\cdots\otimes X^\sa} etc.  

\smallskip 
\hh{(c)}\quad 
Suppose 
\hb{1\leq i\leq\sa} and 
\hb{1\leq j\leq\tau}. We denote by 
\hb{\sC_j^i\sco T_\tau^\sa M\ra T_{\tau-1}^{\sa-1}M} the contraction with 
respect to positions $i$ and~$j$, defined by 
$$ 
\sC_j^i\Bigl(
\bigotimes_{k=1}^\sa X^k 
\otimes\bigotimes_{\ell=1}^\tau X_\ell^*\Bigr) 
:=\dl X_j^*,X^i\dr 
\bigotimes_{\dlim{k=1}{k\neq i}}^\sa X^k 
\otimes\bigotimes_{\dlim{\ell=1}{\ell\neq j}}^\tau X_\ell^* 
\qa X^k\in\Ga(M,TM) 
\qb X_\ell^*\in\Ga(M,T^*M). 
$$ 
It follows from (a) and~(b) that 
$$ 
\sC_j^i\sco T_{\tau_1}^{\sa_1}M\oplus T_{\tau_2}^{\sa_2}M 
\ra T_{\tau_1+\tau_2-1}^{\sa_1+\sa_2-1}M 
\qa (a,b)\mt\sC_j^i(a\otimes b) 
\npb 
$$ 
is a bundle multiplication, where 
\hb{1\leq i\leq\sa_1+\sa_2} and 
\hb{1\leq j\leq\tau_1+\tau_2}.  

\smallskip 
\hh{(d)}\quad 
Let 
\hb{W_{\coW j}=(W_{\coW j},h_{W_{\coW j}})}, 
\ \hb{j=0,1,2}, be metric vector bundles and 
\hb{\sa_j,\tau_j\in\BN}. Suppose 
$$ 
\sfw\sco W_0\oplus W_1\ra W_2 
\qb \sft\sco T_{\tau_0}^{\sa_0}M\oplus T_{\tau_1}^{\sa_1}M 
    \ra T_{\tau_2}^{\sa_2}M 
$$ 
are bundle multiplications. Set 
\hb{T_{\tau_j}^{\sa_j}(M,W_{\coW j}) 
   :=(T_{\tau_j}^{\sa_j}M\otimes W_{\coW j},h_j)} with 
\hb{h_j:=\pr_{\sa_j}^{\tau_j}\otimes h_{W_{\coW j}}}. Then 
$$ 
\sft\otimes\sfw 
\sco T_{\tau_0}^{\sa_0}(M,W_0)\oplus T_{\tau_1}^{\sa_1}(M,W_1) 
\ra T_{\tau_2}^{\sa_2}(M,W_2), 
\npb 
$$ 
defined by 
\hb{\sft\otimes\sfw (a_0\otimes u_0,\ a_1\otimes u_1)
   :=\sft(a_0,a_1)\otimes \sfw(u_0,u_1)}, is a bundle 
multiplication.\hfill$\Box$ 
\end{examples} 
Let $\sfm$ be a bundle multiplication from 
\hb{V_0\oplus V_1} into~$V_2$. Then 
$$ 
\Ga(M,V_0\oplus V_1)^J\ra\Ga(M,V_2)^J 
\qb \bigl(v_0(t),v_1(t)\bigr)\mt\sfm\bigl(v_0(t),v_1(t)\bigr) 
\qa t\in J, 
\npb 
$$ 
is the \emph{point-wise extension} of~$\sfm$, denoted by~$\sfm$ also. 

\smallskip 
After these preparations we can prove the following point-wise multiplier 
theorem which is the basis of the more specific results of the next section. 
\begin{theorem}\LabelT{thm-P.M} 
Let 
\hb{W_{\coW j}=(W_{\coW j},h_{W_{\coW j}},D_j)}, 
\ \hb{j=0,1,2}, be fully uniformly regular vector bundles over~$M$. Assume 
\hb{\sa_j,\tau_j\in\BN} satisfy 
\beq\Label{P.st} 
\sa_2-\tau_2=\sa_0+\sa_1-\tau_0-\tau_1. 
\eeq 
Set 
$$ 
V_{\coV j}=(V_{\coV j},h_j,\na_{\cona j})  
:=\bigl(T_{\tau_j}^{\sa_j}(M,W_{\coW j}), 
\pr_{\sa_j}^{\tau_j}\otimes h_{W_{\coW j}},\na(\na_{\cona g},D_j)\bigr) 
\po 
$$ 
and suppose 
\hb{\sfm\sco V_0\oplus V_1\ra V_2} is a bundle multiplication, 
\hb{\lda_0,\lda_1\in\BR}, 
\ \hb{\lda_2:=\lda_0+\lda_1}, and 
\hb{\vec\om_j:=(\lda_j,\mu)}. Then 
{\samepage 
\begin{itemize} 
\item[{\rm(i)}] 
${}$
\hb{\sfm\in\cL\bigl(\cB^{s/\vec r,\vec\om_0}\JVn, 
\cB^{s/\vec r,\vec\om_1}\JVe;\cB^{s/\vec r,\vec\om_2}\JVz\bigr)}, where 
either 
\hb{s\in r\BN} and 
\hb{\cB\in\{BC,bc\}}, or 
\hb{s>0} and 
\hb{\cB\in\{B_\iy,b_\iy\}}.\po 
\item[{\rm(ii)}] 
${}$
\hb{\sfm\in\cL\bigl(BC^{s/\vec r,\vec\om_0}\JVn, 
W_{\coW p}^{s/\vec r,\vec\om_1}\JVe; 
W_{\coW p}^{s/\vec r,\vec\om_2}\JVz\bigr)}, 
\ \hb{s\in r\BN}. 
\item[{\rm(iii)}] 
${}$
\hb{\sfm\in\cL\bigl(B_\iy^{s_0/\vec r,\vec\om_0}\JVn, 
\gF_p^{s/\vec r,\vec\om_1}\JVe; 
\gF_p^{s/\vec r,\vec\om_2}\JVz\bigr)}, 
\ \hb{0<s<s_0}.\po  
\end{itemize} } 
\end{theorem}
\begin{proof} 
(1) 
Suppose 
$$  
M=\BX\in\{\BR^m,\BH^m\} 
\qb g=g_m 
\qb \rho\sim\mf{1} 
\qb W_{\coW j}=\bigl(M\times F_j,\pr_{F_j},d_{W_{\coW j}}\bigr), 
$$ 
where $d_{W_{\coW j}}$~is the \hbox{$F_j$-valued} differential. Set 
$$ 
E_j:=\bigl(E_{\tau_j}^{\sa_j}(F_j),\prsn_{HS}\bigr) 
\qb V_{\coV j}=\bigl(\BX\times E_j,\pr_j,d_{E_j}\bigr), 
\npb 
$$ 
where 
\hb{\prsn_j:=\prsn_{E_j}}. 

\smallskip 
Introducing bases, we define isomorphisms 
\hb{E_j\simeq\BK^{N_j}}. By means of them $\sfm$~is transported onto 
an element of $\cL(\BK^{N_0},\BK^{N_1};\BK^{N_2})^\BX$ which has the 
`matrix representation'
$$  
\BK^{N_0}\times\BK^{N_1}\ni(\xi,\eta) 
\mt(\sfm_{\nu_0\nu_1}^{\nu_2}(x)\xi^{\nu_0}\eta^{\nu_1}) 
_{1\leq\nu_2\leq N_2}\in\BK^{N_2}. 
\npb 
$$ 
Assume 
\hb{\sfm\in BC^\iy\bigl(\BX,\cL(E_0,E_1;E_2)\bigr)}. Then the assertion 
follows from Theorem~\ref{thm-P.bt}. 

\smallskip 
(2) 
Now we consider the general case. We choose uniformly regular atlases 
\hb{\gK\slt\Phi_j} for~$W_{\coW j}$ over~$\gK$ with model fiber~$F_j$. 
Given 
\hb{\ka\slt\vp_j\in\gK\slt\Phi_j} we define, recalling \Eqref{L.EstF}, 
\ \hb{\sfm_\ka\in\cD\bigl(\BX_\ka,\cL(E_0,E_1;E_2)\bigr)} by 
$$ 
\sfm_\ka(\eta_0,\eta_1) 
:=(\ka\slt\vp_2)_*\bigl(\chi_\ka\sfm 
\bigl((\ka\slt\vp_0)^*\eta_0,(\ka\slt\vp_1)^*\eta_1\bigr)\bigr) 
$$ 
for 
\hb{\eta_j\in E_j^{\BX_\ka}}. It follows from \Eqref{L.Nh} and the fact 
that $\sfm$~is a bundle multiplication that 
$$ 
|\sfm_\ka(\eta_0,\eta_1)|_2 
\leq c\,\rho_\ka^{\tau_2-\sa_2}\rho_\ka^{\sa_0-\tau_0}
 \rho_\ka^{\sa_1-\tau_1}\,|\eta_0|_0\,|\eta_1|_1 
\qa \eta_j\in E_j^{\BX_\ka}. 
$$ 
Hence we infer from \Eqref{P.st} 
$$ 
\sfm_\ka\in BC^k\bigl(\BX_\ka,\cL(E_0,E_1;E_2)\bigr) 
\qb \|\sfm_\ka\|_{k,\iy}\leq c(k) 
\qa \ka\slt\vp_j\in\gK\slt\Phi_j 
\qb k\in\BN. 
$$ 

\smallskip 
(3) 
In the following, it is understood that $\vp_q^{\vec\om_j}$~is defined by 
means of 
\hb{\gK\slt\Phi_j} for
\hb{1\leq q\leq\iy}. Then, given 
\hb{v_j\in\Ga\MVj^J},  
$$ 
\vp_{q,\ka}^{\vec\om_2}\bigl(\sfm(v_0,v_1)\bigr) 
=\rho_\ka^{\lda_0}\rho_\ka^{\lda_1+m/q}\Ta_{q,\ka}^\mu 
 (\ka\slt\vp_2)_*\bigl(\pi_\ka\sfm(v_0,v_1)\bigr) 
=\sfm_\ka(\vp_{\iy,\ka}^{\vec\om_0}v_0, 
 \wh{\vp}_{q,\ka}^{\kern1pt\vec\om_1}v_1). 
$$ 
Consequently, we get from \Eqref{P.TS} 
\beq\Label{P.phk} 
\vp_{q,\ka}^{\vec\om_2}\bigl(\sfm(v_0,v_1)\bigr) 
=\sum_{\wt{\ka}\in\gN(\ka)}a_{\wt{\ka}\ka}\sfm_\ka
 (\vp_{\iy,\ka}^{\vec\om_0}v_0,R_{\wt{\ka}\ka}\vp_{q,\ka}^{\vec\om_1}v_1). 
\eeq 

(4) 
Suppose 
either 
\hb{s\in\BN} and 
\hb{\cB\in\{BC,bc\}}, or 
\hb{s>0} and 
\hb{\cB\in\{B_\iy,b_\iy\}}. Then we infer from \Eqref{P.abd}, 
Lemma~\ref{lem-P.T}, Theorem~\ref{thm-P.bt}, and steps (1) and~(3) that 
$$ 
\|a_{\wt{\ka}\ka}\sfm_\ka(\eta_0,\eta_1)\|_{\cB^{s/\vec r}\BYkEz} 
\leq c\,\|\eta_0\|_{\cB^{s/\vec r}\BYkEn} 
\,\|\eta_1\|_{\cB^{s/\vec r}\BYkEe}, 
$$ 
uniformly with respect to 
\hb{\ka\slt\vp,\,\wt{\ka}\slt\wt{\vp}\in\gK\slt\Phi_2}. Hence we get from 
\Eqref{P.phk} and the finite multiplicity of~$\gK$ 
\beq\Label{P.linf}  
\big\|\vp_\iy^{\vec\om_2}\bigl(\sfm(v_0,v_1)\bigr)\big\| 
 _{\ell_\iy(\mf{\cB}^{s/\vec r}\BYEz)} 
\leq c\|v_0\|_{\ell_\iy(\mf{\cB}^{s/\vec r}\BYEn)} 
\,\|v_1\|_{\ell_\iy(\mf{\cB}^{s/\vec r}\BYEe)}.  
\eeq  
Thus Theorems \ref{thm-H.Rk} and~\ref{thm-H.R} imply, due to \Eqref{R.rN}, 
$$ 
\|\sfm(v_0,v_1)\| _{\cB^{s/\vec r,\vec\om_2}\JVz} 
\leq c\,\|v_0\|_{\cB^{s/\vec r,\vec\om_0}\JVn} 
\,\|v_1\|_{\cB^{s/\vec r,\vec\om_1}\JVe},  
\npb 
$$ 
provided either 
\hb{s\in r\BN} and 
\hb{\cB=BC}, or 
\hb{s>0} and 
\hb{\cB=B_\iy}. 

\smallskip 
If 
\hb{s\in r\BN} and 
\hb{\cB=bc}, or 
\hb{s>0} and 
\hb{\cB=b_\iy}, then \Eqref{P.linf} holds with $\ell_\iy$~replaced 
by~$\ell_{\iy,\unif}$ everywhere. Thanks to Theorems \ref{thm-H.Rbck}(i) 
and~\ref{thm-H.Rbc}(i) this proves assertion~(i). The proofs for 
(ii) and (iii) are similar. 
\end{proof} 
It is clear that obvious analogues of the results of this section hold in 
the case of time-independent isotropic spaces. This generalizes and 
improves \cite[Theorem~9.2]{Ama12b}. 
\section{Contractions}\LabelT{sec-C}%
In practice, most pointwise multiplications in tensor bundles occur through 
contractions of tensor fields. For this reason we specialize in this section 
the general multiplier Theorem~\ref{thm-P.M} to this setting and study the 
problem of right invertibility of multiplier operators induced by 
contraction. 

\smallskip 
Let 
\hb{V_{\coV i}=(V_{\coV i},h_i)}, 
\ \hb{i=1,2}, be uniformly regular metric vector bundles of rank~$n_i$ 
over~$M$ with model fiber~$E_i$. Set 
\hb{V_0=(V_0,h_0):=\bigl(\Hom\VeVz,h_{12}\bigr)}. By 
Example~\ref{exa-U.ex}(f), \ $V_0$~is a uniformly regular vector bundle of 
rank~$n_1n_2$ over~$M$ with model fiber $\cL\EeEz$. The 
\emph{evaluation map} 
$$ 
\ev\sco\Ga(M,V_0\otimes V_1)\ra\Ga\MVz 
\qb (a,v)\mt av 
\npb 
$$ 
is defined by 
\hb{av(p):=a(p)v(p)} for 
\hb{p\in M}. 
\begin{lemma}\LabelT{lem-C.E} 
The evaluation map is a bundle multiplication.
\end{lemma} 
\begin{proof} 
We fix uniformly regular atlases 
\hb{\gK\slt\Phi_i}, 
\ \hb{i=1,2}, for~$V_{\coV i}$ over~$\gK$. Then, using the notation of 
Section~\ref{sec-V}, it follows from \Eqref{V.Aad} 
$$ 
(a^*a)_{\wt{\nu}_1}^{\nu_1} 
=h_1^{*\nu_1\wh{\nu}_1} 
\,\ol{a_{\wh{\nu}_1}^{\wh{\nu}_2}} 
\,\,\ol{h_{2,\wh{\nu}_2\wt{\nu}_2}} 
\,a_{\wt{\nu}_1}^{\wt{\nu}_2}. 
$$ 
Hence we infer from \Eqref{U.h} 
$$ 
\ka_*(|a|_{h_0}^2) 
=\ka_*\bigl(\tr(a^*a)\bigr) 
=\ka_*h_1^{*\nu_1\wh{\nu}_1} 
\,\ka_*a_{\nu_1}^{\wt{\nu}_2} 
\,\ka_*h_{2,\wt{\nu}_2\wh{\nu}_2} 
\,\ol{\ka_*a_{\wh{\nu}_1}^{\wh{\nu}_2}}  
\sim\sum_{\nu_1,\nu_2}|\ka_*a_{\nu_1}^{\nu_2}|^2 
=\tr\bigl([\ka_*a]^*[\ka_*a]\bigr), 
$$   
uniformly with respect to 
\hb{\ka\in\gK}. Furthermore, \Eqref{V.pN} and \Eqref{U.h} imply 
\beq\Label{C.kLb} 
\bal 
\ka_*(|au|_{h_2}) 
&=\big|\bigl((\ka\slt\vp_{12})_*a\bigr)(\ka\slt\vp_1)_*u\big| 
 _{(\ka\slt\vp_2)_*h_2}\\ 
&\sim\big|\bigl((\ka\slt\vp_{12})_*a\bigr)(\ka\slt\vp_1)_*u\big|_{E_2} 
 \leq|(\ka\slt\vp_{12})_*a|_{\cL\EeEz}\,|(\ka\slt\vp_1)_*u|_{E_1} 
\eal 
\eeq 
for 
\hb{u\in\Ga\MVe} and 
\hb{\ka\slt\vp_i\in\gK\slt\Phi_i}. Since $\cL\EeEz$ is finite-dimensional 
the operator norm~%
\hb{\vsdot_{\cL\EeEz}} is equivalent to the trace norm. Hence, using 
\hb{\cL\EeEz\simeq\BK^{n_2\times n_1}} and \Eqref{V.pN} and \Eqref{U.h} once 
more, we deduce from \Eqref{C.kLb} that 
\hb{\ka_*(|au|_{h_2})\leq c\ka_*(|a|_{h_0})\ka_*(|u|_{h_1})} for 
\hb{\ka\in\gK}. Consequently, 
$$ 
|au|_{h_2}\leq c\,|a|_{h_0}\,|u|_{h_1} 
\qa (a,u)\in\Ga(M,V_0\oplus V_1). 
\npb 
$$ 
This proves the lemma. 
\end{proof} 
Suppose 
\hb{\sa,\sa_i,\tau,\tau_i\in\BN} for 
\hb{i=1,2} with 
\hb{\sa+\tau>0}. We define the \emph{center contraction} of order 
\hb{\sa+\tau}, 
\beq\Label{C.CC} 
\sC=\sC_{[\tau]}^{[\sa]} 
\sco\Ga(M,T_{\tau_2+\sa}^{\sa_2+\tau}M\oplus T_{\tau+\tau_1}^{\sa+\sa_1}M) 
\ra\Ga(M,T_{\tau_1+\tau_2}^{\sa_1+\sa_2}M), 
\eeq 
as follows: Given 
\hb{(i_k)\in\BJ_{\sa_k}}, 
\ \hb{(j_k)\in\BJ_{\tau_k}} for 
\hb{k=1,2}, and 
\hb{\sa\in\BJ_\sa}, 
\ \hb{\tau\in\BJ_\tau} we set 
$$ 
(i_2;j):=(i_{2,1},\ldots,i_{2,\sa_2},j_1,\ldots,j_\tau) 
\in\BJ_{\sa_2+\tau} 
$$ 
etc. Assume 
\hb{a\in\Ga(M,T_{\tau_2+\sa}^{\sa_2+\tau}M)} is locally represented 
on~$U_{\coU\ka}$ by 
$$ 
a=a_{(j_2;i)}^{(i_2;j)}\,\frac\pl{\pl x^{(i_2)}} 
\otimes\frac\pl{\pl x^{(j)}} 
\otimes dx^{(j_2)}\otimes dx^{(i)} 
$$ 
and $b$~has a corresponding representation. Then the local representation 
of~$\sC(a,b)$ on~$U_{\coU\ka}$ is given by 
$$ 
a_{(j_2;i)}^{(i_2;j)}\,b_{(j;j_1)}^{(i;i_1)}\,\frac\pl{\pl x^{(i_2)}} 
\otimes\frac\pl{\pl x^{(i_1)}} 
\otimes dx^{(j_2)}\otimes dx^{(j_1)}. 
$$ 
A~center contraction~\Eqref{C.CC} is~a \emph{complete contraction} (on the 
right) if 
\hb{\sa_1=\tau_1=0}. If $\sC$~is a complete contraction, then we usually 
simply write 
\hb{a\cdot u} for~$\sC(a,u)$. 
\begin{lemma}\LabelT{lem-C.CE} 
The \emph{center contraction} associated with the evaluation map~$\ev$, 
$$ 
\sC\otimes\ev 
\sco\Ga\bigl(M,T_{\tau_2+\sa}^{\sa_2+\tau}\MVn 
\oplus T_{\tau+\tau_1}^{\sa+\sa_1}\MVe\bigr) 
\ra\Ga\bigl(M,T_{\tau_1+\tau_2}^{\sa_1+\sa_2}\MVz\bigr), 
\npb 
$$ 
is a bundle multiplication. 
\end{lemma} 
\begin{proof} 
Note that $\sC$~is a composition of 
\hb{\sa+\tau} simple contractions of type~$\sC_j^i$. Hence the assertion 
follows from Lemma~\ref{lem-C.E} and Examples~\ref{exa-P.ex}(c) and~(d). 
\end{proof} 
Henceforth, we write again~$\sC$ for 
\hb{\sC\otimes\ev}, if no confusion seems likely. Furthermore, we use the 
same symbol for point-wise extensions to time-dependent tensor fields. 
In addition, we do not indicate notationally the tensor bundles on which 
$\sC$~is operating. This will always be clear from the context. 

\smallskip 
Throughout the rest of this section we presuppose 
$$ 
\bal 
\bt\quad 
&W_{\coW i}=(W_{\coW i},h_i,D_i),\ \ i=1,2,3,
 \text{ are fully uniformly regular vector bundles}\\ 
\noalign{\vskip-.5\jot}  
&\text{of rank $n_i$ over $M$ with model fiber }F_i.  
\eal 
$$ 
For 
\hb{i,j\in\{1,2,3\}} we set 
$$ 
W_{\coW ij}=(W_{\coW ij},h_{W_{\coW ij}},D_{ij}) 
:=\bigl(\Hom\WiWj,\prsn_{HS},\na\DiDj\bigr). 
\npb 
$$ 
Example~\ref{exa-U.ex}(f) guarantees that $W_{\coW ij}$~is a fully uniformly 
regular vector bundle over~$M$. 

\smallskip 
We also assume for 
\hb{i,j\in\{1,2,3\}} 
$$ 
\bal 
\bt\quad 
&\sa_i,\tau_i,\sa_{ij},\tau_{ij}\in\BN;\\ 
\bt\quad 
&V_{\coV i}=(V_{\coV i},h_i,\na_{\cona i}) 
 :=\bigl(T_{\tau_i}^{\sa_i}\MWi, 
 \prsn_{\sa_i}^{\tau_i}\otimes h_{W_{\coW i}}, 
 \na(\na_{\cona g},D_i)\bigr);\\ 
\bt\quad 
&V_{\coV ij}=(V_{\coV ij},h_{ij},\na_{\cona ij}) 
 :=\bigl(T_{\tau_{ij}}^{\sa_{ij}}\MWij, 
 \prsn_{\sa_{ij}}^{\tau_{ij}}\otimes h_{W_{\coW ij}}, 
 \na(\na_{\cona g},D_{ij})\bigr);\\ 
\bt\quad 
&\lda_i,\lda_{ij}\in\BR,\ \ \vec\om_i=(\lda_i,\mu),  
 \ \ \vec\om_{ij}=(\lda_{ij},\mu). 
\eal 
$$ 

\smallskip 
Due to Lemma~\ref{lem-C.CE} we can apply Theorem~\ref{thm-P.M} and its 
corollary with 
\hb{\sfm=\sC}. For simplicity and for their importance in the theory of 
differential and pseudodifferential operators, \emph{we restrict ourselves 
in the following to complete contractions}. It should be observed that 
condition~\Eqref{C.su} below is void if 
\hb{\pl M=\es} and 
\hb{J=\BR}.\po 
{\samepage 
\begin{theorem}\LabelT{thm-C.CC} 
\begin{itemize} 
\item[{\rm(i)}] 
Suppose 
\beq\Label{C.l12} 
\lda_2=\lda_{12}+\lda_1 
\qb \sa_2=\sa_{12}-\tau_1 
\qb \tau_2=\tau_{12}-\sa_1, 
\po 
\eeq 
and 
\beq\Label{C.su} 
s> 
\left\{
{}      
\bal
-1+1/p  &     &&\quad \text{if \ }\pl M\neq\es,\\
r(-1+1/p&)    &&\quad \text{if \ }\pl M=\es\text{ and }J=\BR^+. 
\eal
\right. 
\po 
\eeq 
Let one of the following additional conditions be satisfied: 
$$ 
\bal
\text{\rm($\al$)}  \qquad  &\ph{|}s\ph{|}=t\in r\BN, \ \  
        &&q:=\iy,           &&\ \ \cB=\gG\in\{BC,bc\}\text{;}\\
\text{\rm($\ba$)}  \qquad  &\ph{|}s\ph{|}=t\in r\BN, \ \   
        &&q:=p,             &&\ \ \cB=BC, \ \gG=W\text{;}\\
\text{\rm($\ga$)}  \qquad  &\ph{|}s\ph{|}=t>0, \ \ 
        &&q:=\iy,           &&\ \ \cB=\gG\in\{B_\iy,bc_\iy\}\text{;}\\
\text{\rm($\da$)}  \qquad  &|s|<t, \ \ 
        &&q:=p,             &&\ \ \cB=B_\iy, \ \gG=\gF\text{.}
\eal 
\po 
$$ 
Assume 
\hb{a\in\cB^{t/\vec r,\vec\om_{12}}\JVez}. Then 
$$ 
A:=(u\mt a\cdot u) 
\in\cL\bigl(\gG_q^{s/\vec r,\vec\om_1}\JVe, 
\gG_q^{s/\vec r,\vec\om_2}\JVz\bigr),  
\npb 
$$ 
where 
\hb{BC_\iy:=BC} and 
\hb{bc_\iy:=bc} if\/ \text{\rm($\al$)}~applies. The map 
\hb{a\mt A} is linear and continuous.\po 
\item[{\rm(ii)}] 
Assume, in addition, 
$$ 
\lda_3=\lda_{23}+\lda_2 
\qb \sa_3=\sa_{23}-\tau_2 
\qb \tau_3=\tau_{23}-\sa_2 
\po   
$$ 
and 
\hb{b\in\cB^{t/\vec r,\vec\om_{23}}\JVzd}. Set 
\hb{B:=(v\mt b\cdot v)}. Then 
$$ 
BA=\Bigl(u\mt\sC_{[\tau_2]}^{[\sa_2]}(b,a)\cdot u\Bigr) 
\in\cL\bigl(\gG_q^{s/\vec r,\vec\om_1}\JVe, 
\gG_q^{s/\vec r,\vec\om_3}\JVd\bigr). 
\po 
$$
\end{itemize} 
\end{theorem} }
\begin{proof} 
(1) 
Suppose 
\hb{s\geq0} with 
\hb{s>0} if 
\hb{\gF=B}. Then, due to Lemma~\ref{lem-C.CE}, assertion~(i) is immediate 
from Theorem~\ref{thm-P.M}. 

\smallskip 
(2) 
Choose uniformly regular atlases 
\hb{\gK\slt\Phi_i}, 
\ \hb{i=1,2}, for~$W_{\coW i}$ over~$\gK$. Let 
\beq\Label{C.aloc1} 
a=a_{(j_{12}),\nu_1}^{(i_{12}),\nu_2}(t) 
\,\frac\pl{\pl x^{(i_{12})}} 
\otimes dx^{(j_{12})} 
\otimes b_{\nu_2}^2 
\otimes\ba_1^{\nu_1} 
\qa t\in J, 
\eeq 
be the local representation of~$a$ in the local coordinate frame 
for~$V_{12}$ over~$U_{\coU\ka}$ associated with 
\hb{\ka\slt\vp_{12}\in\gK\slt\Phi_{12}}, where 
\hb{(b_1^i,\ldots,b_{n_i}^i)} is the local coordinate frame 
for~$W_{\coW i}$  over~$U_{\coU\ka}$ associated with 
\hb{\ka\slt\Phi_i}, and 
\hb{(\ba_i^1,\ldots,\ba_i^{n_i})} is its dual frame 
(cf.~Example~\ref{exa-V.ex}(b) and \Eqref{L.bas}). Write 
\hb{(i_{k\ell})=(i_\ell;j_k)\in\BJ_{\sa_{k\ell}}} and 
\hb{(j_{k\ell})=(j_\ell;i_k)\in\BJ_{\tau_{k\ell}}} for 
\hb{k,\ell\in\{1,2\}} with 
\hb{k\neq\ell}, where 
\hb{(i_k)\in\BJ_{\sa_k}} and 
\hb{(j_k)\in\BJ_{\tau_k}}. 

\smallskip 
We define 
\hb{a'\in\Ga 
   \bigl(M,T_{\tau_{12}}^{\sa_{12}}\bigl(M,\Hom(W_2',W_1')\bigr)\bigr)^J} by 
$$ 
a'\vph{a}_{(i_1;j_2),\nu_1}^{(j_1;i_2),\nu_2}(t) 
\,\frac\pl{\pl x^{(j_1)}} 
\otimes\frac\pl{\pl x^{(i_2)}} 
\otimes dx^{(i_1)}\otimes dx^{(j_2)} 
\otimes\ba_1^{\nu_1} 
\otimes b_{\nu_2}^2 
\qa t\in J, 
\npb 
$$ 
where 
\hb{a'\vph{a}_{(i_1;j_2),\nu_1}^{(j_1;i_2),\nu_2}
   :=a_{(j_2;i_1),\nu_1}^{(i_2;j_1),\nu_2}}. 

\smallskip 
\smallskip 
It is obvious that 
\beq\Label{C.adu1} 
a'\in\cB^{t/\vec r,\vec\om_{12}}
   \bigl(J,T_{\tau_{12}}^{\sa_{12}}\bigl(M,\Hom(W_2',W_1')\bigr)\bigr), 
\eeq 
the map 
\hb{a\mt a'} is linear and continuous, and \hb{(a')'=a}. Furthermore, 
since 
\hb{V_{\coV i}'=T_{\sa_i}^{\tau_i}(M,W_{\coW i}')}, 
\beq\Label{C.vau} 
\dl v,a\cdot u\dr_{V_2}=\dl a'\cdot v,u\dr_{V_1} 
\qa (v,u)\in\Ga(M,V_2'\oplus V_1)^J. 
\eeq 

\smallskip 
(3) 
Suppose condition~($\da$) is satisfied and 
\hb{s<0}. It follows from \Eqref{C.l12}, step~(1), and \Eqref{C.adu1} 
\beq\Label{C.Cad} 
\sC(a'):=(v\mt a'\cdot v)  
\in\cL\bigl(\gF_{p'}^{-s/\vec r,-\vec\om_2}\JVzs, 
\gF_{p'}^{-s/\vec r,-\vec\om_1}\JVes\bigr). 
\eeq 
From Theorem~\ref{thm-A.HBH}(ii) and assumption~\Eqref{C.su} we infer 
$$ 
\gF_{p'}^{-s/\vec r,-\vec\om_i}\JVis 
=\ci\gF_{p'}^{-s/\vec r,-\vec\om_i}\JVis 
\qa i=1,2. 
$$ 
Thus we deduce from \Eqref{A.def1}, \,\Eqref{C.vau}, and \Eqref{C.Cad} that 
\hb{\sC(a')=\sC(a)'}. Hence, using 
\hb{(a')'=a}, we get the remaining part of assertion~(i), provided 
\hb{s\neq0} if 
\hb{\gF=B}. Now this gap is closed by interpolation. 

\smallskip 
(4) 
It is clear that 
\hb{\sC(b)\sC(a)\sco T_{\tau_1}^{\sa_1}\MVe\ra T_{\tau_3}^{\sa_3}\MVd} is 
given by 
\beq\Label{C.Cba} 
v\mt\sC\bigl(b,\sC(a,v)\bigr) 
=\sC\Bigl(\sC_{[\tau_2]}^{[\sa_2]}(b,a),v\Bigr) 
=\Bigl(v\mt\sC_{[\tau_2]}^{[\sa_2]}(b,a)v)\Bigr). 
\eeq 
Set 
\hb{\sfm=\sC_{[\tau_2]}^{[\sa_2]}} in Theorem~\ref{thm-P.M}. Also set 
\hb{V_0:=W_{23}}, 
\ \hb{V_1:=W_{12}}, and 
\hb{V_2:=W_{13}} in Lemma~\ref{lem-C.CE}. Then it follows from that lemma 
and Theorem~\ref{thm-P.M} that 
$$ 
\sC_{[\tau_2]}^{[\sa_2]}(b,a) 
\in\cB^{t/\vec r,(\lda_3-\lda_1,\mu)} 
\bigl(J,T_{\tau_3+\sa_1}^{\sa_3+\tau_1}\MWed\bigr).  
\npb 
$$ 
Thus claim~(ii) is a consequence of \Eqref{C.Cba} and assertion~(i). 
\end{proof} 
Next we study the invertibility of the linear map~$A$. We introduce 
the following definition: Suppose 
\hb{t>0} and 
\beq\Label{C.ast} 
a\in B_\iy^{t/\vec r,\vec\om_{11}}\JVee 
\qb \sa_{11}=\tau_{11}=\sa_1+\tau_1. 
\eeq 
Then $a$~ia said to be \emph{\hbox{$\lda_{11}$-uniformly} contraction 
invertible} if there exists 
\hb{a^{-1}\in\Ga\MVee^J} satisfying 
\beq\Label{C.Cin} 
a^{-1}\cdot(a\cdot u)=u 
\qb a\cdot(a^{-1}\cdot u)=u 
\qa u\in\Ga\MVe^J, 
\eeq 
and 
\beq\Label{C.Crho} 
\rho^{-\lda_{11}}\,|a^{-1}(t)|_{h_{11}}\leq c 
\qa t\in J. 
\eeq 
Note that the second part of \Eqref{C.ast} guarantees that the complete 
contractions in \Eqref{C.Cin} are well-defined. Also note that there 
exists at most one~$a^{-1}$ satisfying \Eqref{C.Cin}, the 
\emph{contraction inverse} of~$a$. For abbreviation, we put 
$$ 
B_{\iy,\inv}^{t/\vec r,\vec\om_{11}}\JVee 
:=\bigl\{\,a\in B_\iy^{t/\vec r,\vec\om_{11}}\JVee  
\ ;\ a\text{ is $\lda_{11}$-uniformly contraction invertible}\,\bigr\}. 
$$ 

\smallskip 
Let $\cX$ and~$\cY$ be Banach spaces and let $U$ be open in~$\cX$. Then 
\hb{f\sco U\ra\cY} is \emph{analytic} if each 
\hb{x_0\in U} has a neighborhood in which $f$~can be represented by a 
convergent series of continuous monomials. If $f$~is analytic, then $f$~is 
smooth and it can be locally represented by its Taylor series. If 
\hb{\BK=\BC}, and $f$~is (Fr\'echet) differentiable, then it is analytic. 
For this and further details we refer to E.~Hille and 
R.S.~Phillips~\cite{HiP57a}. 

\smallskip 
To simplify the presentation we restrict ourselves now to the most important 
cases in which 
\hb{\cB=B_\iy}. We leave it to the reader to carry out the obvious 
modifications in the following considerations needed to cover the remaining 
instances as well. 
\begin{proposition}\LabelT{pro-C.in} 
Suppose 
\hb{\sa_{11}=\tau_{11}=\sa_1+\tau_1}. Then 
$B_{\iy,\inv}^{t/\vec r,\vec\om_{11}}\JVee$ is open in 
$B_\iy^{t/\vec r,\vec\om_{11}}\JVee$. If 
\hb{a\in B_{\iy,\inv}^{t/\vec r,\vec\om_{11}}\JVee}, then 
\hb{a^{-1}\in B_{\iy,\inv}^{t/\vec r,(-\lda_{11},\mu)}\JVee}. The map 
$$ 
B_{\iy,\inv}^{t/\vec r,\vec\om_{11}}\JVee 
\ra B_\iy^{t/\vec r,(-\lda_{11},\mu)}\JVee 
\qb a\mt a^{-1} 
\npb 
$$ 
is analytic. 
\end{proposition} 
\begin{proof} 
(1) 
Without loss of generality we let 
\hb{F_1=\BK^n} and set 
\hb{\sa:=\sa_{11}}. Note that 
\hb{E:=\cL(\BK^n)^{m^\sa\times m^\sa}} is a Banach algebra with 
unit of dimension 
\hb{N^2:=(nm^\sa)^2}. It is obvious that we can fix an algebra isomorphism 
from~$E$ onto~$\BK^{N\times N}$ by which we identify~$E$ 
with~$\BK^{N\times N}$. 

\smallskip 
For 
\hb{b\in\BK^{N\times N}} we denote by~$b^\nat$ the 
\hb{(N\times N)}-matrix of cofactors of~$b$. Thus 
\hb{b^\nat=[b_{ij}^\nat]} with 
\beq\Label{C.bij} 
b_{ij}^\nat:=\det[b_1,\ldots,b_{i-1},e_j,b_{i+1},\ldots,b_N], 
\eeq 
where $b_1,\ldots,b_N$ are the columns of~$b$ and $e_j$~is the \hbox{$j$-th} 
standard basis vector of~$\BK^N$. Then, if $b$~is invertible, 
\beq\Label{C.bin} 
b^{-1}=\bigl(\det(b)\bigr)^{-1}b^\nat. 
\eeq 

\smallskip 
(2) 
Suppose either 
\hb{X:=(Q^m,g_m)} or 
\hb{X:=(Q^m\cap\BH^m,g_m)}, and 
\hb{Y=X\times J}. Set 
\beq\Label{C.XY} 
\cX^{t/\vec r}\YE 
:=B\bigl(J,B_\iy^t\XE\bigr)\cap B_\iy^{t/r}\bigl(J,B\XE\bigr), 
\eeq 
where $B_\iy^t\XE$ is obtained from $B_\iy^t\RmE$ by restriction, of 
course. Note 
\beq\Label{C.XBC} 
\cX^{t/\vec r}\YE\hr B_\iy\YE. 
\eeq 
It follows from Theorem~\ref{thm-P.M} that $\cX^{t/\vec r}\YE$ is a Banach 
algebra with respect to the point-wise extension of the (matrix) product 
of~$E$. 

\smallskip 
Assume 
\hb{b\in\cX^{t/\vec r}\YE} and $b(y)$~is invertible for 
\hb{y\in Y} such that 
\beq\Label{C.bi0} 
|b^{-1}(y)|_E\leq c_0 
\qa y\in Y. 
\eeq 
Then the spectrum~$\sa\bigl(b(y)\bigr)$ of~$b(y)$ is bounded and has a 
positive distance from 
\hb{0\in\BC}, uniformly with respect to 
\hb{y\in Y}. Hence 
\beq\Label{C.det} 
1/c(c_0)\leq\big|\det\big(b(y)\big)\big|\leq c(c_0) 
\qa y\in Y, 
\eeq 
due to the fact that $\det\big(b(y)\big)$ can be represented as the product 
of the eigenvalues of~$b(y)$, counted with multiplicities. 

\smallskip 
Since $\det\big(b(y)\big)$ is a polynomial in the entries of~$b(y)$ and 
\hb{\cX^{t/\vec r}(Y):=\cX^{t/\vec r}\YBK} is a multiplication algebra we 
infer 
\beq\Label{C.detr} 
\det(b)\in\cX^{t/\vec r}(Y). 
\eeq 
Using the chain rule if 
\hb{t\geq1} (cf.~Lemma~1.4.2 of~\cite{Ama09a}), we get 
\hb{\big(\det(b)\big)^{-1}\in\cX^{t/\vec r}(Y)} from \Eqref{C.det} 
and \Eqref{C.detr}. Now we deduce from 
\Eqref{C.bij}, \,\Eqref{C.bin}, and the fact that $\cX^{t/\vec r}(Y)$ is a 
multiplication algebra, that 
$$ 
b^{-1}\in\cX^{t/\vec r}\YE 
\qb \|b^{-1}\|_{\cX^{t/\vec r}\YE}\leq c(c_0), 
\npb 
$$ 
whenever 
\hb{b\in\cX^{t/\vec r}\YE} satisfies \Eqref{C.bi0}. 

\smallskip 
By \Eqref{C.XBC} it is obvious that the set of all invertible elements of 
$\cX^{t/\vec r}\YE$ satisfying \Eqref{C.bi0} for some 
\hb{c_0=c_0(b)\geq1} is open in $\cX^{t/\vec r}\YE$. 

\smallskip 
(3) 
Assume 
\hb{\gK\slt\Phi_1} is a uniformly regular atlas for~$W_1$ over~$\gK$. Given 
\hb{\ka\slt\vp_1\in\gK\slt\Phi_1}, put 
$$ 
\chi_\ka^{\vec\om_1}v 
:=\rho_\ka^{\lda_1}\Ta_{\iy,\ka}^\mu(\ka\slt\vp_1)_*v 
\qb \chi_\ka^{\vec\om_{11}}a
:=\rho_\ka^{\lda_{11}}\Ta_{\iy,\ka}^\mu(\ka\slt\vp_{11})_*a  
\npb 
$$ 
for 
\hb{v\in\Ga\MVe^J} and 
\hb{a\in\Ga\MVee^J}, respectively, and 
\hb{Y_\ka:=Q_\ka^m\times J}. 

\smallskip 
Suppose 
\hb{a\in B_{\iy,\inv}^{t/\vec r,\vec\om_{11}}\JVee}. Then we deduce from 
\Eqref{C.Cin} (see Example~\ref{exa-U.ex}(f)) and
\beq\Label{C.thk} 
\chi_\ka^{\vec\om_1}v 
=\chi_\ka^{\vec\om_1}\bigl(a^{-1}\cdot(a\cdot v)\bigr) 
=(\chi_\ka^{(-\lda_{11},\mu)}a^{-1}) 
(\chi_\ka^{\vec\om_{11}}a)\chi_\ka^{\vec\om_1}v 
\eeq 
for 
\hb{\ka\slt\vp_1\in\gK\slt\Phi_1} and 
\hb{v\in\Ga(U_{\coU\ka},V_1)^J}. Note that $\chi_\ka^{\vec\om_1}$~is a 
bijection from  $\Ga(U_{\coU\ka},V_1)^J$ onto 
$(E_{\tau_1}^{\sa_1})^{Y_\ka}$. Thus it follows from \Eqref{C.thk} that 
$\chi_\ka^{(-\lda_{11},\mu)}a^{-1}$ is a left inverse for 
$\chi_\ka^{\vec\om_{11}}a$ in~$B_\iy\YkE$. Similarly, we see that it is also 
a right inverse. Hence 
\hb{b_\ka:=\chi_\ka^{\vec\om_{11}}a} is invertible in $B_\iy\YkE$ and 
\beq\Label{C.bia} 
b_\ka^{-1}=\chi_\ka^{(-\lda_{11},\mu)}a^{-1}. 
\eeq 

\smallskip 
We infer from \Eqref{S.sd}(iv), \,\Eqref{L.Nh}, \,\Eqref{U.h}, 
\,\Eqref{U.h12}, \,\Eqref{C.ast}, and \Eqref{C.Crho} that 
\beq\Label{C.rtc} 
|b_\ka^{-1}|_E 
\leq c\,\Ta_{\iy,\ka}^\mu\ka_*(\rho^{-\lda_{11}}\ |a^{-1}|_{h_{11}})\leq c 
\qa \ka\slt\vp_1\in\gK\slt\Phi_1. 
\eeq 
Recalling \Eqref{S.err}, \,\Eqref{R.S}, and \Eqref{P.thta} we find 
\beq\Label{C.S} 
b_\ka 
=\chi_\ka^{\vec\om_{11}} 
 \Bigl(\sum_{\wt{\ka}\in\gN(\ka)}\pi_{\wt{\ka}}^2a\Bigr)  
=\sum_{\wt{\ka}\in\gN(\ka)}S_{\wt{\ka}\ka}(\wt{\ka}_*\pi_{\wt{\ka}}) 
 R_{\wt{\ka}\ka}\vp_{\iy,\wt{\ka}}^{\vec\om_{11}}a. 
\eeq 
Since 
\hb{a\in B_\iy^{t/\vec r,\vec\om_{11}}\JVee} implies 
\hb{\vp_\iy^{\vec\om_{11}}a\in\ell_\iy(\mf{B}_\iy^{t/\vec r})} we deduce 
from \Eqref{C.S}, \,\Eqref{R.LS}(iii), Lemmas \ref{lem-H.S} 
and \ref{lem-P.T}, Theorem~\ref{thm-P.M}, and definition~\Eqref{C.XY} 
\beq\Label{C.bbd} 
\|b_\ka\|_{\cX^{t/\vec r}\YkE}\leq c\,\|a\|_{t/\vec r,\iy;\vec\om_{11}}  
\qa \ka\slt\vp_1\in\gK\slt\Phi_1. 
\eeq 

\smallskip 
Set 
\hb{a_\ka:=\rho_\ka^{-\lda_{11}}b_\ka}. Then it follows from \Eqref{C.rtc} 
and \Eqref{C.bbd} that 
\beq\Label{C.bib} 
\rho_\ka^{-\lda_{11}}a_\ka^{-1}\in\cX^{t/\vec r}\YkE 
\qb \|\rho_\ka^{-\lda_{11}}a_\ka^{-1}\|_{\cX^{t/\vec r}\YkE}\leq c 
\qa \ka\slt\vp_1\in\gK\slt\Phi_1. 
\eeq 
Employing \Eqref{R.LS}(iii) and Theorem~\ref{thm-P.M} once more we 
derive from \Eqref{C.bib} 
\beq\Label{C.php} 
\vp_{\iy,\ka}^{(-\lda_{11},\mu)}a^{-1} 
=\chi_\ka^{(-\lda_{11},\mu)}(\pi_\ka a^{-1}) 
=(\ka_*\pi_\ka)b_\ka^{-1} 
\in B_\iy^{t/\vec r}\YkE 
=B_{\iy,\ka}^{t/\vec r}  
\eeq 
and 
\hb{\vp_\iy^{(-\lda_{11},\mu)}a^{-1} 
   \in\ell_\iy(\mf{B}_\iy^{t/\vec r})}. Hence 
Theorem~\ref{thm-H.R} implies 
\beq\Label{C.aps} 
a^{-1}=\psi_\iy^{(-\lda_{11},\mu)}(\vp_\iy^{(-\lda_{11},\mu)}a^{-1}) 
\in B_\iy^{t/\vec r,(-\lda_{11},\mu)}\JVee. 
\eeq 

\smallskip 
(4) 
Let $\cX$ be a Banach algebra with unit~$e$. Denote by~$\cG$ the 
group of invertible elements of~$\cX$. For 
\hb{b_0\in\cX} and 
\hb{\da>0} let $\cX(b_0,\da)$ be the open ball in~$\cX$ of radius~$\da$, 
centered at~$b_0$. Suppose 
\hb{b_0\in\cG}. Then 
\hb{b=b_0-(b_0-b)=\bigl(e-(b_0-b)b_0^{-1}\bigr)b_0} and 
$$ 
\|(b_0-b)b_0^{-1}\|\leq\|b_0-b\|\,\|b_0^{-1}\|<1/2 
\qa b\in\cX(b_0,\|b_0^{-1}\|/2), 
$$ 
imply that 
\hb{b\in\cX(b_0,\|b_0^{-1}\|/2)} is invertible and 
\beq\Label{C.bNS} 
b^{-1}=b_0^{-1}\bigl(e-(b_0-b)b_0^{-1}\bigr)^{-1} 
=b_0^{-1}\sum_{i=0}^\iy\bigl((b_0-b)b_0^{-1}\bigr)^i. 
\eeq 
In fact, this Neumann series has the convergent majorant~$\sum_i2^{-i}$. 
Note that 
\hb{p_i(x):=(-1)^ib_0^{-1}(xb_0^{-1})^i} is a continuous homogenous 
polynomial in 
\hb{x\in\cX}. Hence it follows from \Eqref{C.bNS} 
$$ 
b^{-1}=\sum_{i=0}^\iy p_i(b-b_0) 
\qa b\in\cX(b_0,\|b_0^{-1}\|/2), 
$$ 
and this series converges uniformly on $\cX(b_0,\|b_0^{-1}\|/2)$. Thus 
$\cG$~is open and the inversion map 
\hb{\inv\sco\cG\ra\cX}, 
\ \hb{b\mt b^{-1}} is analytic. 

\smallskip 
(5) 
We set 
\hb{\cX:=B\bigl(J,B\MVee\bigr)} and define a multiplication by 
\hb{(a,b)\mt\sC_{[\tau_1]}^{[\sa_1]}(a,b)}. Then $\cX$~is a Banach 
algebra with unit 
\hb{e:=\bigl((p,t)\mt\id_{\cL((V_1)_p)})}. 

\smallskip 
Consider the continuous linear map 
$$ 
f\sco B_\iy^{t/\vec r,\vec\om_{11}}\JVee\ra\cX 
\qb a\mt\rho^{\lda_{11}}a. 
$$ 
Then 
\hb{G:=f^{-1}(\cG)} is open in $B_\iy^{t/\vec r,\vec\om_{11}}\JVee$. 
Consequently, 
$$ 
f_0:=\inv\circ(f\sn G)\sco G\ra\cX 
\qb a\mt(\rho^{\lda_{11}}a)^{-1} 
$$ 
is continuous (in fact, analytic) by step~(4). Note that 
\hb{a^{-1}=\rho^{\lda_{11}}f_0(a)} is the contraction inverse of~$a$. 
Furthermore, 
\hb{f_0(a)\in\cX} implies 
$$ 
\rho^{-\lda_{11}}\,|a^{-1}(t)|_{h_{11}}=|f_0(a)(t)|_{h_{11}}\leq c 
\qa t\in J. 
$$ 
Hence each 
\hb{a\in G} is \hbox{$\lda_{11}$-uniformly} contraction invertible. 
Conversely, if 
\hb{a\in B_\iy^{t/\vec r,\vec\om_{11}}\JVee} is \hbox{$\lda_{11}$-uniformly} 
contraction invertible, then $a$~belongs to~$G$. Thus 
\hb{G=B_{\iy,\inv}^{t/\vec r,\vec\om_{11}}\JVee} which shows that 
$B_{\iy,\inv}^{t/\vec r,\vec\om_{11}}\JVee$ is open. 

\smallskip 
(6) 
We denote by~$\cG_\ka$ the group of invertible elements 
of~$B_{\iy,\ka}^{t/\vec r}$. Suppose 
\hb{a_0\in G}. Then step~(3) (see \Eqref{C.bbd} and \Eqref{C.bib}) 
guarantees that 
\hb{b_{0,\ka}:=\chi_\ka^{\vec\om_{11}}a_0\in\cG_\ka} and 
$$ 
\|b_{0,\ka}\|_{B_{\iy,\ka}^{t/\vec r}} 
+\|b_{0,\ka}^{-1}\|_{B_{\iy,\ka}^{t/\vec r}}\leq c 
\qa \ka\slt\vp_1\in\gK\slt\Phi_1. 
$$ 
Hence we infer from step~(4) that there exists 
\hb{\da>0} such that the open 
ball~$B_{\iy,\ka}^{t/\vec r}(b_{0,\ka},\da)$ belongs to~$\cG_\ka$ for 
\hb{\ka\in\gK} and the inversion map 
\hb{\inv_\ka\sco\cG_\ka\ra B_{\iy,\ka}^{t/\vec r}} is analytic on 
$B_{\iy,\ka}^{t/\vec r}(b_{0,\ka},\da)$, uniformly with respect to 
\hb{\ka\in\gK} in the sense that the series 
$$ 
\sum_ib_{0,\ka}^{-1}\bigl((b_{0,\ka}-b_\ka)b_{0,\ka}^{-1}\bigr)^i 
\npb 
$$ 
converges in~$B_{\iy,\ka}^{t/\vec r}$, uniformly with respect to 
\hb{b_\ka\in B_{\iy,\ka}^{t/\vec r}(b_{0,\ka},\da)} and 
\hb{\ka\in\gK}. 

\smallskip 
Note that 
$$ 
\mf{B}_\iy^{t/\vec r}(\mf{b}_0,\da) 
:=\prod_\ka B_{\iy,\ka}^{t/\vec r}(b_{0,\ka},\da)
$$ 
is open in~$\ell_\iy(\mf{B}_\iy^{t/\vec r})$. The above 
considerations show that 
\beq\Label{C.inv} 
\mf{\inv}\sco\mf{B}_\iy^{t/\vec r}(\mf{b}_0,\da) 
\ra\ell_\iy(\mf{B}_\iy^{t/\vec r}) 
\qb \mf{b}\mt\bigl(\inv_\ka(b_\ka)\bigr) 
\eeq 
is analytic. It follows from \Eqref{C.bbd} that the linear map 
\beq\Label{C.chi} 
\chi^{\vec\om_{11}} 
\sco B_\iy^{t/\vec r,\vec\om_{11}}\JVee\ra\ell_\iy(\mf{B}_\iy^{t/\vec r}) 
\qb v\mt(\chi_\ka^{\vec\om_{11}}v) 
\eeq 
is continuous. Hence 
\hb{G_0:=(\chi^{\vec\om_{11}})^{-1} 
   \bigl(\mf{B}_\iy^{t/\vec r}(\mf{b}_0,\da)\bigr)\cap G} 
is an open neighborhood of~$a_0$ in~$G$. It is a consequence of 
\Eqref{C.inv} and \Eqref{C.chi} that 
\hb{\mf{\inv}\circ\chi^{\vec\om_{11}}} is an analytic map from~$G_0$ into 
$\ell_\iy(\mf{B}_\iy^{t/\vec r})$.  

\smallskip 
Consider the point-wise multiplication operator 
$$ 
\mf{\pi}\sco\ell_\iy(\mf{B}_\iy^{t/\vec r}) 
\ra \ell_\iy(\mf{B}_\iy^{t/\vec r}) 
\qb \mf{b}\mt\bigl((\ka_*\pi_\ka)b_\ka\bigr). 
\npb 
$$ 
It follows from \Eqref{R.LS} and Theorem~\ref{thm-P.M} that it is a 
well-defined continuous linear map. 

\smallskip 
If 
\hb{a\in G_0}, then we know from \Eqref{C.bia} and \Eqref{C.bbd} that 
$$ 
\mf{\inv}\circ\chi^{\vec\om_{11}}(a) 
=(\chi_\ka^{(-\lda_{11},\mu)}a^{-1}\bigr) 
\in\ell_\iy(\mf{B}_\iy^{t/\vec r}). 
$$ 
Hence we see by \Eqref{C.php} and \Eqref{C.aps} that 
\hb{a^{-1}=\psi_\iy^{(-\lda_{11},\mu)}\circ\mf{\pi}\circ\mf{\inv} 
   \circ\chi^{\vec\om_{11}}a}. Thus 
$$ 
(a\mt a^{-1}) 
=\psi_\iy^{(-\lda_{11},\mu)}\circ\mf{\pi}\circ\mf{\inv} 
\circ\chi^{\vec\om_{11}} 
\sco G_0\ra B_\iy^{t/\vec r,(-\lda_{11},\mu)}\JVee 
\npb 
$$ 
is analytic, being a composition of analytic maps. This proves the 
proposition.
\end{proof} 
Henceforth, we set 
\hb{\gF_\iy:=B_\iy} so that $\gF_q$~is defined for 
\hb{1<q\leq\iy}. 
\begin{theorem}\LabelT{thm-C.in} 
Suppose 
\hb{1<q\leq\iy} and 
$$ 
\bal 
t>0\text{ and  $s$ satisfies \Eqref{C.su} with 
$|s|<t$ if\/ $q=p$, and $s=t$ if\/ }q=\iy. 
\eal 
$$ 
Assume 
\hb{\sa_{11}=\tau_{11}=\sa_1+\tau_1} and 
\hb{\lda_2=\lda_{11}+\lda_1}. If 
\hb{a\in B_{\iy,\inv}^{t/\vec r,\vec\om_{11}}\JVee}, then 
\beq\Label{C.A} 
A=\sC(a) 
\in\Lis\bigl(\gF_q^{s/\vec r,\vec\om_1}\JVe, 
\gF_q^{s/\vec r,\vec\om_2}\JVe\bigr) 
\npb 
\eeq 
and 
\hb{A^{-1}=\sC(a^{-1})}. The map 
\hb{a\mt A^{-1}} is analytic. 
\end{theorem} 
\begin{proof} 
It follows from Theorem~\ref{thm-C.CC}(i) and Proposition~\ref{pro-C.in} 
that \Eqref{C.A} applies and 
\hb{a\mt\sC(a^{-1})} is analytic. Part~(ii) of that theorem implies 
\hb{A^{-1}=\sC(a^{-1})}. 
\end{proof} 
Next we study the problem of the right invertibility of the operator~$A$ 
of Theorem~\ref{thm-C.CC}. This is of particular importance in 
connection with boundary value problems. First we need some preparation. 

\smallskip 
We assume 
\beq\Label{C.stst} 
\sa_{12}=\sa_2+\tau_1 
\qb \tau_{12}=\tau_2+\sa_1 
\qb \sa_{21}=\tau_{12} 
\qb \tau_{21}=\sa_{12}. 
\eeq 
Then, given 
\hb{a\in\Ga\MVez^J}, there exists a unique 
\hb{a^*\in\Ga\MVze^J}, the \emph{complete contraction adjoint of}~$a$, 
such that 
\beq\Label{C.hh} 
h_2(a\cdot u,v)=h_1(u,a^*\cdot v) 
\qa (u,v)\in\Ga(M,V_1\oplus V_2)^J. 
\eeq 
Indeed, recalling \Eqref{C.aloc1} set 
\beq\Label{C.adu2} 
(a^*)_{(j_{21}),\nu_2} 
     ^{(i_{21}),\nu_1} 
:=g_{(j_1)(\wt{\jmath}_1)} 
   ^{(i_1)(\wt{\imath}_1)}  
\,h_{W_1}^{*\nu_1\wt{\nu}_1} 
\,\ol{a_{(\wt{\jmath}_2;\wt{\imath}_1),\wt{\nu}_1} 
       ^{(\wt{\imath}_2;\wt{\jmath}_1),\wt{\nu}_2}} 
\,g_{(\wt{\imath}_2)(i_2)} 
   ^{(\wt{\jmath}_2)(j_2)}   
\,\ol{h_{W_2,\wt{\nu}_2\nu_2}}.  
\eeq 
Then it follows from \Eqref{V.Aad} and 
\hb{h_{k\ell}=\prsn_{\sa_{k\ell}}^{\tau_{k\ell}}\otimes h_{W_{k\ell}}} that 
\beq\Label{C.aduk} 
(a^*)_{(j_{21}),\nu_2} 
     ^{(i_{21}),\nu_1} 
\,\frac\pl{\pl x^{(i_{21})}} 
\otimes dx^{(j_{21})} 
\otimes b_{\nu_1}^1 
\otimes\ba_2^{\nu_2} 
\eeq 
is the local representation of~$a^*$ over~$U_{\coU\ka}$ with respect to the 
coordinate frame for~$V_{21}$ over~$U_{\coU\ka}$ associated with 
\hb{\ka\slt\vp_{21}}, 

\smallskip 
We set 
\beq\Label{C.lst} 
\lda_{21}^*:=\lda_{12}+\sa_{21}-\tau_{21} 
\qb \vec\om_{21}^*:=(\lda_{21}^*,\mu) 
\eeq 
and suppose 
\hb{a\in B_\iy^{t/\vec r,\vec\om_{12}}\JVez}. Then it is a consequence of 
\Eqref{U.h}, \,\Eqref{L.kgg}, \,\Eqref{L.kgk}, \,\Eqref{C.adu2}, and 
\Eqref{C.aduk} that 
$$ 
\|\vp_{\iy,\ka}^{\vec\om_{21}^*}a^*\| 
_{B_\iy^{t/\vec r}(\BY_\ka,E_{\tau_{21}}^{\sa_{21}})} 
\sim 
\|\vp_{\iy,\ka}^{\vec\om_{12}}a\| 
_{B_\iy^{t/\vec r}(\BY_\ka,E_{\tau_{12}}^{\sa_{12}})} 
\qa \ka\slt\vp_i\in\gK\slt\Phi_i 
\qb i=1,2. 
$$ 
From this, Theorem~\ref{thm-H.R}, \ \Eqref{R.rN}, and \Eqref{C.adu2} we 
infer 
\beq\Label{C.asta} 
(a\mt a^*) 
\in\cL\bigl(B_\iy^{t/\vec r,\vec\om_{12}}\JVez, 
B_\iy^{t/\vec r,\vec\om_{21}^*}\JVze\bigr). 
\eeq 

\smallskip 
Assume 
\hb{a^*(p,t)\in\cL\bigl((V_2)_p,(V_1)_p\bigr)} is injective for 
\hb{(p,t)\in M\times J}. Then 
\hb{a(p,t)\in\cL\bigl((V_1)_p,(V_2)_p\bigr)} is surjective. This motivates 
the following definition: 
\beq\Label{C.us} 
\bal 
{}      
&a\in B_\iy^{t/\vec r,\vec\om_{12}}\JVez 
 \text{ is }\lda_{12}\text{\emph{-uniformly contraction surjective} if}\\ 
&\rho^{\lda_{12}+(\tau_{12}-\sa_{12})/2}\,|a^*(t)\cdot u|_{h_1} 
 \geq|u|_{h_2}/c 
 \qa u\in\Ga\MVz 
 \qb t\in J. 
\eal 
\eeq 
The reason for the specific choice of the exponent of~$\rho$ will become 
apparent below. We set 
$$ 
B_{\iy,\surj}^{t/\vec r,\vec\om_{12}}\JVez 
:=\bigl\{\,a\in B_\iy^{t/\vec r,\vec\om_{12}}\JVez 
\ ;\ a\text{ is $\lda_{12}$-uniformly contraction surjective}\,\bigr\}. 
$$ 

\smallskip 
For abbreviation, we put 
$$ 
a\odot a^*:=\sC_{[\tau_1]}^{[\sa_1]}(a,a^*) 
\qb \sa_{22}:=\tau_{22}:=\sa_2+\tau_2 
\qb \lda_{22}:=2\lda_{12}+\tau_{12}-\sa_{12}.  
$$ 
It follows from \Eqref{C.asta} and Theorem~\ref{thm-P.M} that 
\beq\Label{C.a2} 
B_\iy^{t/\vec r,\vec\om_{12}}\JVez\ra B_\iy^{t/\vec r,\vec\om_{22}}\JVzz 
\qb a\mt a\odot a^* 
\npb 
\eeq 
is a well-defined continuous quadratic map. Hence it is analytic. 
\begin{lemma}\LabelT{lem-C.CS} 
\hb{a\in B_{\iy,\surj}^{t/\vec r,\vec\om_{12}}\JVez} iff\/ 
\hb{a\odot a^*\in B_{\iy,\inv}^{t/\vec r,\vec\om_{22}}\JVzz}. 
\end{lemma} 
\begin{proof} 
It follows from \Eqref{C.hh} that 
\beq\Label{C.haa} 
h_2\bigl((a\odot a^*)\cdot u,v\bigr) 
=h_2\bigl(a\cdot(a^*\cdot u),v\bigr) 
=h_1(a^*\cdot u,a^*\cdot v) 
\qa (u,v)\in\Ga(M,V_2\oplus V_2)^J. 
\eeq 
Hence 
\hb{\sC(a\odot a^*)} is symmetric and positive semi-definite. We see from 
\Eqref{C.haa} that \Eqref{C.us} is equivalent to 
$$ 
\rho^{\lda_{22}}h_2\bigl((a\odot a^*)(t)\cdot u,u\bigr) 
\geq|u|_{h_2}^2/c 
\qa u\in\Ga\MVz 
\qb t\in J. 
\npb 
$$ 
By symmetry this inequality is equivalent to the 
\hbox{$\lda_{22}$-uniform} contraction invertibility of 
\hb{a\odot a^*}. 
\end{proof} 
In the next proposition we give a local criterion for checking 
\hbox{$\lda_{12}$-uniform} surjectivity. 
\begin{proposition}\LabelT{pro-C.S} 
Suppose 
\hb{a\in B_\iy^{t/\vec r,\vec\om_{12}}\JVez}. Let 
\hb{\gK\slt\Phi_i}, 
\ \hb{i=1,2}, be uniformly regular atlases for~$V_{\coV i}$ over~$\gK$. Set 
$$ 
\sfa_\ka(t)(\za,\za) 
:=
\sum_{\dlim{(i_1)\in\BJ_{\sa_1},\,(j_1)\in\BJ_{\tau_1}}{1\leq\nu_1\leq n_1}}  
\Big|\ka_*a_{(j_2;i_1),\nu_1}^{(i_2;j_1),\nu_2}(t) 
\,\za_{(i_2),\nu_2}^{(j_2)}\Big|^2 
$$ 
for 
\hb{\za\in E_{\sa_2}^{\tau_2}(F_2^*)^{Q_\ka^m}} and 
\hb{t\in J}. Then $a$~is \hbox{$\lda_{12}$-uniformly} contraction surjective 
iff 
$$ 
\rho_\ka^{2\lda_{12}}\sfa_\ka(t)(\za,\za)\sim|\za|^2 
\qa \za\in E_{\sa_2}^{\tau_2}(F_2^*)^{Q_\ka^m} 
\qb \ka\in\gK 
\qb t\in J. 
$$ 
\end{proposition} 
\begin{proof} 
Assume 
\hb{v\in\Ga\MVz^J} and put 
\hb{w:=(h_2)_\flat v\in\Ga\MVzs^J}, where 
\hb{V_2'=T_{\sa_2}^{\tau_2}\MWzs}.  
Then, locally on~$U_{\coU\ka}$, 
$$ 
v=v_{(j_2)}^{(i_2),\nu_2}\frac\pl{\pl x^{(i_2)}}\otimes dx^{(j_2)} 
\otimes b_{\nu_2}^2 
\qb w=w_{(i_2),\nu_2}^{(j_2)}\,\frac\pl{\pl x^{(j_2)}}\otimes dx^{(i_2)} 
\otimes\ba_2^{\nu_2}, 
$$ 
where 
$$ 
w_{(i_2),\nu_2}^{(j_2)}=g_{(i_2)(\wt{\imath}_2)}^{(j_2)(\wt{\jmath}_2)} 
\,h_{W_2,\nu_2\wt{\nu}_2} 
\,\ol{v_{(\wt{\jmath}_2)}^{(\wt{\imath}_2),\wt{\nu}_2}}, 
$$ 
due to 
\hb{h_2=\prsn_{\sa_2}^{\tau_2}\otimes h_{W_2}}. Thus it follows from 
\Eqref{C.adu2} that, locally on~$U_{\coU\ka}$, 
$$ 
\bigl((a\odot a^*)\cdot v\bigr)_{(j_2)}^{(i_2),\nu_2} 
=a_{(j_2;i_1),\nu_1}^{(i_2;j_1),\nu_2} 
\,g_{(j_1)(\wt{\jmath}_1)}
   ^{(i_1)(\wt{\imath}_1)}  
\,h_{W_1}^{*\nu_1\wt{\nu}_1} 
\,\ol{a_{(\wt{\jmath}_2;\wt{\imath}_1),\wt{\nu}_1} 
       ^{(\wt{\imath}_2;\wt{\jmath}_1),\wt{\nu}_2}} 
\,\,\ol{w_{(\wt{\imath}_2),\wt{\nu}_2} 
         ^{(\wt{\jmath}_2)}}.   
$$ 
Hence 
$$ 
\bal 
h_2\bigl((a\odot a^*)\cdot v,v\bigr)
&=g_{(i_2)(\wh{\imath}_2)}^{(j_2)(\wh{\jmath}_2)} 
 \,h_{W_2,\nu_2\wh{\nu}_2} 
 \bigl((a\odot a^*)\cdot v\bigr)_{(j_2)}^{(i_2),\nu_2} 
 \,\ol{v_{(\wh{\jmath}_2)}^{(\wh{\imath}_2),\wh{\nu}_2}}\\  
&=\bigl((a\odot a^*)\cdot v\bigr)_{(j_2)}^{(i_2),\nu_2} 
 \,w_{(i_2),\nu_2}^{(j_2)}\\  
&=a_{(j_2;i_1),\nu_1}^{(i_2;j_1),\nu_2} 
 \,w_{(i_2),\nu_2}^{(j_2)} 
 \,g_{(j_1)(\wt{\jmath}_1)}^{(i_1)(\wt{\imath}_1)} 
 \,h_{W_1}^{*\nu_1\wt{\nu}_1} 
 \,\ol{a_{(\wt{\jmath}_2;\wt{\imath}_1),\wt{\nu}_1} 
       ^{(\wt{\imath}_2;\wt{\jmath}_1),\wt{\nu}_2}} 
 \,\,\ol{w_{(\wt{\imath}_2),\wt{\nu}_2}^{(\wt{\jmath}_2)}}.  
\eal 
$$ 
Thus we deduce from \Eqref{U.h}, \,\Eqref{S.err}, \,\Eqref{L.hl}, 
\,\Eqref{L.kgg} \ (applied to~$W_{\coW i}$) 
\beq\Label{C.h2z} 
\ka_*\bigl(\rho^{\lda_{22}}h_2\bigl((a\odot a^*)\cdot v,v\bigr)\bigr) 
\sim\rho_\ka^{\lda_{22}+2(\tau_1-\sa_1)}\sfa_\ka(\za,\za) 
\eeq 
for 
\hb{\ka\slt\vp_i\in\gK\slt\Phi_i},  
\ \hb{i=1,2}, and  
\hb{v\in\Ga\MVz^J}, where 
\beq\Label{C.z} 
\za:=(\ka\slt\vp_2)_*\bigl((h_2)_\flat v\bigr) 
\in\bigl(E_{\sa_2}^{\tau_2}(F_2^*)\bigr)^{Q_\ka^m\times J}. 
\eeq 
Since $(h_2)_\flat$~is an isometry and $h_2^*$~is the bundle 
metric of~$V_2^*$ we get from \Eqref{L.Nh} 
\beq\Label{C.v} 
\ka_*(|v|_{h_2}^2)=\ka_*(|w|_{h_2^*}^2) 
\sim\rho_\ka^{2(\tau_2-\sa_2)}\,|\za|_{E_{\sa_2}^{\tau_2}(F_2^*)}^2 
\qa \ka\in\gK, 
\npb 
\eeq 
with $v$ and~$\za$ being related by \Eqref{C.z}. Now the assertion 
follows from 
\Eqref{C.stst}, \,\Eqref{C.lst}, \,\Eqref{C.h2z}, and \Eqref{C.v}. 
\end{proof} 
Suppose 
\hb{a\in B_\iy^{t/\vec r,\vec\om_{12}}\JVez} and 
\hb{a^c\in B_\iy^{t/\vec r,(-\lda_{12},\mu)}\JVze} are such that 
\hb{a\cdot(a^c\cdot v)=v} for $v$ belonging to $\Ga\MVz^J$. Then 
$a^c$~is a \emph{right contraction inverse of}~$a$. 
\begin{proposition}\LabelT{pro-C.R} 
Let conditions \Eqref{C.stst} be satisfied. Then 
$B_{\iy,\surj}^{t/\vec r,\vec\om_{12}}\JVez$ is open in
$B_\iy^{t/\vec r,\vec\om_{12}}\JVez$ and there exists an analytic map 
$$ 
I^c\sco B_{\iy,\surj}^{t/\vec r,\vec\om_{12}}\JVez 
\ra B_\iy^{t/\vec r,(-\lda_{12},\mu)}\JVze 
\npb 
$$ 
such that $I^c(a)$~is a right contraction inverse for~$a$. 
\end{proposition} 
\begin{proof} 
It follows from \Eqref{C.a2}, Proposition~\ref{pro-C.in}, and 
Lemma~\ref{lem-C.CS} that 
\hb{S:=B_{\iy,\surj}^{t/\vec r,\vec\om_{12}}\JVez} is open in 
$B_\iy^{t/\vec r,\vec\om_{12}}\JVez$. Set 
$$ 
I^c(a):=\sC_{[\tau_2]}^{[\sa_2]}\bigl(a^*,(a\odot a^*)^{-1}\bigr) 
\qa a\in S, 
$$ 
where 
\hb{(a\odot a^*)^{-1}} is the contraction inverse of 
\hb{a\odot a^*\in B_\iy^{t/\vec r,\vec\om_{22}}\JVzz}. Then 
\Eqref{C.asta}, \,\Eqref{C.a2}, and Theorem~\ref{thm-P.M} imply that 
$I^c$~is an analytic map from~$S$ into 
\hb{B_\iy^{t/\vec r,(-\lda_{12},\mu)}\JVze}. Since 
$$ 
a\cdot\bigl(I^c(a)\cdot v\bigr) 
=a\cdot\bigl(a^*\cdot\bigl((a\odot a^*)^{-1}\cdot v\bigr)\bigr) 
=(a\odot a^*)\cdot\bigl((a\odot a^*)^{-1}\cdot v\bigr) 
=v 
\qa v\in\Ga\MVz, 
\npb 
$$ 
the assertion follows. 
\end{proof} 
After these preparations it is easy to prove the second main theorem of 
this section. For this it should be noted that definition~\Eqref{C.us} 
applies equally well if 
\hb{a\in\cB^{t/\vec r,\vec\om_{12}}\JVez} where either 
\hb{\cB=b_\iy}, or 
\hb{t\in r\BN} and 
\hb{\cB\in\{BC,bc\}}. Hence $\cB_\surj^{t/\vec r,\vec\om_{12}}\JVez$ is 
defined in these cases also.\po 
{\samepage 
\begin{theorem}\LabelT{thm-C.R} 
Let assumptions \Eqref{C.l12} and \Eqref{C.su} be satisfied and 
\hb{1<q\leq\iy}.\po  
\begin{itemize} 
\item[{\rm(i)}] 
Assume 
\hb{|s|<t} if\/ 
\hb{q=p}, and 
\hb{s=t>0} if\/ 
\hb{q=\iy}. Then there exists an analytic map 
$$ 
A^c\sco B_{\iy,\surj}^{t/\vec r,\vec\om_{12}}\JVez 
\ra\cL\bigl(\gF_q^{s/\vec r,\vec\om_2}\JVz, 
\gF_q^{s/\vec r,\vec\om_1}\JVe\bigr) 
\npb 
$$ 
such that $A^c(a)$~is a right inverse for 
\hb{A(a)=(v\mt a\cdot v)}.\po 
\item[{\rm(ii)}] 
There exists an analytic map 
$$ 
A^c\sco\cB_\surj^{s/\vec r,\vec\om_{12}}\JVez 
\ra\cL\bigl(\cB^{s/\vec r,\vec\om_2}\JVz, 
\cB^{s/\vec r,\vec\om_1}\JVe\bigr) 
\npb 
$$ 
such that $A^c(a)$~is a right inverse for~$A(a)$ if either 
\hb{s\in r\BN} and 
\hb{\cB\in\{BC,bc\}}, or 
\hb{s>0} and 
\hb{\cB=b_\iy}. 
\end{itemize} 
\end{theorem} }
\begin{proof} 
The first assertion is an obvious consequence of Theorem~\ref{thm-C.CC} and 
Proposition~\ref{pro-C.R}. The second claim is obtained by modifying the 
above arguments in the apparent way.  
\end{proof} 
As in the preceding section, the above results possess obvious analogues 
applying in the isotropic case. 
\section{Embeddings}\LabelT{sec-E}%
Now we complement the embedding theorems of Section~\ref{sec-A} by 
establishing further inclusions between anisotropic weighted spaces. 
\begin{theorem}\LabelT{thm-E.p} 
Suppose 
\hb{\lda_0<\lda_1} and put 
\hb{\vec\om_i:=(\lda_i,\mu)} for 
\hb{i=0,1}. Then 
\hb{\gF_p^{s/\vec r,\vec\om_0}\sdh\gF_p^{s/\vec r,\vec\om_1}} if 
\hb{\rho\leq1}, whereas 
\hb{\rho\geq1} implies 
\hb{\gF_p^{s/\vec r,\vec\om_1}\sdh\gF_p^{s/\vec r,\vec\om_0}} for 
\hb{s\in\BR}. 

\smallskip 
Similarly, 
\hb{\cB^{s/\vec r,\vec\om_0}\hr\cB^{s/\vec r,\vec\om_1}} if 
\hb{\rho\leq1}, and 
\hb{\cB^{s/\vec r,\vec\om_1}\hr\cB^{s/\vec r,\vec\om_0}} for 
\hb{\rho\geq1}, if either 
\hb{s>0} and 
\hb{\cB\in\{B_\iy,b_\iy\}}, or 
\hb{s\in r\BN} and 
\hb{\cB\in\{BC,bc\}}. 
\end{theorem} 
\begin{proof} 
If 
\hb{\rho\leq1}, then it is obvious that 
$$ 
W_{\coW p}^{kr/\vec r,\vec\om_0}\sdh W_{\coW p}^{kr/\vec r,\vec\om_1} 
\qb \ci W_{\coW p}^{kr/\vec r,\vec\om_0} 
    \sdh \ci W_{\coW p}^{kr/\vec r,\vec\om_1} 
\qb BC^{kr/\vec r,\vec\om_0}\hr BC^{kr/\vec r,\vec\om_1} 
$$ 
for 
\hb{k\in\BN}. Thus, by duality, 
$$ 
W_{\coW p}^{kr/\vec r,\vec\om_0}\sdh W_{\coW p}^{kr/\vec r,\vec\om_1} 
\qa k\in-\BN^\times. 
\npb 
$$ 
By interpolation 
\hb{\gF_p^{s/\vec r,\vec\om_0}\sdh\gF_p^{s/\vec r,\vec\om_1}} follows. 
The proof of the other embeddings is similar. 
\end{proof} 
The next theorem contains Sobolev-type embedding results. In the 
anisotropic case they involve the weight exponents as well as the 
regularity parameters.\po 
{\samepage 
\begin{theorem}\LabelT{thm-E.Sob} 
\begin{itemize} 
\item[{\rm(i)}] 
Suppose 
\hb{s_0<s_1} and 
\hb{p_0,p_1\in(1,\iy)} satisfy 
\beq\Label{E.ss} 
s_1-(m+r)/p_1=s_0-(m+r)/p_0. 
\npb  
\eeq  
Set 
\hb{\vec\om_0:=\bigl(\lda+(m+\mu)(1/p_1-1/p_0),\mu\bigr)}. Then 
\hb{\gF_{p_1}^{s_1/\vec r,\vec\om} 
   \sdh\gF_{p_0}^{s_0/\vec r,\vec\om_0}}.\po  
\item[{\rm(ii)}] 
Assume 
\hb{t>0} and 
\hb{s\geq t+(m+r)/p}. Set 
\hb{\vec\om_\iy:=\bigl(\lda+(m+\mu)/p,\mu\bigr)}. Then 
\hb{\gF_p^{s/\vec r,\vec\om}\hr b_\iy^{t/\vec r,\vec\om_\iy}}.\po 
\end{itemize} 
\end{theorem} }
\begin{proof} 
(1) 
Note that 
\hb{s_1>s_0} and \Eqref{E.ss} imply 
\hb{p_0>p_1}. Hence it follows from \Eqref{E.ss} and Theorems 3.3.2, 
3.7.5, and 4.4.1 of~\cite{Ama09a} that 
$$ 
\ell_{p_1}(\mf{\gF}_{p_1}^{s_1/\vec r}) 
\sdh\ell_{p_1}(\mf{\gF}_{p_0}^{s_0/\vec r}) 
\sdh\ell_{p_0}(\mf{\gF}_{p_0}^{s_0/\vec r}). 
$$ 
Also note that we get 
\hb{\psi_{p_1}^{\vec\om}=\psi_{p_0}^{\vec\om_0}} from \Eqref{E.ss}. Thus we 
infer from Theorem~\ref{thm-RA.R} that 
 \bspezeq 
 \bal 
 \begin{picture}(121,67)(-55,-6)        
 \put(0,50){\makebox(0,0)[b]{\small{$\vp_{p_1}^{\vec\om}$}}}
 \put(0,5){\makebox(0,0)[b]{\small{$\psi_{p_0}^{\vec\om_0}$}}}
 \put(-39,45){\makebox(0,0)[c]{\small{$\gF_{p_1}^{s_1/\vec r,\vec\om}$}}}
 \put(-39,0){\makebox(0,0)[c]{\small{$\gF_{p_0}^{s_0/\vec r,\vec\om_0}$}}}
 \put(46,45){\makebox(0,0)[c]{\small{$\ell_{p_1}(\mf{\gF}_{p_1}^{s_1/ 
 \vec r})$}}}
 \put(46,0){\makebox(0,0)[c]{\small{$\ell_{p_0}(\mf{\gF}_{p_0}^{s_0/ 
 \vec r})$}}}
 \put(51,22.5){\makebox(0,0)[l]{\small{$d$}}}
 \put(-20,45){\vector(1,0){40}}
 \put(20,0){\vector(-1,0){40}}
 \put(-39,33){\vector(0,-1){25}}
 \put(-37,33){\oval(4,4)[t]}
 \put(46,33){\vector(0,-1){25}}
 \put(48,33){\oval(4,4)[t]}
 \end{picture} 
 \eal 
 \npb 
 \espezeq 
is commuting. From this we obtain assertion~(i). 

\smallskip 
(2) 
We infer from Lemma~\ref{lem-RA.WW} and \cite[Theorem~3.3.2]{Ama09a} that 
$$ 
B_{p,\ka}^{s/\vec r}=B_{p,\ka}^{s/\mf{\nu}}\hr B_{p,\iy,\ka}^{s/\mf{\nu}} 
\hr B_{\iy,\iy,\ka}^{t/\mf{\nu}}=B_{\iy,\ka}^{t/\vec r} 
$$ 
and from \cite[Theorem~3.7.1]{Ama09a} that 
\hb{H_{p,\ka}^{s/\vec r}\hr B_{p,\iy,\ka}^{s/\mf{\nu}}}. Consequently, 
\hb{\gF_{p,\ka}^{s/\vec r}\hr B_{\iy,\ka}^{t/\vec r}}. From this and the 
density of $\cD\BYkE$ in~$\gF_{p,\ka}^{s/\vec r}$ it follows, due to 
\hb{\cD\BYkE\hr b_{\iy,\ka}^{t/\vec r}}, that 
\hb{\gF_{p,\ka}^{s/\vec r}\hr b_{\iy,\ka}^{t/\vec r}}. Thus, 
by \Eqref{R.ll}, 
\beq\Label{E.lpi} 
\ell_p(\mf{\gF}_p^{s/\vec r})\hr\ell_\iy(\mf{b}_\iy^{t/\vec r}). 
\eeq 
It is obvious that 
\hb{\mf{\cD}\BYE\hr\ell_{\iy,\unif}(\mf{b}_\iy^{t/\vec r})}. By 
Theorem~\ref{thm-RA.R} we know that $\mf{\cD}\BYE$ is dense 
in~$\ell_p(\mf{\gF}_p^{s/\vec r})$. From this and \Eqref{E.lpi} we deduce 
\hb{\ell_p(\mf{\gF}_p^{s/\vec r}) 
   \hr\ell_{\iy,\unif}(\mf{b}_\iy^{t/\vec r})}. Observing 
\hb{\psi_p^{\vec\om}=\psi_\iy^{\vec\om_\iy}}, we infer from 
Theorems \ref{thm-RA.R} and \ref{thm-H.Rbc} that the diagram 
 \bspezeq 
 \bal 
 \begin{picture}(131,67)(-55,-6)        
 \put(0,50){\makebox(0,0)[b]{\small{$\vp_p^{\vec\om}$}}}
 \put(0,5){\makebox(0,0)[b]{\small{$\psi_\iy^{\vec\om_\iy}$}}}
 \put(-38,45){\makebox(0,0)[c]{\small{$\gF_p^{s/\vec r,\vec\om}$}}}
 \put(-38,0){\makebox(0,0)[c]{\small{$b_\iy^{t/\vec r,\vec\om_\iy}$}}}
 \put(52,45){\makebox(0,0)[c]{\small{$\ell_p(\mf{\gF}_p^{s/\vec r})$}}}
 \put(52,0){\makebox(0,0)[c]{\small{$\ell_{\iy,\unif}(\mf{b}_\iy^{t/ 
 \vec r})$}}} \put(-20,45){\vector(1,0){40}}
 \put(20,0){\vector(-1,0){40}}
 \put(-38,33){\vector(0,-1){25}}
 \put(-36,33){\oval(4,4)[t]}
 \put(52,33){\vector(0,-1){25}}
 \put(54,33){\oval(4,4)[t]}
 \end{picture} 
 \eal 
 \npb 
 \espezeq 
is commuting. This proves~(ii). 
\end{proof} 
\begin{remark}\LabelT{rem-E.ss} 
Define the \emph{anisotropic small H\"older space} 
\hb{C_0^{s/\vec r,\vec\om}=C_0^{s/\vec r,\vec\om}\JV} to be the closure 
of~$\cD\JcD$ in~$B_\iy^{s/\vec r,\vec\om}$ for 
\hb{s>0}. Then the above proof shows 
\hb{\gF_p^{s/\vec r,\vec\om}\hr C_0^{t/\vec r,\vec\om_\iy}} if the 
hypotheses of~(ii) are satisfied. ${}$\hfill$\Box$ 
\end{remark} 
\section{Differential Operators}\LabelT{sec-M}%
First we establish the mapping properties of $\na$ and~$\pl$ in anisotropic 
weighted Bessel potential and Besov spaces. They are, of course, of 
fundamental importance for the theory of differential equations. 
\begin{theorem}\LabelT{thm-M.D} 
Suppose either 
\hb{s\geq0} and 
\hb{\gG=\gF_p}, or 
\hb{s>0} and 
\hb{\gG\in\{B_\iy,b_\iy\}}. Then 
$$ 
\na\in\cL\bigl(\gG^{s+1,\lda},\gG^{s,\lda}(V_{\tau+1}^\sa)\bigr) 
\cap\cL\bigl(\gG^{(s+1)/\vec r,\vec\om}, 
\gG^{s/\vec r,\vec\om}(J,V_{\tau+1}^\sa)\bigr) 
\npb 
$$ 
and 
\hb{\pl\in\cL\bigl(\gG^{(s+r)/\vec r,\vec\om}, 
   \gG^{s/\vec r,(\lda+\mu,\mu)}\bigr)}. 
\end{theorem} 
\begin{proof} 
We consider the time-dependent case. The proof in the stationary setting 
is similar. 

\smallskip 
(1) 
From \Eqref{L.Dr} and \Eqref{L.abd} we know that 
$$ 
(\ka\slt\vp)_*\na v=\pl_xv+a_\ka v 
\qa v\in C\bigl(J,C^1\BXkE\bigr), 
$$ 
where 
\hb{a_\ka\in C^\iy\bigl(Q_\ka^m,\cL(E_\tau^\sa,E_{\tau+1}^\sa)\bigr)} 
satisfies 
\hb{\|a_\ka\|\leq c(k)} for 
\hb{\ka\slt\vp\in\gK\slt\Phi}. Hence it follows from Theorem \ref{thm-P.b} 
and 
\hb{\gG_\ka^{(s+1)/\vec r}\hr\gG_\ka^{s/\vec r}} that 
$$ 
A_\ka:=(v\mt a_\ka\chi v) 
\in\cL\bigl(\gG_\ka^{(s+1)/\vec r}, 
\gG^{s/\vec r}(\BY_\ka,E_{\tau+1}^\sa)\bigr) 
\qa \|A_\ka\|\leq c 
\qb \ka\slt\vp\in\gK\slt\Phi. 
$$ 
By \cite[Theorem~4.4.2]{Ama09a} and Theorems \ref{thm-PH.BJ} and 
\ref{thm-PH.bcintJ} we get 
\beq\Label{M.D} 
\pl_x\in\cL\bigl(\gG_\ka^{(s+1)/\vec r}, 
\gG^{s/\vec r}(\BY_\ka,E_{\tau+1}^\sa)\bigr) 
\qb \pl\in\cL(\gG_\ka^{(s+r)/\vec r},\gG_\ka^{s/\vec r}).  
\eeq  

\smallskip 
(2) 
Set 
\hb{q:=p} if 
\hb{\gG=\gF_p}, and 
\hb{q:=\iy} otherwise. Then, given 
\hb{u\in\gG_\ka^{(s+1)/\vec r,\vec\om}}, 
$$ 
\vp_{q,\ka}^{\vec\om}(\na u) 
=\rho_\ka^{\lda+m/q}\Ta_{q,\ka}^\mu(\ka\slt\vp)_*(\pi_\ka\na u) 
=(\ka_*\pi_\ka)\bigl((\ka\slt\vp)_*\na\bigr) 
 (\wh{\vp}_{q,\ka}^{\kern1pt\vec\om}u).
$$ 
Hence we get from \Eqref{P.TS} 
$$ 
\vp_{q,\ka}^{\vec\om}(\na u) 
=\sum_{\wt{\ka}\in\gK}b_{\wt{\ka}\ka}(\ka\slt\vp)_*\na 
(R_{\wt{\ka}\ka}\vp_{q,\wt{\ka}}^{\vec\om}u), 
$$ 
where 
\hb{b_{\wt{\ka}\ka}=(\ka_*\pi_\ka)a_{\wt{\ka}\ka}} and 
$a_{\wt{\ka}\ka}$~is defined by \Eqref{P.a}. From this, 
\Eqref{R.LS}, \,\Eqref{P.abd}, Lemma~\ref{lem-P.T}, step~(1), 
Theorem~\ref{thm-P.M}, and the finite multiplicity of~$\gK$ we infer 
$$ 
\|\vp_q^{\vec\om}(\na u)\| 
_{\ell_q(\mf{\gG}^{s/\vec r}(\BY,E_{\tau+1}^\sa))} 
\leq c\,\|\vp_q^{\vec\om}u\|_{\ell_q(\mf{\gG}^{(s+1)/\vec r})} 
$$ 
for 
\hb{u\in\gG^{(s+1)/\vec r,\vec\om}}. Using Theorems \ref{thm-RA.R}, 
\,\ref{thm-H.R}, and \ref{thm-H.Rbc} we thus obtain 
$$ 
\|\vp_q^{\vec\om}(\na u)\| 
_{\ell_q(\mf{\gG}^{s/\vec r}(\BY,E_{\tau+1}^\sa))} 
\leq c\,\|u\|_{\gG^{(s+1)/\vec r,\vec\om}}  
\qa u\in\gG^{(s+1)/\vec r,\vec\om}.  
\npb 
$$ 
Thus the first assertion follows from 
\hb{\na u=\psi_q^{\vec\om}\bigl(\vp_q^{\vec\om}(\na u)\bigr)} by invoking 
these theorems once more. 
 
\smallskip 
(3) 
Since (see \Eqref{RA.tk}) 
$$ 
\vp_{q,\ka}^{(\lda+\mu,\mu)}(\pl u) 
=\rho_\ka^\mu\vp_{q,\ka}^{\vec\om}\pl u 
=\pl(\vp_{q,\ka}^{\vec\om}u), 
\npb 
$$ 
the second assertion is implied by the second part of \Eqref{M.D} and 
the arguments of step~(2). 
\end{proof} 
By combining this result with Theorem~\ref{thm-C.CC} and embedding theorems 
of the preceding section we can derive mapping properties of differential 
operators. To be more precise, for 
\hb{k\in\BN^\times} we consider operators of the form 
$$ 
\cA=\sum_{i+jr\leq kr}a_{ij}\cdot\na^i\pl^j 
$$ 
where $a_{ij}$~are suitably regular time-dependent 
vector-bundle-valued tensor field homomorphisms and 
\hb{a_{ij}\cdot\na^i\pl^j} equals 
\hb{\bigl(u\mt a_{ij}\cdot(\na^i\pl^ju)\bigr)}, 
of course. Recall that 
\hb{\gF_\iy=B_\iy}. 
\begin{theorem}\LabelT{thm-M.D0} 
Let\/ 
\hb{\bar W=(\bar W,h_{\bar W},D_{\bar W})} be a fully uniformly regular 
vector bundle over~$M$. Suppose $k$, $\bar\sa$,~$\bar\tau$ 
belong to~$\BN$ and 
\hb{\bar\lda\in\BR}. For 
\hb{0\leq i\leq k} set 
$$ 
\sa_i:=\bar\sa+\tau+i 
\qb \tau_i:=\bar\tau+\sa 
\qb \vec{\bar\om}:=(\bar\lda,\mu). 
\po 
$$ 
{\samepage 
\begin{itemize} 
\item[{\rm(i)}] 
Given 
\hb{i,j\in\BN} with 
\hb{i+jr\leq k}, put 
$$ 
\lda_{ij}:=\bar\lda-\lda-j\mu 
\qb \vec\om_{ij}:=(\lda_{ij},\mu). 
\po  
$$ 
Let condition~\Eqref{C.su} be satisfied. Suppose 
\hb{\wh{s}>|s|} if\/ 
\hb{q=p}, and 
\hb{\wh{s}=|s|>0} if\/ 
\hb{q=\iy}, and 
\beq\Label{M.Ba} 
a_{ij}\in B_\iy^{\wh{s}/\vec r,\vec\om_{ij}} 
\bigl(J,T_{\tau_i}^{\sa_i}\bigl(M,\Hom(W,\bar W)\bigr)\bigr) 
\qa i+jr\leq k. 
\po 
\eeq 
Then 
$$ 
\cA\in\cL\bigl(\gF_q^{(s+kr)/\vec r,\vec\om}, 
\gF_q^{s/\vec r,\vec{\bar\om}} 
\bigl(J,V_{\bar\tau}^{\bar\sa}(\bar W)\bigr)\bigr) 
\qa 1<q\leq\iy. 
\po 
$$ 
If $B_\iy^{\wh{s}/\vec r,\vec\om_{ij}}$ in \Eqref{M.Ba} is replaced by 
$b_\iy^{\wh{s}/\vec r,\vec\om_{ij}}$, then 
$$ 
\cA\in\cL\bigl(b_\iy^{(s+kr)/\vec r,\vec\om}, 
b_\iy^{s/\vec r,\vec{\bar\om}} 
\bigl(J,V_{\bar\tau}^{\bar\sa}(\bar W)\bigr)\bigr). 
\po 
$$ 
\item[{\rm(ii)}] 
Fix 
$$ 
p_{ij}
\left\{
\bal
{}      
&=(m+r)/(kr-i-jr),   &&\quad i+jr>kr-(m+r)/p,\\
&>p,                 &&\quad i+jr=kr-(m+r)/p,\\
&=p,                 &&\quad i+jr<kr-(m+r)/p, 
\eal
\right. 
\po 
$$ 
and set 
$$ 
\lda_{ij}:=\bar\lda-\lda-j\mu-(m+\mu)/p_{ij} 
\qb \vec\om_{ij}:=(\lda_{ij},\mu) 
\po 
$$ 
for 
\hb{i+jr\leq kr}. Suppose 
$$ 
a_{ij}\in L_{p_{ij}}
\bigl(J,L_{p_{ij}}^{\lda_{ij}}\bigl(T_{\tau_i}^{\sa_i} 
(M,\Hom(W,\bar W))\bigr)\bigr). 
\po 
$$ 
Then 
$$ 
\cA\in\cL\bigl(W_{\coW p}^{kr/\vec r,\vec\om}, 
L_p\bigl(J,L_p^{\bar\lda}(V_{\bar\tau}^{\bar\sa}(\bar W))\bigr)\bigr). 
$$ 
\item[{\rm(iii)}] 
In either case the map 
\hb{(a_{ij}\mt\cA)} is linear and continuous.\po 
\end{itemize} }  
\end{theorem} 
\begin{proof} 
(1) 
Theorem~\ref{thm-M.D} implies 
\beq\Label{M.ij} 
\na^i\pl^j 
\in\cL\bigl(\gF_q^{(s+kr)/\vec r,\vec\om}, 
\gF_q^{(s-i+(k-j)r)/\vec r,(\lda+j\mu,\mu)}(J,V_{\tau+i}^\sa)\bigr) 
\eeq 
and this is also true if $\gF_\iy$~is replaced by~$b_\iy$. Since 
$$ 
\gF_q^{(s-i+(k-j)r)/\vec r,(\lda+j\mu,\mu)}(J,V_{\tau+i}^\sa) 
\hr\gF_q^{s/\vec r,(\lda+j\mu,\mu)}(J,V_{\tau+i}^\sa) 
\npb 
$$ 
assertion~(i) follows from Theorem~\ref{thm-C.CC}. 

\smallskip 
(2) 
If 
\hb{i+jr>kr-(m+r)/p}, then we get from Theorem~\ref{thm-E.Sob}(i) 
$$ 
H_p^{(kr-i-jr)/\vec r,(\lda+j\mu,\mu)}(J,V_{\tau+i}^\sa) 
\hr L_{q_{ij}}\bigl(J,L_{q_{ij}}^{\bar\lda-\lda_{ij}}(V_{\tau+i}^\sa)\bigr), 
\npb 
$$ 
where 
\hb{1/q_{ij}:=1/p-1/p_{ij}}. 

\smallskip 
Suppose 
\hb{i+jr=kr-(m+r)/p}. Then 
\hb{p_{ij}>p} implies 
\hb{s:=i+jr+(m+r)/p_{ij}<kr}. Thus, invoking Theorem~\ref{thm-E.Sob}(i) 
once more, 
$$ 
H_p^{(kr-i-jr)/\vec r,(\lda+j\mu,\mu)}(J,V_{\tau+i}^\sa) 
\hr H_p^{(s-i-jr)/\vec r,(\lda+j\mu,\mu)}(J,V_{\tau+i}^\sa) 
\hr L_{q_{ij}}\bigl(J,L_{q_{ij}}^{\bar\lda-\lda_{ij}}(V_{\tau+i}^\sa)\bigr). 
$$ 

\smallskip 
If 
\hb{i+jr<kr-(m+r)/p}, then we deduce from Theorem~\ref{thm-E.Sob}(i) 
$$ 
H_p^{(kr-i-jr)/\vec r,(\lda+j\mu,\mu)}(J,V_{\tau+i}^\sa) 
\hr L_\iy\bigl(J,L_\iy^{\bar\lda-\lda_{ij}}(V_{\tau+i}^\sa)\bigr). 
$$ 

\smallskip 
Since 
\hb{q_{ij}=\iy} if 
\hb{p_{ij}=p} we get in either case from \Eqref{M.ij} 
$$ 
\na^i\pl^ju 
\in L_{q_{ij}}\bigl(J,L_{q_{ij}}^{\bar\lda-\lda_{ij}}(V_{\tau+i}^\sa)\bigr) 
=:L_{q_{ij}}(J,X_{ij}) 
\qa u\in H_p^{kr/\vec r,\vec\om}. 
$$ 
Note 
\hb{\bar\lda+\bar\tau-\bar\sa 
   =\lda_{ij}+\tau_i-\sa_i+\bar\lda-\lda_{ij}+\tau+i-\sa} implies, due to 
Lemma~\ref{lem-C.CE}, 
$$ 
\rho^{\bar\lda+\bar\tau-\bar\sa}\,|a_{ij}\cdot\na^i\pl^ju|_{\bar h} 
\leq c\rho^{\lda_{ij}+\tau_i-\sa_i}\,|a_{ij}|_{h_{ij}} 
\,\rho^{\bar\lda-\lda_{ij}+\tau+i-\sa}\,|\na^i\pl^ju|_{h_i}, 
$$ 
where 
\hb{\bar h:=\pr_{\bar\sa}^{\bar\tau}\otimes h_{\bar W}}, 
\ \hb{h_{ij}:=\pr_{\sa_i}^{\tau_i}\otimes h_{W\bar W}}, and 
\hb{h_i:=\pr_\sa^{\tau+i}\otimes h_W}. Hence, by H\"older's inequality, 
$$ 
\|a_{ij}\cdot\na^i\pl^ju\| 
_{L_p(J,L_p^{\bar\lda}(V_{\bar\tau}^{\bar\sa}(\bar W)))} 
\leq\|a_{ij}\|_{L_{p_{ij}}(J,Y_{ij})} 
\,\|\na^i\pl^ju\|_{L_{q_{ij}}(J,X_{ij})}, 
$$ 
where 
\hb{Y_{ij}:=L_{p_{ij}}^{\lda_{ij}} 
   \bigl(T_{\tau_i}^{\sa_i}\bigl(M,\Hom(W,\bar W)\bigr)\bigr)}. By combining 
this with \Eqref{M.ij} and using 
\hb{W_{\coW p}^{kr/\vec r,\vec\om}\doteq\gF_p^{kr/\vec r,\vec\om}} we get 
assertion~(ii). 

\smallskip 
(3) 
The last claim is obvious. 
\end{proof} 
It is clear which changes have to be made to get analogous results for 
`stationary' differential operators in the time-independent isotropic case. 
Details are left to the reader. 
\section{Extensions and Restrictions}\LabelT{sec-ER}%
In many situations it is easier to consider anisotropic function spaces on 
the whole line rather than on the half-line. Therefore we investigate in 
this section the possibility of extending half-line spaces to spaces on 
all of~$\BR$. 

\smallskip 
We fix 
\hb{h\in C^\iy\bigl((0,\iy),\BR\bigr)} satisfying 
\beq\Label{ER.h} 
\int_0^\iy t^s\,|h(t)|\,dt<\iy 
\qb s\in\BR 
\qa (-1)^k\int_0^\iy t^kh(t)\,dt=1 
\qb k\in\BZ, 
\eeq 
and 
\hb{h(1/t)=-th(t)} for 
\hb{t>0}. Lemma~4.1.1 of~\cite{Ama09a}, which is taken from~\cite{Ham75a}, 
guarantees the existence of such a function. 

\smallskip 
Let $\cX$ be a locally convex space. Then the \emph{point-wise restriction}, 
\beq\Label{ER.r} 
r^+\sco C\BRcX\ra C\BRpcX 
\qb u\mt u\sn\BR^+, 
\eeq 
is a continuous linear map. For 
\hb{v\in C\BRpcX} we set 
\beq\Label{ER.eps} 
\ve v(t):=\int_0^\iy h(s)v(-st)\,ds 
\qa t<0, 
\eeq 
and 
\beq\Label{ER.e} 
e^+v:=
\left\{
\bal
{}      
&&v      &\quad \text{on }\BR^+,\\
&&\ve v  &\quad \text{on }(-\iy,0).
\eal
\right.
\eeq 
It follows from \Eqref{ER.h} that $e^+$~is a 
continuous linear map from $C\BRpcX$ into $C\BRcX$, and 
\hb{r^+e^+=\id}. Thus point-wise restriction~\Eqref{ER.r} is a retraction, 
and $e^+$~is a coretraction. 

\smallskip 
By replacing~$\BR^+$ in \Eqref{ER.r} by~$-\BR^+$ and using obvious 
modifications we get the point-wise restriction~$r^-$ `to the negative 
half-line' and a corresponding extension operator~$e^-$. The \emph{trivial 
extension operator} 
$$ 
e_0^+\sco C_{(0)}\BRpcX:=\bigl\{\,u\in C\BRpcX\ ;\ u(0)=0\,\bigr\} 
\ra C\BRcX 
$$ 
is defined by 
\hb{e_0^+v:=v} on~$\BR^+$ and 
\hb{e_0^+v:=0} on~$(-\iy,0)$. Then 
\beq\Label{ER.r0} 
r_0^+:=r^+(1-e^-r^-)\sco C\BRcX\ra C_{(0)}\BRpcX 
\npb 
\eeq 
is a retraction, and $e_0^+$~is a coretraction. 

\smallskip 
We define: 
$$ 
\gF_p^{s/\vec r,\vec\om}\bigl((0,\iy),V\bigr) 
\text{ is the closure of $\cD\bigl((0,\iy),\cD\bigr)$ in } 
\gF_p^{s/\vec r,\vec\om}\BRpV. 
$$ 
Thus 
$$ 
\ci\gF_p^{s/\vec r,\vec\om}\BRpV 
\hr\gF_p^{s/\vec r,\vec\om}\bigl((0,\iy),V\bigr) 
\hr\gF_p^{s/\vec r,\vec\om}\BRpV . 
\npb 
$$ 
Now we can prove an extension theorem `from the half-cylinder 
\hb{M\times\BR^+} to the full cylinder 
\hb{M\times\BR}.'\po 
{\samepage 
\begin{theorem}\LabelT{thm-ER.ER} 
\begin{itemize} 
\item[{\rm(i)}] 
Suppose 
\hb{s\in\BR} where 
\hb{s>-1+1/p} if\/ 
\hb{\pl M\neq\es}. Then the diagram 
 \bspezeq 
  \bal 
 \begin{picture}(248,88)(-110,-41)        
 \put(0,40){\makebox(0,0)[b]{\small{$d$}}}
 \put(0,5){\makebox(0,0)[b]{\small{$d$}}}
 \put(0,-30){\makebox(0,0)[b]{\small{$d$}}}
 \put(-90,35){\makebox(0,0)[c]{\small{$\cD\BRpcD$}}}
 \put(-90,-35){\makebox(0,0)[c]{\small{$\cD\BRpcD$}}}
 \put(-40,0){\makebox(0,0)[c]{\small{$\cD\BRcD$}}}
 \put(53,0){\makebox(0,0)[c]{\small{$\gF_p^{s/\vec r,\vec\om}\BRV$}}}
 \put(111,35){\makebox(0,0)[c]{\small{$\gF_p^{s/\vec r,\vec\om}\BRpV$}}}
 \put(111,-35){\makebox(0,0)[c]{\small{$\gF_p^{s/\vec r,\vec\om}\BRpV$}}}
 \put(-53,17.5){\makebox(0,0)[r]{\small{$e^+$}}}
 \put(75,17.5){\makebox(0,0)[l]{\small{$e^+$}}}
 \put(-54,-17.5){\makebox(0,0)[r]{\small{$r^+$}}}
 \put(75,-17.5){\makebox(0,0)[l]{\small{$r^+$}}}
 \put(-95,0){\makebox(0,0)[r]{\small{$\id$}}}
 \put(116,0){\makebox(0,0)[l]{\small{$\id$}}}
 \put(-64,35){\vector(1,0){141}}
 \put(-64,37){\oval(4,4)[l]}
 \put(-18,0){\vector(1,0){40}}
 \put(-18,2){\oval(4,4)[l]}
 \put(-64,-35){\vector(1,0){141}}
 \put(-64,-33){\oval(4,4)[l]}
 \put(-90,25){\vector(0,-1){50}}
 \put(111,25){\vector(0,-1){50}}
 \put(-80,25){\vector(1,-1){20}}
 \put(-60,-5){\vector(-1,-1){20}}
 \put(101,25){\vector(-1,-1){20}}
 \put(81,-5){\vector(1,-1){20}}
 \end{picture} 
 \eal 
 \npb 
\espezeq 
is commuting.\po 
\item[{\rm(ii)}] 
If 
\hb{s>0}, then 
 \bspezeq 
 \bal 
 \begin{picture}(261,88)(-117,-41)        
 \put(0,40){\makebox(0,0)[b]{\small{$d$}}}
 \put(0,5){\makebox(0,0)[b]{\small{$d$}}}
 \put(0,-30){\makebox(0,0)[b]{\small{$d$}}}
 \put(-90,35){\makebox(0,0)[c]{\small{$\cD\bigl((0,\iy),\cD\bigr)$}}}
 \put(-90,-35){\makebox(0,0)[c]{\small{$\cD\bigl((0,\iy),\cD\bigr)$}}}
 \put(-40,0){\makebox(0,0)[c]{\small{$\cD\BRcD$}}}
 \put(53,0){\makebox(0,0)[c]{\small{$\gF_p^{s/\vec r,\vec\om}\BRV$}}}
 \put(111,35){\makebox(0,0)[c]{\small{$\gF_p^{s/\vec r,\vec\om} 
 \bigl((0,\iy),V\bigr)$}}}
 \put(111,-35){\makebox(0,0)[c]{\small{$\gF_p^{s/\vec r,\vec\om} 
 \bigl((0,\iy),V\bigr)$}}}
 \put(-53,17.5){\makebox(0,0)[r]{\small{$e_0^+$}}}
 \put(75,17.5){\makebox(0,0)[l]{\small{$e_0^+$}}}
 \put(-54,-17.5){\makebox(0,0)[r]{\small{$r_0^+$}}}
 \put(75,-17.5){\makebox(0,0)[l]{\small{$r_0^+$}}}
 \put(-95,0){\makebox(0,0)[r]{\small{$\id$}}}
 \put(116,0){\makebox(0,0)[l]{\small{$\id$}}}
 \put(-58,35){\vector(1,0){129}}
 \put(-58,37){\oval(4,4)[l]}
 \put(-18,0){\vector(1,0){40}}
 \put(-18,2){\oval(4,4)[l]}
 \put(-58,-35){\vector(1,0){129}}
 \put(-58,-33){\oval(4,4)[l]}
 \put(-90,25){\vector(0,-1){50}}
 \put(111,25){\vector(0,-1){50}}
 \put(-80,25){\vector(1,-1){20}}
 \put(-60,-5){\vector(-1,-1){20}}
 \put(101,25){\vector(-1,-1){20}}
 \put(81,-5){\vector(1,-1){20}}
 \end{picture} 
 \eal 
 \npb 
 \espezeq 
is a commuting diagram as well.\po 
\end{itemize} 
\end{theorem} }
\begin{proof} 
(1)
Suppose 
\beq\Label{ER.M} 
M=(\BX,g_m)\text{ with }\BX\in\{\BR^m,\BH^m\} 
\qb \rho=\mf{1} 
\qb W=\BX\times F 
\qb D=d_F. 
\eeq 
If 
\hb{k\in\BN}, then it is not difficult to see that $r^+$~is a retraction 
from 
\hb{W_{\coW p}^{kr/\vec r}(\BX\times\BR,E)} onto 
\hb{W_{\coW p}^{kr/\vec r}(\BX\times\BR^+,E)}, and $e^+$~is a coretraction. 
(cf.~steps (1) and~(2) of the proof of Theorem~4.4.3 of~\cite{Ama09a}). 
Thus, if 
\hb{s>0}, the first assertion follows by interpolation.  

\smallskip 
(2) 
Let \Eqref{ER.M} be satisfied. Suppose 
\hb{s>0} and 
\hb{J=\BR^+}. It is an easy consequence of 
\beq\Label{ER.FLF} 
\gF_p^{s/\vec r}\JV 
=L_p\bigl(J,\gF_p^s(V)\bigr)\cap\gF_p^{s/r}\bigl(J,L_p(V)\bigr)\ph{.} 
\eeq 
that 
\beq\Label{ER.FLF0} 
\gF_p^{s/\vec r}\ciJV 
=L_p\bigl(J,\gF_p^s(V)\bigr)\cap\gF_p^{s/r}\bigl(\ci J,L_p(V)\bigr). 
\eeq 
From this it is obvious that 
$$ 
e_0^+\in\cL\bigl(\gF_p^{s/\vec r}\ciJV,\gF_p^{s/\vec r}\BRV\bigr). 
$$ 
Note that 
\hb{L_p(V)=L_p\BXE} is a UMD space 
(e.g.;~\cite[Theorem~III.4.5.2]{Ama95a}). Hence 
\cite[Lemma~4.1.4]{Ama09a}, definition~\Eqref{ER.r0}, and the 
arguments of step~(1) show 
$$ 
r_0^+\in\cL\bigl( \gF_p^{s/r}\bigl(\BR,L_p(V)\bigr), 
\gF_p^{s/r}\bigl(\ci J,L_p(V)\bigr)\bigr). 
\npb 
$$ 
From this, \Eqref{ER.FLF}, and \Eqref{ER.FLF0} we deduce assertion~(ii) 
in this setting. 

\smallskip 
(3) 
Assume \Eqref{ER.M} and 
\hb{s<0} with 
\hb{s>-1+1/p} if 
\hb{\BX=\BH^m}. Then 
\hb{\gF_{p'}^{-s}(V')=\ci\gF_{p'}^{-s}(V')} by 
Theorem~4.7.1(ii) of~\cite{Ama09a}. Hence 
$$ 
\bal 
\gF_{p'}^{-s/\vec r}\ciJVs 
&=L_{p'}\bigl(J,\gF_{p'}^{-s}(V')\bigr) 
 \cap\gF_{p'}^{-s/r}\bigl(\ci J,L_{p'}(V')\bigr)  
 =L_{p'}\bigl(J,\ci\gF_{p'}^{-s}(V')\bigr) 
 \cap\gF_{p'}^{-s/r}\bigl(\ci J,L_{p'}(V')\bigr)\\ 
&=\ci\gF_{p'}^{s/\vec r}\JVs. 
\eal 
$$ 
Thus, by \Eqref{A.def1}, 
$$ 
\gF_p^{s/\vec r}\JV\doteq\bigl(\gF_{p'}^{-s/\vec r}\ciJVs\bigr)'. 
$$ 
The results of Section~4.2 of~\cite{Ama09a} imply~$r^+$, 
respectively~$e^+$, is the dual of~$e_0^+$, respectively~$r_0^+$. From this 
and step~(2) it follows (see \cite[(4.2.3)]{Ama09a} that assertion~(i) holds 
in the present setting if 
\hb{s<0}, provided 
\hb{s>-1+1/p} if 
\hb{\BX=\BH^m}. 
 
\smallskip 
(4) 
It follows from \Eqref{RA.tdef} and \hbox{\Eqref{ER.r}--\Eqref{ER.e}} 
that 
$$ 
r^+\circ\Ta_{q,\ka}^\mu=\Ta_{q,\ka}^\mu\circ r^+ 
\qb e^+\circ\Ta_{q,\ka}^\mu=\Ta_{q,\ka}^\mu\circ e^+  
$$ 
for 
\hb{1\leq q\leq\iy}. Hence 
\beq\Label{ER.rot} 
r_0^+\circ\Ta_{q,\ka}^\mu=\Ta_{q,\ka}^\mu\circ r_0^+ 
\qb e_0^+\circ\Ta_{q,\ka}^\mu=\Ta_{q,\ka}^\mu\circ e_0^+.  
\eeq 
Thus 
$$ 
\vp_{q,\ka}^{\vec\om}(r^+u) 
=\rho_\ka^{\lda+m/q}\Ta_{q,\ka}^\mu(\ka\slt\vp)_*(\pi_\ka r^+u) 
=r^+(\vp_{q,\ka}^{\vec\om}u) 
$$ 
and, similarly, 
$$ 
\psi_{q,\ka}^{\vec\om}(r^+v_\ka) 
=r^+(\psi_{q,\ka}^{\vec\om}v_\ka). 
$$ 
This implies that $\vp_q^{\vec\om}$ and~$\psi_q^{\vec\om}$ commute with 
$r^+$, $r_0^+$, $e^+$, and~$e_0^+$. Hence the statements follow from 
steps \hbox{(1)--(3)} and 
Theorem~\ref{thm-RA.R}. 
\end{proof} 
The next theorem concerns the extension of Besov-H\"older spaces from 
half- to full cylinders. 
\begin{theorem}\LabelT{thm-ER.B} 
Suppose either 
\hb{s\in r\BN} and 
\hb{\cB\in\{BC,bc\}}, or 
\hb{s>0} and 
\hb{\cB\in\{B_\iy,b_\iy\}}. Then $r^+$~is a retraction from 
$\cB^{s/\vec r,\vec\om}\BRV$ onto 
$\cB^{s/\vec r,\vec\om}\BRpV$, and $e^+$~is a coretraction. 
\end{theorem} 
\begin{proof} 
(1) 
Let 
\hb{k\in\BN} and 
\hb{\cB\in\{BC,bc\}}. It is obvious that 
$$ 
r^+\in\cL\bigl(\cB^{kr/\vec r,\vec\om}\BRV, 
\cB^{kr/\vec r,\vec\om}\BRpV\bigr). 
$$ 
It follows from \Eqref{ER.h} that 
\hb{\ve\in\cL\bigl(\cB^{kr/\vec r,\vec\om}\BRpV, 
   \cB^{kr/\vec r,\vec\om}\mBRpV\bigr)}. Thus, by the second part of 
\Eqref{ER.h} and \Eqref{ER.e}, 
$$ 
e^+\in\cL\bigl(\cB^{kr/\vec r,\vec\om}\BRpV, 
\cB^{kr/\vec r,\vec\om}\BRV\bigr).  
\npb 
$$  
From this we get the assertion in this case. 

\smallskip 
(2) 
If 
\hb{s>0} and 
\hb{\cB\in\{B_\iy,b_\iy\}}, then, due to Corollary~\ref{cor-H.R}, we obtain 
the statement by interpolation from the results of step~(1). 
\end{proof} 
Lastly, we consider little Besov-H\"older spaces `with vanishing initial 
values'. They are defined as follows: If 
\hb{k\in\BN}, then 
\beq\Label{ER.bc} 
\bal 
{}      
&u\in bc^{kr/\vec r,\vec\om}\bigl((0,\iy),V\bigr)\text{ iff}\\ 
&u\in bc^{kr/\vec r,\vec\om}\BRpV\text{ and $\pl^ju(0)=0$ for } 
 0\leq j\leq k. 
\eal 
\eeq 
Furthermore, 
\hb{b_\iy^{s/\vec r,\vec\om}\bigl((0,\iy),V\bigr)} is defined by 
\beq\Label{ER.bc1} 
\left\{ 
\bal 
{}          
&\bigl(bc^{kr/\vec r,\vec\om}((0,\iy),V), 
 bc^{(k+1)r/\vec r,\vec\om}((0,\iy),V)\bigr)_{(s-kr)/r,\iy}^0,
    &&\quad  &kr<s &<(k+1)r,\\ 
&\bigl(bc^{kr/\vec r,\vec\om}((0,\iy),V), 
 bc^{(k+2)r/\vec r,\vec\om}((0,\iy),V)\bigr)_{1/2,\iy}^0, 
    &&\quad  &  s  &=(k+1)r, 
\eal 
\right. 
\npb 
\eeq 
where 
\hb{k\in\BN}. 
\begin{theorem}\LabelT{thm-ER.b} 
Let 
\hb{k\in\BN} and 
\hb{s>0}. Then $r_0^+$~is a retraction from 
$bc^{kr/\vec r,\vec\om}\BRV$ onto 
$bc^{kr/\vec r,\vec\om}\bigl((0,\iy),V\bigr)$ and from 
$b_\iy^{s/\vec r,\vec\om}\BRV$ onto 
$b_\iy^{s/\vec r,\vec\om}\bigl((0,\iy),V\bigr)$, and $e_0^+$~is a 
coretraction.
\end{theorem} 
\begin{proof} 
It is easily seen by \Eqref{ER.r0} and the preceding theorem that the 
assertion is true for $bc^{kr/\vec r,\vec\om}$ spaces. The stated 
results in the remaining cases now follow by interpolation.
\end{proof} 
\section{Trace Theorems}\LabelT{sec-T}%
Suppose $\Ga$~is a union of connected components of~$\pl M$.  We denote by 
\hb{\thia\sco\Ga\hr M} the natural injection and endow~$\Ga$ with 
the induced Riemannian metric 
\hb{\thg:=\thia\,^*g}. Let $\rgK$ be a singularity datum for~$M$. For 
\hb{\ka\in\gK_\Ga} we put 
\hb{U_{\coU\ithka}:=\pl U_{\coU\ka}=U_{\coU\ka}\cap\Ga} and 
\hb{\thka:=\ia_0\circ(\thia\,^*\ka)\sco U_{\coU\ithka}\ra\BR^{m-1}}, where 
\hb{\ia_0\sco\{0\}\times\BR^{m-1}\ra\BR^{m-1}}, 
\ \hb{(0,x')\mt x'}. Then 
\hb{\thgK:=\{\,\thka\ ;\ \ka\in\gK_\Ga\,\}} is 
a normalized atlas for~$\Ga$, the one induced by~$\gK$. We set 
\hb{\thrho:=\thia\,^*\rho=\rho\sn\Ga}. It follows that $(\thrho,\thgK)$ 
is a singularity datum for~$\Ga$, so that $\Ga$~is 
singular of type~$[\![\thrho]\!]$. Henceforth, it is understood that 
$\Ga$~is given this singularity structure \emph{induced by}~$\gT(M)$. 

\smallskip 
We denote by 
\hb{\thW=W_\Ga} the restriction of~$W$ to~$\Ga$ and by 
\hb{h_\ithW:=\thia\,^*h_W} the bundle metric on~$\Ga$ induced by~$h_W$. 
Furthermore, the connection~$D_\ithW$ for~$\thW$, induced by~$D$, is defined 
by restricting 
$$ 
D\sco\cT M\times C^\iy\MW\ra C^\iy\MW 
\quad\text{to}\quad  
\cT\Ga\times C^\iy\GathW, 
$$ 
considered as a map into $C^\iy\GathW$. Then 
\hb{\thW=(\thW,h_\ithW,D_\ithW)} is a fully  uniformly regular vector 
bundle over~$\Ga$. 

\smallskip 
We set 
\hb{\thV:=T_\tau^\sa\GathW} and endow it with the bundle metric 
\hb{\thh:=\prsn_{T_\tau^\sa\Ga}\otimes h_\ithW}, where 
\hb{\prsn_{T_\tau^\sa\Ga}} is the bundle metric on~$T_\tau^\sa\Ga$ 
induced by~$\thg$. Then we equip~$\thV$ with the metric connection 
\hb{\thna:=\na(\na_{\cona\cona\ithg},D_\ithW)}. Hence 
\hb{\thV=(\thV,\thh,\thna)}. It follows that 
$\gF_p^{s/\vec r,\vec\om}\JthV$ is a well-defined anisotropic 
weighted space with respect to the boundary weight function~$\thrho$. 

\smallskip 
We write 
\hb{\mf{n}=\mf{n}(\Ga)} for the inward pointing unit normal 
on~$\Ga$. In local coordinates, 
\hb{\ka=(x^1,\ldots,x^m)},
\beq\Label{T.n} 
\mf{n} 
=\bigl(\sqrt{g_{11}\sn\pl U_{\coU\ka}}\,\bigr)^{-1}\,\frac\pl{\pl x^1}.
\eeq 

\smallskip 
Let
\hb{u\in\cD=\cD\MV} and
\hb{k\in\BN}. The \emph{trace of order}~$k$ \emph{of}~$u$ \emph{on}~$\Ga$, 
\ \hb{\pl_{\mf{n}}^ku=\pl_{\mf{n}(\Ga)}^ku\in\cD\GathV}, is defined by 
\beq\Label{T.kn} 
\dl\pl_{\mf{n}}^ku,a\dr_{\ithV\vph{V}^*} 
:=\bigl\dl\na^ku\sn\Ga,a\otimes\mf{n}^{\otimes k}\bigr\dr_{\ithV\vph{V}^*} 
\qa a\in\cD(\Ga,\thV\vph{V}^*). 
\eeq 
We also set 
\hb{\ga_\Ga:=\pl_{\mf{n}(\Ga)}^0} and call it \emph{trace operator 
on}~$\Ga$. We write again 
\hb{\pl_{\mf{n}}^k=\pl_{\mf{n}(\Ga)}^k} for the point-wise extension 
of~$\pl_{\mf{n}(\Ga)}^k$ over~$J$, that is, 
\hb{(\pl_{\mf{n}}^ku)(t):=\pl_{\mf{n}}^k\bigl(u(t)\bigr)} for 
\hb{t\in J} and 
\hb{u\in\cD\JcD}, and call it \emph{lateral trace operator of order}~$k$ 
\emph{on} 
\hb{\Ga\times J}. Correspondingly, the \emph{lateral trace operator on} 
\hb{\Ga\times J} is the point-wise extension of~$\ga_\Ga$, denoted 
by~$\ga_\Ga$ as well. Moreover,  
$$ 
\pl_{\mf{n},0}^k 
\sco\cD\bigl((0,\iy),\cD\bigr)\ra\cD\bigl((0,\iy),\cD\GathV\bigr) 
\qb u\mt\pl_{\mf{n}}^ku 
\npb 
$$ 
is the restriction of~$\pl_{\mf{n}}^k$ to $\cD\bigl((0,\iy),\cD\bigr)$. 

\smallskip 
Assume 
\hb{J=\BR^+}. Then 
\hb{M_0:=M\times\{0\}} is the \emph{initial boundary} of the space-time 
\hbox{(half-)}cylinder 
\hb{M\times\BR^+}. The \emph{initial trace operator} is the linear map 
$$ 
\ga_{M_0}\sco\cD\BRpcD\ra\cD 
\qb u\mt u(0), 
$$ 
where $M_0$~is identified with~$M$. Furthermore, 
$$ 
\pl_{t=0}^k:=\ga_{M_0}\circ\pl^k 
\sco\cD\BRpcD\ra\cD 
\qb u\mt(\pl^ku)(0) 
\npb 
$$ 
is the \emph{initial trace operator of order}~$k$. 

\smallskip 
Suppose 
\hb{s_0>1/p}. The following theorem shows, in particular, that there exists 
a unique 
$$ 
(\ga_\Ga)_{s_0} 
\in\cL\bigl(\gF_p^{s_0/\vec r,\vec\om}, 
B_p^{(s_0-1/p)/\vec r,(\lda+1/p,\mu)}\JthV\bigr) 
$$ 
extending~$\ga_\Ga$ and being a retraction. Furthermore, there exists a 
coretraction~$(\ga_\Ga^c)_{s_0}$ such that, for each 
\hb{s\in\BR}, there is 
$$ 
(\ga_\Ga^c)_s  
\in\cL\bigl(B_p^{(s-1/p)/\vec r,(\lda+1/p,\mu)}\JthV, 
\gF_p^{s/\vec r,\vec\om}\bigr) 
$$ 
such that 
\beq\Label{T.ucr} 
\bal 
{\rm(i)} 
&\quad (\ga_\Ga^c)_s\bsn\cD\bigl(J,\cD\GathV\bigr) 
      =(\ga_\Ga^c)_{s_0}\bsn\cD\bigl(J,\cD\GathV\bigr),\\ 
{\rm(ii)} 
&\quad (\ga_\Ga^c)_s\text{ is for each $s>1/p$ a coretraction for } 
 (\ga_\Ga)_s. 
\eal 
\eeq 
Thus $(\ga_\Ga)_{s_0}$~is for each 
\hb{s_0>1/p} uniquely determined by~$\ga_\Ga$ and $(\ga_\Ga^c)_s$ can be 
obtained for any 
\hb{s\in\BR} by unique continuous extension or 
restriction of~$(\ga_\Ga^c)_{s_0}$ for any 
\hb{s_0>1/p}. Hence we simply write $\ga_\Ga$ and~$\ga_\Ga^c$ for 
$(\ga_\Ga)_s$ and~$(\ga_\Ga^c)_s$, respectively, without fearing 
confusion. So we can say $\ga_\ga^c$~is a \emph{universal coretraction} 
for the retraction 
$$ 
\ga_\Ga 
\in\cL\bigl(\gF_p^{s/\vec r,\vec\om}, 
B_p^{(s-1/p)/\vec r,(\lda+1/p,\mu)}\JthV\bigr) 
\qa s>1/p, 
$$ 
herewith expressing properties~\Eqref{T.ucr}. Similar conventions hold for 
higher order trace operators and traces occurring below.\po 
{\samepage 
\begin{theorem}\LabelT{thm-T.T} 
Suppose 
\hb{k\in\BN}.  
\begin{itemize} 
\item[{\rm(i)}] 
Assume 
\hb{\Ga\neq\es} and 
\hb{s>k+1/p}. Then $\pl_{\mf{n}}^k$~is a retraction 
$$ 
\text{from $\gF_p^{s/\vec r,(\lda,\mu)}\JV$ 
onto }B_p^{(s-k-1/p)/\vec r,(\lda+k+1/p,\mu)}\JthV. 
\npb 
$$ 
It possesses a universal coretraction~$(\ga_{\mf{n}}^k)^c$ satisfying 
\hb{\pl_{\mf{n}}^i\circ(\ga_{\mf{n}}^k)^c=0} for 
\hb{0\leq i\leq k-1}.\po 
\item[{\rm(ii)}] 
Suppose 
\hb{s>r(k+1/p)}. Then $\pl_{t=0}^k$~is a retraction 
$$ 
\text{from $\gF_p^{s/\vec r,(\lda,\mu)}\BRpV$ 
onto }B_p^{s-r(k+1/p),\lda+\mu(k+1/p)}(V). 
\npb 
$$ 
There exists a universal coretraction~$(\ga_{t=0}^k)^c$ such that 
\hb{\pl_{t=0}^i\circ(\ga_{t=0}^k)^c=0} for 
\hb{0\leq i\leq k-1}.\po 
\item[{\rm(iii)}] 
Let 
\hb{\Ga\neq\es} and 
\hb{s>k+1/p}. Then $\pl_{\mf{n},0}^k$~is a retraction 
\beq\Label{T.F0} 
\text{from $\gF_p^{s/\vec r,(\lda,\mu)}\bigl((0,\iy),V\bigr)$ 
onto }B_p^{(s-k-1/p)/\vec r,(\lda+k+1/p,\mu)}\bigl((0,\iy),\thV\bigr). 
\npb 
\eeq  
\end{itemize} 
The restriction of~$(\ga_{\mf{n}}^k)^c$ to the space on the right side of\/ 
\Eqref{T.F0} is a universal coretraction.\po 
\end{theorem} }
\begin{proof} 
(1) 
Suppose 
\hb{\BX\in\{\BR^m,\BH^m\}}, 
\ \hb{M=(\BX,g_m)}, 
\ \hb{\rho=\mf{1}}, 
\ \hb{W=\BX\times F}, and 
\hb{D=d_F} so that 
\hb{V=\BX\times E}. Put 
\hb{\BY:=\BX\times J}. Assume either 
\hb{\Ga\neq\es} or 
\hb{J=\BR^+}. If 
\hb{J=\BR}, then 
\hb{M\times J=\BY=\BH^{m+1}}. If 
\hb{J=\BR^+} and 
\hb{\Ga=\es}, then 
\hb{M\times J=\BR^m\times\BR^+\simeq\BH^{m+1}}. Finally, if 
\hb{J=\BR^+} and 
\hb{\Ga\neq\es}, then 
$$ 
M\times J=\BH^m\times\BR^+\simeq\BR^+\times\BR^+\times\BR^{m-1}, 
$$ 
that is, 
\hb{M\times J} is a closed \hbox{$2$-corner} in the sense of Section~4.3 
of~\cite{Ama09a}. In each case 
\hb{{}\simeq{}}~is simply a permutation diffeomorphism. 

\smallskip 
If either 
\hb{J=\BR} or 
\hb{\Ga=\es}, then assertions (i) and~(ii) follow from Theorem~4.6.3 
of~\cite{Ama09a}. If 
\hb{J=\BR^+} and 
\hb{\Ga\neq\es}, then assertion~(i) follows from Theorem~4.6.3 and the 
definition of the trace operator for a face of 
\hb{\BR^+\times\BR^+\times\BR^{m-1}}, that is, formula~(4.10.12) 
of~\cite{Ama09a}. Claim~(iii) is a consequence of Theorem~4.10.3 
of~\cite{Ama09a} \ (choose any~$\ka$ therein with 
\hb{\ka>s+1}). 

\smallskip 
(2) 
Now we consider the general case. Suppose 
\hb{\Ga\neq\es}. For 
\hb{t>1/p} we set 
$$ 
\thB_{p,\ka}^{(t-1/p)/\vec r}:= 
\left\{ 
\bal 
{}      
&B_p^{(t-1/p)/\vec r}(\pl\BY_\ka,E)&&\quad\text{if }\ka\in\gK_\Ga,\\ 
&\{0\}                             &&\quad \text{otherwise}. 
\eal 
\right. 
$$ 
Let $\ga_\ka$ be the trace operator on 
\hb{\pl\BY_\ka=\{0\}\times\BR^{m-1}\times J} if 
\hb{\ka\in\gK_\Ga}, 
and 
\hb{\ga_\ka:=0} otherwise. Set 
$$ 
\ga_{k,\ka}:=\rho_\ka^k\bigl(\sqrt{\ga_\ka(\ka_*g_{11})}\,\bigr)^{-k} 
\ga_\ka\circ\pl_1^k 
\qa \ka\in\gK. 
$$ 
It follows from step~(1), \,\Eqref{T.n}, and \Eqref{T.kn} that  
\hb{\ga_\ka\circ\pl_1^k} is a retraction from~$\gF_{p,\ka}^{s/\vec r}$ 
onto~$\thB_{p,\ka}^{(s-1/p)/\vec r}$ and that there exists a universal 
coretraction~$\wt{\ga}_{k,\ka}^c$  satisfying 
\beq\Label{T.gig} 
(\ga_\ka\circ\pl_1^i)\circ\wt{\ga}_{k,\ka}^c=0 
\qa 0\leq i\leq k-1, 
\eeq 
(setting 
\hb{\wt{\ga}_{k,\ka}^c:=0} if 
\hb{\ka\in\gK\ssm\gK_\Ga}). We put 
$$ 
\ga_{k,\ka}^c:=\rho_\ka^{-k}\bigl(\sqrt{\ga_\ka(\ka_*g_{11})}\,\bigr)^k 
\,\wt{\ga}_{k,\ka}^c 
\qa \ka\in\gK. 
$$ 
Then \Eqref{U.Rr} and \Eqref{S.sd} imply 
$$ 
\ga_{k,\ka}\in\cL(\gF_{p,\ka}^{s/\vec r},\thB_{p,\ka}^{(s-k-1/p)/\vec r}) 
\qb \ga_{k,\ka}^c\in\cL(\thB_{p,\ka}^{(s-k-1/p)/\vec r}, 
\gF_{p,\ka}^{s/\vec r}) 
$$ 
and 
$$ 
\|\ga_{k,\ka}\|+\|\ga_{k,\ka}^c\|\leq c 
\qa \ka\in\gK. 
$$ 
From \Eqref{T.gig} and Leibniz' rule we thus infer 
\beq\Label{T.gik} 
\ga_{i,\ka}\circ\ga_{k,\ka}^c=\da_{ik}\id 
\qa 0\leq i\leq k. 
\eeq 

\smallskip 
(3) 
We set 
\hb{(\thpi_{\ithka},\thchi_{\ithka}):=(\pi_\ka,\chi_\ka)\sn U_{\coU\ithka}} 
for 
\hb{\thka\in\thgK}. Then it is verified that 
\hb{\bigl\{\,(\thpi_{\ithka},\thchi_{\ithka})\ ;\ \thka\in\thgK\,\bigr\}} 
is a localization system subordinate to~$\thgK$. We denote by 
$$ 
\thpsi_p^{\vec\om} 
\sco\ell_p(\thmfB_p^{(s-k-1/p)/\vec r}) 
\ra B_p^{(s-k-1/p)/\vec r,\vec\om}\JthV 
\npb 
$$
the `boundary retraction' defined analogously to~$\psi_p^{\vec\om}$. 
Correspondingly, $\thvp_p^{\vec\om}$~is the `boundary coretraction'. 

\smallskip 
We write 
\hb{\thka\slt\thvp} for the restriction of 
\hb{\ka\slt\vp\in\gK\slt\Phi} to~$\Ga$ and put 
$$ 
C_{k,\ka} 
:=\thrho_{\ithka}^k(\thka\slt\thvp)_*\circ\pl_{\mf{n}}^k\circ(\ka\slt\vp)^* 
\qa \ka\slt\vp\in\gK_\Ga\slt\Phi, 
$$ 
and 
\hb{C_{k,\ka}:=0} otherwise. Note 
\hb{\thrho_{\ithka}=\rho_\ka} for 
\hb{\ka\in\gK_\Ga}. It follows from \Eqref{L.Dr}, \,\Eqref{T.n}, and 
\Eqref{T.kn} that 
\beq\Label{T.Cck} 
C_{k,\ka}v=\ga_{k,\ka}v+\sum_{\ell=0}^{k-1}a_{\ell,\ka}\ga_{\ell,\ka}v 
\qa v\in\cD\bigl(J,\cD(\pl\BX_\ka,E)\bigr), 
\eeq 
and \Eqref{L.abd} implies 
\hb{\|a_{\ell,k}\|_{k-1,\iy}\leq c} for 
\hb{0\leq\ell\leq k-1} and 
\hb{\ka\in\gK}. Hence, using 
\hb{\gF_{p,\ka}^{s/\vec r}\hr\gF_{p,\ka}^{(s-k+\ell)/\vec r}} and 
Theorem~\ref{thm-P.M}, we find 
\beq\Label{T.Cc} 
C_{k,\ka}\in\cL(\gF_{p,\ka}^{s/\vec r},\thB_{p,\ka}^{(s-k-1/p)/\vec r}) 
\qb \|C_{k,\ka}\|\leq c 
\qa \ka\in\gK. 
\eeq 

\smallskip 
(4) 
For 
\hb{u\in\cD\JcD} 
\beq\Label{T.pDk}
\thpi_{\ithka}\pl_{\mf{n}}^ku 
=\pl_{\mf{n}}^k(\pi_\ka u)-\sum_{j=0}^{k-1} 
\bid kj(\pl_{\mf{n}}^{k-j}\pi_\ka)\pl_{\mf{n}}^j(\chi_\ka u) 
\qa \ka\in\gK, 
\eeq 
setting 
\hb{\pl_{\mf{n}}^kv:=0} if 
\hb{\supp(v)\cap\Ga=\es}. Note 
\beq\Label{T.Cph} 
\bal 
{}      
&\thrho_{\ithka}^{\lda+k+1/p+(m-1)/p}\Ta_{p,\ithka}^\mu(\thka\slt\thvp)_* 
 \bigl(\pl_{\mf{n}}^k(\pi_\ka u)\bigr)\\  
&\qquad\qquad={} 
 \thrho_\ka^k(\thka\slt\thvp)_*\circ\pl_{\mf{n}}^k\circ(\ka\slt\vp)^* 
 \bigl(\rho_\ka^{\lda+m/p}\Ta_{p,\ka}^\mu(\ka\slt\vp)_*(\pi_\ka u)\bigr)
 =C_{k,\ka}(\vp_{p,\ka}^{\vec\om}u), 
\eal 
\eeq 
since 
\hb{\Ta_{p,\ithka}^\mu=\Ta_{p,\ka}^\mu} for 
\hb{\ka\in\gK_\Ga}. Similarly, using \Eqref{P.TS} and \Eqref{P.a} also, 
$$ 
\bal 
{}      
&\thrho_{\ithka}^{\lda+k+1/p+(m-1)/p}\Ta_{p,\ithka}^\mu(\thka\slt\thvp)_* 
 \bigl((\pl_{\mf{n}}^{k-j}\pi_\ka)\pl_{\mf{n}}^j(\chi_\ka u)\bigr)\\ 
&\qquad\qquad={} 
 C_{k-j,\ka}(\ka_*\pi_\ka)C_{j,\ka}(\wh{\vp}_{p,\ka}^{\kern1pt\vec\om}u) 
 =\sum_{\wt{\ka}\in\gN(\ka)} 
 C_{k-j,\ka}(\ka_*\pi_\ka)C_{j,\ka}(a_{\wt{\ka}\ka}R_{\wt{\ka}\ka} 
 \vp_{p,\wt{\ka}}^{\kern1pt\vec\om}u). 
\eal 
$$ 
From this, \Eqref{T.pDk}, and \Eqref{T.Cph} we get 
\beq\Label{T.phC} 
\thvp_{p,\ithka}^{(\lda+k+1/p,\mu)}(\pl_{\mf{n}}^ku) 
=C_{k,\ka}(\vp_{p,\ka}^{\vec\om}u) 
+\sum_{\wt{\ka}\in\gN(\ka)}A_{k-1,\wt{\ka}\ka}(\vp_{p,\wt{\ka}}^{\vec\om}u), 
\eeq 
where 
$$ 
A_{k-1,\wt{\ka}\ka}
:=\sum_{i=0}^{k-1}b_{i,\wt{\ka}\ka}C_{i,\ka}\circ R_{\wt{\ka}\ka} 
\qb b_{i,\wt{\ka}\ka} 
:=-\sum_{j=i}^{k-1}\bid kj\bid jiC_{k-j,\ka} 
(\ka_*\pi_\ka)C_{j-i}a_{\wt{\ka}\ka}. 
$$ 

\smallskip 
It is obvious that 
$$ 
C_{\ell,\ka} 
\in\cL\bigl(BC_\ka^{n+\ell},BC^n(\pl\BX_\ka,E)\bigr) 
\qb \|C_{\ell,\ka}\|\leq c(n) 
\qa \ka\in\gK 
\qb 0\leq\ell\leq k 
\qb n\in\BN.  
$$ 
From this, \Eqref{R.LS}, \,\Eqref{P.abd}, and Theorem~\ref{thm-P.M} we 
obtain 
$$ 
b_{i,\wt{\ka}\ka}\in BC^\iy(\pl\BX_\ka,E) 
\qb \|b_{i,\wt{\ka}\ka}\|_{n,\iy}\in c 
\qa \wt{\ka}\in\gN(\ka) 
\qb \ka\in\gK 
\qb 0\leq i\leq k 
\qb n\in\BN. 
$$ 
Hence, using Theorem~\ref{thm-P.M} once more, we get from \Eqref{T.Cc} 
and Lemma~\ref{lem-P.T} 
\beq\Label{T.Akk} 
A_{k-1,\wt{\ka}\ka} 
\in\cL(\gF_{p,\wt{\ka}}^{s/\vec r},\thB_{p,\ka}^{(s-k-1/p)/\vec r}) 
\qb \|A_{k-1,\wt{\ka}\ka}\|\leq c  
\qa \wt{\ka}\in\gN(\ka) 
\qb \ka\in\gK. 
\eeq 

\smallskip 
(5) 
We define~$\mf{C}_k$ by 
$$ 
\mf{C}_k\mf{v}:=\Bigl(C_{k,\ka}v_\ka 
+\sum_{\wt{\ka}\in\gN(\ka)}A_{k-1,\wt{\ka}\ka}v_{\wt{\ka}}\Bigr)_{\ka\in\gK}  
\qa \mf{v}=(v_\ka). 
$$ 
Then we deduce from \Eqref{T.Cc}, \,\Eqref{T.Akk}, and the finite 
multiplicity of~$\gK$ 
\beq\Label{T.C} 
\mf{C}_k\in\cL\bigl(\ell_p(\mf{\gF}_p^{s/\vec r}), 
\ell_p(\thmfB_p^{(s-k-1/p)/\vec r})\bigr). 
\eeq 
Employing \Eqref{T.gik} and \Eqref{T.Cck} we infer 
\hb{C_{k,\ka}\circ\ga_{k,\ka}^c=\id}. Furthermore, recalling 
\Eqref{P.RTS} and using 
\hb{\thrho_\ka=\rho_\ka} for 
\hb{\ka\in\gK_\Ga}, 
$$ 
\bal 
C_{i,\ka}\circ R_{\wt{\ka}\ka} 
&=\thrho_\ka^k(\thka\slt\thvp)_*\circ\pl_{\mf{n}}^i\circ(\ka\slt\vp)^* 
 \circ T_{\wt{\ka}\ka}\circ(\ka\slt\vp)_*(\wt{\ka}\slt\wt{\vp})^*  
 (\chi\cdot)\\ 
&=(\rho_\ka/\rho_{\wt{\ka}})^k 
 T_{\ithwtka\ithka}S_{\ithwtka\ithka}C_{i,\wt{\ka}} 
 =(\rho_\ka/\rho_{\wt{\ka}})^kR_{\ithwtka\ithka}C_{i,\wt{\ka}}. 
\eal 
$$ 
By this, \Eqref{T.gik}, and \Eqref{T.Cck} it follows 
\hb{C_{i,\ka}\circ R_{\wt{\ka}\ka}\circ\ga_{k,\wt{\ka}}^c=0} for 
\hb{0\leq i\leq k-1}. Thus, setting 
\hb{\mfga_k^c\mf{v}:=(\ga_{k,\ka}^cv_\ka)}, 
\beq\Label{T.gc} 
\mfga_k^c\in\cL\bigl(\ell_p(\thmfB_p^{(s-k-1/p)/\vec r}), 
\ell_p(\mf{\gF}_p^{s/\vec r})\bigr) 
\qb \mf{C}_i\circ\mfga_k^c=\da_{ik}\id 
\qa 0\leq i\leq k. 
\eeq 
From \Eqref{T.phC}, \,\Eqref{T.C}, and the first claim of 
Theorem~\ref{thm-RA.R} we get 
$$ 
\pl_{\mf{n}}^k 
=\thpsi_p^{(\lda+k+1/p,\mu)}\circ\mf{C}_k\circ\vp_p^{\vec\om} 
\in\cL\bigl(\gF_p^{s/\vec r,\vec\om}, 
B_p^{(s-k-1/p)/\vec r,(\lda+k+1/p,\mu)}\JthV\bigr).  
$$ 

\smallskip 
(6) 
Given 
\hb{\mf{v}\in\thmfB_p^{(s-k-1/p)/\vec r}}, 
$$ 
\bal 
\pl_{\mf{n}}^i(\psi_p^{\vec\om}\mf{v}) 
&=\sum_\ka\rho_\ka^{-(\lda+m/p)}\pl_{\mf{n}}^i 
 \bigl(\Ta_{p,\ka}^{-\mu}\pi_\ka(\ka\slt\vp)^*v_\ka\bigr)\\ 
&=\sum_{\ithka}\thrho_\ithka^{-(\lda+i+m/p)}\Ta_{p,\ithka}^{-\mu} 
 \Bigl(\thpi_\ithka(\thka\slt\thvp)^*C_{i,\ka}v_\ka+\thrho_\ithka^i 
 \sum_{j=0}^{i-1}\bid ij(\pl_{\mf{n}}^{i-j}\pi_\ka)\pl_{\mf{n}}^j 
 \bigl((\ka\slt\vp)^*v_\ka\bigr)\Bigr)\\ 
&=\thpsi_p^{(\lda+i+1/p,\mu)}C_{i,\ka}v_\ka 
 +\sum_\ka\thrho_\ithka^{-(\lda+i+m/p)} 
 \Ta_{p,\ithka}^{-\mu}(\thka\slt\thvp)^* 
 \sum_{j=0}^{i-1}\bid ijC_{i-j,\ka}(\pi_\ka)C_{j,\ka}v_\ka. 
\eal 
$$ 
Thus we infer from \Eqref{T.gik}, \,\Eqref{T.Cck}, and \Eqref{T.gc} 
$$ 
\pl_{\mf{n}}^i(\psi_p^{\vec\om}\mfga_k^c\mf{w}) 
=\da_{ik}\thpsi_p^{(\lda+i+1/p,\mu)}\mf{w} 
\qa \mf{w}\in\thmfB_p^{(s-k-1/p)/\vec r} 
\qb 0\leq i\leq k. 
$$ 
Now \Eqref{T.gc} and the first part of Theorem~\ref{thm-RA.R} imply 
$$ 
(\ga_{\mf{n}}^k)^c 
:=\psi_p^{\vec\om}\circ\mfga_k^c\circ\thvp_p^{(\lda+k+1/p,\mu)} 
\in\cL\bigl(B_p^{(s-k-1/p)/\vec r,(\lda+k+1/p,\mu)}\JthV, 
\gF_p^{s/\vec r,\vec\om}\bigr)  
\npb 
$$ 
and 
\hb{\pl_{\mf{n}}^i(\ga_{\mf{n}}^k)^c=\da_{ik}\id} for 
\hb{0\leq i\leq k}. This proves assertion~(i). 

\smallskip 
(7) 
By invoking in the preceding argumentation the second statement of 
Theorem~\ref{thm-RA.R} we see that assertion~(iii) is true. 

\smallskip 
(8) 
We denote by~$\pl_{t=0,\ka}^k$ the initial trace operator of order~$k$ for 
\hb{\BY_\ka=\BX_\ka\times\BR^+}. It follows from step~(1) that 
$$ 
\mf{\pl}_{t=0}^k 
\sco\ell_p(\mf{\gF}_p^{s/\vec r})\ra\ell_p(\mf{B}_p^{s-r(k+1/p)}) 
\qb \mf{v}\mt(\pl_{t=0,\ka}^kv_\ka) 
$$ 
is a retraction and there exists a universal coretraction 
$$ 
(\mf{\pl}_{t=0}^k)^c 
\sco\ell_p(\mf{B}_p^{s-r(k+1/p)})\ra\ell_p(\mf{\gF}_p^{s/\vec r})
\qb \mf{w}\mt\bigl((\pl_{t=0,\ka}^k)^cw_\ka\bigr) 
$$ 
such that 
\beq\Label{T.dd} 
\mf{\pl}_{t=0}^j\circ(\mf{\pl}_{t=0}^k)^c=\da_{jk}\id 
\qa 0\leq j\leq k. 
\eeq 

\smallskip 
(9) 
We deduce from \Eqref{RA.tk} and step~(1)  
\beq\Label{T.dkt} 
\pl_{t=0,\ka}^k\circ\vp_{p,\ka}^{\vec\om} 
=\vp_{p,\ka}^{\lda+\mu(k+1/p)}\circ\pl_{t=0}^k 
\qa \ka\in\gK.  
\eeq 
Hence 
$$ 
\mf{\pl}_{t=0}^k\circ\vp_p^{\vec\om} 
=\vp_p^{\lda+\mu(k+1/p)}\circ\pl_{t=0}^k. 
$$ 
From this and Theorems \ref{thm-R.R} and \ref{thm-RA.R} we infer 
$$ 
\pl_{t=0}^k 
=\psi_p^{\lda+\mu(k+1/p)}\circ\mf{\pl}_{t=0}^k\circ\vp_p^{\vec\om} 
\in\cL\bigl(\gF_p^{s/\vec r,\vec\om}\BRpV, 
B_p^{s-r(k+1/p),\lda+\mu(k+1/p)}(V)\bigr).  
$$ 

\smallskip 
(10) 
Set 
$$ 
(\ga_{t=0}^k)^c:=\psi_p^{\vec\om}\circ(\mf{\pl}_{t=0}^k)^c 
\circ\vp_p^{\lda+\mu(k+1/p)}. 
$$ 
Then, similarly as above, 
$$ 
(\ga_{t=0}^k)^c
\in\cL\bigl(B_p^{s-r(k+1/p),\lda+\mu(k+1/p)}(V),  
\gF_p^{s/\vec r,\vec\om}\BRpV\bigr). 
$$ 
For 
\hb{0\leq j\leq k} we get from \Eqref{RA.tk} and \Eqref{T.dd} 
$$ 
\bal 
\pl_{t=0}^j(\ga_{t=0}^k)^cw 
&=\pl_{t=0}^j\Bigl(\sum_\ka\psi_{p,\ka}^{\vec\om}\circ(\pl_{t=0,\ka}^k)^c 
 \circ\vp_{p,\ka}^{\lda+\mu(k+1/p)}w\Bigr)\\ 
&=\sum_\ka\psi_{p,\ka}^{\lda+\mu(j+1/p)}\circ\pl_{t=0,\ka}^j 
 \circ(\pl_{t=0,\ka}^k)^c\circ\vp_{p,\ka}^{\lda+\mu(k+1/p)}w\\ 
&=\da_{jk}\psi_p^{\lda+\mu(j+1/p)}\circ\vp_{p,\ka}^{\lda+\mu(k+1/p)}w 
 =\da_{jk}w 
\eal 
\npb 
$$ 
for 
\hb{w\in\cD}. Since $\cD$~is dense in~$B_p^{s-r(k+1/p),\lda+\mu(k+1/p)}$, 
assertion~(ii) follows. 
\end{proof} 

\smallskip 
Suppose $M$~is a compact \hbox{$m$-dimensional} submanifold of~$\BR^m$. 
In this setting and if 
\hb{s=r\in2\BN^\times} assertions (i) and~(ii) reduce to the trace theorems 
for anisotropic Sobolev spaces  due to P.~Grisvard~\cite{Gri66a}; also see 
O.A. Ladyzhenskaya, V.A. Solonnikov, and N.N. Ural'ceva~\cite{LSU68a} and 
R.~Denk, M.~Hieber, and J.~Pr{\"u}ss~\cite{DHP07a}. (In the latter paper 
the authors consider vector-valued spaces.) The much simpler Hilbertian case 
\hb{p=2} has been presented by J.-L. Lions and E.~Magenes 
in~\cite[Chapter~4, Section~2]{LiM72a} following the approach 
by P.~Grisvard~\cite{Gri67a}. 
\section{Spaces With Vanishing Traces}\LabelT{sec-VT}%
In this section we characterize $\ci\gF_p^{s/\vec r,\vec\om}$ and 
$\gF_p^{s/\vec r,\vec\om}\ciJV$ by the vanishing of certain traces. In fact, 
we need to characterize those subspaces of $\gF_p^{s/\vec r,\vec\om}\JV$ 
whose traces vanish on~$\Ga$ even if 
\hb{\Ga\neq\pl M}. More precisely, we denote by 
\beq\Label{VT.F0} 
\ci\gF_{p,\Ga}^{s/\vec r,\vec\om} 
=\ci\gF_{p,\Ga}^{s/\vec r,\vec\om}\JV\text{ the closure of 
$\cD\bigl(\ci J,\cD(M\ssm\Ga,V)\bigr)$ in }\gF_p^{s/\vec r,\vec\om}\JV. 
\eeq 
Note that 
\hb{\ci\gF_{p,\pl M}^{s/\vec r,\vec\om}=\ci\gF_p^{s/\vec r,\vec\om}}. 
By Theorem~\ref{thm-A.HBH}(ii) we know already 
\beq\Label{VT.F0F} 
\ci\gF_p^{s/\vec r,\vec\om}=\gF_p^{s/\vec r,\vec\om} 
\qa s<1/p, 
\eeq 
and, trivially, 
\hb{\ci\gF_p^{s/\vec r,\vec\om}=\gF_p^{s/\vec r,\vec\om}} if 
\hb{\pl M=\es} and 
\hb{J=\BR}. The following theorem concerns the case 
\hb{s>1/p} and 
\hb{\GaJ\neq(\es,\BR)}.\po  
{\samepage 
\begin{theorem}\LabelT{thm-VT.V} 
\begin{itemize} 
\item[{\rm(i)}] 
If\/ 
\hb{\Ga\neq\es} and 
\hb{k+1/p<s<k+1+1/p} with 
\hb{k\in\BN}, then 
\beq\Label{VT.Fdu} 
\ci\gF_{p,\Ga}^{s/\vec r,\vec\om} 
=\{\,u\in\gF_p^{s/\vec r,\vec\om}\ ;\ \pl_{\mf{n}}^iu=0, 
\ i\leq k\,\}. 
\po 
\eeq 
\item[{\rm(ii)}] 
Assume 
\hb{r(\ell+1/p)<s<r(\ell+1+1/p)} with 
\hb{\ell\in\BN}. Then 
\beq\Label{VT.Fdt} 
\gF_p^{s/\vec r,\vec\om}\bigl((0,\iy),V\bigr) 
=\bigl\{\,u\in\gF_p^{s/\vec r,\vec\om}\BRpV\ ;\ \pl_{t=0}^ju=0, 
\ j\leq\ell\,\bigr\}. 
\npb 
\eeq 
Suppose 
\hb{s<r/p} with 
\hb{s>r(-1+1/p)} if\/ 
\hb{\Ga\neq\es}. Then 
\hb{\gF_p^{s/\vec r,\vec\om}\bigl((0,\iy),V\bigr) 
=\gF_p^{s/\vec r,\vec\om}\BRpV}.\po 
\end{itemize} 
\end{theorem} }
\begin{proof} 
(1) 
Let the assumptions of~(i) be satisfied. Since $\pl_{\mf{n}}^i$~is 
continuous and vanishes on the dense subset 
$\cD\bigl(\ci J,\cD(M\ssm\Ga)\bigr)$ of 
$\ci\gF_{p,\Ga}^{s/\vec r,\vec\om}$ it follows that the latter space is 
contained in the one on the right side of \Eqref{VT.Fdu}. 

\smallskip 
Conversely, let 
\hb{u\in\gF_p^{s/\vec r,\vec\om}} satisfy 
\hb{\pl_{\mf{n}}^iu=0} for 
\hb{i\leq k}. Suppose 
\hb{\al\in\cD\bigl(\ci M\cup\Ga,[0,1]\bigr)} and 
\hb{\al=1} in a neighborhood of~$\Ga$. Then 
\hb{v:=\al u\in\gF_p^{s/\vec r,\vec\om}} and 
\hb{\pl_{\mf{n}}^iv=0} for 
\hb{i\leq k}. We infer from \Eqref{T.Cck}, \,\Eqref{T.pDk}, and 
\Eqref{T.Cph} that 
\hb{\ga_\ka\circ\pl_1^i(\vp_{p,\ka}^{\vec\om}v)=0} for 
\hb{i\leq k} and 
\hb{\ka\slt\vp\in\gK\slt\Phi}. Since 
\hb{\ga_\ka=0} for 
\hb{\ka\notin\gK_\Ga} it follows from \cite[Theorem~4.7.1]{Ama09a} that 
\hb{\vp_{p,\ka}^{\vec\om}v\in\ci\gF_{p,\ka}^{s/\vec r}} for 
\hb{\ka\slt\vp\in\gK_\Ga\slt\Phi}. If 
\hb{\ka\in\gK\ssm\gK_\Ga}, then $\vp_{p,\ka}^{\vec\om}v$~belongs to 
$\ci\gF_{p,\ka}^{s/\vec r}$ as well. Moreover, $v$~vanishes near 
\hb{\pl M\ssm\Ga} and 
\hb{\ci\gF_{p,\ka}^{s/\vec r}=\gF_{p,\ka}^{s/\vec r}} for 
\hb{\ka\in\gK\ssm\gK_{\pl M}}. Hence we deduce from 
Theorem~\ref{thm-RA.R} that 
\hb{\vp_p^{\vec\om}v\in\ell_p(\cimfgF_p^{s/\vec r})}. Now part~(ii) of that 
theorem guarantees 
\hb{v=\psi_p^{\vec\om}(\vp_p^{\vec\om}v)\in\ci\gF_p^{s/\vec r,\vec\om}}. 
Consequently, 
\hb{u\in\ci\gF_{p,\Ga}^{s/\vec r,\vec\om}}. This proves claim~(i). 

\smallskip 
(2) 
Assume 
\hb{J=\BR^+} and 
\hb{r(\ell+1/p)<s<r(\ell+1+1/p)}. As above, we see that 
\hb{\gF_p^{s/\vec r,\vec\om}\bigl((0,\iy),V\bigr)} is contained in the 
space on the right side of \Eqref{VT.Fdt}. 

\smallskip 
Let 
\hb{u\in\gF_p^{s/\vec r,\vec\om}\BRpV} satisfy 
\hb{\pl_{t=0}^ju=0} for 
\hb{0\leq j\leq\ell}. We get from \Eqref{T.dkt} that 
\hb{\pl_{t=0,\ka}^j(\vp_{p,\ka}^{\vec\om}u)=0} for 
\hb{j\leq\ell} and 
\hb{\ka\slt\vp\in\gK\slt\Phi}. 

\smallskip 
Suppose 
\hb{\ka\in\gK\ssm\gK_{\pl M}}. Then \cite[Theorem~4.7.1]{Ama09a} implies 
\hb{\vp_{p,\ka}^{\vec\om}u 
   \in\gF_p^{s/\vec r}\bigl(\BX_\ka\times(0,\iy),E\bigr)}. If 
\hb{\ka\in\gK_{\pl M}}, then we obtain the latter result by extending 
\hb{v_\ka:=\vp_{p,\ka}^{\vec\om}u} first from 
\hb{\BH^m\times\BR^+} to 
\hb{\BR^m\times\BR^+} (as in Section~4.1 of~\cite{Ama09a}), then applying 
\cite[Theorem~4.7.1]{Ama09a}, and restricting afterwards to 
\hb{\BH^m\times\BR^+}. From this and 
Theorems \ref{thm-RA.R}  and~\ref{thm-ER.ER} we obtain 
\beq\Label{VT.e} 
e_0^+(\vp_p^{\vec\om}u) 
\in\ell_p\bigl(\mf{\gF}^{s/\vec r}(\BX\times\BR,E)\bigr). 
\eeq 
Thus, using these theorems once more and the fact that, by~\Eqref{ER.rot}, 
\ $e_0^+$~commutes with~$\psi_p^{\vec\om}$, we find 
\beq\Label{VT.re} 
u=r_0^+\circ e_0^+\circ\psi_p^{\vec\om}\circ\vp_p^{\vec\om}u 
=r_0^+\circ\psi_p^{\vec\om}\circ e_0^+\circ\vp_p^{\vec\om}u 
\in\gF_p^{s/\vec r,\vec\om}\bigl((0,\iy),V\bigr). 
\npb 
\eeq 
This implies the first part of claim~(ii). 

\smallskip 
Assume 
\hb{s<r/p}. If 
\hb{\pl M=\es}, then 
\hb{\cD\bigl((0,\iy),M\bigr)=\cD\ciJciM}. Hence 
\hb{\gF_p^{s/\vec r,\vec\om}\bigl((0,\iy),V\bigr) 
   =\ci\gF_p^{s/\vec r,\vec\om}\BRpV}. Thus, by \Eqref{VT.F0F}, 
\ \hb{\gF_p^{s/\vec r,\vec\om}\bigl((0,\iy),V\bigr) 
   =\gF_p^{s/\vec r,\vec\om}\BRpV} for 
\hb{s<1/p} and 
\hb{\pl M=\es}. This shows that in either case
\hb{s>r(-1+1/p)}. Consequently, as above, we deduce 
\hb{\gF_p^{s/\vec r}(\BX_\ka\times\ci J,E)=\gF_{p,\ka}^{s/\vec r}} from  
\cite[Theorem~4.7.1(ii)]{Ama09a}. 
Now the second part of assertion~(ii) is implied by \Eqref{VT.e} and 
\Eqref{VT.re}. 
\end{proof} 
\section{Boundary Operators}\LabelT{sec-B}%
Throughout this section we suppose 
\hb{\Ga\neq\es}. 

\smallskip 
For 
\hb{k\in\BN} we consider differential operators on~$\Ga$ of the form 
$$ 
\sum_{i=0}^kb_i(\thna)\circ\pl_{\mf{n}}^i 
\qb b_i(\thna):=\sum_{j=0}^{k-i}b_{ij}\cdot\thna^j, 
$$ 
where 
\hb{b_{ij}\cdot\thna^j:=(u\mt b_{ij}\cdot\thna^ju)}, of course. Thus 
$b_i(\thna)$~is a tangential differential operator of order at most 
\hb{k-i}. 

\smallskip 
In fact, we consider systems of such operators. Thus we assume 
\beq\Label{B.ass} 
\bal
\bt\quad
&k,r_i\in\BN\text{ with }r_0<\cdots<r_k,\\ 
\bt\quad
&\sa_i,\tau_i\in\BN\text{ and }\lda_i\in\BR,\\ 
\bt\quad
&G_i=(G_i,h_{G_i},D_{G_i})\text{ is a fully uniformly regular vector 
 bundle over }\Ga 
\eal
\eeq 
for 
\hb{0\leq i\leq k}. For abbreviation, 
$$ 
\nu_i:=(r_i,\sa_i,\tau_i,\lda_i) 
\qa 0\leq i\leq k 
\qb \mf{\nu}_k:=(\nu_0,\ldots,\nu_k). 
$$ 
Then we define \emph{boundary operators} on~$\Ga$ of order at most~$r_i$ by 
$$ 
\cB_i(\mf{b}_i):=\sum_{j=0}^{r_i}\cB_{ij}(\mf{b}_{ij})\circ\pl_{\mf{n}}^j 
\qb \cB_{ij}(\mf{b}_{ij}):=\sum_{\ell=0}^{r_i-j}b_{ij,\ell}\cdot\thna^\ell,  
$$ 
where 
\hb{\mf{b}_i:=(\mf{b}_{i0},\ldots,\mf{b}_{ir_i})} and 
\hb{\mf{b}_{ij}:=(\mf{b}_{ij,0},\ldots,\mf{b}_{ij,r_i-j})} with 
$b_{ij,\ell}$ being time-dependent $\Hom\thWGi$-valued tensor fields 
on~$\Ga$. To be more precise, we introduce data spaces for 
\hb{s>r_i} by 
$$ 
\gB_{ij}^s\GaGi 
=\gB_{ij}^s(\Ga,G_i,\nu_i,\mu) 
:=\prod_{\ell=0}^{r_i-j}B_\iy^{(s-r_i)/\vec r,(\lda_i+r_i-j,\mu)} 
\bigl(\BR,T_{\tau_i+\sa}^{\sa_i+\tau+\ell} 
\bigl(\Ga,\Hom\thWGi\bigr)\bigr) 
$$ 
with general point~$(\mf{b}_{ij})$, and 
$$ 
\gB_i^s\GaGi 
=\gB_i^s(\Ga,G_i,\nu_i,\mu) 
:=\prod_{j=0}^{r_i}\gB_{ij}^s\GaGi 
\npb 
$$ 
whose general point is~$\mf{b}_i$. 
\begin{remarks}\LabelT{rem-B.B} 
\hh{(a)}\quad  
For the ease of writing we assume that these data spaces are defined on 
the whole line~$\BR$. In the following treatment, when studying 
function spaces on~$\BR^+$ or~$(0,\iy)$ it suffices, of course, to consider 
data defined on~$\BR^+$ only. It follows from Theorem~\ref{thm-ER.B} that 
this is no restriction of generality to assume that the data are given 
on all of~$\BR$. 

\smallskip 
\hh{(b)}\quad  
It should be observed that everything which follows below remains valid 
if we replace the data space $B_\iy^{(s-r_i)/\vec r,(\lda_i+r_i-j,\mu)}$ 
by $B_\iy^{\bar s_i/\vec r,(\lda_i+r_i-j,\mu)}$ with 
\hb{\bar s>s-r_i-1/p}. The selected choice has the advantage that 
$\gB_i^s\GaGi$ is independent of~$p$.\hfill$\Box$ 
\end{remarks} 
Henceforth, 
\hb{I\in\bigl\{\,J,(0,\iy)\,\bigr\}}. Given 
\hb{\mf{b}_i\in\gB_i^s\GaGi}, it follows from Theorem~\ref{thm-M.D0}, by 
taking also Theorem~\ref{thm-VT.V}(ii) into consideration if 
\hb{I=(0,\iy)}, that 
\beq\Label{B.Bij} 
\cB_{ij}(\mf{b}_{ij}) 
\in\cL\bigl(B_p^{(s-j-1/p)/\vec r,(\lda+j+1/p,\mu)}\IthV, 
B_p^{(s-r_i-1/p)/\vec r,(\lda+\lda_i+r_i+1/p,\mu)} 
\bigl(I,T_{\tau_i}^{\sa_i}\GaGi\bigr)\bigr). 
\eeq 
Hence, by Theorem~\ref{thm-T.T}, 
\beq\Label{B.BiL} 
\cB_i(\mf{b}_i) 
\in\cL\bigl(\gF_p^{s/\vec r,\vec\om}\IV, 
B_p^{(s-r_i-1/p)/\vec r,(\lda+\lda_i+r_i+1/p,\mu)} 
\bigl(I,T_{\tau_i}^{\sa_i}\GaGi\bigr)\bigr). 
\eeq 
Finally, we set 
\hb{G:=G_0\oplus\cdots\oplus G_k}, 
$$ 
\gB^s\GaG=\gB^s(\Ga,G,\mf{\nu},\mu) 
:=\prod_{i=0}^k\gB_i^s\GaGi 
$$ 
and  
$$  
\cB(\mf{b}):=\bigl(\cB_0(\mf{b}_0),\ldots,\cB_k(\mf{b}_k)\bigr) 
\qb \mf{b}:=(\mf{b}_0,\ldots,\mf{b}_k)\in\gB^s\GaG. 
$$ 
The boundary operator~$\cB_i(\mf{b}_i)$ is \emph{normal\/} if 
\hb{b_{ir_i}:=b_{ir_i,0}} is \hbox{$\lda_i$-uniformly} contraction 
surjective, and $\cB(\mf{b})$ is \emph{normal\/} if each~$\cB_i(\mf{b}_i)$, 
\ \hb{0\leq i\leq k}, has this property. Then 
$$ 
\gB_\norm^s\GaG:=\bigl\{\,\mf{b}\in\gB^s\GaG 
\ ;\ \cB(\mf{b})\text{ is normal}\,\bigr\}. 
$$ 

\smallskip 
It should be observed that 
\hb{\Ga\neq\pl M}, in general. This will allow us to consider boundary 
value problems where the order of the boundary operators may be different 
on different parts of~$\pl M$. 

\smallskip 
Lastly, we introduce the `boundary space' 
$$ 
\pl_{\Ga\times I}\gF_p^{s/\vec r,\vec\om}(G) 
=\pl_{\Ga\times I}\gF_p^{s/\vec r,\vec\om}(G,\mf{\nu},\mu) 
:=\prod_{i=0}^kB_p^{(s-r_i-1/p)/\vec r,(\lda+\lda_i+r_i+1/p,\mu)} 
\bigl(I,T_{\tau_i}^{\sa_i}\GaGi\bigr). 
\npb 
$$ 
The following lemma shows that it is an image space for the boundary 
operators under consideration. 
\begin{lemma}\LabelT{lem-B.C} 
If 
\hb{s>r_k+1/p} and 
\hb{\mf{b}\in\gB^s\GaG}, then 
$$ 
\cB(\mf{b})\in\cL\bigl(\gF_p^{s/\vec r,\vec\om}\IV, 
\pl_{\Ga\times I}\gF_p^{s/\vec r,\vec\om}(G)\bigr). 
\npb 
$$ 
The map 
\hb{\cB(\cdot)=\bigl(\mf{b}\mt\cB(\mf{b})\bigr)} is linear and continuous, 
and $\gB_\norm^s\GaG$ is open in~$\gB^s\GaG$. 
\end{lemma} 
\begin{proof} 
The first assertion is immediate from \Eqref{B.BiL}. The second one is 
obvious, and the last one is a consequence of Proposition~\ref{pro-C.R}. 
\end{proof} 
\begin{theorem}\LabelT{thm-B.R} 
Suppose assumption\/~\Eqref{B.ass} applies. Let 
\hb{s>r_k+1/p} and 
\hb{\mf{b}\in\gB_\norm^s\GaG}. Then $\cB(\mf{b})$~is a retraction from 
$\gF_p^{s/\vec r,\vec\om}\JV$ onto 
$\pl_{\Ga\times J}\gF_p^{s/\vec r,\vec\om}(G)$. There exists an analytic 
map 
$$ 
\cB^c(\cdot)\sco\gB_\norm^s\GaG 
\ra\cL\bigl(\pl_{\Ga\times J}\gF_p^{s/\vec r,\vec\om}(G), 
\gF_p^{s/\vec r,\vec\om}\JV\bigr) 
$$ 
such that 
\[ 
\bal 
{\rm(i)} 
&\quad \cB^c(\mf{b})\text{ is a coretraction for }\cB(\mf{b}),\\ 
{\rm(ii)} 
&\quad \pl_{\mf{n}}^j\circ\cB^c(\mf{b})=0\kern1pt 
 \text{ for\/ $0\leq j<s-1/p$ with }j\notin\{r_0,\ldots,r_k\}. 
\eal 
\npb 
\] 
If 
\hb{J=\BR^+}, then 
\hb{\cB^c(\mf{b})g\in\gF_p^{s/\vec r,\vec\om}\bigl((0,\iy),V\bigr)} 
whenever 
\hb{g\in\pl_{\Ga\times(0,\iy)}\gF_p^{s/\vec r,\vec\om}(G)}. 
\end{theorem} 
\begin{proof} 
(1) 
We deduce from Theorem~\ref{thm-C.R} for 
\hb{0\leq i\leq k} the existence of an analytic map~$A_i^c(\cdot)$ from 
$$ 
B_{\iy,\surj}^{(s-r_i)/\vec r,(\lda_i,\mu)} 
\bigl(J,T_{\tau_i+\sa}^{\sa_i+\tau}\bigl(\Ga,\Hom\thWGi\bigr)\bigr) 
$$ 
into 
$$ 
\cL\bigl(B_p^{(s-r_i-1/p)/\vec r,(\lda+\lda_i+r_i+1/p,\mu)} 
\bigl(J,T_{\tau_i}^{\sa_i}\GaGi\bigr), 
B_p^{(s-r_i-1/p)/\vec r,(\lda+r_i+1/p,\mu)}\JthV\bigr) 
\npb 
$$ 
such that $A_i^c(a)$~is a right inverse for 
\hb{A_i(a):=(u\mt a\cdot u)}.

\smallskip 
(2) 
Suppose 
\hb{\mf{b}\in\gB_\norm^s\GaG}. For 
\hb{0\leq i\leq k} we set 
$$ 
\cC_{r_i}(\mf{b}_i) 
:=-\sum_{j=0}^{r_i-1}A_i^c(b_{ir_i}) 
\cB_{ij}(\mf{b}_{ij})\circ\pl_{\mf{n}}^j. 
$$ 
It follows from \Eqref{B.Bij}, step~(1), and Theorem~\ref{thm-T.T} that 
\beq\Label{B.Ci} 
\cC_{r_i}(\mf{b}_i) 
\in\cL\bigl(\gF_p^{s/\vec r,\vec\om}\JV, 
B_p^{(s-j-1/p)/\vec r,(\lda+r_i+1/p,\mu)}\JthV\bigr) 
\npb 
\eeq 
and the map 
\hb{\mf{b}_i\ra\cC_{r_i}(\mf{b}_i)} is analytic. 

\smallskip 
Let 
\hb{N:=[s-1/p]_-} and define 
$$ 
\cC=(\cC_0,\ldots,\cC_N) 
\in\cL\Bigl(\gF_p^{s/\vec r,\vec\om}\JV,\prod_{\ell=0}^N 
B_p^{(s-\ell-1/p)/\vec r,(\lda+\ell+1/p,\mu)}\JthV\Bigr)   
\npb 
$$ 
by setting 
\hb{\cC_\ell:=0} for 
\hb{0\leq\ell\leq N} with 
\hb{\ell\notin\{r_0,\ldots,r_k\}}. 

\smallskip 
(3) 
Assume 
\hb{g=(g_0,\ldots,g_k)\in\pl_{\Ga\times J}\gF_p^{s/\vec r}(G)}. Define 
$$ 
h=(h_0,\ldots,h_N) 
\in\prod_{\ell=0}^NB_p^{(s-\ell-1/p)/\vec r,(\lda+\ell+1/p,\mu)}\JthV  
\npb 
$$ 
by 
\hb{h_{r_i}:=A_i^c(b_{ir_i})g_i} for 
\hb{0\leq i\leq k}, and 
\hb{h_\ell:=0} otherwise. 

\smallskip 
By Theorem~\ref{thm-T.T} there exists for 
\hb{j\in\{0,\ldots,N\}} a universal coretraction~$(\ga_{\mf{n}}^j)^c$ 
for~$\pl_{\mf{n}}^j$ satisfying 
\beq\Label{B.ljc} 
\pl_{\mf{n}}^\ell\circ(\ga_{\mf{n}}^j)^c=\da^{\ell j}\id 
\qa 0\leq\ell\leq j. 
\eeq 
We put 
\hb{u_0:=(\ga_{\mf{n}}^0)^ch_0\in\gF_p^{s/\vec r,\vec\om}\JV}. Suppose 
\hb{1\leq j\leq N} and 
\hb{u_0,u_1,\ldots,u_{j-1}} have already been defined. Set 
\beq\Label{B.u} 
u_j:=u_{j-1}+(\ga_{\mf{n}}^j)^c\,(h_j+\cC_ju_{j-1}-\pl_{\mf{n}}^ju_{j-1}). 
\eeq 
This defines 
\hb{u_0,u_1,\ldots,u_N\in\gF_p^{s/\vec r,\vec\om}\JV}. It follows from 
\Eqref{B.ljc} and \Eqref{B.u} 
\beq\Label{B.dj} 
\pl_{\mf{n}}^ju_j=h_j+\cC_ju_{j-1} 
\qa 0\leq j\leq N, 
\eeq 
and 
$$ 
\pl_{\mf{n}}^\ell u_j=\pl_{\mf{n}}^\ell u_{j-1} 
\qa 0\leq\ell\leq j-1 
\qb 1\leq j\leq N. 
$$ 
The latter relation implies 
$$ 
\pl_{\mf{n}}^\ell u_j=\pl_{\mf{n}}^\ell u_n 
\qa 0\leq\ell<j<n\leq N. 
$$ 
Hence, since $\cC_j$~involves 
\hb{\pl_{\mf{n}}^0,\ldots,\pl_{\mf{n}}^{j-1}} only, we deduce from 
\Eqref{B.dj} 
$$ 
\pl_{\mf{n}}^ju_n=h_j+\cC_ju_n 
\qa 0\leq j\leq n\leq N. 
$$ 
If 
\hb{j=r_i}, then we apply $A_i(b_{ir_i})$ to this equation to find 
\beq\Label{B.Big} 
\cB_iu_n=g_i 
\qa r_i\leq n\leq N. 
\eeq 
For 
\hb{0\leq i\leq k} we set 
\hb{G^i:=G_0\oplus\cdots\oplus G_i} and 
\hb{\mf{\nu}^i:=(\nu_1,\ldots,\nu_i)} as well as 
\hb{\mf{b}^i:=(\mf{b}_0,\ldots,\mf{b}_i)}. Then it follows from 
\Eqref{B.BiL} that 
\beq\Label{B.B0} 
\cB^i(\mf{b}^i):=\bigl(\cB_0(\mf{b}_0),\ldots,\cB_i(\mf{b}_i)\bigr) 
\in\cL\bigl(\gF_p^{t/\vec r,\vec\om}\JV, 
\pl_{\Ga\times J}\gF_p^{t/\vec r,\vec\om}(G^i,\mf{\nu}^i,\mu)\bigr) 
\eeq 
for 
\hb{r_i+1/p<t\leq s}. We define~$\cB^{ic}(\mf{b}^i)$ by 
$$ 
\cB^{ic}(\mf{b}^i)\,(g_0,\ldots,g_i):=u_{r_i}. 
$$ 
It follows from \Eqref{B.Ci}, \,\Eqref{B.u}, and Theorem~\ref{thm-T.T} 
that 
\beq\Label{B.B1} 
\cB^{ic}(\mf{b}^i) 
\in\cL\bigl(\pl_{\Ga\times J}\gF_p^{t/\vec r,\vec\om}(G^i,\mf{\nu}^i,\mu), 
\gF_p^{t/\vec r,\vec\om}\JV\bigr) 
\qa r_i+1/p<t\leq s. 
\eeq 
Furthermore, \Eqref{B.Big} and the definition of~$h$ imply 
\beq\Label{B.B2} 
\cB_\al(\mf{b}_\al)\cB^{ic}(\mf{b}^i)(g_0,\ldots,g_i)
=\cB_\al(\mf{b}_\al)\cB^{jc}(\mf{b}^j)(g_0,\ldots,g_j) 
=g_\al 
\qa 0\leq\al\leq i\leq j\leq k, 
\npb 
\eeq 
and 
\hb{\pl_{\mf{n}}^j\cB^{ic}(\mf{b}^i)=0} for 
\hb{0\leq j<t-1/p} with 
\hb{j\notin\{r_0,\ldots,r_k\}}. 

\smallskip 
Now we set 
\hb{\cB^c(\mf{b}):=\cB^{kc}(\mf{b})}. Then \Eqref{B.B0} and \Eqref{B.B2} 
show that it is a right inverse for~$\cB(\mf{b})$. It is a consequence of 
step~(2) and \Eqref{B.u} that $\cB^c(\cdot)$~is analytic. Due to 
Theorem~\ref{thm-T.T} it is easy to see that the last assertion applies as 
well. 
\end{proof} 
There is a similar, though much simpler result concerning the `extension of 
initial conditions'. 
\begin{theorem}\LabelT{thm-B.Rt} 
Suppose 
\hb{0\leq j_0<\cdots<j_\ell} and 
\hb{s>r(j_\ell+1/p)}. Set 
\hb{\cC:=(\pl_{t=0}^{j_0},\ldots,\pl_{t=0}^{j_\ell})} and  
\beq\Label{B.V} 
B_p^{s-r(\mf{j}_\ell+1/p),\lda+\mu(\mf{j}_\ell+1/p)}(V) 
:=\prod_{i=0}^\ell B_p^{s-r(j_i+1/p),\lda+\mu(j_i+1/p)}(V). 
\eeq 
Then $\cC$~is a retraction from\/ $\gF_p^{s/\vec r,\vec\om}\BRpV$ onto 
$B_p^{s-\vec r(\mf{j}_\ell+1/p),\lda+\mu(\mf{j}_\ell+1/p)}(V)$, and there 
exists a coretraction~$\cC^c$ satisfying 
\hb{\pl_{t=0}^j\circ\cC^c=0} for 
\hb{0\leq j<s/r-1/p} with 
\hb{j\notin\{j_0,\ldots,j_\ell\}}. 
\end{theorem} 
\begin{proof} 
Theorem~\ref{thm-T.T}(ii) guarantees that $\cC$~is a continuous linear map 
from $\gF_p^{s/\vec r,\vec\om}\BRpV$ into~\Eqref{B.V}. Due to that theorem 
the assertion follows from step~(3) of the proof of Theorem~\ref{thm-B.R} 
using the following modifications: 
\hb{h_{j_i}:=g_i} for 
\hb{0\leq i\leq\ell} and  
$$ 
u_j:=u_{j-1}+(\ga_{t=0}^j)^c(h_j-\pl_{t=0}^ju_{j-1}) 
\npb 
$$ 
with 
\hb{u_{-1}:=0}. 
\end{proof} 
Now we suppose 
\hb{\Ga\neq\es} and 
\hb{J=\BR^+}. We write 
\hb{\Sa:=\Ga\times\BR^+} for the \emph{lateral boundary over}~$\Ga$ and 
recall that 
\hb{M_0:=M\times\{0\}} is the initial boundary. Then 
\hb{\Sa\cap M_0=\Ga\times\{0\}=:\Ga_0} is the \emph{corner manifold 
over}~$\Ga$. We suppose 
\beq\Label{B.assl} 
\bal
\bt\quad
&\text{assumption \Eqref{B.ass} is satisfied},\\ 
\bt\quad
&\ell\in\BN\text{ and }s>\max\bigl\{r_k+1/p,\ r(\ell+1/p)\bigr\}. 
\eal
\eeq 
We set 
\hb{\cC:=\ora{\pl_{t=0}^\ell}:=(\pl_{t=0}^0,\ldots,\pl_{t=0}^\ell)}. Then, 
by Theorem~\ref{thm-B.Rt}, 
\ $\cC$~is a retraction from $\gF_p^{s/\vec r,\vec\om}\BRpV$ onto 
$$ 
B_p^{s-r(\mf{\ell}+1/p),\lda+\mu(\mf{\ell}+1/p)}(V) 
:=\prod_{j=0}^\ell B_p^{s-r(j+1/p),\lda+\mu(j+1/p)}(V). 
$$ 
By Theorem~\ref{thm-B.R} 
\ $\cB(\mf{b})$~is for 
\hb{\mf{b}\in\gB_\norm^s\GaG} a~retraction from 
$\gF_p^{s/\vec r,\vec\om}\BRpV$ onto $\pl_\Sa\gF_p^{s/\vec r,\vec\om}(G)$. 
We put 
$$ 
\pl_{\Sa\cup M_0}\gF_p^{s/\vec r,\vec\om}(G) 
:=\pl_\Sa\gF_p^{s/\vec r,\vec\om}(G) 
\times B_p^{s-r(\mf{\ell}+1/p),\lda+\mu(\mf{\ell}+1/p)}(V)
$$ 
and 
\hb{\vec\cB(\cdot):=\bigl(\cB(\cdot),\cC\bigr)}. Then 
$$ 
\vec\cB(\cdot) 
\sco\gB_\norm^s\GaG 
\ra\cL\bigl(\gF_p^{s/\vec r,\vec\om}\BRpV, 
\pl_{\Sa\cup M_0}\gF_p^{s/\vec r,\vec\om}(G)\bigr) 
\npb 
$$ 
is the restriction of a continuous linear map to the open subset 
$\gB_\norm^s\GaG$ of~$\gB^s\GaG$, hence analytic. 

\smallskip 
However, $\vec\cB(\mf{b})$~is not surjective, in general. Indeed, suppose 
$$ 
0\leq i\leq k 
\qb 0\leq j\leq\ell 
\qb s>r_i+1/p+r(j+1/p)=:r_{ij}. 
$$ 
Then we deduce from \Eqref{B.BiL} and Theorem~\ref{thm-T.T}(ii) 
$$ 
\pl_{t=0}^j\circ\cB_i(\mf{b}) 
\in\cL\bigl(\gF_p^{s/\vec r,\vec\om}\BRpV, 
B_p^{s-r_{ij},\lda+\lda_i+r_{ij}}\bigl(T_{\tau_i}^{\sa_i}\GanGi\bigr)\bigr).  
$$ 
Furthermore, 
\hb{\pl^j\bigl(\cB_i(\mf{b})u\bigr)=\cB_i^{(j)}(\mf{b})u}, where 
$$ 
\cB_i^{(j)}(\mf{b})u 
=\sum_{\al=0}^j\bid j\al\cB_i(\pl^{j-\al}\mf{b}_i)\circ\pa. 
$$ 
Theorem~\ref{thm-M.D} implies 
$$ 
\pl^{j-\al}\mf{b}_i 
\in\gB_i^s\bigl(\Ga,G_i,\bigl(r_i+r(j-\al),\sa_i,\tau_i, 
\lda_i+\mu(j-\al)\bigr),\mu\bigr). 
$$ 
From this and \Eqref{B.BiL} we infer that $\cB_i^{(j)}(\mf{b})$~possesses 
the same mapping properties as 
\hb{\pl^j\circ\cB_i(\mf{b})}. Set 
$$ 
\cB_i^{(j)}(0)\vec v_j 
=\cB_i^{(j)}(\mf{b},0)\vec v_j 
:=\sum_{\al=0}^j\bid j\al\cB_i(\pl_{t=0}^{j-\al}\mf{b}_i)v_\al 
\qa \vec v_j:=(v_0,\ldots,v_j)  
$$ 
with 
\hb{v_\al\in B_p^{s-r(\al+1/p),\lda+\mu(\al+1/p)}(V)}. Then 
$\cB_i^{(j)}(0)$~is a continuous linear map 
$$ 
\prod_{\al=0}^jB_p^{s-r(\al+1/p),\lda+\mu(\al+1/p)}(V) 
\ra B_p^{s-r_{ij},\lda+\lda_j+r_{ij}}\bigl(T_{\tau_i}^{\sa_i}\GanGi\bigr) 
$$ 
and 
\hb{\mf{b}\mt\cB_i^{(j)}(0)} is the restriction of a linear and continuous 
map to~$\gB_\norm^s\GaG$.  Furthermore, 
\beq\Label{B.cu} 
\pl_{t=0}^j\bigl(\cB_i(\mf{b})u\bigr)=\cB_i^{(j)}(0)\ora{\pl_{t=0}^j}u 
\qa u\in\gF_p^{s/\vec r,\vec\om}\BRpV. 
\eeq 

\smallskip 
We denote for 
\hb{\mf{b}\in\gB_\norm^s\GaG} by 
$$ 
\bal 
{}      
&\pl_{\vec\cB(\mf{b})}^{cc}\gF_p^{s/\vec r,\vec\om}(G)\text{ the set of all 
 $(g,h)\in\pl_{\Sa\cup M_0}\gF_p^{s/\vec r,\vec\om}(G)$ satisfying the 
 \hh{compatibility conditions}}\\ 
&\qquad\qquad\qquad\qquad\qquad\qquad\qquad\qquad 
 \pl_{t=0}^jg_i=\cB_i^{(j)}(0)\vec h_j\\ 
&\text{for $0\leq i\leq k$ and $0\leq j\leq\ell$ with }
 r_i+1/p+r(j+1/p)<s. 
\eal 
$$ 
The linearity and continuity of $\pl_{t=0}^j$ and~$\cB_i^{(j)}(0)$ 
guarantee that 
$\pl_{\vec\cB(\mf{b})}^{cc}\gF_p^{s/\vec r,\vec\om}(G)$ is a closed linear 
subspace of $\pl_{\Sa\cup M_0}\gF_p^{s/\vec r,\vec\om}(G)$. By the 
preceding considerations it contains the range of~$\vec\cB(\mf{b})$. The 
following theorem shows that, in fact, 
\hb{\pl_{\vec\cB(\mf{b})}^{cc}\gF_p^{s/\vec r,\vec\om}(G)=\im(\vec\cB)}, 
provided 
\hb{\mf{b}\in\gB_\norm^s\GaG}. 
\begin{theorem}\LabelT{thm-B.Rcc} 
Let assumption\/~\Eqref{B.assl} be 
satisfied and suppose 
$$ 
s\notin\bigl\{\,r_i+1/p+r(j+1/p) 
\ ;\ 0\leq i\leq k,\ 0\leq j\leq\ell\,\bigr\}. 
$$ 
Then $\vec\cB(\mf{b})$~is for 
\hb{\mf{b}\in\gB_\norm^s\GaG} a~retraction from 
$\gF_p^{s/\vec r,\vec\om}\BRpV$ onto 
$\pl_{\vec\cB(\mf{b})}^{cc}\gF_p^{s/\vec r,\vec\om}(G)$. There exists 
an analytic map 
\beq\Label{B.Bc} 
\vec\cB^c(\cdot)\sco\gB_\norm^s\GaG 
\ra\cL\bigl(\pl_{\Sa\cup M_0}\gF_p^{s/\vec r,\vec\om}(G), 
\gF_p^{s/\vec r,\vec\om}\BRpV\bigr) 
\npb 
\eeq 
such that 
\hb{\vec\cB^c(\mf{b}) 
   \sn\pl_{\vec\cB(\mf{b})}^{cc}\gF_p^{s/\vec r,\vec\om}(G)} 
is a coretraction for~$\vec\cB(\mf{b})$. 
\end{theorem} 
\begin{proof} 
By the preceding remarks it suffices to construct~$\vec\cB^c(\cdot)$ 
satisfying \Eqref{B.Bc} such that its restriction to 
\hb{\pl_{\vec\cB(\mf{b})}^{cc}\gF_p^{s/\vec r,\vec\om}(G)} is a right 
inverse for~$\vec\cB(\mf{b})$. 

\smallskip 
By Theorem~\ref{thm-B.R} there exists an analytic map 
$$ 
\cB^c(\cdot)\sco\gB_\norm^s\GaG 
\ra\cL\bigl(\pl_\Sa\gF_p^{s/\vec r,\vec\om}(G), 
\gF_p^{s/\vec r,\vec\om}\BRpV\bigr) 
$$ 
such that 
\beq\Label{B.v0} 
v\in\pl_{\Ga\times(0,\iy)}\gF_p^{s/\vec r,\vec\om}(G) 
\quad\Ra\quad 
\cB^c(\mf{b})v\in\gF_p^{s/\vec r,\vec\om}\bigl((0,\iy),V\bigr) 
\eeq 
\npb 
for 
\hb{\mf{b}\in\gB_\norm^s\GaG}.  

\smallskip 
Let $\cC^c$ be a coretraction for~$\cC$. Its existence is guaranteed by 
Theorem~\ref{thm-B.Rt}. Given 
\hb{(g,h)\in\pl_{\Sa\cup M_0}\gF_p^{s/\vec r,\vec\om}(G)} and 
\hb{\mf{b}\in\gB_\norm^s\GaG}, set 
$$ 
\vec\cB^c(\mf{b})(g,h) 
:=\cC^ch+\cB^c(\mf{b})\bigl(g-\cB(\mf{b})\cC^ch\bigr). 
$$ 
Then $\cB^c(\cdot)$~satisfies \Eqref{B.Bc} and is analytic. Furthermore, 
\beq\Label{B.g} 
\cB(\mf{b})\bigl(\vec\cB^c(\mf{b})(g,h)\bigr)=g. 
\eeq 

\smallskip 
We fix 
\hb{\mf{b}\in\gB_\norm^s\GaG} and write 
\hb{\cB=\cB(\mf{b})} and 
\hb{\cB^c=\cB^c(\mf{b})}. For 
\hb{(g,h)\in\pl_{\vec\cB(\mf{b})}^{cc}\gF_p^{s/\vec r,\vec\om}(G)} we set 
$$ 
v:=g-\cB\cC^ch\in\pl_\Sa\gF_p^{s/\vec r,\vec\om}(G). 
$$ 
Suppose 
\hb{0\leq i\leq k}. Then 
$$ 
v_i=g_i-\cB_i\cC^ch 
\in B_p^{(s-r_i-1/p)/\vec r,(\lda+\lda_i+r_i+1/p,\mu)}\BRpVi, 
$$ 
where 
\hb{V_{\coV i}:=T_{\tau_i}^{\sa_i}(G_i)}. Let 
\hb{j_i\in\{0,\ldots,\ell\}} be the largest integer satisfying 
\hb{r_i+1/p+r(j_i+1/p)<s}. Then, by \Eqref{B.cu}, 
$$ 
\pl_{t=0}^jv_i 
=\pl_{t=0}^jg_i-\pl_{t=0}^j(\cB_i\cC^ch) 
=\pl_{t=0}^jg_i-\cB_i^{(j)}(0)\ora{\pl_{t=0}^j}\cC^ch  
=\pl_{t=0}^jg_i-\cB_i^{(j)}(0)\vec h_j 
=0 
$$ 
for 
\hb{0\leq j\leq j_i}. Hence 
\hb{r(j_i+1/p)<s-r_i-1/p<r(j_i+1+1/p)} and Theorem~\ref{thm-VT.V}(ii) 
imply 
\beq\Label{B.v} 
v_i\in B_p^{(s-r_i-1/p)/\vec r,(\lda+\lda_i+r_i+1/p,\mu)} 
\bigl((0,\iy),V_{\coV i}\bigr). 
\eeq 
If there is no such~$j_i$, then 
\hb{s-r_i-1/p<r/p}. In this case that theorem guarantees \Eqref{B.v} also. 
This shows that 
\hb{v\in\pl_{\Ga\times(0,\iy)}\gF_p^{s/\vec r,\vec\om}(G)}. Hence 
\hb{\cC\cB^cv=0} by \Eqref{B.v0} and Theorem~\ref{thm-VT.V}(ii) and since 
\hb{s>r(j+1/p)} for 
\hb{0\leq j\leq\ell}. Consequently, 
\hb{\cC\bigl(\vec\cB^c(\mf{b})(g,h)\bigr)=h} for 
\hb{(g,h)\in\pl_{\vec\cB(\mf{b})}^{cc}\gF_p^{s/\vec r,\vec\om}(G)}. 
Together with \Eqref{B.g} this proves the theorem. 
\end{proof} 
\begin{remark}\LabelT{rem-B.Rcc} 
Let assumption~\Eqref{B.ass} be satisfied. Suppose 
$$ 
r_0+1/p<s<r/p 
\qa s\notin\{\,r_i+1/p\ ;\ 1\leq i\leq k\,\}. 
$$ 
Then there is a lateral boundary operator~$\cB$~ only, since there is no 
initial trace. Thus this case is covered by Theorem~\ref{thm-B.R}. 

\smallskip 
Assume 
$$ 
r/p<s<r_0+1/p 
\qa s\notin\bigl\{\,r(j+1/p)\ ;\ 1\leq j\leq\ell\,\bigr\}. 
\npb 
$$ 
Then there is no lateral trace operator and we are in a situation to which 
Theorem~\ref{thm-B.Rt} applies. 

\smallskip 
Lastly, if 
\hb{-1+1/p<s<\min\{r_0+1/p,\ r/p\}}, then there is neither a lateral nor an 
initial trace operator and 
\hb{\gF_p^{s/\vec r,\vec\om}\BRpV 
   =\ci\gF_p^{s/\vec r,\vec\om}\BRpV}.\hfill$\Box$ 
\end{remark} 
The theorems on the `extension of boundary values' proved in this section 
are of great importance in the theory of nonhomogeneous time-dependent 
boundary value problems. The only results of this type available in the 
literature concern the case where $M$~is an \hbox{$m$-dimensional} 
compact submanifold of~$\BR^m$. In this situation an anisotropic 
extension theorem involving compatibility conditions has been proved by 
P.~Grisvard in~\cite{Gri66a} for the case where 
\hb{s\in r\BN^\times}, and in~\cite{Gri67a} if 
\hb{p=2} (also see J.-L. Lions and E.~Magenes 
\cite[Chapter~4, Section~2]{LiM72a} for the Hilbertian case) by means of 
functional analytical techniques. If 
\hb{s=r=2}, then corresponding results are derived in 
O.A. Ladyzhenskaya, V.A. Solonnikov, and N.N. Ural'ceva~\cite{LSU68a} 
by studying heat potentials. In contrast to our work, in all these 
publications the exceptional values 
\hb{r_i+1/p+r(j+1/p)} for~$s$ are considered also. 
\section{Interpolation}\LabelT{sec-IP}%
In Section~\ref{sec-A} the anisotropic spaces~$\gF_p^{s/\vec r,\vec\om}$ 
have been defined for 
\hb{s>0} by interpolating between anisotropic Sobolev spaces. From this we 
could derive some interpolation properties by means of reiteration 
theorems. However, such results would be restricted to spaces with one 
and the same value of~$\lda$. In this section we prove general interpolation 
theorems for anisotropic Bessel potential, Besov, and Besov-H\"older spaces 
involving different values of $s$ and~$\lda$. 

\smallskip 
Reminding that 
\hb{\xi_\ta=(1-\ta)\xi_0+\ta\xi_1} for 
\hb{\xi_0,\xi_1\in\BR} and 
\hb{0\leq\ta\leq1}, we set 
\hb{\vec\om_\ta:=(\lda_\ta,\mu)} for 
\hb{\lda_0,\lda_1\in\BR}. We also recall that 
\hb{\pr_\ta=\pe_\ta} if 
\hb{\gF=H}, and 
\hb{\pr_\ta=\pr_{\ta,p}} if 
\hb{\gF=B}.\po 
{\samepage 
\begin{theorem}\LabelT{thm-IP.I} 
Suppose 
\hb{-\iy<s_0<s_1<\iy}, 
\ \hb{\lda_0,\lda_1\in\BR}, and 
\hb{0<\ta<1}.\po  
\begin{itemize} 
\item[{\rm(i)}] 
Assume 
\hb{s_0>-1+1/p} if\/ 
\hb{\pl M\neq\es}. Then 
\beq\Label{IP.F} 
(\gF_p^{s_0/\vec r,\vec\om_0},\gF_p^{s_1/\vec r,\vec\om_1})_\ta 
\doteq\gF_p^{s_\ta/\vec r,\vec\om_\ta} 
\doteq[\gF_p^{s_0/\vec r,\vec\om_0},\gF_p^{s_1/\vec r,\vec\om_1}]_\ta 
\po 
\eeq 
and 
\beq\Label{IP.HB} 
(H_p^{s_0/\vec r,\vec\om_0},H_p^{s_1/\vec r,\vec\om_1})_{\ta,p} 
\doteq B_p^{s_\ta/\vec r,\vec\om_\ta}. 
\po  
\eeq 
\item[{\rm(ii)}] 
If\/ 
\hb{s_0>0}, then 
$$ 
[B_\iy^{s_0/\vec r,\vec\om_0},B_\iy^{s_1/\vec r,\vec\om_1}]_\ta 
\doteq B_\iy^{s_\ta/\vec r,\vec\om_\ta} 
\po 
$$ 
and 
$$ 
[b_\iy^{s_0/\vec r,\vec\om_0},b_\iy^{s_1/\vec r,\vec\om_1}]_\ta 
\doteq b_\iy^{s_\ta/\vec r,\vec\om_\ta}. 
\po 
$$ 
\end{itemize} 
\end{theorem} }
\begin{proof} 
(1) 
Let $X$ be a Banach space and 
\hb{\da>0}. Then 
\hb{\da X:=(X,\Vsdot_{\da X})}, where 
\hb{\|x\|_{\da X}:=\|\da^{-1}x\|_X} for 
\hb{x\in X}. Thus $\da X$~is the image space of~$X$ under the map 
\hb{x\mt\da x} so that this function is an isometric isomorphism 
from~$X$ onto~$\da X$. 

\smallskip 
Assume $X_\ba$~is a Banach space and 
\hb{\da_\ba>0} for each~$\ba$ in a countable index set~$\sB$. Then we set 
\hb{\mf{\da}\mf{X}:=\prod_\ba\da_\ba X_\ba} and 
\hb{\mf{\da}\mf{x}:=(\da_\ba x_\ba)}. Hence 
\hb{\mf{\da}:=(\mf{x}\mt\mf{\da}\mf{x})\in\Lis(\mf{X},\mf{\da}\mf{X})}.  

\smallskip 
Let $\XnXe$ be a pair of Banach spaces such that $X_j$~is continuously 
injected in some locally convex space for 
\hb{j=0,1}, that is, $\XnXe$ is an interpolation couple. Suppose 
\hb{\pg_\ta\in\bigl\{\pe_\ta,\pr_{\ta,p}\bigr\}}. Then interpolation theory 
guarantees 
\beq\Label{IP.d} 
\{\da_0X_0,\da_1X_1\}_\ta=\da_0^{1-\ta}\da_1^\ta\{X_0,X_1\}_\ta 
\qa \da_0,\da_1>0, 
\npb 
\eeq 
(e.g.,~\cite[formula~(7) in Section~3.4.1]{Tri78a}). 

\smallskip 
(2) 
Let 
\hb{J=\BR} and \hb{s>-1+1/p} if 
\hb{\pl M\neq\es}. Put 
\hb{\xi:=\lda-\lda_0}. Then 
\hb{\vp_{p,\ka}^{\vec\om}=\rho_\ka^\xi\vp_{p,\ka}^{\vec\om_0}} and 
\hb{\psi_{p,\ka}^{\vec\om}=\rho_\ka^{-\xi}\psi_{p,\ka}^{\vec\om_0}} imply, 
due to Theorem~\ref{thm-RA.R}, that the diagram 
 \bspezeq 
 \bal 
 \begin{picture}(184,103)(-93,-56)        
 \put(0,40){\makebox(0,0)[b]{\small{$\id$}}}
 \put(-80,35){\makebox(0,0)[c]{\small{$\gF_p^{s/\vec r,\vec\om}$}}}
 \put(80,35){\makebox(0,0)[c]{\small{$\gF_p^{s/\vec r,\vec\om}$}}}
 \put(0,0){\makebox(0,0)[c]{\small{$\ell_p(\mf{\gF}_p^{s/\vec r})$}}}
 \put(0,-50){\makebox(0,0)[c]{\small{$\ell_p(\mf{\rho}^{-\xi} 
  \mf{\gF}_p^{s/\vec r})$}}}
 \put(-33,5){\makebox(0,0)[r]{\small{$\vp_p^{\vec\om}$}}}
 \put(31,5){\makebox(0,0)[l]{\small{$\psi_p^{\vec\om}$}}}
 \put(-48,-17.5){\makebox(0,0)[r]{\small{$\vp_p^{\vec\om_0}$}}}
 \put(48,-17.5){\makebox(0,0)[l]{\small{$\psi_p^{\vec\om_0}$}}}
 \put(-5,-28.5){\makebox(0,0)[r]{\small{$\cong$}}}
 \put(5,-28.5){\makebox(0,0)[l]{\small{$\mf{\rho}^\xi$}}}
 \put(-60,35){\vector(1,0){120}}
 \put(-60,35){\vector(1,0){120}}
 \put(0,-40){\vector(0,1){25}}
 \put(-60,27){\vector(2,-1){40}}
 \put(20,7){\vector(2,1){40}}
 \put(-65,25){\vector(1,-2){35}}
 \put(30,-45){\vector(1,2){35}}
 \end{picture} 
 \eal 
 \npb 
 \espezeq 
is commuting. Hence $\psi_p^{\vec\om_0}$~is for each~$s$ a~retraction from 
$\ell_p(\mf{\rho}^{-\xi}\mf{\gF}_p^{s/\vec r})$ onto 
$\gF_p^{s/\vec r,\vec\om}$, and $\vp_p^{\vec\om_0}$~is a coretraction. 

\smallskip 
(3) 
Let 
\hb{\xi:=\lda_1-\lda_0}. By Theorem~\ref{thm-RA.R} and the preceding step 
each of the maps 
$$ 
\psi_p^{\vec\om_0} 
\sco\ell_p(\mf{\gF}_p^{s_0/\vec r}) 
\ra\gF_p^{s_0/\vec r,\vec\om_0}  
\qb \psi_p^{\vec\om_0} 
\sco\ell_p(\mf{\rho}^{-\xi}\mf{\gF}_p^{s_1/\vec r}) 
\ra\gF_p^{s_1/\vec r,\vec\om_1}  
$$ 
is a retraction, and $\vp_p^{\vec\om_0}$~is a coretraction. Thus, by 
interpolation, 
\beq\Label{IP.0} 
\psi_p^{\vec\om_0} 
\sco\bigl\{\ell_p(\mf{\gF}_p^{s_0/\vec r}), 
\ell_p(\mf{\rho}^{-\xi}\mf{\gF}_p^{s_1/\vec r})\bigr\}_\ta 
\ra\bigl\{\gF_p^{s_0/\vec r,\vec\om_0},\gF_p^{s_1/\vec r,\vec\om_1} 
\bigr\}_\ta
\eeq 
is a retraction, and $\vp_p^{\vec\om_0}$~is a coretraction. From 
\cite[Theorem~1.18.1]{Tri78a} and \Eqref{IP.d} we infer 
\beq\Label{IP.1} 
\bigl\{\ell_p(\mf{\gF}_p^{s_0/\vec r}), 
\ell_p(\mf{\rho}^{-\xi}\mf{\gF}_p^{s_1/\vec r})\bigr\}_\ta  
\doteq 
\ell_p\bigl(\mf{\rho}^{-\ta\xi}\{\mf{\gF}_p^{s_0/\vec r}, 
\mf{\gF}_p^{s_1/\vec r}\}_\ta\bigr).  
\eeq 
(Recall the definition of $(\mf{E},\mf{F})_\ta$ after \Eqref{RA.lL}.) 
Suppose  
\hb{\pl M=\es}. Then \cite[formulas (3.3.12) and (3.4.1) and 
Theorem~3.7.1(iv)]{Ama09a} imply 
\beq\Label{IP.2} 
(\gF_{p,\ka}^{s_0/\vec r},\gF_{p,\ka}^{s_1/\vec r})_\ta 
\doteq\gF_{p,\ka}^{s_\ta/\vec r} 
\qb (H_{p,\ka}^{s_0/\vec r},H_{p,\ka}^{s_1/\vec r})_{\ta,p} 
\doteq B_{p,\ka}^{s_\ta/\vec r} 
\doteq[B_{p,\ka}^{s_0/\vec r},B_{p,\ka}^{s_1/\vec r}]_\ta. 
\eeq 
This is due to the fact that, on account of \cite[Corollary~3.3.4 and 
Theorem~3.7.1(i)]{Ama09a}, the definition of~$\gF_{p,\ka}^{s/\vec r}$ for 
\hb{s<0} used in that publication coincides with the definition by duality 
employed in this paper. 

\smallskip 
If 
\hb{\pl M\neq\es}, then it follows from 
\hb{s>-1+1/p} and Theorem~4.7.1(ii) of~\cite{Ama09a} by the same 
arguments that \Eqref{IP.2} holds in this case as well. 

\smallskip 
Thus, in either case, due to 
\hbox{\Eqref{IP.0}--\Eqref{IP.2}} 
\ $\psi_p^{\vec\om_0}$~is a~retraction from 
$\ell_p(\mf{\rho}^{-\ta\xi}\mf{\gF}_p^{s_\ta/\vec r})$ onto 
$(\gF_p^{s_0/\vec r,\vec\om_0},\gF_p^{s_1/\vec r,\vec\om_1})_\ta$ and onto 
$[\gF_p^{s_0/\vec r,\vec\om_0},\gF_p^{s_1/\vec r,\vec\om_1}]_\ta$, and 
$\vp_p^{\vec\om_0}$~is a~coretraction. On the other hand, we infer from 
step~(2), setting 
\hb{\xi=\ta(\lda_1-\lda_0)}, that $\psi_p^{\vec\om_0}$~is a~retraction 
from $\ell_p(\mf{\rho}^{-\ta\xi}\mf{\gF}_p^{s_\ta/\vec r})$ onto 
$\gF_p^{s_\ta/\vec r,\vec\om_\ta}$, and $\vp_p^{\vec\om_0}$~is 
a~coretraction. This implies the validity of \Eqref{IP.F} if 
\hb{J=\BR}. The proof for \Eqref{IP.HB} is similar. 

\smallskip 
(4) 
Assume 
\hb{J=\BR^+}. In this case we get assertion~(i) by 
Theorem~\ref{thm-ER.ER}(i) in conjunction with what has just been proved. 

\smallskip 
(5) 
Set 
\hb{\xi=\lda_1-\lda_0}. Then as above, we infer from Theorem~\ref{thm-H.R} 
that $\psi_\iy^{\vec\om_0}$~is a~retraction from 
$\ell_\iy(\mf{B}_\iy^{s_0/\vec r})$ onto~$B_\iy^{s_0/\vec r,\vec\om_0}$ 
and from $\ell_\iy(\mf{\rho}^{-\xi}\mf{B}_\iy^{s_1/\vec r})$ 
onto~$B_\iy^{s_1/\vec r,\vec\om_0}$, and $\vp_\iy^{\vec\om_0}$~is 
a~coretraction. Hence 
\beq\Label{IP.linf} 
\psi_\iy^{\vec\om_0} 
\sco\bigl[\ell_\iy(\mf{B}_\iy^{s_0/\vec r}), 
\ell_\iy(\mf{\rho}^{-\xi}\mf{B}_\iy^{s_1/\vec r})\bigr]_\ta 
\ra[B_\iy^{s_0/\vec r,\vec\om_0},B_\iy^{s_1/\vec r,\vec\om_1}]_\ta 
\npb 
\eeq 
is a~retraction, and $\vp_\iy^{\vec\om_0}$~is a~coretraction. 

\smallskip 
We use the notation of Sections \ref{sec-PH} and~\ref{sec-H}. Then, setting 
\hb{B_\iy^{s/\vec r}(\BM\times J,\mf{\rho}^{-\xi}\mf{E}) 
   :=\prod_\ka B_\iy^{s/\vec r}(\BM\times J,\rho_\ka^{-\xi}E)}, it is  
not difficult to verify (cf.~Lemma~\ref{lem-PH.fJ}) that  
\beq\Label{IP.f} 
\wt{\mf{f}}\in\Lis\bigl(B_\iy^{s/\vec r}\bigl(\BM\times J, 
\ell_\iy(\mf{\rho}^{-\xi}\mf{E})\bigr), 
\ell_\iy(\mf{\rho}^{-\xi}\mf{B}_\iy^{s/\vec r})\bigr) 
\eeq 
for 
\hb{s>0}. Hence we deduce from \Eqref{IP.linf} that the map 
\hb{\Psi_\iy^{\vec\om_0}=\psi_\iy^{\vec\om_0}\circ\wt{\mf{f}}} 
\beq\Label{IP.Bint} 
\bigl[B_\iy^{s/\vec r}\bigl(\BM\times J, 
\ell_\iy(\mf{E})\bigr),B_\iy^{s/\vec r}\bigl(\BM\times J, 
\ell_\iy(\mf{\rho}^{-\xi}\mf{B}_\iy^{s/\vec r})\bigr)\bigr]_\ta 
\ra[B_\iy^{s_0/\vec r,\vec\om_0},B_\iy^{s_1/\vec r,\vec\om_1}]_\ta  
\npb 
\eeq 
is a retraction, and $\Phi_\iy^{\vec\om_0}$~is a~coretraction. 

\smallskip 
(6) 
Set 
\hb{\gK_0:=\{\,\ka\in\gK\ ;\ \rho_\ka\leq1\,\}} and 
\hb{\gK_1:=\gK\ssm\gK_0}. Let $X_\ka$ be a Banach space for 
\hb{\ka\in\gK} and set 
\hb{\mf{X}:=\prod_\ka X_\ka} and 
\hb{\mf{X}_{\coX j}:=\prod_{\ka\in\gK_j}X_\ka} as well as 
\hb{\ell_\iy^j(\mf{X}):=\ell_\iy(\mf{X}_{\coX j})} for 
\hb{j=0,1}. Then 
\hb{\ell_\iy(\mf{X})\doteq\ell_\iy^0(\mf{X})\oplus\ell_\iy^1(\mf{X})}. 
Consequently, 
\beq\Label{IP.BS0} 
B_\iy^{s/\vec r}\bigl(\BM\times J,\ell_\iy(\mf{\rho}^{-\eta}\mf{E})\bigr) 
\doteq 
B_\iy^{s/\vec r}\bigl(\BM\times J,\ell_\iy^0(\mf{\rho}^{-\eta}\mf{E})\bigr) 
\oplus 
B_\iy^{s/\vec r}\bigl(\BM\times J,\ell_\iy^1(\mf{\rho}^{-\eta}\mf{E})\bigr) 
\npb 
\eeq 
for 
\hb{\eta\in\{0,\xi\}}. 

\smallskip 
(7) 
Put 
\hb{Y_0:=B_\iy^{s/\vec r}\bigl(\BM\times J, 
   \ell_\iy^0(\mf{\rho}^{-\xi}\mf{E})\bigr)} and 
\hb{Y_1:=B_\iy^{s/\vec r}\bigl(\BM\times J,\ell_\iy^0(\mf{E})\bigr)}. It 
follows from 
\hb{\rho_\ka\leq1} for 
\hb{\ka\in\gK_0} that 
\hb{Y_1\hr Y_0}. Define a linear operator~$A_0$ in~$Y_0$ with domain~$Y_1$ 
by 
\hb{A_0u=\mf{\rho}^{-\xi}u}. Then $A_0$~is closed, $-A_0$~contains the 
sector~$S_{\pi/4}$ in its resolvent set and satisfies 
\hb{\|(\lda+A_0)^{-1}\|_{\cL(Y_0)}\leq c/|\lda|} for 
\hb{\lda\in S_{\pi/4}}. Furthermore, 
$$ 
\|(-A_0)^z\|_{\cL(Y_0)}\leq\sup_{\ka\in\gK_0}\rho_\ka^{-\xi\Re z}\leq1 
\qa \Re z\leq0. 
$$ 
Hence Seeley's theorem, alluded to in the proof of Theorem~\ref{thm-PH.B}, 
implies 
\beq\Label{IP.BS1} 
\bal 
\big[B_\iy^{s/\vec r}\bigl(\BM\times J,\ell_\iy^0(\mf{E})\bigr), 
B_\iy^{s/\vec r}\bigl(\BM\times J,\ell_\iy^0(\mf{\rho}^{-\xi}\mf{E})\bigr) 
\big]_\ta 
&=[Y_0,Y_1]_{1-\ta}\\ 
&\doteq 
B_\iy^{s/\vec r}\bigl(\BM\times J, 
\ell_\iy^0(\mf{\rho}^{-\ta\xi}\mf{E})\bigr), 
\eal 
\eeq 
due to the fact that the space on the right side equals, except for 
equivalent norms, $\dom(A_0^{1-\ta})$~equipped with the graph norm. 

\smallskip 
(8) 
Set 
\hb{Z_0:=B_\iy^{s/\vec r}\bigl(\BM\times J,\ell_\iy^1(\mf{E})\bigr)} and 
\hb{Z_1:=B_\iy^{s/\vec r}\bigl(\BM\times 
   J,\ell_\iy^1(\mf{\rho}^{-\xi}\mf{E})\bigr)}. Then 
\hb{\rho_\ka>1} for 
\hb{\ka\in\gK_1} implies 
\hb{Z_1\hr Z_0}. Define a linear map~$A_1$ in~$Z_0$ with domain~$Z_1$ by 
\hb{A_1u:=\mf{\rho}^\xi u}. Then $A_1$~is closed and satisfies 
\hb{\|(\lda+A_1)^{-1}\|_{\cL(Z_0)}\leq c/|\lda|} for 
\hb{\lda\in S_{\pi/4}} as well as 
$$ 
\|(-A_1)^z\|_{\cL(Z_0)}\leq\sup_{\ka\in\gK_0}\rho_\ka^{\xi\Re z}\leq1 
\qa \Re z\leq0. 
$$ 
Thus, using Seeley's theorem once more, 
\beq\Label{IP.BS2} 
\big[B_\iy^{s/\vec r}\bigl(\BM\times J,\ell_\iy^1(\mf{E})\bigr), 
B_\iy^{s/\vec r}\bigl(\BM\times J,\ell_\iy^1(\mf{\rho}^{-\xi}\mf{E})\bigr) 
\big]_\ta 
\doteq 
B_\iy^{s/\vec r}\bigl(\BM\times J, 
\ell_\iy^1(\mf{\rho}^{-\ta\xi}\mf{E})\bigr).  
\eeq 
Now we deduce from \hbox{\Eqref{IP.BS0}--\Eqref{IP.BS2}} 
that 
$$ 
\big[B_\iy^{s/\vec r}\bigl(\BM\times J,\ell_\iy(\mf{E})\bigr), 
B_\iy^{s/\vec r}\bigl(\BM\times J,\ell_\iy(\mf{\rho}^{-\xi}\mf{E})\bigr) 
\big]_\ta 
\doteq B_\iy^{s/\vec r}\bigl(\BM\times J, 
\ell_\iy(\mf{\rho}^{-\ta\xi}\mf{E})\bigr). 
$$ 
Thus \Eqref{IP.Bint} shows that 
$$ 
\Psi_\iy^{\vec\om} 
\sco B_\iy^{s/\vec r}\bigl(\BM\times J, 
\ell_\iy(\mf{\rho}^{-\ta\xi}\mf{E})\bigr) 
\ra[B_\iy^{s_0/\vec r,\vec\om_0},B_\iy^{s_1/\vec r,\vec\om_1}]_\ta 
$$ 
is a retraction, and $\Phi_\iy^{\vec\om}$~is a coretraction. Hence 
\Eqref{H.phpsP} and \Eqref{IP.f} imply that 
$$ 
\psi_\iy^{\vec\om} 
\sco\ell_\iy(\mf{\rho}^{-\ta\xi}\mf{B}_\iy^{s/\vec r}) 
\ra[B_\iy^{s_0/\vec r},B_\iy^{s_1/\vec r}]_\ta 
$$ 
is a retraction, and $\vp_\iy^{\vec\om}$~is a coretraction. From this and 
the observation at the beginning of step~(5) we derive that the first part 
of the second statement is true.  

\smallskip 
(9) 
By replacing~$\ell_\iy$ in the preceding considerations 
by~$\ell_{\iy,\unif}$ and invoking Theorem~\ref{thm-H.Rbc} instead of 
Theorem~\ref{thm-H.R} we see that the second part of claim~(ii) is also 
true. 
\end{proof} 
For completeness and complementing the results of~\cite{Ama12b} we include 
the following interpolation theorem for isotropic Besov-H\"older spaces. 
\begin{remark}\LabelT{rem-IP.I} 
Suppose 
\hb{0<s_0<s_1}, 
\ \hb{\lda_0,\lda_1\in\BR}, and 
\hb{0<\ta<1}. Then 
$$ 
[B_\iy^{s_0,\lda_0},B_\iy^{s_1,\lda_1}]_\ta 
\doteq B_\iy^{s_\ta,\lda_\ta} 
\qb [b_\iy^{s_0,\lda_0},b_\iy^{s_1,\lda_1}]_\ta 
\doteq b_\iy^{s_\ta,\lda_\ta}. 
$$ 
\end{remark} 
\begin{proof} 
This follows from the above proof by relying on the corresponding 
isotropic results of Sections \ref{sec-PH} and~\ref{sec-H}. 
\end{proof}  
Throughout the rest of this section we suppose 
\beq\Label{IP.assB} 
\bal 
\bt\quad
&\Ga\neq\es.\\ 
\bt\quad
&\text{assumption \Eqref{B.ass} is satisfied}.\\ 
\bt\quad
&\bar s>r_k+1/p\text{ and }\mf{b}\in\gB_\norm^{\bar s}.\\ 
\bt\quad
&\cB=(\cB_0,\ldots,\cB_k):=\cB(\mf{b}). 
\eal
\eeq 
Let 
\hb{I\in\bigl\{J,(0,\iy)\bigr\}}. For 
\hb{-1+1/p<s\leq\bar s} with 
\hb{s\notin\{\,k_i+1/p\ ;\ 0\leq i\leq k\,\}} we set 
$$ 
\gF_{p,\cB}^{s/\vec r,\vec\om}(I) 
:=\bigl\{\,u\in\gF_p^{s/\vec r,\vec\om}(I) 
=\gF_p^{s/\vec r,\vec\om}\IV 
\ ;\ \cB_iu=0\text{ for }r_i<s-1/p\,\bigr\}. 
\npb 
$$ 
Thus 
\hb{\gF_{p,\cB}^{s/\vec r,\vec\om}(I) 
   =\gF_p^{s/\vec r,\vec\om}(I)} if 
\hb{s<k_0+1/p}.  

\smallskip 
Suppose $\mf{b}$~is independent of~$t$. Then we can define stationary 
isotropic spaces with vanishing boundary conditions analogously, that is, 
$$ 
\gF_{p,\cB}^{s,\lda}
:=\bigl\{\,u\in\gF_p^{s,\lda}(V) 
\ ;\ \cB_iu=0\text{ for }r_i<s-1/p\,\bigr\}. 
$$ 
\begin{theorem}\LabelT{thm-IP.B} 
Let\/ \Eqref{IP.assB} be satisfied. Suppose 
\hb{-1+1/p<s_0<s_1\leq\bar s} and\/ 
\hb{0<\ta<1} satisfy 
$$ 
s_0,s_1,s_\ta\notin\{\,r_i+1/p\ ;\ 0\leq i\leq k\,\} 
$$ 
and 
\hb{\lda_0,\lda_1\in\BR}. Then 
\beq\Label{IP.FF} 
\bigl(\gF_{p,\cB}^{s_0/\vec r,\vec\om_0}(I), 
\gF_{p,\cB}^{s_1/\vec r,\vec\om_1}(I)\bigr)_\ta 
\doteq\gF_{p,\cB}^{s_\ta/\vec r,\vec\om_\ta}(I)  
\doteq\bigl[\gF_{p,\cB}^{s_0/\vec r,\vec\om_0}(I), 
\gF_{p,\cB}^{s_1/\vec r,\vec\om_1}(I)\bigr]_\ta 
\eeq 
and 
$$ 
\bigl(H_{p,\cB}^{s_0/\vec r,\vec\om_0}(I), 
H_{p,\cB}^{s_1/\vec r,\vec\om_1}(I)\bigr)_{\ta,p} 
\doteq B_{p,\cB}^{s_\ta/\vec r,\vec\om_\ta}(I). 
$$ 
If\/ $\mf{b}$~is independent of 
\hb{t\in\BR}, then 
$$ 
(\gF_{p,\cB}^{s_0,\lda_0},\gF_{p,\cB}^{s_1,\lda_1})_\ta 
\doteq\gF_{p,\cB}^{s_\ta,\lda_\ta}
\doteq[\gF_{p,\cB}^{s_0,\lda_0},\gF_{p,\cB}^{s_1,\lda_1}]_\ta 
$$ 
and 
$$ 
(H_{p,\cB}^{s_0,\lda_0},H_{p,\cB}^{s_1,\lda_1})_{\ta,p} 
\doteq B_{p,\cB}^{s_\ta,\lda_\ta}. 
$$ 
\end{theorem} 
\begin{proof} 
Theorem~\ref{thm-B.R} guarantees the existence of a coretraction 
$$ 
\cB^c\in\cL\bigl(\pl_{\Ga\times I}\gF_p^{s/\vec r,\vec\om}(G), 
\gF_p^{s/\vec r,\vec\om}(I)\bigr) 
\qa r_k+1/p<s\leq\bar s. 
$$ 
Hence 
\hb{\cB^c\cB\in\cL\bigl(\gF_p^{s/\vec r,\vec\om}(I)\bigr)} is a 
projection. Note that $\cB^c\cB$~depends on~$\mf{b}$ and the universal 
extension operators~\Eqref{B.ljc} only. Thus we do not need to indicate 
the parameters $s$, $\lda$, and~$p$ with 
\hb{r_k+1/p<s\leq\bar s} which characterize the 
domain~$\gF_p^{s/\vec r,\vec\om}(I)$. 

\smallskip 
Taking this into account and using the notation of the proof of 
Theorem~\ref{thm-B.R} we set 
\hb{X_\ell:=\gF_p^{s_\ell/\vec r,\vec\om_\ell}(I)} for 
\hb{\ell\in\{0,1,\ta\}} and, putting 
\hb{r_{k+1}:=\iy}, 
\beq\Label{IP.P} 
P_\ell:=
\left\{
\bal
{}      
&\id_\ell,  
    &\quad         s_\ell   &<r_0+1/p,\\ 
&\id_\ell-\cB^{ic}\cB^i,   
    &\quad r_i+1/p<s_\ell   &<r_{i+1}+1/p. 
\eal
\right.
\eeq 
Since 
\hb{X_\ell\hr\cD\ciJcicD'} the sum space 
\hb{X_0+X_1} is well-defined, that is, $\XnXe$~is an interpolation couple. 
It follows from \hbox{\Eqref{B.B0}--\Eqref{B.B2}} 
that 
\hb{P_0\in\cL(X_0+X_1)} and 
\hb{P_\ell\in\cL(X_\ell)} with 
\hb{P_0\sn X_\ell=P_\ell} for 
\hb{\ell\in\{0,1,\ta\}}. 

\smallskip 
Theorem~\ref{thm-IP.I} guarantees 
\hb{\XnXe_\ta\doteq X_\ta}. 
Theorem~\ref{thm-B.R}, definition~\Eqref{IP.P}, and 
\cite[Lemma~I.2.3.1]{Ama95a} (also see \cite[Lemma~4.1.5]{Ama09a}) imply 
that $P_\ell$~is a projection onto 
\hb{X_{\ell,\cB}:=\gF_{p,\cB}^{s_\ell/\vec r,\vec\om_\ell}(I)} for 
\hb{\ell\in\{0,1,\ta\}}. Thus it is a retraction from~$X_\ell$ 
onto~$X_{\ell,\cB}$ possessing the natural injection 
\hb{X_{\ell,\cB}\hr X_\ell} as a coretraction. Consequently, $P_\ta$~is a 
retraction from 
\hb{X_\ta\doteq\XnXe_\ta} onto $(X_{0,\cB},X_{1,\cB})_\ta$. From this we 
get 
\hb{(X_{0,\cB},X_{1,\cB})_\ta\doteq X_{\ta,\cB}}. This proves the first 
equivalence of \Eqref{IP.FF}. The remaining 
statements for the anisotropic case follow analogously. 

\smallskip 
Due to the observation at the end of Section~\ref{sec-C} it is clear that 
the above proof applies to the isotropic case as well. 
\end{proof} 
There is a similar result concerning interpolations of spaces with vanishing 
initial conditions. For this we assume 
\beq\Label{IP.asst} 
\bal 
\bt\quad
&\ell,j_0,\ldots,j_\ell\in\BN\text{ with }j_0<j_1<\cdots<j_\ell.\\ 
\bt\quad
&\cC:=(\pl_{t=0}^{j_0},\ldots,\pl_{t=0}^{j_\ell}). 
\eal 
\eeq 
Then, given 
\hb{s>-1+1/p}, we put 
$$ 
\gF_{p,\cC}^{s/\vec r,\vec\om}\Rp
:=\bigl\{\,u\in\gF_p^{s/\vec r,\vec\om}\Rp 
\ ;\ \pl_{t=0}^{j_i}u=0\text{ if }r(j_i+1/p)<s\,\bigr\}. 
$$ 
\begin{theorem}\LabelT{thm-IP.It} 
Let\/ \Eqref{IP.asst} be satisfied. Suppose 
\hb{-1+1/p<s_0<s_1} and\/ 
\hb{\ta\in(0,1)} satisfy 
$$ 
s_0,s_1,s_\ta 
\notin\bigl\{\,r(j_i+1/p)\ ;\ 0\leq i\leq\ell\,\bigr\} 
$$ 
and 
\hb{\lda_0,\lda_1\in\BR}. Then 
$$ 
\bigl(\gF_{p,\cC}^{s_0/\vec r,\vec\om_0}\Rp, 
\gF_{p,\cC}^{s_1/\vec r,\vec\om_1}\Rp\bigr)_\ta 
\doteq\gF_{p,\cC}^{s_\ta/\vec r,\vec\om_\ta}\Rp 
\doteq\bigl[\gF_{p,\cC}^{s_0/\vec r,\vec\om_0}\Rp, 
\gF_{p,\cC}^{s_1/\vec r,\vec\om_1}\Rp\bigr]_\ta 
$$ 
and 
$$ 
\bigl(H_{p,\cC}^{s_0/\vec r,\vec\om_0}\Rp, 
H_{p,\cC}^{s_1/\vec r,\vec\om_1}\Rp\bigr)_{\ta,p} 
\doteq B_{p,\cC}^{s_\ta/\vec r,\vec\om_\ta}\Rp. 
$$ 
\end{theorem} 
\begin{proof} 
This is shown by the preceding proof using 
Theorem~\ref{thm-B.Rt} instead of Theorem~\ref{thm-B.R}. 
\end{proof} 
Now we suppose, in addition to \Eqref{IP.assB}, that 
\hb{\ell\in\BN} and 
\hb{\bar s>r(\ell+1/p)}. Then we set 
$$ 
\gF_{p,\vec\cB}^{s/\vec r,\vec\om} 
:=\bigl\{\,u\in\gF_p^{s/\vec r,\vec\om}\Rp 
\ ;\ \cB_iu=0, 
\ \pl_{t=0}^ju=0\text{ if }r_i+1/p+r(j+1/p)<s\,\bigr\} 
$$ 
if 
\hb{\max\{r_0+1/p,\ r/p\}<s\leq\bar s} and 
\hb{s\notin\bigl\{\,r_i+1/p+r(j+1/p) 
\ ;\ 0\leq i\leq k, 
\ 0\leq j\leq\ell\,\bigr\}},  
$$ 
\gF_{p,\vec\cB}^{s/\vec r,\vec\om}
:=\gF_{p,\cB}^{s/\vec r,\vec\om}\Rp  
$$ 
if 
\hb{r_0+1/p<s<r/p} and 
\hb{s\notin\{\,r_i+1/p\ ;\ 1\leq i\leq k\,\}},  
$$ 
\gF_{p,\vec\cB}^{s/\vec r,\vec\om} 
:=\gF_{p,\ora{\pl_{t=0}^\ell}}^{s/\vec r,\vec\om}\Rp  
$$ 
if 
\hb{r/p<s<r_0+1/p} with 
\hb{s\notin\bigl\{\,r(j+1/p)\ ;\ 1\leq j\leq\ell\,\bigr\}}, and 
$$ 
\gF_{p,\vec\cB}^{s/\vec r,\vec\om}
:=\gF_p^{s/\vec r,\vec\om}\Rp  
\npb 
$$ 
if 
\hb{-1+1/p<s<\max\{r_0+1/p,\ r/p\}}. 

\smallskip 
The following theorem is analogue to Theorem~\ref{thm-IP.B}. It 
describes the interpolation behavior of anisotropic function spaces with 
vanishing boundary and initial conditions. 
\begin{theorem}\LabelT{thm-IP.BI} 
Let assumption \Eqref{IP.assB} be satisfied. Also assume 
\hb{\ell\in\BN} and 
\hb{r(\ell+1/p)<\bar s}. Suppose 
\hb{-1+1/p<s_0<s_1\leq\bar s} and\/ 
\hb{0<\ta<1} satisfy 
$$ 
s_0,s_1,s_\ta\notin\{\,r_i+1/p+r(j+1/p), 
\ r_i+1/p, 
\ r(j+1/p) 
\ ;\ 0\leq i\leq k, 
\ 0\leq j\leq\ell\,\} 
$$ 
and 
\hb{\lda_0,\lda_1\in\BR}. Then 
$$ 
(\gF_{p,\vec\cB}^{s_0/\vec r,\vec\om_0}, 
\gF_{p,\vec\cB}^{s_1/\vec r,\vec\om_1})_\ta 
\doteq\gF_{p,\vec\cB}^{s_\ta/\vec r,\vec\om_\ta}  
\doteq[\gF_{p,\vec\cB}^{s_0/\vec r,\vec\om_0}, 
\gF_{p,\vec\cB}^{s_1/\vec r,\vec\om_1}]_\ta 
$$ 
and 
$$ 
(H_{p,\vec\cB}^{s_0/\vec r,\vec\om_0}, 
H_{p,\vec\cB}^{s_1/\vec r,\vec\om_1})_{\ta,p} 
\doteq B_{p,\vec\cB}^{s_\ta/\vec r,\vec\om_\ta}. 
$$ 
\end{theorem} 
\begin{proof} 
This follows by the arguments of the proof of Theorem~\ref{thm-IP.B} by 
invoking Theorem~\ref{thm-B.Rcc} and Remark~\ref{rem-B.Rcc}. 
\end{proof} 
The preceding interpolation theorems combined with the characterization 
statements of Section~\ref{sec-VT} lead to interpolation results for 
spaces with vanishing traces. For abbreviation, 
\hb{\gF_{p,\Ga}^{s/\vec r,\vec\om}(I) 
   =\gF_{p,\Ga}^{s/\vec r,\vec\om}\IV},~etc.\po  
{\samepage 
\begin{theorem}\LabelT{thm-IP.I0} 
Suppose 
\hb{-1+1/p<s_0<s_1<\iy}, 
\ \hb{0<\ta<1}, and 
\hb{\lda_0,\lda_1\in\BR}.\po 
\begin{itemize} 
\item[{\rm(i)}] 
If\/ 
\hb{s_0,s_1,s_\ta\notin\BN+1/p}, then 
$$ 
\bigl(\ci\gF_{p,\Ga}^{s_0/\vec r,\vec\om_0}(J), 
\ci\gF_{p,\Ga}^{s_1/\vec r,\vec\om_1}(J)\bigr)_\ta 
\doteq\ci\gF_{p,\Ga}^{s_\ta/\vec r,\vec\om_\ta}(J) 
\doteq\bigl[\ci\gF_{p,\Ga}^{s_0/\vec r,\vec\om_0}(J), 
\ci\gF_{p,\Ga}^{s_1/\vec r,\vec\om_1}(J)\bigr]_\ta 
\po 
$$ 
and 
$$ 
\bigl(\ci H_{p,\Ga}^{s_0/\vec r,\vec\om_0}(J), 
\ci H_{p,\Ga}^{s_1/\vec r,\vec\om_1}(J)\bigr)_{\ta,p} 
\doteq\ci B_{p,\Ga}^{s_\ta/\vec r,\vec\om_\ta}(J). 
\po 
$$ 
\item[{\rm(ii)}] 
Assume 
\hb{s_0,s_1,s_\ta\notin r(\BN+1/p)}. Then 
$$ 
\bigl(\gF_p^{s_0/\vec r,\vec\om_0}(0,\iy), 
\gF_p^{s_1/\vec r,\vec\om_1}(0,\iy)\bigr)_\ta 
\doteq\gF_p^{s_\ta/\vec r,\vec\om_\ta}(0,\iy) 
\doteq\bigl[\gF_p^{s_0/\vec r,\vec\om_0}(0,\iy), 
\gF_p^{s_1/\vec r,\vec\om_1}(0,\iy)\bigr]_\ta 
\po 
$$ 
and 
$$ 
\bigl(H_p^{s_0/\vec r,\vec\om_0}(0,\iy), 
H_p^{s_1/\vec r,\vec\om_1}(0,\iy)\bigr)_{\ta,p} 
\doteq B_p^{s_\ta/\vec r,\vec\om_\ta}(0,\iy).  
\po 
$$ 
\item[{\rm(iii)}] 
Suppose 
\hb{s_0,s_1,s_\ta\notin\BN+1/p} with 
\hb{s_0,s_1,s_\ta\notin r(\BN+1/p)} if 
\hb{J=\BR^+}. Then 
$$ 
(\ci\gF_p^{s_0/\vec r,\vec\om_0}, 
\ci\gF_p^{s_1/\vec r,\vec\om_1})_\ta 
\doteq\ci\gF_p^{s_\ta/\vec r,\vec\om_\ta} 
\doteq[\ci\gF_p^{s_0/\vec r,\vec\om_0}, 
\ci\gF_p^{s_1/\vec r,\vec\om_1}]_\ta 
\po 
$$ 
and 
$$ 
(\ci H_p^{s_0/\vec r,\vec\om_0}, 
\ci H_p^{s_1/\vec r,\vec\om_1})_{\ta,p} 
\doteq\ci B_p^{s_\ta/\vec r,\vec\om_\ta}.  
\po 
$$ 
\end{itemize} 
\end{theorem} }
\begin{proof} 
To prove~(i) we can assume 
\hb{s_1>1/p}, due to Theorem~\ref{thm-A.HBH}(ii). Hence 
\hb{k:=[s_1-1/p]_-\geq0}. Set 
\hb{\cB:=(\pl_{\mf{n}}^0,\ldots,\pl_{\mf{n}}^k)} on 
\hb{\Ga\times J}. Then Theorem~\ref{thm-VT.V}(i) guarantees 
\hb{\ci\gF_{p,\Ga}^{s_j/\vec r,\vec\om_j}(J) 
   =\gF_{p,\cB}^{s_j/\vec r,\vec\om_j}(J)} for
\hb{j\in\{0,1,\ta\}}. Hence assertion~(i) is a consequence of 
Theorem~\ref{thm-IP.I}. The proofs for claims (ii) and~(iii) follow 
analogous lines. 
\end{proof} 
Since, in \Eqref{A.def1}, the negative order spaces have been defined by 
duality we can now prove interpolation theorems for these spaces as well. 
\begin{theorem}\LabelT{thm-IP.In} 
Suppose 
\hb{-\iy<s_0<s_1<1/p}, 
\ \hb{0<\ta<1}, and 
\hb{\lda_0,\lda_1\in\BR}. Assume 
\hb{s_0,s_1,s_\ta\notin-\BN+1/p} and, if
\hb{J=\BR^+}, also 
\hb{s_0,s_1,s_\ta\notin r(-\BN+1/p)}. Then 
$$ 
(\gF_p^{s_0/\vec r,\vec\om_0}, 
\gF_p^{s_1/\vec r,\vec\om_1})_\ta 
\doteq\gF_p^{s_\ta/\vec r,\vec\om_\ta} 
\doteq[\gF_p^{s_0/\vec r,\vec\om_0}, 
\gF_p^{s_1/\vec r,\vec\om_1}]_\ta 
$$ 
and 
$$ 
(H_p^{s_0/\vec r,\vec\om_0}, 
H_p^{s_1/\vec r,\vec\om_1})_{\ta,p} 
\doteq B_p^{s_\ta/\vec r,\vec\om_\ta}.  
$$ 
\end{theorem} 
\begin{proof} 
This follows easily from Theorem~\ref{thm-IP.I0}(iii), the duality 
properties of~%
\hb{\pr_\ta}, and Theorem~\ref{thm-A.HBH}(ii) and 
Corollary~\ref{cor-A.FF}(ii). 
\end{proof} 
Suppose $M$~is an \hbox{$m$-dimensional} compact submanifold of~$\BR^m$ 
with boundary and 
\hb{W=M\times\BC^n}. In this situation it has been shown by 
R.~Seeley~\cite{See71a} that 
\beq\Label{IP.LH} 
[L_p,H_{p,\cB}^s]_\ta=H_{p,\cB}^{\ta s} \qa s>0, 
\eeq 
with $\cB$ a~normal system of boundary operators (with smooth coefficients). 
This generalizes the earlier result by P.~Grisvard~\cite{Gri67a} who 
obtained \Eqref{IP.LH} in the case 
\hb{p=2} and 
\hb{n=1}. The latter author proved in~\cite{Gri69a} that 
\hb{(L_p,W_{\coW p,\cB}^k)_{\ta,p}\doteq B_{p,\cB}^{\ta k}} and 
\hb{(L_p,B_{p,\cB}^s)_{\ta,p}\doteq B_{p,\cB}^{\ta s}} for 
\hb{k\in\BN^\times} and 
\hb{s>0}. An extension of these results to arbitrary Banach spaces is due 
to D.~Guidetti~\cite{Guie91b}. In each of those papers the `singular values' 
\hb{\BN+1/p} are considered also. (If 
\hb{s\in\BN+1/p}, then $H_{p,\cB}^s$ and~$B_{p,\cB}^s$ are no longer closed 
subspaces of $H_p^s$ and~$B_p^s$, respectively.) 

\smallskip 
Following the ideas of R.~Seeley and D.~Guidetti we have given 
in~\cite[Theorem~4.9.1]{Ama09a} a~proof of the anisotropic part of 
Theorem~\ref{thm-IP.B} in the special case where 
\hb{M=\BH^m} and 
\hb{J=\BR}, respectively 
\hb{M=\BR^m} and 
\hb{J=\BR^+} (to remain in the setting of this paper), 
\hb{W=M\times\BC^n}, and $B$~has constant coefficients. The proof given 
here, which is solely based on Theorem~\ref{thm-B.R} and general properties 
of interpolation functors, is new even in this simple Euclidean setting. 
\section{Bounded Cylinders}\LabelT{sec-BC}%
So far we have developed the theory of weighted anisotropic function spaces 
on full and half-cylinders, making use of the dilation invariance of~$J$. 
In this final section we now show that all preceding results not explicitly 
depending on this dilation invariance remain valid in the case of cylinders 
of finite height.   

\smallskip 
Throughout this section 
$$ 
\bt\quad 
J=\BR^+ 
\qb 0<T<\iy 
\qb J_T:=[0,T]. 
\npb 
$$ 
Furthermore, 
\hb{\gF_p^{s/\vec r,\vec\om}(J)=\gF_p^{s/\vec r,\vec\om}\JV} etc. 

\smallskip 
For 
\hb{k\in\BN} we introduce $W_{\coW p}^{kr/\vec r}(J_T)$ by replacing~$J$ 
in definition~\Eqref{A.W} by~$J_T$. Then $\gF_p^{s/\vec r,\vec\om}(J_T)$ 
is defined for 
\hb{s>0} analogously to \Eqref{A.def}. Similarly as in \Eqref{VT.F0} 
$$ 
\ci\gF_{p,\Ga}^{s/\vec r,\vec\om}(J_T) 
\text{ is the closure of $\cD\bigl((0,T],\cD(M\ssm\Ga,V)\bigr)$
in $\gF_p^{s/\vec r,\vec\om}(J_T)$ for }s>0. 
$$ 
Moreover, 
$$ 
\ci\gF_p^{s/\vec r,\vec\om}(J_T) 
:=\ci\gF_{p,\pl M}^{s/\vec r,\vec\om}(J_T) 
\qb \gF_p^{s/\vec r,\vec\om}(0,T] 
:=\ci\gF_{p,\es}^{s/\vec r,\vec\om}(J_T). 
$$ 
Note that we do not require that 
\hb{u\in\ci\gF_{p,\Ga}^{s/\vec r,\vec\om}(J_T)} approaches zero near~$T$. 
To take care of this situation also we define: 
$$ 
\ci\gF_p^{s/\vec r,\vec\om}(0,T) 
\text{ is the closure of $\cD\bigl((0,T),\ci\cD\bigr)$
in }\gF_p^{s/\vec r,\vec\om}(J_T). 
$$ 
Then 
$$ 
\gF_p^{-s/\vec r,\vec\om}(J_T) 
:=\bigl(\ci\gF_{p'}^{s/\vec r,\vec\om}\bigl((0,T),V'\bigr)\bigr)' 
\qa s>0, 
$$ 
and 
$$ 
H_p^{0/\vec r,\vec\om}(J_T):=L_p(J_T,L_p^\lda) 
\qb B_p^{0/\vec r,\vec\om}(J_T) 
:=\bigl(B_p^{-s(p)/\vec r,\vec\om}(J_T),B_p^{s(p)/\vec r,\vec\om}(J_T) 
\bigr)_{1/2,p}. 
$$ 
This defines the weighted anisotropic Bessel potential space scale 
\hb{\bigl[\,H_p^{s/\vec r,\vec\om}(J_T)\ ;\ s\in\BR\,\bigr]} and 
Besov space scale 
\hb{\bigl[\,B_p^{s/\vec r,\vec\om}(J_T)\ ;\ s\in\BR\,\bigr]} on~$J_T$. 

\smallskip 
As for H\"older space scales, $BC^{kr/\vec r,\vec\om}(J_T)$ is obtained 
by replacing~$J$ in \Eqref{H.krN0} and \Eqref{H.krN1} by~$J_T$. Then 
$bc^{kr/\vec r,\vec\om}(J_T)$ is the closure of 
$$ 
BC^{\iy/\vec r,\vec\om}(J_T) 
:=\bigcap_{i\in\BN}BC^{ir/\vec r,\vec\om}(J_T)
$$ 
in $BC^{kr/\vec r,\vec\om}(J_T)$. Besov-H\"older spaces are defined for 
\hb{s>0} by 
\beq\Label{BC.B} 
B_\iy^{s/\vec r,\vec\om}(J_T) 
:=\left\{
\bal
{}      
&\bigl(bc^{kr/\vec r,\vec\om}(J_T), 
 bc^{(k+1)r/\vec r,\vec\om}(J_T)\bigr)_{(s-k)/r,\iy}, 
    &\quad  kr<s &<(k+1)r,\\ 
&\bigl(bc^{kr/\vec r,\vec\om}(J_T), 
 bc^{(k+2)r/\vec r,\vec\om}(J_T)\bigr)_{1/2,\iy}, 
    &\quad     s &=(k+1)r. 
\eal 
\right. 
\eeq 
Moreover, $b_\iy^{s/\vec r,\vec\om}(J_T)$ is the closure of 
$BC^{\iy/\vec r,\vec\om}(J_T)$ in~$B_\iy^{s/\vec r,\vec\om}(J_T)$. 
Lastly, 
$b_\iy^{s/\vec r,\vec\om}(0,T]$ is obtained by substituting $(0,T]$ for 
$(0,\iy)$ in \Eqref{ER.bc} and \Eqref{ER.bc1}. 

\smallskip 
Given a locally convex space~$\cX$, the continuous linear map 
$$ 
r_T\sco C\JcX\ra C\JTcX 
\qb u\mt u\sn J_T 
$$ 
is the \emph{point-wise restriction to}~$J_T$. As usual, we use the same 
symbol for~$r_T$ and any of its restrictions or (unique) continuous 
extensions. 
\begin{theorem}\LabelT{thm-BC.R} 
Let one of the following conditions be satisfied: 
$$ 
\bal
\text{\rm($\al$)}  \qquad  &s\in\BR 
    &\text{ and\/ }
    &\gG=\gF_p^{s/\vec r,\vec\om}\text{;}\\
\text{\rm($\ba$)}  \qquad  &s>0 
    &\text{ and\/ }
    &\gG\in\{B_\iy^{s/\vec r,\vec\om},b_\iy^{s/\vec r,\vec\om}\}\text{;}\\
\text{\rm($\ga$)}  \qquad  &k\in\BN 
    &\text{ and\/ }
    &\gG\in\{BC^{kr/\vec r,\vec\om},bc^{kr/\vec r,\vec\om}\}\text{;}\\
\text{\rm($\da$)}  \qquad  &s>0 
    &\text{ and\/ }
    &\gG=\ci\gF_{p,\Ga}^{s/\vec r,\vec\om}\text{.}
\eal
$$ 
Then $r_T$~is a retraction from~$\gG(J)$ onto~$\gG(J_T)$ possessing a 
universal coretraction~$e_T$. It is also a retraction from $\gG(0,\iy)$ 
onto $\gG(0,T]$ with coretraction~$e_T$ if either 
\hb{s>0} and\/ 
\hb{\gG=b_\iy^{s/\vec r,\vec\om}} or 
\hb{k\in\BN} and\/ 
\hb{\gG=bc^{kr/\vec r,\vec\om}}. 
\end{theorem} 
\begin{proof} 
(1) 
Suppose 
\hb{k\in\BN}. It is obvious that 
$$ 
r^+\in\cL\bigl(W_{\coW p}^{kr/\vec r,\vec\om}(J), 
W_{\coW p}^{kr/\vec r,\vec\om}(J_T)\bigr). 
$$ 
Thus we get 
\hb{r^+\in\cL\bigl(\gG(J),\gG(J_T)\bigr)} if $(\al)$~is satisfied with 
\hb{s>0} by interpolation, due to the definition of 
$\gF_p^{s/\vec r,\vec\om}(I)$ for 
\hb{I\in\{J,J_T\}}. 

\smallskip 
(2) 
It is also clear that 
\hb{r^+\in\cL\bigl(BC^{kr/\vec r,\vec\om}(J), 
   BC^{kr/\vec r,\vec\om}(J_T)\bigr)} for 
\hb{k\in\BN}. Hence 
\beq\Label{BC.BC} 
r^+\in\cL\bigl(BC^{\iy/\vec r,\vec\om}(J), 
BC^{\iy/\vec r,\vec\om}(J_T)\bigr). 
\eeq 
From this we obtain 
\beq\Label{BC.G} 
r^+\in\cL\bigl(\gG(J),\gG(J_T)\bigr)
\eeq 
if either $(\ba)$ or $(\ga)$~is satisfied. In fact, this is obvious from 
\Eqref{BC.BC} if $(\ga)$~applies. If 
\hb{s>0} and 
\hb{\gG=B_\iy^{s/\vec r,\vec\om}}, then \Eqref{BC.G} is obtained by 
interpolation on account of \Eqref{H.Bdef2}, Corollary~\ref{cor-H.R}(ii), 
and \Eqref{BC.B}. From this and \Eqref{BC.BC} it follows that \Eqref{BC.G} 
is valid if 
\hb{\gG=b_\iy^{s/\vec r,\vec\om}}, due to \Eqref{H.bdef} and the definition 
of~$bc^{s/\vec r,\vec\om}(J_T)$. 

\smallskip 
Clearly, $r^+$~maps 
\hb{\cD\bigl(\ci J,\cD(M\ssm\Ga,V)\bigr)} into 
\hb{\cD\bigl((0,T],\cD(M\ssm\Ga,V)\bigr)}. From this and step~(1) we infer 
that \Eqref{BC.G} is true if $(\da)$~applies. It is equally clear that 
\hb{r^+\in\cL\bigl(\gG(0,\iy),\gG(0,T)\bigr)} if either 
\hb{s>0} and 
\hb{\gG=b_\iy^{s/\vec r,\vec\om}} or 
\hb{k\in\BN} and 
\hb{\gG=bc^{kr/\vec r,\vec\om}}. 

\smallskip 
(3) 
We set 
\hb{\da_T(t):=t+T} for 
\hb{t\in\BR}. We fix 
\hb{\al\in\cD\bigl((-T,0],\BR\bigr)} satisfying 
\hb{\al(t)=1} for 
\hb{-T/2\leq t\leq0} and put 
\hb{\ba u(t):=\al\da_T^*u(t)} for 
\hb{t\leq0} and 
\hb{u\sco J_T\ra C(V)}. It follows that 
\hb{\ba\in\cL\bigl(\gG(J_T),\gG(-\BR^+)\bigr)}, provided 
\hb{s>0} if $(\al)$~holds. Indeed, this is easily verified if $\gG$~is one 
of the spaces $W_{\coW p}^{kr/\vec r,\vec\om}$ and~$BC^{kr/\vec r,\vec\om}$. 
From this we get the claim by interpolation, similarly as in steps (1) 
and~(2). 

\smallskip 
(4) 
We recall from Section~\ref{sec-ER} the definition of the extension 
operator~$e^-$ associated with the point-wise restriction~$r^-$ 
to~$-\BR^+$. Then we define a linear map 
$$ 
\ve_T\sco C\bigl(J_T,C(V)\bigr) 
\ra C\bigl([T,\iy),C(V)\bigr)
\qb u\mt\da_{-T}^*(e^-\ba u). 
$$ 
Finally, we put 
\hb{e_Tu(t):=u(t)} for 
\hb{t\in J_T} and 
\hb{e_Tu(t):=\ve_Tu(t)} for 
\hb{T<t<\iy}. It follows from step~(3) and Theorems \ref{thm-ER.ER} 
and~\ref{thm-ER.B} that 
\beq\Label{BC.e1} 
e_T\in\cL\bigl(\gG(J_T),\gG(J)\bigr) 
\npb 
\eeq 
if one of conditions \hbox{$(\al)$--$(\ga)$} is satisfied, provided 
\hb{s>0} if $(\al)$~applies. 

\smallskip 
Since $\al$~is compactly supported it follows from \Eqref{ER.h} that 
$\ve_Tu$~is smooth and rapidly decreasing if $u$~is smooth. By the density 
of  
\hb{\cD\bigl([T,\iy),\cD(M\ssm\Ga,V)\bigr)} in the Schwartz 
space of smooth rapidly decreasing 
\hb{\cD(M\ssm\Ga,V)}-valued functions on~$[T,\iy)$ we get 
\hb{e_Tu\in\ci\gF_{p,\Ga}^{s/\vec r,\vec\om}(J)} if 
\hb{u\in\ci\gF_{p,\Ga}^{s/\vec r,\vec\om}(J_T)}. From this we see that 
\Eqref{BC.e1} holds if $(\da)$~is satisfied. It is obvious that 
\hb{r_Te_T=\id}. Thus the assertion is proved, provided 
\hb{s>0} if $(\al)$~is satisfied. 

\smallskip 
(5) 
As in \Eqref{ER.r0} we introduce the trivial extension map 
\hb{e_0^-\sco C_{(0)}\mBRpcX\ra C\BRcX} by  
\hb{e_0^-u(t):=u(t)} if 
\hb{t\leq0}, and 
\hb{e_0^-u(t):=0} if 
\hb{t>0}. Then 
$$ 
r_0^-:=r^-(1-e^+r^+)\sco C\BRcX\ra C_0\mBRpcX 
$$ 
is a retraction possessing~$e_0^-$ as coretraction. We also set 
\hb{r_{0,T}:=\da_{-T}^*r_0^-\da_T^*e_0^+}. Then  
\beq\Label{BC.D} 
r_{0,T}\bigl(\cD\bigl((0,T),\ci\cD\bigr)\bigr) 
\is\cD\bigl((0,\iy),\ci\cD\bigr). 
\eeq 
The mapping properties of $e^+$ and~$r^+$ described in Theorems 
\ref{thm-ER.ER} and~\ref{thm-ER.B}, and the analogous ones for~$r^-$, 
imply, similarly as above, that 
$$ 
r_{0,T}\in\cL\bigl(\gF_p^{s/\vec r,\vec\om}(J), 
\gF_p^{s/\vec r,\vec\om}(J_T)\bigr) 
\qa s>0. 
$$ 
Consequently, we get from \Eqref{BC.D} 
$$ 
r_{0,T}\in\cL\bigl(\ci\gF_p^{s/\vec r,\vec\om}(J), 
\ci\gF_p^{s/\vec r,\vec\om}(0,T)\bigr) 
\qa s>0. 
$$ 
We define 
\hb{e_{0,T}\sco\cD\bigl((0,T),\ci\cD\bigr)\ra 
   \cD\bigl((0,\iy),\ci\cD\bigr)} by 
\hb{e_{0,T}u\sn J_T:=u} and 
\hb{e_{0,T}u\sn[0,\iy):=0}. Then $e_{0,T}$~extends to a continuous linear map 
from $\ci\gF_p^{s/\vec r,\vec\om}(0,T)$ into 
$\ci\gF_p^{s/\vec r,\vec\om}(J)$ for 
\hb{s>0}, the trivial extension. Moreover, 
\hb{r_{0,T}e_{0,T}=\id}. Thus $r_{0,T}$~is a retraction 
possessing~$e_{0,T}$ as coretraction. 

\smallskip 
(6) 
Let 
\hb{s>0}. For 
\hb{u\in\cD\JcD} and 
\hb{\vp\in\cD\bigl((0,T),\cD\ciMVs\bigr)} we get 
$$ 
\int_0^T\int_M\dl\vp,r_Tu\dr_V\,dV_{\coV g}\,dt 
=\int_0^\iy\int_M\dl e_{0,T}\vp,u\dr_V\,dV_{\coV g}\,dt. 
$$ 
Hence, by step~(5) and the definition of the negative order spaces, 
$$ 
|\dl\vp,r_Tu\dr_{M\times J}| 
\leq c\,\|\vp\|_{\gF_{p'}^{s/\vec r,\vec\om}(\ci J_T,V')} 
\,\|u\|_{\gF_p^{-s/\vec r,\vec\om}(J)} 
$$ 
for 
\hb{u\in\gF_p^{-s/\vec r,\vec\om}(J)} and 
\hb{\vp\in\ci\gF_{p'}^{s/\vec r,\vec\om}\bigl((0,T),V'\bigr)}. Thus 
$$ 
r_T\in\cL\bigl(\gF_p^{-s/\vec r,\vec\om}(J), 
\gF_p^{-s/\vec r,\vec\om}(J_T)\bigr). 
$$ 

\smallskip 
(7) 
For 
\hb{v\in C\mJcD} we set 
$$ 
\ve^-v(t):=\int_0^\iy h(s)v(-st)\,ds 
\qa t\geq0. 
$$ 
Then, given 
\hb{\vp\in\cD\bigl((0,\iy),\cD(\ci M,V')\bigr)}, we obtain from 
\hb{h(1/s)=-sh(s)} for 
\hb{s>0} and \Eqref{ER.eps} 
\beq\Label{BC.e2} 
\bal 
\int_0^\iy\bigl\dl\vp(t),\ve^-v(t)\bigr\dr_M\,dt 
&=\int_0^\iy\int_0^\iy\bigl\dl\vp(t),h(s)v(-st)\bigr\dr_M\,ds\,dt\\ 
&=\int_{-\iy}^0\int_0^\iy 
 \bigl\dl s^{-1}\vp(-\tau/s)h(s),v(\tau)\bigr\dr_M\,ds\,d\tau\\ 
&=\int_{-\iy}^0\int_0^\iy\sa^{-1} 
 \bigl\dl\vp(-\tau\sa)h(1/\sa),v(\tau)\bigr\dr_M\,d\sa\,d\tau\\ 
&=-\int_{-\iy}^0\Bigl\dl\int_0^\iy 
 h(\sa)\vp(-\sa\tau)\,d\sa,v(\tau)\Bigr\dr_M\,d\tau 
 =-\int_{-\iy}^0\dl\ve\vp,v\dr_M\,d\tau. 
\eal 
\eeq 
Thus, by the definition of~$e_T$, given  
\hb{u\in\cD\JTcD}, 
$$ 
\int_0^\iy\dl\vp,e_Tu\dr_M\,dt 
=\int_0^T\dl\vp,u\dr_M\,dt
 +\int_T^\iy\dl\vp,\da_{-T}^*e^-\al\da_T^*u\dr_M\,dt. 
$$ 
The last integral equals, due to 
\hb{e^-w(t)=\ve^-w(t)} for 
\hb{t\geq0} and \Eqref{BC.e2}, 
$$ 
\bal 
\int_0^\iy\dl\da_T^*\vp,\ve^-\al\da_T^*u\dr_M\,dt 
&=-\int_{-\iy}^0\dl\ve\da_T^*\vp,\al\da_T^*u\dr_M\,dt\\ 
&=-\int_0^T\int_0^\iy 
 \bigl\dl h(\sa)\vp\bigl(-\sa(s-T)+T\bigr)\,d\sa, 
 \al(s-T)u(s)\bigr\dr_M\,ds 
\eal 
$$ 
since $\al$~is supported in~$(-T,0]$. From this we infer as in steps (3) 
and~(4) that, given 
\hb{s>0}, 
$$ 
|\dl\vp,e_Tu\dr_{M\times J}| 
\leq c\,\|\vp\|_{\gF_{p'}^{s/\vec r,\vec\om}(J,V')} 
\,\|u\|_{\gF_p^{-s/\vec r,\vec\om}(J_T)} 
$$ 
for 
\hb{\vp\in\cD\bigl((0,\iy),\cD(\ci M,V')\bigr)} and 
\hb{u\in\cD\JTcD}. Thus 
$$ 
e_T\in\cL\bigl(\gF_p^{-s/\vec r,\vec\om}(J_T), 
\gF_{p'}^{-s/\vec r,\vec\om}(J)\bigr) 
\npb 
$$ 
for 
\hb{s>0}. This and step~(6) imply that the assertion holds if $(\al)$~is 
satisfied with 
\hb{s<0}. 

\smallskip 
The case 
\hb{s=0} and 
\hb{\gF=H} is covered by step~(1). If 
\hb{s=0} and 
\hb{\gF=B}, we now obtain the claim by interpolation, due to the 
definition of $B_p^{0/\vec r,\vec\om}(I)$ for 
\hb{I\in\{J,J_T\}}. 
\end{proof} 
\begin{corollary}\LabelT{cor-BC.R} 
Suppose 
\hb{s>0}. There exists a universal retraction~$r_{0,T}$ from 
$\ci\gF_p^{s/\vec r,\vec\om}(J)$ onto 
$\ci\gF_p^{s/\vec r,\vec\om}(\ci J_T)$ such that the trivial extension 
is a coretraction for it. 
\end{corollary} 
\begin{proof} 
This has been shown in step~(5). 
\end{proof} 
As a consequence of this retraction theorem we find that, modulo obvious 
adaptions, everything proved in the preceding sections remains valid for 
cylinders of finite height. 
\begin{theorem}\LabelT{thm-BC.T} 
All embedding, interpolation, trace, and point-wise contraction multiplier 
theorems, as well as the theorems involving boundary conditions, remain 
valid if $J$~is replaced by~$J_T$. Furthermore, all retraction theorems 
for the anisotropic spaces stay in force, provided $\vp_q^{\vec\om}$ 
and~$\psi_q^{\vec\om}$ are replaced by 
\hb{\vp_q^{\vec\om}\circ e_T} and 
\hb{r_T\circ\psi_q^{\vec\om}}, respectively. 
\end{theorem} 
\begin{proof} 
This is an immediate consequence of Theorem~\ref{thm-BC.R} and the fact 
that all contraction multiplication and boundary operators are local ones. 
\end{proof}


\def\cprime{$'$} \def\polhk#1{\setbox0=\hbox{#1}{\ooalign{\hidewidth
  \lower1.5ex\hbox{`}\hidewidth\crcr\unhbox0}}}

\end{document}